\documentclass[12pt]{book}

\usepackage[utf8]{inputenc}  

\usepackage{hyperref} 
\usepackage{amssymb, amsmath, amsfonts,amsthm,color,nicefrac,geometry}
\usepackage{mathtools}
\DeclarePairedDelimiter\ceil{\lceil}{\rceil}

\DeclarePairedDelimiter\floorgrid{\llbracket}{\rrbracket}
\usepackage{stmaryrd} 
\usepackage{url}
\usepackage{cleveref}
\usepackage{enumerate}

\usepackage[toc,page]{appendix} 
\usepackage{bbm}
\usepackage{fancyhdr}
\usepackage[nocompress]{cite}

\usepackage{xcolor}
\hypersetup{
	colorlinks,
	linkcolor={red!60!black},
	citecolor={green!60!black},
	urlcolor={blue!60!black}
}

\renewcommand{\epsilon}{\varepsilon}
\newcommand{\R}{\mathbb{R}}

\newcommand{\N}{\mathbb{N}}
\newcommand{\Z}{\mathbb{Z}}
\newcommand{\Q}{\mathbb{Q}}
\newcommand{\F}{\mathcal{F}}
\newcommand{\E}{\mathbb{E}}
\renewcommand{\P}{\mathbb{P}}
\newcommand{\lb}{\lbrace}
\newcommand{\rb}{\rbrace}
\newcommand{\la}{\left|}
\newcommand{\ra}{\right|}

\newcommand{\norm}[1]{\left\lVert#1\right\rVert} 
\newcommand{\pt}{\tfrac{\partial}{\partial t}}
\newcommand{\pv}{\tfrac{\partial}{\partial \vartheta}}

\newcommand\numberthis{\addtocounter{equation}{1}\tag{\theequation}}

\newtheorem{theorem}{Theorem}
\newtheorem{lemma}[theorem]{Lemma}

\newtheorem{setting}[theorem]{Setting}
\newtheorem{cor}[theorem]{Corollary}

\newtheorem{prop}[theorem]{Proposition}

\numberwithin{theorem}{section}

\title{Weak error analysis for stochastic gradient\\ descent optimization algorithms}

\author{Aritz Bercher$^1$, Lukas Gonon$^{2,3}$, \\
	 Arnulf Jentzen$^{4,5}$, and Diyora Salimova$^{6,7}$
	\bigskip
	\\
	\small{$^1$ Department of Mathematics, ETH Zurich,}\\
	\small{Switzerland, e-mail:  
		abercher@outlook.com}
	\smallskip
	\\
	\small{$^2$ Faculty of Mathematics and Statistics, University of}
	\\
	\small{St.~Gallen, Switzerland, e-mail: lukas.gonon@unisg.ch}
		\smallskip
	\\
	\small{$^3$ Department of Mathematics, ETH Zurich,}\\
	\small{Switzerland, e-mail:  lukas.gonon@math.ethz.ch}
	\smallskip
	\\
	\small{$^4$ Department of Mathematics, ETH Zurich,}\\
	\small{Switzerland, e-mail:  arnulf.jentzen@sam.math.ethz.ch}
	\smallskip
	\\
	\small{$^5$ Faculty of Mathematics and Computer Science, University of}\\
	\small{M\"unster, Germany, e-mail: ajentzen@uni-muenster.de}
	\smallskip
	\\
	\small{$^6$ Department of Mathematics, ETH Zurich,}\\
	\small{Switzerland, e-mail:  diyora.salimova@sam.math.ethz.ch}
	    \\
	\small{$^7$ Department of Information Technology and Electrical Engineering,}\\
	\small{ETH Zurich, Switzerland, e-mail:  sdiyora@mins.ee.ethz.ch}}

\date{\today}

\begin{document}

\maketitle

\chapter*{Abstract}

Stochastic gradient descent (SGD) type optimization schemes are fundamental ingredients in a large number of machine learning  based algorithms. In particular, SGD type optimization schemes are frequently employed in applications involving natural language processing, 
	object and face recognition, 
	fraud detection, computational advertisement, 
	and numerical approximations of partial differential equations. In mathematical convergence results for SGD type optimization schemes there are usually two types of error criteria studied in the scientific literature, that is,  the error in the strong sense  and  the error with respect to the objective function. In applications  one is  often not only  interested   in the size of the error with respect to 
	 the objective function but also in the size of the error with respect to  a test function which is possibly different from the objective function. The analysis of the size of this error
	is  the subject of this article. In particular, the main result of this article proves  under suitable assumptions that the size of this error decays at the same speed as in the special case where the test function coincides with the objective function.

\tableofcontents

\chapter{Introduction}

Stochastic gradient descent (SGD) type optimization schemes are fundamental ingredients in a large number of machine learning  based algorithms. In particular, SGD type optimization schemes are frequently employed in applications involving
natural language processing (cf., e.g., \cite{DahlEtAl2012,GravesMohamedHinton13,HintonETAL12,HuLuLiChen2014,KalchbrennerEtAl2014,WuEtAl2016}), 
 object and face recognition  (cf., e.g., \cite{HuangEtAl2017,KrizhevskySutskeverHinton12,SimonyanZisserman2014,TaigmanEtAl2014,WangEtAl2015}), 
fraud detection
(cf., e.g., \cite{ChouiekhHaj2018,RoyEtAl2018}), computational advertisement
(cf., e.g., \cite{WangEtAl2017,ZhaiEtAl2016}), 
price formation (cf., e.g., \cite{SirignanoCont2019}),
portfolio hedging (cf., e.g., \cite{BuehlerGanonTeichmann2019}),
 financial model calibration (cf., e.g., \cite{LiuBorovykhEtAl2019,BayerHorvathEthAl2019}),
and numerical approximations of partial differential equations (PDEs) 
(cf., e.g., \cite{BeckEtAL2018,BeckJentzenE2019, EHanJentzen2017,EYu2018, HanJentzenE2018,Henry2017,Mishra2018, NabianMeidani2018,SirignanoSpiliopoulos2017}).
In view of the success of the SGD type optimization schemes in the above sketched applications, SGD type optimization schemes  have also been intensively studied in the scientific literature. 
In particular,
we refer, e.g., to \cite{BercuFort13,BottouCurtisNocedal2018,Ruder16} for overview articles on SGD type optimization schemes,
we refer, e.g., to \cite{BengioLewandowskiPascanu2013,BordesBottouGallinari09,DauphinDeVriesBengio15,DauphinEtAl2014,DefossezBach17,Dozat16,DuchiHazanSinger11,KingmaBa14,LangfordLiZhang09,LeRouxSchmidtBach12,McmahanStreeter14,NiuRechtChristopherWright11,Polyak98,PolyakJuditsky92,PolyakTsypkin80,Schraudolph99,SchraudolphYuGunter07,Shalev-ShwartzShingerSrebroCotter11,Sohl-DicksteinPooleGanguli14,Zeiler12,ZhangChoromanskaLeCun15,LiTaiE2017,ZarembaSutskever2014,LovasEtAl2020} and the references mentioned therein for the proposal and the derivation  of SGD type optimization schemes,
we refer, e.g.,  to
\cite{BachMoulines11,BachMoulines13,Bottou12,BottouLeCun04,DarkenChangMoody92,
DieuleveutDurmusBach17,
InoueParkOkada03,
LeCunBottouOrrMuller98,
MizutaniDreyfus10,
PascanuBengio13,
Pillaud-VivienRudiBach17,
RakhlinShamirSridharan12,
RattraySaadAmari98,
SutskeverMartensDahlHinton13,
Xu11,
Zhang04,
JohnsonZhangNIPS2013,
NguyenEtAl2018} and the references mentioned therein 
for numerical simulations for
SGD type optimization schemes,
and
we refer, e.g., to \cite{DeanETAL12,DengETAL13,Graves13,GravesMohamedHinton13,HintonETAL12,HintonSalakhutdinov06,DahlSainathHinton2013,KrizhevskySutskeverHinton12,LeCunBottouBengioHaffner98,SchaulZhangLeCun12,IoffeSzegedy2015,Xu11,Zhang04,
SankararamanEtAl2019,BeckerEtAl2019,BeckerCheriditoJentzen2019,BeckBeckerCheridito2019,becker2018deep} and the references mentioned therein 
for applications involving  neural networks and SGD type optimization schemes.
There are also a number of rigorous mathematical results on SGD type optimization schemes which aim to contribute to an understanding toward the success and the limitations of SGD type optimization schemes (cf., e.g., 
\cite{DieuleveutDurmusBach17,Wurstemberger2018,JentzenWurstemberger2018,JohnsonZhangNIPS2013,LanZhou2018,NguyenEtAl2018,Pillaud-VivienRudiBach17,RakhlinShamirSridharan12,TangMonteleoni15}
  for mathematical results in case of strongly convex objective functions, cf., e.g., \cite{Bach14,BachMoulines11,BachMoulines13,BottouBousquet11,WoodworthSrebro2016} for mathematical results in case of convex but possibly non-strongly convex objective functions, and  cf., e.g., 
  \cite{FehrmanGessJentzen2019,GhadimiLanZhang2016,LeiHuLiTang2019,CheriditoEtAl2020,ZeyuanEtAl2018,LiLiang2018,BrutzkusEtAl2017,ChauEtAl2019,LovasEtAl2020}
  for mathematical results in case of possibly non-convex objective functions).
In mathematical convergence results for SGD type optimization schemes there are usually two types of error criteria studied in the scientific literature, that is, (I) the error in the strong sense (cf., e.g., 
\cite{Bach14,BachMoulines11,BottouLeCun04,DieuleveutDurmusBach17,Wurstemberger2018,JentzenWurstemberger2018,NguyenEtAl2018,NguyenEtAl2018}) and (II) the error with respect to the objective function (cf., e.g.,
\cite{Bach14,BachMoulines11,BachMoulines13,DieuleveutDurmusBach17,JentzenWurstemberger2018,JohnsonZhangNIPS2013,LanZhou2018,Pillaud-VivienRudiBach17,RakhlinShamirSridharan12,TangMonteleoni15,WoodworthSrebro2016}). More specifically, 
suppose that the objective function  $f\colon\R^d\to\R$  which we intend to minimize by means of an SGD type optimization scheme satisfies for all $x \in \R^d$ that $f(x) =\E[F(x,Z)]$, where $d \in \N = \{1, 2, 3, \ldots\}$, where $Z \colon \Omega \to S$  is a random variable on a probability space  $(\Omega,\F,\P)$ with values in a measurable space $(S,\mathcal{S})$, and where $F \colon \R^d\times S\to \R$ is a sufficiently regular function  (cf., e.g.,  \cite[Section~1]{DereichMuellerGronbach2019}, \cite[Theorem~1.1]{Wurstemberger2018}, and \cite[Theorem~1.1]{JentzenWurstemberger2018}). Moreover, suppose that $\Xi \in \R^d$ is a minimum point of the objective function $f \colon \R^d \to \R$ and 
suppose that $\Theta \colon \N_0\times \Omega \to \R^d$ is the stochastic process induced by the considered SGD type optimization scheme (cf.~\eqref{eq:intro:SGD} in Theorem~\ref{thm:intro} below). Then in the case of (I) one is interested  in the size of the strong $L^2$-error between the minimum point $\Xi$ and $\Theta_n$  as $n \to \infty$ and in the case of (II) one is interested in the size of the error between the objective function $f$ evaluated at the minimum point $\Xi$ and the expectation  of the
objective function $f$ evaluated at $\Theta_n$ as $n \to \infty$.
In the case of (II) the error is in some sense weaker but in many situations one can establish quicker convergence rates for (II), namely, twice the convergence rate in (I) 
(see, e.g., \cite[items (ii) and (iii) in Theorem~1.1]{JentzenWurstemberger2018}).
In applications one is usually    not only interested in the objective function $f$ evaluated at the minimum point $\Xi$ but also in some other functional evaluated at the minimum point $\Xi$ and the analysis of the error corresponding to this approximation problem 
 is the subject of this article. More formally,  the main contribution of this work is to study an error criteria which is different from (I) and (II) and which essentially generalizes (II), that is, in this work we study the size of the error between  $\psi(\Xi)$ and $\E [\psi(\Theta_n)]$ as $n \to \infty$ for any sufficiently regular function $\psi \colon \R^d \to \R$ (in particular, including the objective function $f \colon \R^d \to \R$ as a special case).
More specifically,  the main result of this article, Theorem~\ref{cor:no_setting_discrete_version} below,  establishes that under suitable convexity type assumptions
the convergence rate of this error   is the same convergence rate
 as in the special case (II) where the sufficiently regular function $\psi \colon \R^d \to \R$ coincides with the objective function $f \colon \R^d \to \R$.
To illustrate the  findings of Theorem~\ref{cor:no_setting_discrete_version} we now present a special case of the main result of this article. 
\begin{theorem}
\label{thm:intro}
Let  $d\in \N$, $ \xi,\,\Xi\in \R^d$, $\epsilon\in (0,1)$,  $\eta, L, c \in (0,\infty)$, $\psi\in C^2(\R^d,\R)$,  let $(S,\mathcal{S})$ be a measurable space, let $(\Omega,\F,\P)$ be a probability space, 
let  $F  = \linebreak (F(\theta, s))_{(\theta, s) \in \R^d\times S } \colon \R^d\times S\to \R$ be  $(\mathcal{B}(\R^d)\otimes \mathcal{S})/\mathcal{B}(\R)$-measurable,
let $Z_n\colon\Omega \to S$, $n \in \N$, be i.i.d.\ random variables,
 assume for all $s \in S$ that $(\R^d \ni \theta \mapsto F(\theta,s)\in\R )\in C^3(\R^d,\R)$,  assume  for all $\theta,  \vartheta \in \R^d$  that 
\begin{gather}
\E\big[\|(\nabla_{\theta} F)(\theta, Z_1)\|_{\R^d}^2 \big] \leq c \big[1+\|\theta\|_{\R^d} \big]^2,\\
\textstyle\sum_{i=2}^3 \inf\nolimits_{\delta\in(0,\infty)}\sup\nolimits_{u\in [-\delta,\delta]^d}\E\big[ |F(\theta, Z_1)| + \| (\tfrac{\partial^i}{\partial \theta^i}F)(\theta+u,Z_1)\|_{L^{(i)}(\R^d,\R)}^{1+\delta}\big]<\infty,\\
\langle \theta - \vartheta, \E[ ( \nabla_{\theta} F)(\theta, Z_1) ]-\E[ ( \nabla_{\theta} F)(\vartheta,Z_1) ]\rangle_{\R^d} \geq L \|\theta-\vartheta\|_{\R^d}^2,\\
\big\|\E\big[ (\tfrac{\partial^3}{\partial {\theta}^3}F)(\theta, Z_1) \big]\big\|_{L^{(3)}(\R^d,\R)} 
+ \max\nolimits_{i \in \{1,2\}}
\|\psi^{(i)}(\theta)\|_{L^{(i)}(\R^d,\R)} < \infty,
\end{gather}
and $\| \E[ ( \nabla_{\theta} F)(\theta, Z_1) ]\|_{\R^d} \leq c \|\theta - \Xi\|_{\R^d}$,
and let $\Theta \colon \N_0\times \Omega \to \R^d$ satisfy for all $n\in\N$   that $\Theta_0 = \xi$ and 
\begin{equation}
\label{eq:intro:SGD}
\Theta_n = \Theta_{n-1} - \tfrac{\eta}{n^{1-(\epsilon/2)}} (\nabla_{\theta} F)(\Theta_{n-1},Z_n).
\end{equation}
Then 
\begin{enumerate}[(i)]
	\item
	\label{item1:thm:intro} we have that $\lb \theta \in \R^d \colon ( \E[F(\theta, Z_1)] = \inf\nolimits_{\vartheta \in \R^d} \E[F(\vartheta, Z_1)])\rb = \lb \Xi\rb$
	and 
	\item
	\label{item2:thm:intro} there exists  $C \in \R$ such that for all $n\in\N$ we have that 
	\begin{equation}
	\label{eq:intro:statement}
		|\psi(\Xi) - \E [\psi(\Theta_{n})]|  \leq C n^{\epsilon-1}.
		\end{equation}
\end{enumerate}
\end{theorem}
Theorem~\ref{thm:intro} is an immediate consequence of Corollary~\ref{cor:conv:sgd:nof} below. Corollary~\ref{cor:conv:sgd:nof}, in turn, follows from
Theorem~\ref{cor:no_setting_discrete_version} which is the main result of this article.
We now introduce some of the notation which we have used in 
 Theorem~\ref{thm:intro} above and which we will use in the later part of this article. For every $d \in \N$ we denote by $\left\| \cdot \right\|_{\R^d} \colon \R^d \to [0, \infty)$  the standard norm on $\R^d$, for every $d \in \N$ we denote by $\langle \cdot, \cdot \rangle_{\R^d} \colon \R^d \times \R^d \to \R$ the standard scalar product on $\R^d$,
for every $k, m, n \in\N$ we denote by  $L^{(k)}(\R^m,\R^n)$ the set of all continuous $k$-linear functions from $\R^m \times \R^m \times \ldots \times \R^m= (\R^m)^k$ to $\R^n$,
for every $k, m, n \in\N$ 
we denote by $\left\|\cdot\right\|_{L^{(k)}(\R^m,\R^n)}\colon L^{(k)}(\R^m, \R^n)\to [0,\infty) $ the function which satisfies for all $A\in L^{(k)}(\R^m,\R^n)$ that 
\begin{equation}
\| A \|_{L^{(k)}(\R^m,\R^n)} = \sup_{u_1,u_2,\dots,u_k\in \R^m\backslash \lb 0\rb} \frac{\|A(u_1,u_2,\dots,u_k)\|_{\R^n}}{\|u_1\|_{\R^m}\|u_2\|_{\R^m}\cdots\|u_k\|_{\R^m}},
\end{equation}
for every $m, n \in\N$  we denote by $L^{(0)}(\R^m,\R^n)$ the set given by $L^{(0)}(\R^m,\R^n) = \R^n$, and for every $ m, n \in\N$ 
 we denote by $\left\|\cdot\right\|_{L^{(0)}(\R^m,\R^n)}\colon \R^n\to [0,\infty) $ the function which satisfies for all $x \in \R^n$ that $\| x \|_{L^{(0)}(\R^m,\R^n)} =\|x\|_{\R^n}$.
Note that for all $m, n \in \N$, $A \in L(\R^m, \R^n)$
we have that $L(\R^m, \R^n)= L^{(1)}(\R^m, \R^n)$ and $\| A \|_{L(\R^m,\R^n)} = \| A \|_{L^{(1)}(\R^m,\R^n)}$.
Let us also add a few further comments  on some of the mathematical objects appearing appearing in  Theorem~\ref{thm:intro} above. In 
Theorem~\ref{thm:intro} above we intend to approximately solve the stochastic optimization problem in item~\eqref{item1:thm:intro} above. More specifically, in Theorem~\ref{thm:intro} above we intend to  weakly approximate the global minimizer $\Xi \in \R^d$ of the function $\R^d \ni \theta \mapsto  \E[F(\theta, Z_1)] \in \R$, where $F \colon \R^d\times S\to \R$ is a sufficiently regular function and where $Z_1 \colon \Omega \to S$ is a random variable on the probability space  $(\Omega,\F,\P)$ with values on the measurable space $(S, \mathcal{S})$.  In Theorem~\ref{thm:intro} above we intend to accomplish this by means of the stochastic gradient descent process $\Theta \colon \N_0\times \Omega \to \R^d$  defined recursively in \eqref{eq:intro:SGD}. In \eqref{eq:intro:statement} in item~\eqref{item2:thm:intro}  in Theorem~\ref{thm:intro} above we establish that for every  sufficiently regular function $\psi \colon \R^d \to \R$ and every arbitrarily small $\varepsilon \in (0, 1)$ we have that the weak error $|\psi(\Xi) - \E [\psi(\Theta_{n})]|$ converges with convergence rate $1-\varepsilon$ to $0$ as $n \to \infty$.  The weak error analysis which we use in our proof of Theorem~\ref{thm:intro} above is strongly based on employing first-order  Kolmogorov  backward PDEs associated to ordinary differential equations (ODEs). In that aspect our strategy of our proof of Theorem~\ref{thm:intro} is inspired by the weak error analysis for numerical approximations of stochastic differential equations (SDEs). In particular, the weak error analysis for numerical approximations of SDEs is often based on employing second-order Kolmogorov PDEs associated to SDEs; see, e.g.,  Kloeden \& Platen \cite[Chapter 14]{kp92}, Rößler~\cite[Subsection~2.2.1]{r03}, M{\"u}ller-Gronbach \& Ritter~\cite[Section~4]{mr08} and the references mentioned therein for weak error analyses for numerical approximations of SDEs.

The rest of this article is structured in the following way.  As we mentioned earlier, the weak error analysis which we use in our proof of Theorem~\ref{thm:intro} above is strongly based on employing first-order  Kolmogorov  backward PDEs associated to ODEs. To this end, we recall in Chapter~\ref{section:existence} existence and regularity properties for solutions of such  first-order  Kolmogorov  backward PDEs. In Chapter~\ref{sec:SAA:general} we use the analysis for first-order  Kolmogorov  backward PDEs from Chapter~\ref{section:existence} to study weak approximation errors for stochastic approximation algorithms (SAAs) in the case of general learning rates. In 
Chapter~\ref{sec:SAA:poly} we specialize the weak error analysis for SAAs in the case of general learning rates from Chapter~\ref{sec:SAA:general} to accomplish weak error estimates for SAAs in the case of polynomially decaying learning rates. In Chapter~\ref{sec:SGD} we apply the weak error analysis results for SAAs from Chapter~\ref{sec:SAA:poly} to establish weak error estimates for SGD optimization methods.

\chapter{Existence results for solutions of first-order  Kolmogorov  backward partial differential equations (PDEs)}
\label{section:existence}
\chaptermark{}

The weak error analysis which we use in our proof of Theorem~\ref{thm:intro} above is strongly based on employing first-order  Kolmogorov  backward PDEs associated to ODEs. In this chapter we present in \Cref{thm:Kolm_back_eq} in Section~\ref{subsec:Kolmogorov} below an elementary existence  result for solutions of such  first-order  Kolmogorov  backward PDEs. In our proof of  \Cref{thm:Kolm_back_eq}  we use the well-known  regularity result for solutions of ODEs in  Lemma~\ref{lem:regularity_flow} in Section~\ref{subsec:existence} below and we use the elementary  uniqueness result  for solutions of ODEs in Lemma~\ref{lem:unique}
in Section~\ref{subsec:existence} below. 
Our proof of Lemma~\ref{lem:unique}, in turn, employs the well-known  result for continuous functions on compact topological spaces in  \Cref{prop:cont_fct_bnd_on_cmpcts} in Section~\ref{subsec:sufficient} below and the well-known  Gronwall integral inequality in Lemma \ref{lem:Gronwall} in Section~\ref{subsec:Gronwall} below.
In addition,  our proof of \Cref{thm:Kolm_back_eq} also uses 
the essentially well-known result on the possibility of interchanging derivatives and integrals in Lemma~\ref{lem:diff:induction} in  Section~\ref{subsec:sufficient} below.  A slightly modified version of
 Lemma~\ref{lem:diff:induction} can, e.g., be found 
in Durrett~\cite[Theorem A.5.1]{Durrett2010}.  In order to formulate the statement of Lemma~\ref{lem:diff:induction}  we employ the essentially well-known measurability result for derivatives of sufficiently regular functions in \Cref{cor:derivative:gen} in  Section~\ref{subsec:sufficient} below. \Cref{cor:derivative:gen}  follows directly from the elementary measurability results in Lemmas~\ref{thm:cont:meas}, \ref{lem:derivative:measd}, and \ref{lem:product:meas}  in  Section~\ref{subsec:sufficient} below.
Moreover, in this chapter we present in \Cref{lem:Gronwall_differential} a well-known Gronwall-type differential inequality, we present   in \Cref{lem:diff^n_under_int} a direct generalization of the result on the possibility of interchanging derivatives and integrals in Lemma~\ref{lem:diff:induction}, we present  in \Cref{thm:regularity_flow_dynamical_system} an essentially well-known existence and  uniqueness result for solutions of ODEs, and we present  in \Cref{cor:regularity_flow} a direct generalization of  the  regularity result for solutions of ODEs in  Lemma~\ref{lem:regularity_flow}. In Chapter~\ref{sec:SAA:general} below   we employ \Cref{lem:Gronwall_differential}, Lemma~\ref{lem:Gronwall}, Lemma~\ref{prop:cont_fct_bnd_on_cmpcts}, Corollary~\ref{cor:derivative:gen},
Lemma~\ref{lem:diff:induction},  \Cref{lem:diff^n_under_int}, \Cref{thm:regularity_flow_dynamical_system},    Lemma~\ref{lem:regularity_flow},  \Cref{cor:regularity_flow}, and \Cref{thm:Kolm_back_eq} 
to study  weak approximation errors for SAAs.

\section{Gronwall-type inequalities}
\label{subsec:Gronwall}
\sectionmark{}

\begin{lemma}\label{lem:Gronwall_differential}
Let $t\in \R$, $T\in (t,\infty)$, $b\in C([t,T],\R)$, $f\in C^1([t,T],\R)$ satisfy for all $s\in [t,T]$ that
$
f'(s)\leq b(s)f(s)
$.
Then we have for all $s\in [t,T]$ that
\begin{equation}\label{eq:Gronwall_differential_2}
f(s)\leq f(t)\exp\!\left(\int_{t}^s b(u)\, du\right).
\end{equation}
\end{lemma}
\begin{proof}[Proof of Lemma~\ref{lem:Gronwall_differential}]
Throughout this proof let $v\colon [t,T]\to (0,\infty)$ satisfy for all $s\in [t,T]$ that
\begin{equation}\label{eq:Gronwall_differential_3}
v(s)=\exp\!\left( \int_{t}^sb(u)\, du \right)
\end{equation}
and let $g\colon [t,T]\to \R$ satisfy for all $s\in [t,T]$ that
\begin{equation}\label{eq:Gronwall_differential_4}
g(s)=\frac{f(s)}{v(s)}.
\end{equation}
Observe that for all $s\in [t,T]$ we have that
$
v'(s)=b(s)v(s)
$.
This  implies that for all $s\in [t,T]$ we have that
\begin{equation}
\begin{split}
g'(s)
&=\frac{f'(s)v(s)-f(s)v'(s)}{v(s)^2}\\
&=\frac{f'(s)v(s)-f(s)b(s)v(s)}{v(s)^2}\\
&\leq \frac{b(s)f(s)v(s)-f(s)b(s)v(s)}{v(s)^2}=0.
\end{split}
\end{equation}
This assures that $g$ is non-increasing. This reveals that for all $s\in [t,T]$  it holds that
\begin{equation}
\frac{f(s)}{v(s)}=g(s)\leq g(t)=\frac{f(t)}{v(t)}=f(t).
\end{equation}
This establishes  \eqref{eq:Gronwall_differential_2}. The proof of Lemma~\ref{lem:Gronwall_differential} is thus completed.
\end{proof}

\begin{lemma}\label{lem:Gronwall}
Let $T\in (0,\infty)$, $a,b\in [0,\infty)$, let $f\colon[0,T]\to [0,\infty)$ be  $ \mathcal{B}([0,T])$/ \allowbreak$\mathcal{B}([0,\infty)) $-measurable, and assume  for all $t\in [0,T]$ that
\begin{equation}\label{eq:Gronwall_1}
\int_0^T|f(s)|\, ds<\infty \qquad\text{and} \qquad
f(t)\leq a + b\int_0^tf(s)\, ds.
\end{equation}
Then we have for all $t\in [0,T]$ that $f(t)\leq a \exp(bt)$.
\end{lemma}
\begin{proof}[Proof of Lemma~\ref{lem:Gronwall}]
We claim that for all $n\in\N_0$, $t\in [0,T]$  we have that
\begin{equation}\label{eq:Gronwall_4}
f(t)\leq a \left(\textstyle\sum\limits_{k=0}^{n} \displaystyle \frac{(bt)^k}{k!} \right) + b^{n+1}\int_0^t \frac{(t-s)^n}{n!}f(s)\, ds.
\end{equation}
We now establish \eqref{eq:Gronwall_4} by induction on $n \in \N_0$. The base case $n=0$ is an immediate consequence of \eqref{eq:Gronwall_1}. For the induction step $\N_0\ni n\to n+1\in\N_0$ assume that \eqref{eq:Gronwall_4} holds for a given $n \in \N_0$.
Observe that the induction hypothesis and \eqref{eq:Gronwall_1} ensure that for all $t\in [0,T]$ we have that
\begin{equation}\label{eq:Gronwall_5}
\begin{split}
f(t) & \leq a \left( \textstyle\sum\limits_{k=0}^{n} \displaystyle \frac{(bt)^k}{k!} \right)+ b^{n+1}\int_0^t \frac{(t-s)^n}{n!}f(s)\, ds\\
& \leq a \left(\textstyle\sum\limits_{k=0}^{n} \displaystyle \frac{(bt)^k}{k!} \right) + b^{n+1}\int_0^t \frac{(t-s)^n}{n!}\bigg( a + b\int_0^sf(v)\, dv \bigg) ds.
\end{split}
\end{equation}
Moreover, note that for all $t\in [0,T]$ we have that
\begin{equation}\label{eq:Gronwall_6}
\int_0^t\frac{(t-s)^n}{n!}\, ds= \frac{1}{n!}\bigg[\frac{-(t-s)^{n+1}}{n+1}\bigg]_{s=0}^{s=t}=\frac{ t^{n+1}}{(n+1)!}.
\end{equation}
Furthermore, observe that Tonelli's theorem implies that for all $t\in [0,T]$ we have that
\begin{equation}\label{eq:Gronwall_7}
\begin{split}
\int_0^t(t-s)^n\int_0^sf(v)\, dv\, ds 
& = \int_0^t\int_0^t(t-s)^nf(v)\mathbbm{1}_{\lbrace 0\leq v\leq s \leq t\rbrace}\, dv\, ds\\
& = \int_0^t f(v) \int_v^t(t-s)^n\, ds\, dv\\
& = \int_0^t f(v) \left[ -\frac{(t-s)^{n+1}}{n+1} \right]_{s=v}^{s=t}\, dv\\
& = \int_0^tf(v)\frac{(t-v)^{n+1}}{n+1}\, dv.
\end{split}
\end{equation}
Combining this, \eqref{eq:Gronwall_5}, and \eqref{eq:Gronwall_6} establishes that for all $t\in [0,T]$ we have that
\begin{equation}\label{eq:Gronwall_8}
\begin{split}
f(t) & \leq a \left( \textstyle\sum\limits_{k=0}^{n} \displaystyle \frac{(bt)^k}{k!} \right) + b^{n+1}\int_0^t \frac{(t-s)^n}{n!}\left( a + b\int_0^sf(v)\, dv \right) ds\\
& = a \left( \textstyle\sum\limits_{k=0}^{n} \displaystyle \frac{(bt)^k}{k!} \right)+ a\frac{(bt)^{n+1}}{(n+1)!} +b^{n+2}\int_0^t \frac{(t-s)^{n+1}}{(n+1)!}f(s)\, ds\\
& = a \left( \textstyle\sum\limits_{k=0}^{n+1} \displaystyle \frac{(bt)^k}{k!} \right) +b^{n+2}\int_0^t \frac{(t-s)^{n+1}}{(n+1)!}f(s)\, ds.
\end{split}
\end{equation}
This proves \eqref{eq:Gronwall_4} in the case $n+1$.
This finishes the proof of the induction step. Induction hence establishes \eqref{eq:Gronwall_4}.
Next observe that \eqref{eq:Gronwall_4}  implies that for all $t\in[0,T]$, $n\in\N_0$ we have that
\begin{equation}\label{eq:Gronwall_9}
\begin{split}
f(t) \leq  a \, e^{bt} + b^{n+1}\int_0^t \frac{(t-s)^n}{n!}f(s)\, ds 
\leq  a \, e^{bt} + b^{n+1}\frac{t^n}{n!}\int_0^tf(s)\, ds.
\end{split}
\end{equation}
Moreover, note  that \eqref{eq:Gronwall_1} ensures that for all $t\in[0,T]$ we have that
\begin{equation}\label{eq:Gronwall_10}
\limsup_{n\to \infty} \left[ b^{n+1}\frac{t^n}{n!}\int_0^tf(s)\, ds \right]=0.
\end{equation}
Combining this and \eqref{eq:Gronwall_9} establishes that for all $t\in [0,T]$ we have that $f(t)\leq a \exp(bt)$. The proof of Lemma~\ref{lem:Gronwall} is thus completed.
\end{proof}

\section{Sufficient conditions for interchanging derivatives and integrals}
\label{subsec:sufficient}
\sectionmark{}

\begin{lemma}
\label{thm:cont:meas}
Let $(S,\mathcal{S})$ be a measurable space, let $(X, d_X)$ be a compact metric space, let $(Y, d_Y)$ be a separable metric space, let $C(X,Y)$ be the space of continuous functions endowed with the topology of $d_Y$-uniform  convergence,  let $f \colon X \times S \to Y$ be a function, assume for all  $x \in X$ that $ ( S \ni s  \mapsto f(x, s) \in Y)$ is $\mathcal{S}$/$\mathcal{B}(Y)$-measurable, and assume for all $s \in S$ that $(X \ni x \mapsto f(x,s) \in Y )\in C(X,Y)$. Then we have that
\begin{equation}
(S \ni s \mapsto (X \ni x \mapsto f(x,s) \in Y) \in C(X,Y)) 
\end{equation}
is $\mathcal{S}$/$\mathcal{B}(C(X,Y))$-measurable.
\end{lemma}

\begin{lemma}\label{prop:cont_fct_bnd_on_cmpcts}
	Let $(X,\mathcal{X})$ be a  compact topological space, let $(M,d)$ be a metric space, and let $u\in M$, $f \in C(X, M)$. Then we have that
	\begin{equation}\label{eq:cont_fct_bnd_on_cmpcts_1}
	\sup\! \big(\{ d(f(x),u) \in \R \colon x \in X \} \cup \{0\} \big)<\infty.
	\end{equation}
\end{lemma}

\begin{lemma}
\label{lem:derivative:measd}
Let $(X, \left\| \cdot \right\|_{X})$ and $(Y, \left\| \cdot \right\|_{Y})$ be finite dimensional normed vector spaces, let $(S,\mathcal{S})$ be a measurable space,  let $F  = (F(x, s))_{(x, s) \in X\times S} \colon X\times S \to Y$ be  $(\mathcal{B}(X)\otimes \mathcal{S})$/$\mathcal{B}(Y)$-measurable, and assume for all  $s \in S$ that $(X \ni x \mapsto F(x,s) \in Y )\in C^1(X,Y)$. Then we have for all $x \in X$  that 
\begin{equation}
\label{eq:partial:measd}
\big(S \ni s \mapsto (\tfrac{\partial}{\partial x} F)(x,s) \in L(X, Y) \big)
\end{equation}
is $\mathcal{S}$/$\mathcal{B}(L(X, Y))$-measurable.
\end{lemma}
\begin{proof}[Proof of Lemma~\ref{lem:derivative:measd}]
	Throughout this proof let $V = C(\{w \in X \colon \|w\|_{X} \leq 1\}, Y)$, let $\left\|\cdot\right\|_V \colon V \to [0, \infty)$ satisfy for all $f \in V$ that
\begin{equation}
\label{eq:meas:norm}
\|f\|_V = \sup\nolimits_{h \in \{w \in X \colon \|w\|_{X} \leq 1 \}} \|f(h)\|_{Y}
\end{equation}
(cf.~\Cref{prop:cont_fct_bnd_on_cmpcts}), let $\iota \colon C(X, Y) \to V$	satisfy for all $\varphi \in C(X, Y)$ that
\begin{equation}
\iota (\varphi) = (\{ w \in X \colon \|w\|_{X} \leq 1\} \ni h \mapsto \varphi(h) \in Y),
\end{equation}
and let $\psi \colon \{f \in V \colon f \text{ is linear}\} \to L(X, Y)$  satisfy for all $\mathcal{A} \in  \{f \in V \colon f \text{ is linear}\}$, $h \in X \backslash \{0\}$ that 
\begin{equation}
\psi(\mathcal{A})(h) = \|h\|_{X} \mathcal{A} \Big(\tfrac{h}{\|h\|_{X}}\Big).
\end{equation}
Observe that
the assumption that  $\forall \, s \in S \colon (X \ni x \mapsto F(x,s) \in Y )\in C^1(X,Y)$ implies that for all $x \in X$,  $s \in S$, $\varepsilon \in (0, \infty)$ there exists $r \in (0, \infty)$ such that for all $h \in X \backslash \{0\}$ with $\|h\|_X \leq r$ we have that
\begin{equation}
 \frac{\|F(x+h,s) - F(x,s) - (\tfrac{\partial}{\partial x} F)(x,s) h\|_Y}{\|h\|_X} \leq \varepsilon.
\end{equation}
This reveals that for all $x \in X$,  $s \in S$, $\varepsilon \in (0, \infty)$  there exists $r \in (0, \infty)$ such that
\begin{equation}
\sup_{h \in \{w \in X \colon 0<\|w\|_{X} \leq r \}} \frac{\|F(x+h,s) - F(x,s) - (\tfrac{\partial}{\partial x} F)(x,s) h\|_Y}{\|h\|_X} \leq \varepsilon.
\end{equation}
This ensures that for all $x \in X$,  $s \in S$, $\varepsilon \in (0, \infty)$  there exists $r \in (0, \infty)$ such that for all $\delta \in (0, r]$ we have that
\begin{equation}
\sup_{h \in \{w \in X \colon 0<\|w\|_{X} \leq \delta \}} \frac{\|F(x+h,s) - F(x,s) - (\tfrac{\partial}{\partial x} F)(x,s) h\|_Y}{\|h\|_X} \leq \varepsilon.
\end{equation}
This reveals that for all $x \in X$, $s \in S$  it holds that
\begin{equation}
\limsup_{(0,\infty) \ni r \to 0} \sup_{h \in \{w \in X \colon 0<\|w\|_{X} \leq r \}} \! \left[ \frac{\|F(x+h,s) - F(x,s) - (\tfrac{\partial}{\partial x} F)(x,s) h\|_Y}{\|h\|_X} \right]\! =0.
\end{equation}
This assures that  for all $x \in X$,  $s \in S$ we have that
\begin{equation}
\limsup_{(0,\infty) \ni r \to 0} \sup_{h \in \{w \in X \colon 0<\|w\|_{X} \leq 1 \}} \! \left[ \frac{\|F(x+rh,s) - F(x,s) - (\tfrac{\partial}{\partial x} F)(x,s) rh\|_Y}{r \|h\|_X } \right]\! =0.
\end{equation}
This reveals that for all $x \in X$,  $s \in S$  it holds that
\begin{equation}
\limsup_{(0,\infty) \ni r \to 0} \sup_{h \in \{w \in X \colon 0<\|w\|_{X} \leq 1 \}} \Big\| \tfrac{F(x+rh,s) - F(x,s)}{r}  -  (\tfrac{\partial}{\partial x} F)(x,s)h \Big\|_{Y} =0.
\end{equation}
This and \eqref{eq:meas:norm} demonstrate that for all $x \in X$, $s \in S$ we have that 
\begin{align}
\label{eq:uniform:limit}
&\limsup_{(0,\infty) \ni r \to 0} \Big\| \big(\{ w \in X \colon \|w\|_{X} \leq 1\} \ni h \mapsto \tfrac{F(x+rh,s) - F(x,s)}{r} \in Y\big) - \iota\big(  \big(\tfrac{\partial}{\partial x} F\big)(x,s) \big)  \Big\|_V \nonumber\\
& = \limsup_{(0,\infty) \ni r \to 0} \sup_{h \in \{w \in X \colon 0<\|w\|_{X} \leq 1 \}} \Big\| \tfrac{F(x+rh,s) - F(x,s)}{r}  -  (\tfrac{\partial}{\partial x} F)(x,s)h \Big\|_{Y} = 0.
\end{align}
Next observe that \Cref{thm:cont:meas} (with $S=S$, $\mathcal{S}= \mathcal{S}$, $X = \{ w \in X \colon \|w\|_{X} \leq 1\}$, $d_X = (\{ w \in X \colon \|w\|_{X} \leq 1\} \times \{ w \in X \colon \|w\|_{X} \leq 1\} \ni (y,z) \mapsto \|y-z\|_{X} \in [0,\infty))$, $Y=Y$, $d_Y = (Y \times Y \ni (y,z) \mapsto \|y-z\|_{Y} \in [0, \infty) )$, $f = (\{ w \in X \colon \|w\|_{X} \leq 1\} \times S \ni (h,s) \mapsto F(x+rh,s) \in Y)$ for  $x \in X$, $r \in (0, \infty)$ in the notation of \Cref{thm:cont:meas})  implies that for all $x \in X$, $r \in (0, \infty)$ we have that
\begin{equation}
\begin{split}
(S \ni s  \mapsto (\{ w \in X \colon \|w\|_{X} \leq 1\} \ni h \mapsto F(x+rh,s) \in Y) \in V) 
\end{split}
\end{equation}
is $\mathcal{S}$/$\mathcal{B}(V)$-measurable. This and \eqref{eq:uniform:limit} prove that for all $x \in X$ we have that
\begin{equation}
\label{eq:iota}
\big(S \ni s \mapsto \iota \big( (\tfrac{\partial}{\partial x} F)(x,s)\big) \in V \big) 
\end{equation}
is $\mathcal{S}$/$\mathcal{B}(V)$-measurable. 
Moreover, note that for all $f_n \in  \{f \in V \colon f \text{ is linear}\}$, $n \in \N$, and all functions  $g \colon \{w \in X \colon \|w\|_{X} \leq 1\} \to Y$ with 
\begin{equation}
\limsup_{n \to \infty}  \sup_{h \in \{w \in X \colon \|w\|_{X} \leq 1 \}} \|f_n(h) - g(h)\|_{Y} =0
\end{equation}
we have that
\begin{equation}
g \in \{f \in V \colon f \text{ is linear}\}.
\end{equation}
This ensures that
\begin{equation}
\label{eq:linear:borel}
\{f \in V \colon f \text{ is linear}\} \in \mathcal{B}(V).
\end{equation}
Combining this, the fact that $\forall \, A \in L(X,Y)  \colon  \iota(A) \in  \{f \in V \colon f \text{ is linear}\}$, and \eqref{eq:iota} proves that 
\begin{equation}
\label{eq:iota2}
\big(S \ni s \mapsto \iota \big( (\tfrac{\partial}{\partial x} F)(x,s)\big) \in  \{f \in V \colon f \text{ is linear}\}  \big) 
\end{equation}
is $\mathcal{S}$/$\mathcal{B}(\{f \in V \colon f \text{ is linear}\})$-measurable. 
Furthermore, observe that for all $A \in L(X,Y)$, $x \in X \backslash \{0\}$ we have that 
\begin{equation}
\psi(\iota(A)) x = \|x\|_X (\iota(A)) \Big(\tfrac{x}{\|x\|_X}\Big) = \|x\|_X A  \Big(\tfrac{x}{\|x\|_X}\Big) = Ax.
\end{equation}
This implies that for all $A \in L(X,Y)$ we have that
\begin{equation}
\label{eq:psi:iota}
\psi(\iota(A)) = A.
\end{equation}
Next note that for all $f_1, f_2 \in \{f \in V \colon f \text{ is linear}\}$  we have that
\begin{equation}
\begin{split}
\|\psi(f_1) - \psi(f_2)\|_{L(X,Y)}  &= \sup_{h \in X \backslash \{0\}} \frac{\|\psi(f_1)h - \psi(f_2)h\|_Y}{\|h\|_X} \\
 &=  \sup_{h \in X \backslash \{0\}} \frac{ \|h\|_X \big\|f_1 \big(\frac{h}{\|h\|_X}\big) - f_2 \big(\frac{h}{\|h\|_X}\big)\big\|_Y}{\|h\|_X}\\
& = \sup\nolimits_{h \in \{w \in X \colon \|w\|_{X} \leq 1 \}} \|(f_1 - f_2)(h)\|_Y.
\end{split}
\end{equation}
Combining this and \eqref{eq:meas:norm}  establishes that 
\begin{equation}
\psi \in C(\{f \in V \colon f \text{ is linear}\}, L(X, Y)).
\end{equation}
This, \eqref{eq:linear:borel}, \eqref{eq:iota2}, and \eqref{eq:psi:iota}  demonstrate that for all $x \in X$ we have that
\begin{equation}
\big(S \ni s \mapsto  (\tfrac{\partial}{\partial x} F)(x,s) \in L(X, Y)\big) 
\end{equation}
is $\mathcal{S}$/$\mathcal{B}(L(X, Y))$-measurable.
The proof of Lemma~\ref{lem:derivative:measd} is thus completed.
\end{proof}

\begin{lemma}
	\label{lem:product:meas}
Let $(S,\mathcal{S})$ be a measurable space, let $(X, d_X)$ be a separable metric space, let $(Y, d_Y)$ be a metric space, let $F \colon X \times S \to Y$ satisfy for all $s \in S$, $x \in X$ that $(X \ni y \mapsto F(y,s) \in Y )\in C(X, Y)$ and $(S \ni w \mapsto F(x, w) \in Y)$  is $\mathcal{S}$/$\mathcal{B}(Y)$-measurable. Then $F$ is $(\mathcal{B}(X)\otimes \mathcal{S})$/$\mathcal{B}(Y)$-measurable.
\end{lemma}
\begin{proof}[Proof of \Cref{lem:product:meas}]
This is a direct consequence of, e.g., Aliprantis \& Border~\cite[Lemma~4.51]{AliprantisBorder2006}. The proof of \Cref{lem:product:meas} is thus completed. 
\end{proof}

\begin{cor}
\label{cor:derivative:gen}
Let $d, m, n \in \N$,  let $(S,\mathcal{S})$ be a measurable space,  let $F  = \linebreak (F(x, s))_{(x, s) \in \R^{d} \times S} \colon \R^{d} \times S \to \R^{m}$ be  $(\mathcal{B}(\R^{d})\otimes \mathcal{S})$/$\mathcal{B}(\R^{m})$-measurable, and assume  for all  $s \in S$ that $(\R^{d} \ni x \mapsto F(x,s) \in \R^{m} )\in C^n(\R^{d},\R^{m})$. Then we have for all $k \in \{1,2, \ldots, n\}$, $x \in \R^d$  that 
\begin{equation}
\label{eq:derivative:measd}
\big(S \ni s \mapsto (\tfrac{\partial^k}{\partial x^k} F)(x,s) \in L^{(k)}(\R^d, \R^m) \big)
\end{equation}
is $\mathcal{S}$/$\mathcal{B}(L^{(k)}(\R^d, \R^m))$-measurable.
\end{cor}
\begin{proof}[Proof of \Cref{cor:derivative:gen}]
This is a direct consequence of Lemma~\ref{lem:derivative:measd} and Lemma~\ref{lem:product:meas}. The proof of \Cref{cor:derivative:gen} is thus completed. 
\end{proof}

\begin{lemma}\label{lem:diff:induction}
Let $d, m, n \in\N$,  let $(S,\mathcal{S},\mu)$ be a finite measure space, let $F   = (F(x, s))_{(x, s) \in \R^{d} \times S} \colon \R^d\times S\to\R^m$ be  $(\mathcal{B}(\R^d)\otimes \mathcal{S})$/$\mathcal{B}(\R^m)$-measurable, 
 let $f \colon \R^d \to \R^m$ be  $(n-1)$-times differentiable,  assume for all  $s\in S$ that $(\R^d\ni x\mapsto F(x,s)\in\R^m)\in C^{n}(\R^d,\R^m)$, and assume for all $x\in\R^d$ that
\begin{equation}
\begin{split}
\label{eq:diff:induction}
\inf_{\delta\in (0,\infty)}\sup_{u\in [-\delta,\delta]^d}\int_S & \|(\tfrac{\partial^{n-1}}{\partial x^{n-1}}F)(x,z)\|_{L^{(n-1)}(\R^d,\R^m)}\\
& \quad + \|(\tfrac{\partial^n}{\partial x^n}F)(x+u,z)\|_{L^{(n)}(\R^d,\R^m)}^{1+\delta}\, \mu(dz)<\infty
\end{split}
\end{equation}
(cf.~Corollary~\ref{cor:derivative:gen}) and 
\begin{equation}
f^{(n-1)}(x)=\int_S (\tfrac{\partial^{n-1}}{\partial x^{n-1}}F)(x,s)\, \mu(ds).
\end{equation}
Then
\begin{enumerate}[(i)]
\item\label{item1:diff:induction} we have that $f\in C^n(\R^d,\R^m)$ and 
\item\label{item2:diff:induction} we have for all  $x\in\R^d$ that
\begin{equation}
f^{(n)}(x)=\int_S (\tfrac{\partial^n}{\partial x^n}F)(x,s)\, \mu(ds).
\end{equation}
\end{enumerate}
\end{lemma}
\begin{proof}[Proof of Lemma~\ref{lem:diff:induction}]
Throughout this proof let $f_1, f_2, \dots,f_m\colon\R^d\to\R$  satisfy for all $x\in\R^d$ that
\begin{equation}
f(x)=(f_1(x), f_2(x), \dots,f_m(x)),
\end{equation}
let $F_1, F_2, \dots,F_m\colon\R^d\times S\to\R$  satisfy for all $x\in\R^d$, $s\in S$ that
\begin{equation}
\label{eq:F:components}
F(x,s)=(F_1(x,s), F_2(x, s), \dots,F_m(x,s)),
\end{equation}
let $\delta_x\in (0,\infty)$, $x \in \R^d$, satisfy for all $x \in \R^d$ that 
\begin{equation}\label{eq:diff^n:delta}
\sup_{v\in [-\delta_x,\delta_x]^d}\int_{ S} \| (\tfrac{\partial^{n}}{\partial x^{n}} F)(x+v,s)\|_{L^{(n)}(\R^d,\R^m)}^{1+\delta_x}\,  \mu(ds) < \infty,
\end{equation}	
and let $e_1=(1,0,\dots,0)$, $e_2=(0,1,0,\dots,0)$,  \ldots, $e_d=(0,0, \dots,0,1)\in\R^d$. 
Note that, e.g., Coleman~\cite[pages 93-94, Section~4.5]{Coleman2012} assures that for all  $x\in\R^d$, $i_1, i_2, \ldots, i_n \in \{1,2, \ldots, d\}$, $j \in \{1, 2, \ldots, m\}$, $s\in S$ we have that
\begin{equation}
(\tfrac{\partial^{n}}{\partial x_{i_1} \partial x_{i_2} \dots \partial x_{i_{n}} }F_j)(x,s)=(\tfrac{\partial^{n}}{\partial x^{n}} F_j)(x,s)(e_{i_1}, e_{i_2},\dots,e_{i_{n}}).
\end{equation}
This ensures that for all  $x\in\R^d$, $i_1, i_2, \ldots, i_n \in \{1,2, \ldots, d\}$, $j \in \{1, 2, \ldots, m\}$ we have that
\begin{align*}
\label{eq:diff^n_under_int_Fj}
&  \sup_{v\in [-\delta_x,\delta_x]^d}\int_{ S}  \big| (\tfrac{\partial^{n}}{\partial x_{i_1} \partial x_{i_2} \dots \partial x_{i_{n}} }F_j)(x+v,s)\big|^{1+\delta_x}\,  \mu(ds)\\
& = \sup_{v\in [-\delta_x,\delta_x]^d}\int_{ S} \big| (\tfrac{\partial^{n}}{\partial x^{n}} F_j)(x+v,s)(e_{i_1}, e_{i_2}, \dots,e_{i_{n}})\big|^{1+\delta_x}\,  \mu(ds) \numberthis\\
& \leq \sup_{v\in [-\delta_x,\delta_x]^d}\int_{ S} \big(\| (\tfrac{\partial^{n}}{\partial x^{n}} F_j)(x+v,s)\|_{L^{(n)}(\R^d,\R)}\|e_{i_1}\|_{\R^d} \|e_{i_2}\|_{\R^d} \cdots\|e_{i_{n}}\|_{\R^d}\big)^{1+\delta_x}\,  \mu(ds)\\
& = \sup_{v\in [-\delta_x,\delta_x]^d}\int_{ S} \| (\tfrac{\partial^{n}}{\partial x^{n}} F_j)(x+v,s)\|_{L^{(n)}(\R^d,\R)}^{1+\delta_x}\,  \mu(ds). 
\end{align*}
In addition, note that, e.g., Coleman~\cite[Proposition~4.6]{Coleman2012} and \eqref{eq:F:components} demonstrate that for all $x \in \R^d$, $s \in S$ we have that
\begin{equation}
(\tfrac{\partial^n}{\partial x^n} F)(x,s) = \big( (\tfrac{\partial^n}{\partial x^n} F_1)(x,s), (\tfrac{\partial^n}{\partial x^n} F_2)(x,s), \ldots, (\tfrac{\partial^n}{\partial x^n} F_m)(x,s) \big).
\end{equation}
This reveals that for all $x \in \R^d$, $s \in S$,  $j \in \{1, 2, \ldots, m\}$   it holds that
\begin{align*}
& \| (\tfrac{\partial^{n}}{\partial x^{n}} F_j)(x,s)\|_{L^{(n)}(\R^d,\R)} = \sup_{y_1, y_2, \ldots, y_n \in \R^d \backslash \{0\}} \frac{|(\tfrac{\partial^{n}}{\partial x^{n}} F_j)(x,s) (y_1, y_2, \ldots, y_n)|}{\|y_1\|_{\R^d} \|y_2\|_{\R^d} \ldots \|y_n\|_{\R^d}} \numberthis\\
 & \leq  \sup_{y_1, y_2, \ldots, y_n \in \R^d \backslash \{0\}} \frac{\|(\tfrac{\partial^{n}}{\partial x^{n}} F)(x,s) (y_1, y_2, \ldots, y_n)\|_{\R^m}}{\|y_1\|_{\R^d} \|y_2\|_{\R^d} \ldots \|y_n\|_{\R^d}}
  =  \| (\tfrac{\partial^{n}}{\partial x^{n}} F)(x,s)\|_{L^{(n)}(\R^d,\R^m)}.
\end{align*}
Combining this, \eqref{eq:diff^n_under_int_Fj}, and  \eqref{eq:diff^n:delta} implies that for all  $x\in\R^d$, $i_1, i_2, \ldots, i_n \in \{1,2, \ldots, d\}$, $j \in \{1, 2, \ldots, m\}$ we have that
\begin{equation}
\begin{split}
\label{eq:diff^n_under_int_1_1}
&  \sup_{v\in [-\delta_x,\delta_x]^d}\int_{ S}  \big| (\tfrac{\partial^{n}}{\partial x_{i_1} \partial x_{i_2} \dots \partial x_{i_{n}} }F_j)(x+v,s)\big|^{1+\delta_x}\,  \mu(ds)\\ 
& \leq \sup_{v\in [-\delta_x,\delta_x]^d}\int_{ S} \| (\tfrac{\partial^{n}}{\partial x^{n}} F)(x+v,s)\|_{L^{(n)}(\R^d,\R^m)}^{1+\delta_x}\,  \mu(ds)<\infty.
\end{split}
\end{equation}
Next observe that the assumption that $\forall\, s\in S\colon (\R^d\ni y\mapsto F(y,s)\in \R^m)\in C^{n}(\R^d,\R^m)$ and the fundamental theorem of calculus imply that for all  $x\in\R^d$, $i_1, i_2, \ldots, i_n \in \{1,2, \ldots, d\}$, $j \in \{1, 2, \ldots, m\}$, $h\in\R$ we have that
\begin{align}
\label{eq:diff^n_under_int_2}
\begin{split}
&(\tfrac{\partial^{n-1}}{\partial x_{i_1} \partial x_{i_2} \dots \partial x_{i_{n-1}}}f_j)(x+he_{i_{n}})-(\tfrac{\partial^{n-1}}{\partial {x_{i_1}} \partial x_{i_2}  \dots \partial x_{i_{n-1}}}f_j)(x)\\
&= \int_S(\tfrac{\partial^{n-1}}{\partial x_{i_1} \partial x_{i_2}  \dots \partial x_{i_{n-1}}}F_j)(x+h e_{i_{n}},s)\, \mu(ds) 
- \int_S(\tfrac{\partial^{n-1}}{\partial x_{i_1} \partial x_{i_2} \dots \partial x_{i_{n-1}}}F_j)(x,s)\, \mu(ds)\\
&= \int_S(\tfrac{\partial^{n-1}}{\partial x_{i_1} \partial x_{i_2}  \dots \partial x_{i_{n-1}}}F_j)(x+h e_{i_{n}},s)
-(\tfrac{\partial^{n-1}}{\partial x_{i_1} \partial x_{i_2}  \dots \partial x_{i_{n-1}}}F_j)(x,s)\, \mu(ds)\\
& = \int_{ S}\int_0^h\! (\tfrac{\partial^{n}}{\partial x_{i_1} \partial x_{i_2}  \dots \partial x_{i_{n}}}F_j)(x+u e_{i_{n}},s)\, du\,  \mu(ds). 
\end{split}
\end{align}
Moreover, note that Tonelli's theorem, H\" older's inequality, and \eqref{eq:diff^n_under_int_1_1} prove that for all  $x\in\R^d$, $i_1, i_2, \allowbreak \ldots, i_n \in \{1,2, \ldots, d\}$, $j \in \{1, 2, \ldots, m\}$, $h\in [-\delta_x,\delta_x]$ we have that
\begin{align*}
& \bigg| \int_{ S}\int_0^h \big| (\tfrac{\partial^{n}}{\partial x_{i_1} \partial x_{i_2}  \dots \partial x_{i_{n}}}F_j)(x+u e_{i_{n}},s)\big|\, du\,  \mu(ds) \bigg| \\
&= \bigg| \int_0^h\!\int_{ S} \big| (\tfrac{\partial^{n}}{\partial x_{i_1} \partial x_{i_2}  \dots \partial x_{i_{n}}}F_j)(x+u e_{i_{n}},s)\big|\,  \mu(ds)\, du \bigg| \\
&\leq |h|\sup_{v\in [-\delta_x,\delta_x]^d}\int_{ S} \big| (\tfrac{\partial^{n}}{\partial x_{i_1} \partial x_{i_2}  \dots \partial x_{i_{n}}}F_j)(x+v,s)\big|\,  \mu(ds) \numberthis\\
&\leq |h|\sup_{v\in [-\delta_x,\delta_x]^d}\bigg(\int_{ S} \big| (\tfrac{\partial^{n}}{\partial x_{i_1} \partial x_{i_2}  \dots \partial x_{i_{n}}}F_j)(x+v,s)\big|^{1+\delta_x}\,  \mu(ds)\bigg)^{\!\frac{1}{1+\delta_x}} \cdot |\mu(S)|^{(1-\frac{1}{1+\delta_x})}  \\
& = |h|\bigg(\sup_{v\in [-\delta_x,\delta_x]^d}\int_{ S} \big| (\tfrac{\partial^{n}}{\partial x_{i_1} \partial x_{i_2}  \dots \partial x_{i_{n}}}F_j)(x+v,s)\big|^{1+\delta_x}\,  \mu(ds)\bigg)^{\!\frac{1}{1+\delta_x}} \cdot |\mu(S)|^{(1-\frac{1}{1+\delta_x})}  <\infty.
\end{align*}
This, Fubini's theorem, and \eqref{eq:diff^n_under_int_2} assure that for all  $x\in\R^d$, $i_1, i_2, \ldots, i_n \in \{1,2, \ldots, d\}$, $j \in \{1, 2, \ldots, m\}$, $h\in [-\delta_x,\delta_x]$ we have that
\begin{equation}\label{eq:diff^n_under_int_4}
\begin{split}
&(\tfrac{\partial^{n-1}}{\partial x_{i_1} \partial x_{i_2}  \dots \partial x_{i_{n-1}}}f_j)(x+he_{i_{n}})-(\tfrac{\partial^{n-1}}{\partial x_{i_1} \partial x_{i_2}  \dots \partial x_{i_{n-1}}}f_j)(x)\\
&= \int_{ S}\int_0^h\! (\tfrac{\partial^{n}}{\partial x_{i_1} \partial x_{i_2}  \dots \partial x_{i_{n}}}F_j)(x+u e_{i_{n}},s)\, du\,  \mu(ds)\\
&= \int_0^h\int_{ S}\! (\tfrac{\partial^{n}}{\partial x_{i_1} \partial x_{i_2}  \dots \partial x_{i_{n}}}F_j)(x+u e_{i_{n}},s)\,  \mu(ds)\, du.
\end{split}
\end{equation}
In addition, observe that, e.g.,  Klenke~\cite[Corollary 6.21]{Klenke2013} and \eqref{eq:diff^n_under_int_1_1} ensure that for all  $x\in\R^d$, $i_1, i_2, \ldots, i_n \in \{1,2, \ldots, d\}$, $j \in \{1, 2, \ldots, m\}$ we have that
\begin{equation}
\label{eq:ui:induction}
\big( S\ni s\mapsto(\tfrac{\partial^{n}}{\partial x_{i_1} \partial x_{i_2}  \dots \partial x_{i_{n}}}F_j)(x+v,s)\in\R\big), v\in [-\delta_x,\delta_x]^d,
\end{equation}
is a uniformly integrable family of functions.
This reveals that  for all  $x\in\R^d$, $i_1, i_2, \ldots, i_n \in \{1,2, \ldots, d\}$, $j \in \{1, 2, \ldots, m\}$, and all  functions $u =(u_k)_{k \in \N} \colon \N \to [-\delta_x, \delta_x]^d$ it holds that                                                                                                                                                                                                                                                                                                                                                                                                                                                                                                                      
\begin{equation}
\big(S\ni s\mapsto (\tfrac{\partial^{n}}{\partial x_{i_1} \partial x_{i_2}  \dots \partial x_{i_{n}}}F_j)(x+u_k,s)\in\R\big),  k \in \N,
\end{equation} 
is a uniformly integrable sequence of functions. This and the Vitali convergence theorem (see, e.g., Klenke~\cite[Theorem 6.25]{Klenke2013}) assure that for all  $x\in\R^d$, $i_1, i_2, \ldots, i_n \in \{1,2, \ldots, d\}$, $j \in \{1,2,  \ldots, m\}$, and all  functions $u=(u_k)_{k \in \N} \colon \N \to [-\delta_x, \delta_x]^d$ with $\limsup_{k \to \infty} \allowbreak \|u_k\|_{\R^d}=0$  we have that                          
\begin{equation}
\label{eq:limsup:int}
\limsup_{k \to \infty} \int_{S}\big|(\tfrac{\partial^{n}}{\partial x_{i_1} \partial x_{i_2}  \dots \partial x_{i_{n}}}F_j)(x+u_k,s) - (\tfrac{\partial^{n}}{\partial x_{i_1}  \partial x_{i_2}  \dots \partial x_{i_{n}}}F_j)(x,s)\big|\, \mu(ds) =0.
\end{equation}
This reveals that for all  $x\in\R^d$, $i_1, i_2, \ldots, i_n \in \{1,2, \ldots, d\}$, $j \in \{1, 2, \ldots, m\}$  it holds that                         
\begin{equation}
\label{eq:cont:F:induction}
\limsup_{v \to 0} \int_{S}\big|(\tfrac{\partial^{n}}{\partial x_{i_1} \partial x_{i_2}  \dots \partial x_{i_{n}}}F_j)(x+v e_{i_n},s) - (\tfrac{\partial^{n}}{\partial x_{i_1} \partial x_{i_2}  \dots \partial x_{i_{n}}}F_j)(x,s)\big|\, \mu(ds) =0.
\end{equation}
This implies that for all $x\in\R^d$, $i_1, i_2, \ldots, i_n \in \{1,2, \ldots, d\}$, $j \in \{1, 2, \ldots, m\}$, $\varepsilon \in (0, \infty)$ there exists  $\delta \in (0, \infty)$ such that 
\begin{equation}
\forall \, v \in (-\delta, \delta) \colon  \int_{S} \big|(\tfrac{\partial^{n}}{\partial x_{i_1} \partial x_{i_2}  \dots \partial x_{i_{n}}}F_j)(x+v e_{i_n},s) - (\tfrac{\partial^{n}}{\partial x_{i_1} \partial x_{i_2}  \dots \partial x_{i_{n}}}F_j)(x,s)\big|\, \mu(ds) < \varepsilon.
\end{equation}
Combining this and \eqref{eq:diff^n_under_int_4} ensures that for all $x\in\R^d$, $i_1, i_2, \ldots, i_n \in \{1,2, \ldots, d\}$, $j \in \{1, 2, \ldots, m\}$, $\varepsilon \in (0, \infty)$ there exists $\delta \in (0, \infty)$ such that  for all $h\in (-\delta, \delta)\backslash \lb 0 \rb$ we have that
\begin{align*}
\label{eq:diff^n_under_int_6}
&\Big| \frac{1}{h}  \!\left((\tfrac{\partial^{n-1}}{\partial x_{i_1} \partial x_{i_2}  \dots \partial x_{i_{n-1}}}f_j)(x+h e_{i_{n}})-(\tfrac{\partial^{n-1}}{\partial x_{i_1} \partial x_{i_2}  \dots \partial x_{i_{n-1}}}f_j)(x)\right)  \\
& \quad- \int_S (\tfrac{\partial^{n}}{\partial x_{i_1} \partial x_{i_2}  \dots \partial x_{i_{n}}}F_j)(x,s)  \, \mu(ds) \Big| \\
& =  \bigg| \frac{1}{h}\int_0^{h}\int_{ S}\! (\tfrac{\partial^{n}}{\partial x_{i_1} \partial x_{i_2}  \dots \partial x_{i_{n}}} F_j)(x+u e_{i_{n}},s)\,  \mu(ds)\, du \numberthis\\
& \quad   - \frac{1}{h}\int_0^{h}\int_{ S}\! (\tfrac{\partial^{n}}{\partial x_{i_1} \partial x_{i_2}  \dots \partial x_{i_{n}}}  F_j)(x,s)\,  \mu(ds)\, du \bigg| \\
& \leq  \frac{1}{h}\int_0^{h} \int_{ S} \big| (\tfrac{\partial^{n}}{\partial x_{i_1}\partial x_{i_2}  \dots \partial x_{i_{n}}}  F_j)(x+u e_{i_{n}},s) - (\tfrac{\partial^{n}}{\partial x_{i_1} \partial x_{i_2}  \dots \partial x_{i_{n}}}  F_j)(x,s) \big|\,  \mu(ds) \, du <\epsilon.
\end{align*}
This demonstrates that   for all $x\in\R^d$, $i_1, i_2, \ldots, i_n \in \{1,2, \ldots, d\}$, $j \in \{1, 2, \ldots, m\}$ we have that
\begin{align}
\begin{split}
\limsup_{\substack{h \to 0\\
		h \in \R \backslash \{0\}}} \Big| &\frac{1}{h} \! \left((\tfrac{\partial^{n-1}}{\partial x_{i_1} \partial x_{i_2}  \dots \partial x_{i_{n-1}}}f_j)(x+h e_{i_{n}})-(\tfrac{\partial^{n-1}}{\partial x_{i_1} \partial x_{i_2}  \dots \partial x_{i_{n-1}}}f_j)(x)\right)\\
	&\quad - \int_S (\tfrac{\partial^{n}}{\partial x_{i_1} \partial x_{i_2}  \dots \partial x_{i_{n}}}F_j)(x,s)  \, \mu(ds) \Big| = 0.
\end{split}
\end{align}
This reveals that  for all  $x\in\R^d$, $i_1, i_2, \ldots, i_n \in \{1,2, \ldots, d\}$, $j \in \{1, 2, \ldots, m\}$  it holds that
\begin{equation}
\label{eq:partial:f:induction}
\big(\tfrac{\partial}{\partial x_{i_{n}} }(\tfrac{\partial^{n-1}}{\partial x_{i_1} \partial x_{i_2}  \dots \partial x_{i_{n-1}}}f_j)\big)(x)=\int_S (\tfrac{\partial^{n}}{\partial x_{i_1} \partial x_{i_2}  \dots \partial x_{i_{n}}}F_j)(x,s)  \, \mu(ds).
\end{equation}
Next observe that \eqref{eq:limsup:int} proves that for all $x\in\R^d$, $i_1, i_2, \ldots, i_n \in \{1,2, \ldots, d\}$, $j \in \{1, 2, \ldots, m\}$, and all functions  $u=(u_k)_{k \in \N} \colon \N \to \R^d$ with $\limsup_{k \to \infty} \|u_k\|_{\R^d}=0$ we have that
\begin{equation}
\begin{split}
&\limsup_{k \to \infty}\la \int_{S} (\tfrac{\partial^{n}}{\partial x_{i_1} \partial x_{i_2}  \dots \partial x_{i_{n}}}F_j)(x+u_k,s)\, \mu(ds)
- \int_{S} (\tfrac{\partial^{n}}{\partial x_{i_1} \partial x_{i_2}  \dots \partial x_{i_{n}}}F_j)(x,s)\, \mu(ds)\ra\\
&\leq \limsup_{k \to \infty} \int_{S}\big| (\tfrac{\partial^{n}}{\partial x_{i_1} \partial x_{i_2}  \dots \partial x_{i_{n}}}F_j)(x+u_k,s)
-(\tfrac{\partial^{n}}{\partial x_{i_1}\partial x_{i_2}  \dots \partial x_{i_{n}}}F_j)(x,s)\big|\, \mu(ds)=0.
\end{split}
\end{equation}
This reveals that for all  $i_1, i_2, \ldots, i_n \in \{1,2, \ldots, d\}$, $j \in \{1, 2, \ldots, m\}$  it holds that
\begin{equation}
\left( \R^d \ni x \mapsto \int_{S} (\tfrac{\partial^{n}}{\partial x_{i_1} \partial x_{i_2}  \dots \partial x_{i_{n}}}F_j)(x,s)\, \mu(ds) \in \R \right) \in C(\R^d, \R).
\end{equation}
Combining this, \eqref{eq:partial:f:induction} and, e.g.,  Coleman~\cite[Corollary~2.2]{Coleman2012} demonstrates that for all
$x\in\R^d$, $i_1, i_2, \ldots, i_n \in \{1,2, \ldots, d\}$, $j \in \{1, 2, \ldots, m\}$ we have that $f \in C^{n}(\R^d, \R^m)$ and
\begin{equation}
(\tfrac{\partial^{n}}{\partial x_{i_1} \partial x_{i_2}  \dots \partial x_{i_{n}}}f_j)(x)=\int_S (\tfrac{\partial^{n}}{\partial x_{i_1} \partial x_{i_2}  \dots \partial x_{i_{n}}}F_j)(x,s)  \, \mu(ds).
\end{equation}
This establishes items~\eqref{item1:diff:induction}--\eqref{item2:diff:induction}.
 The proof of Lemma~\ref{lem:diff:induction} is thus completed.
\end{proof}

\begin{lemma}\label{lem:diff^n_under_int}
Let $d, m, n\in\N$, let $(S,\mathcal{S},\mu)$ be a finite measure space, let $F   = (F(x, s))_{(x, s) \in \R^{d} \times S} \colon \R^d\times S\to\R^m$ be  $(\mathcal{B}(\R^d)\otimes \mathcal{S})$/$\mathcal{B}(\R^m)$-measurable,
let $f\colon \R^d\to\R^m$ be a function, assume  for all  $s\in S$ that $(\R^d\ni x\mapsto F(x,s)\in\R^m)\in C^n(\R^d,\R^m)$, and assume for all $x\in\R^d$ that
\begin{equation}\label{eq:diff^n_under_int_01}
\inf_{\delta\in (0,\infty)}\sup_{u\in [-\delta,\delta]^d}\int_S\|F(x,z)\|_{\R^m}
+ \textstyle\sum\limits_{k=1}^{n} \displaystyle \|(\tfrac{\partial^k}{\partial x^k}F)(x+u,z)\|_{L^{(k)}(\R^d,\R^m)}^{1+\delta}\, \mu(dz)<\infty
\end{equation}
(cf.~Corollary~\ref{cor:derivative:gen}) and   
\begin{equation}\label{eq:diff^n_under_int_03}
f(x)=\int_S F(x,s)\, \mu(ds).
\end{equation}
Then
\begin{enumerate}[(i)]
\item\label{item:diff^n_under_int_1} we have that $f\in C^n(\R^d,\R^m)$ and 
\item\label{item:diff^n_under_int_2} we have for all $k\in\lb 1, 2, \dots,n\rb$, $x\in\R^d$ that
\begin{equation}\label{eq:diff^n_under_int_04}
f^{(k)}(x)=\int_S (\tfrac{\partial^k}{\partial x^k}F)(x,s)\, \mu(ds).
\end{equation}
\end{enumerate}
\end{lemma}
\begin{proof}[Proof of \Cref{lem:diff^n_under_int}]
This is a direct consequence of  Lemma~\ref{lem:diff:induction}. The proof of  \Cref{lem:diff^n_under_int} is thus completed.
\end{proof}

\section{Existence, uniqueness, and regularity results for solutions of ordinary differential equations (ODEs)}
\label{subsec:existence}
\sectionmark{}

\begin{prop}\label{thm:regularity_flow_dynamical_system}
Let $d \in \N$, $L, T\in [0,\infty)$, $f\in C([0,T]\times \R^d,\R^d)$ satisfy for all $t\in [0,T]$, $x,y\in\R^d$  that
\begin{equation}
\label{eq:f:linear}
\|f(t,x)-f(t,y)\|_{\R^d}\leq L \|x-y\|_{\R^d}.
\end{equation}
Then there exists a unique  $\chi\in C( \lb (s,t)\in [0,T]^2\colon s\leq t\rb\times \R^d,\R^d)$ which satisfies for all $x\in\R^d$, $s\in [0,T]$, $t\in [s,T]$ that
\begin{equation}
\chi(s,t,x) = x + \int_s^tf(u,\chi(s,u,x))\, du.
\end{equation}
\end{prop}
\begin{proof}[Proof of \Cref{thm:regularity_flow_dynamical_system}]
Throughout this proof let $g \colon (-1, T+1) \times \R^d \to \R^d$ satisfy for all $s \in (-1,0)$, $t \in [0, T]$, $u \in (T, T+1)$, $x \in \R^d$ that 	$g(s,x) = f(0,x)$, $g(t,x) = f(t,x)$, and $g(u, x) = f(T, x)$. Note that the hypothesis that $f\in C([0,T]\times \R^d,\R^d)$ and \eqref{eq:f:linear} imply that for all $t\in (-1,T+1)$, $x,y\in\R^d$ we have that $g\in C((-1,T+1)\times \R^d,\R^d)$ and
\begin{equation}
\label{eq:g:linear}
\|g(t,x)-g(t,y)\|_{\R^d}\leq L \|x-y\|_{\R^d}.
\end{equation}
This and, e.g.,  Teschl~\cite[Corollary 2.6]{TeschlODE} ensure that there exists a unique  $\chi \colon \lb (s,t)\in [0,T]^2\colon s\leq t\rb\times \R^d \to \R^d$ which satisfies for all $x\in\R^d$, $s\in [0,T]$, $t\in [s,T]$ that $( [s, T] \ni u \mapsto \chi(s,u,x) \in \R^d) \in C([s,T], \R^d)$ and
\begin{equation}
\chi(s,t,x) = x + \int_s^t g(u,\chi(s,u,x))\, du = x + \int_s^t f(u,\chi(s,u,x))\, du.
\end{equation}
Moreover, observe that, e.g.,  Teschl~\cite[Theorem 2.9]{TeschlODE} assures that  $\chi\in C( \lb (s,t)\in [0,T]^2\colon s\leq t\rb\times \R^d,\R^d)$. The proof of \Cref{thm:regularity_flow_dynamical_system} is thus completed.
\end{proof}

\begin{lemma}\label{lem:regularity_flow}
	Let $d,n\in\N$, $T \in (0, \infty)$, $f\in C^{n}(\R^d,\R^d)$ and let $\theta^{\vartheta} \in C( [0, T], \R^{d})$, $\vartheta\in \R^d$,  satisfy for all $\vartheta\in\R^d$, $t\in [0, T]$ that 
	\begin{equation}
	\theta^{\vartheta}_t=\vartheta + \int_0^t f(\theta^{\vartheta}_s)\, ds.
	\end{equation}
	Then we have that 
	$
	( [0, T] \times \R^d\ni (t,\vartheta)\mapsto \theta^{\vartheta}_t\in \R^d )\in C^{n}([0, T]\times\R^d,\R^d)$.
\end{lemma}
\begin{proof}[Proof of Lemma~\ref{lem:regularity_flow}]
This a direct consequence of, e.g., Coleman~\cite[Theorem 10.3]{Coleman2012}. The proof of Lemma~\ref{lem:regularity_flow} is thus completed. 
\end{proof}

\begin{cor}\label{cor:regularity_flow}
Let $d,n\in\N$, $f\in C^{n}(\R^d,\R^d)$ and let $\theta^{\vartheta} \in C([0,\infty),\R^{d})$, $\vartheta\in \R^d$, satisfy for all $\vartheta\in\R^d$, $t\in [0,\infty)$ that 
\begin{equation}\label{eq:regularity_flow_1}
\theta^{\vartheta}_t=\vartheta + \int_0^t f(\theta^{\vartheta}_s)\, ds.
\end{equation}
Then we have that 
$
( [0,\infty)\times \R^d\ni (t,\vartheta)\mapsto \theta^{\vartheta}_t\in \R^d )\in C^{n}([0,\infty)\times\R^d,\R^d)$.
\end{cor}
\begin{proof}[Proof of \Cref{cor:regularity_flow}]
This is a direct consequence of Lemma~\ref{lem:regularity_flow}. The proof of \Cref{cor:regularity_flow} is thus completed. 
\end{proof}

\begin{lemma}
\label{lem:unique}
Let $d \in \N$, $T \in (0, \infty)$, $\vartheta \in \R^d$, $A \in C([0,T], L(\R^d, \R^d))$, let $ y_1, y_2 \colon [0, T] \to \R^d$ be $\mathcal{B}([0, T])$/$\mathcal{B}(\R^d)$-measurable,  and assume for all $ t \in [0, T]$, $i \in \{1, 2\}$ that
\begin{gather}
\label{eq:lem:finite}
\int_0^T \|y_i(s)\|_{\R^d}  \, ds < \infty \qquad \text{and} \qquad y_i(t) = \vartheta + \int_0^t A(s) y_i(s) \, ds.
\end{gather}
Then we have that $ y_1 = y_2$.
\end{lemma}
\begin{proof}[Proof of Lemma~\ref{lem:unique}]
First, note that \eqref{eq:lem:finite} and the triangle inequality ensure that 
\begin{equation}
\label{eq:finite}
\int_0^T \|y_1(s) - y_2(s)\|_{\R^d} \, ds \leq \int_0^T \|y_1(s)\|_{\R^d} + \|y_2(s)\|_{\R^d} \, ds < \infty.
\end{equation}
Next observe that \eqref{eq:lem:finite} and the triangle inequality for the Bochner integral  prove that for all $t \in [0, T]$ we have that 
\begin{equation}
\label{eq:lem:unique:Bochner}
\begin{split}
\| y_1(t) - y_2(t) \|_{\R^d} & = \left\| \int_0^t A(s)(y_1(s) -y_2(s)) \, ds \right\|_{\R^d}\\
& \leq \int_0^t \| A(s)(y_1(s) -y_2(s)) \|_{\R^d} \, ds \\
& \leq \int_0^t \|A(s)\|_{L(\R^d, \R^d)} \|y_1(s) -y_2(s)\|_{\R^d} \, ds\\
& \leq \big[ \sup\nolimits_{v \in [0, T]} \|A(v)\|_{L(\R^d, \R^d)} \big] \int_0^t \|y_1(s) - y_2(s) \|_{\R^d} \, ds.
\end{split}
\end{equation}
Moreover, note that the fact that $[0,T]$ is a compact set, the assumption that $A \in C([0,T], L(\R^d, \R^d))$, and \Cref{prop:cont_fct_bnd_on_cmpcts}  establish that 
\begin{equation}
\sup_{v \in [0,T]} \|A(v)\|_{L(\R^d,\R^d)} < \infty.
\end{equation}
Combining \eqref{eq:finite},  \eqref{eq:lem:unique:Bochner},  and the Gronwall integral inequality in Lemma \ref{lem:Gronwall} hence assures that for all $t\in [0,T]$ we have that
\begin{equation}
y_1(t)=y_2(t).
\end{equation}
The proof of Lemma~\ref{lem:unique} is thus completed.
\end{proof}

\section{Existence results for solutions of first-order   Kolmogorov  backward   PDEs}
\label{subsec:Kolmogorov}
\sectionmark{}

\begin{prop}\label{thm:Kolm_back_eq}
Let $d\in\N$, $T \in (0, \infty)$, $\psi\in C^1(\R^d,\R)$, $f\in C^{1}(\R^d,\R^d)$, let $\theta^{\vartheta}  = (\theta^{\vartheta}_t)_{t \in [0, T]} \in C( [0, T], \R^{d})$, $\vartheta\in \R^d$,  satisfy for all $\vartheta\in\R^d$, $t\in [0, T]$ that 
\begin{equation}\label{eq:Kolm_back_eq_1}
\theta^{\vartheta}_t=\vartheta + \int_0^t f(\theta^{\vartheta}_s)\, ds,
\end{equation}
and let $u  = (u(t,\vartheta))_{(t, \vartheta) \in [0, T] \times \R^d} \colon[0, T] \times \R^d\to\R$ satisfy for all $t\in [0, T]$, $\vartheta \in \R^d$ that $u(t,\vartheta)=\psi(\theta^{\vartheta}_t)$.
Then
\begin{enumerate}[(i)]
\item\label{Kolm_back_eq_item_1} we have that $u\in C^{1}([0, T]\times \R^d,\R)$ and
\begin{equation}
\label{eq:C1:theta}
\big( [0, T] \times \R^d\ni (t,\vartheta)\mapsto \theta^{\vartheta}_t\in \R^d \big)\in C^{1}([0, T]\times\R^d,\R^d),
\end{equation}
\item\label{Kolm_back_eq_item_1_2} we have for all $t\in[0, T]$, $\vartheta\in \R^d$ that $f(\theta_t^{\vartheta})=(\pv \theta^{\vartheta}_t) f(\vartheta)$,
 and
\item\label{Kolm_back_eq_item_2} we have for all $t\in[0, T]$, $\vartheta\in \R^d$ that
\begin{equation}\label{eq:Kolm_back_eq_3}
(\tfrac{\partial}{\partial t}u )(t,\vartheta) = (\tfrac{\partial}{\partial \vartheta} u)(t,\vartheta) f(\vartheta)
=\langle (\nabla_{\vartheta} u)(t,\vartheta), f(\vartheta)\rangle_{\R^d} .
\end{equation}
\end{enumerate}
\end{prop}
\begin{proof}[Proof of  \Cref{thm:Kolm_back_eq}]
First, observe that
Lemma~\ref{lem:regularity_flow},  the assumption that $f\in C^{1}(\R^d,\R^d)$, and the assumption that $\psi\in C^1(\R^d,\R)$ establish  item~\eqref{Kolm_back_eq_item_1}. 
Next  let $y_{\vartheta}\colon[0,T]\to \R^d$, $\vartheta\in\R^d$,  satisfy for all $\vartheta\in\R^d$, $t\in [0,T]$ that
\begin{equation}\label{eq:Kolm_back_eq_7}
y_{\vartheta}(t) = (\pv \theta^{\vartheta}_t) f(\vartheta),
\end{equation}
and let $z_{\vartheta}\colon [0,T]\to \R^d$, $\vartheta\in\R^d$, satisfy for all $\vartheta\in\R^d$, $t\in [0,T]$ that
\begin{equation}\label{eq:Kolm_back_eq_9}
z_{\vartheta}(t)=f(\theta^{\vartheta}_t).
\end{equation}
Observe that the assumption that $f\in C^1(\R^d,\R^d)$ and \eqref{eq:C1:theta} ensure that for all $\vartheta\in\R^d$ we have that $y_{\vartheta}, z_{\vartheta} \in C([0,T], \R^d)$. This and \Cref{prop:cont_fct_bnd_on_cmpcts} prove that for all $\vartheta\in\R^d$ we have that
\begin{equation}
\label{eq:finite:y:z}
\int_0^T \|y_{\vartheta}(s)\|_{\R^d} + \|z_{\vartheta}(s)\|_{\R^d} \, ds \leq T \sup_{s \in [0,T]} \|y_{\vartheta}(s)\|_{\R^d}  + T \sup_{s \in [0,T]} \|z_{\vartheta}(s)\|_{\R^d} < \infty.
\end{equation}
Next note that the assumption that $f\in C^1(\R^d,\R^d)$ and \eqref{eq:C1:theta} demonstrate that for all $\vartheta\in\R^d$ we have that
\begin{equation}
\begin{split}
\big( [0, T] \times \R^d\ni (t,h)\mapsto f'(\theta^{\vartheta +h}_t)(\pv \theta^{\vartheta+h}_t )
 \in L(\R^d, \R^d) \big) \in C([0, T]\times\R^d, L(\R^d, \R^d)).
 \end{split}
\end{equation}
\Cref{prop:cont_fct_bnd_on_cmpcts} hence assures that for all $\vartheta\in\R^d$ we have that 
\begin{equation}
\sup_{(s, h)\in [0,T]\times [-1,1]^d}\|f'(\theta^{\vartheta+h}_s)(\pv \theta^{\vartheta+h}_s )\|_{L(\R^d,\R^d)}<\infty.
\end{equation}
This ensures that for all $\vartheta\in\R^d$, $t\in [0,T]$ we have that
\begin{equation}
\sup_{h\in [-1,1]^d}\int_0^t\|f'(\theta^{\vartheta+h}_s)(\pv \theta^{\vartheta+h}_s )\|_{L(\R^d,\R^d)}^2\, ds<\infty.
\end{equation}
Lemma~\ref{lem:diff:induction}
and \eqref{eq:Kolm_back_eq_1}  hence imply that for all $\vartheta\in\R^d$, $t\in [0,T]$ we have that 
\begin{equation}\label{eq:Kolm_back_eq_4}
\pv \theta^{\vartheta}_t = \operatorname{Id}_{\R^d} + \int_0^t f'(\theta^{\vartheta}_s)(\pv \theta^{\vartheta}_s )\, ds.
\end{equation}
This reveals that for all $\vartheta\in\R^d$, $t\in [0,T]$  it holds that 
\begin{equation}\label{eq:Kolm_back_eq_6}
(\pv \theta^{\vartheta}_t) f(\vartheta) = f(\vartheta) + \int_0^t  f'(\theta^{\vartheta}_s)(\pv \theta^{\vartheta}_s)f(\vartheta)\, ds.
\end{equation}
This and \eqref{eq:Kolm_back_eq_7}  assure that for all $\vartheta\in\R^d$, $t\in [0,T]$ we have that 
\begin{equation}\label{eq:Kolm_back_eq_8}
\begin{split}
y_{\vartheta}(t) &=f(\vartheta)+\int_0^t f'(\theta^{\vartheta}_s) y_{\vartheta}(s)\, ds.
\end{split}
\end{equation}
Moreover, observe that
the assumption that $f\in C^1(\R^d,\R^d)$ and \eqref{eq:C1:theta}
prove that  for all $\vartheta\in\R^d$, $t\in [0,T]$ we have that
\begin{equation}
z_{\vartheta} \in C^1([0,T], \R^d).
\end{equation}
The fundamental theorem of calculus
hence demonstrates that for all $\vartheta\in\R^d$, $t\in [0,T]$ we have that
\begin{equation}
\begin{split}
z_{\vartheta}(t) &= z_{\vartheta}(0) + \int_0^t  z'_{\vartheta}(s)\, ds.
\end{split}
\end{equation}
Combining this and
\eqref{eq:Kolm_back_eq_1}  ensures that for all $\vartheta\in\R^d$, $t\in [0,T]$ we have that
\begin{equation}\label{eq:Kolm_back_eq_10}
\begin{split}
z_{\vartheta}(t) 
= z_{\vartheta}(0) + \int_0^t f'(\theta^{\vartheta}_s) f(\theta_s^{\vartheta})\, ds = z_{\vartheta}(0) + \int_0^t f'(\theta^{\vartheta}_s)z_{\vartheta}(s)\, ds.
\end{split}
\end{equation}
Furthermore, note that the assumption that $f\in C^1(\R^d,\R^d)$ and \eqref{eq:C1:theta} imply that for all $\vartheta\in\R^d$ we have that
\begin{equation}
\big([0,T]\ni s \mapsto f'(\theta^{\vartheta}_s)\in L(\R^d,\R^d)\big) \in C([0,T], L(\R^d, \R^d)).
\end{equation}
This, \eqref{eq:Kolm_back_eq_8}, \eqref{eq:Kolm_back_eq_10}, the fact that $\forall \, \vartheta \in \R^d \colon z_{\vartheta}(0)=f(\vartheta)$, \eqref{eq:finite:y:z}, and Lemma~\ref{lem:unique} demonstrate that for all $\vartheta\in\R^d$, $t\in [0,T]$ we have that
\begin{equation}
\label{eq:y=z}
y_{\vartheta}(t)=z_{\vartheta}(t).
\end{equation}
Combining this with the chain rule and the assumption that $\forall \, t\in [0,T]$, $\vartheta\in\R^d \colon u(t,\vartheta)=\psi(\theta^{\vartheta}_t)$  proves that for all $\vartheta\in\R^d$, $t\in [0,T]$ we have that
\begin{equation}\label{eq:Kolm_back_eq_13}
\begin{split}
(\tfrac{\partial}{\partial t}u)(t,\vartheta) 
& = \psi'(\theta_t^{\vartheta}) (\pt \theta^{\vartheta}_t) = \psi'(\theta_t^{\vartheta}) f(\theta^{\vartheta}_t)= \psi'(\theta_t^{\vartheta}) z_{\vartheta}(t)\\
& = \psi'(\theta_t^{\vartheta}) y_{\vartheta}(t) = \psi'(\theta_t^{\vartheta}) (\pv \theta^{\vartheta}_t) f(\vartheta)  = (\tfrac{\partial}{\partial \vartheta} u)(t,\vartheta) f(\vartheta).
\end{split}
\end{equation}
This and \eqref{eq:y=z} establish item~\eqref{Kolm_back_eq_item_1_2} and item~\eqref{Kolm_back_eq_item_2}. The proof of \Cref{thm:Kolm_back_eq} is thus completed.
\end{proof}

\chapter{Weak error estimates for stochastic approximation algorithms (SAAs) in the case of general learning rates}
\label{sec:SAA:general}
\chaptermark{}

 In this chapter we use the analysis for first-order  Kolmogorov  backward PDEs from Chapter~\ref{section:existence} above to study weak approximation errors for SAAs in the case of general learning rates. In particular, we establish in \Cref{prop:main_theorem}  in Section~\ref{subsec:weak:SAA:general} below weak error estimates for SAAs in the case of general learning rates with mini-batches. Our proof of \Cref{prop:main_theorem}  employs the well-known results on the possibility of interchanging derivatives and expectations in  Section~\ref{subsec:sufficient:exp} below,  the essentially well-known spatial regularity results for flows of certain deterministic ODEs in Section~\ref{subsec:spatial} below,  the auxiliary intermediate results on upper bounds for second-order spatial derivatives of certain deterministic flows in Section~\ref{subsec:upper} below,
   the elementary   temporal regularity result for SAAs  in \Cref{lem:partial_t_of_Theta}  in Section~\ref{subsec:temporal} below, and the auxiliary intermediate results  on a priori estimates for SAAs  in Section~\ref{subsec:apriori} below. In Section~\ref{subsec:SAA:poly} below we combine \Cref{prop:main_theorem}  and the elementary auxiliary results on upper bounds for integrals of certain exponentially decaying functions in Section~\ref{subsec:upper:integrals} below to establish in \Cref{cor:sp_case_gamma} below  weak error estimates for SAAs in the case of polynomially decaying learning rates with mini-batches. In Setting~\ref{sec:setting_1}  in Section~\ref{subsec:setting:3} below we present a mathematical framework for describing SAAs  in the case of general learning rates. In the results of this chapter  we frequently employ Setting~\ref{sec:setting_1}.

\section{Mathematical description for SAAs in the case of general learning rates}
\label{subsec:setting:3}
\sectionmark{}

\begin{setting}
	\label{sec:setting_1}
	
	Let $d\in \N$, $ \xi\in \R^d$, let $(S,\mathcal{S})$ be a measurable space, let $(\Omega,\F,\P)$ be a probability space, let $Z_n\colon\Omega \to S$, $n\in\N$, be i.i.d.\ random variables, for every set $A$ let $\#_A\in \N_0\cup\lb\infty\rb$ be the
	number of elements of $A$, let $\gamma\colon[0,\infty)\to \lb A \subseteq \N\colon\#_A<\infty \rb$ satisfy for all $t\in[0,\infty)$ that $0< \#_{\lb s \in [0,t]\colon\gamma(s)\neq \emptyset\rb}<\infty$,
	let  $G   = (G(x, s))_{(x, s) \in \R^d \times S} \colon \R^d\times S\to \R^d$ be  $(\mathcal{B}(\R^d)\otimes \mathcal{S})$/$\mathcal{B}(\R^d)$-measurable, assume for all $s \in S$ that $( \R^d\ni x \mapsto G(x,s)\in\R^d )\in C^2(\R^d,\R^d)$, assume for all  $x\in\R^d$ that 
	\begin{equation}\label{eq:setting_4}
	\max_{i \in \{1,2\}} \inf_{\delta\in(0,\infty)}\sup_{u\in [-\delta,\delta]^d}  \E\Big[   \|G(x,Z_1)\|_{\R^d} +  \| (\tfrac{\partial^i}{\partial x^i}G)(x+u,Z_1)\|_{L^{(i)}(\R^d,\R^d)}^{1+\delta} \Big]<\infty
	\end{equation}
	(cf.~Corollary~\ref{cor:derivative:gen}), let $g \colon \mathbb{R}^d \to \R^d$ satisfy for all $x\in \R^d$ that $	g(x)=\E[ G(x,Z_1) ]$, 
	let $\theta^{\vartheta} = (\theta^{\vartheta}_t)_{t \in [0, \infty)}\in C( [0,\infty), \R^d)$, $\vartheta\in\R^d$, satisfy for all $t\in [0,\infty)$, $\vartheta\in\R^d$ that
	\begin{equation}\label{eq:equation theta}
	\theta_t^{\vartheta}= \vartheta + \int_0^t g(\theta_s^{\vartheta})\, ds
	\end{equation}
	(cf.~item~\eqref{item1:diff:induction} in Lemma~\ref{lem:diff:induction}), let $\floorgrid{\cdot}\colon[0,\infty)\to [0,\infty)$ satisfy for all $t\in [0,\infty)$ that
	\begin{equation}\label{eq:setting_2}
	\floorgrid{t}=\max\lb s\in [0,t]\colon \gamma(s)\neq \emptyset \rb,
	\end{equation}
	and let $\Theta = (\Theta_t(\omega))_{(t, \omega) \in [0, \infty) \times \Omega} \colon [0,\infty)\times \Omega \to \R^d$ be the stochastic process with continuous sample paths (w.c.s.p.) which satisfies for all $t\in (0,\infty)$  that $\Theta_0 = \xi$ and 
	\begin{equation}\label{eq: equation Theta 2}
	\Theta_t = \Theta_{\floorgrid{t}} + \tfrac{(t-\floorgrid{t})}{\#_{\gamma(\floorgrid{t})}}\Big[ \textstyle\sum_{n\in \gamma(\floorgrid{t})}G(\Theta_{\floorgrid{t}},Z_n)\Big].
	\end{equation}
	
\end{setting}

\section{Sufficient conditions for interchanging derivatives and expectations}
\label{subsec:sufficient:exp}
\sectionmark{}

\begin{lemma}\label{lem:f_C^1}
	Assume Setting~\ref{sec:setting_1}. Then  
	\begin{enumerate}[(i)]
		\item\label{item:f_C^1_1} we have that $g\in C^1(\R^d,\R^d)$ and
		\item\label{item:f_C^1_2} we have for all $x\in \R^d$  that 
		$
		g'(x)=\E [ (\tfrac{\partial}{\partial x}G)(x,Z_1) ]
		$.
	\end{enumerate}
\end{lemma}
\begin{proof}[Proof of Lemma~\ref{lem:f_C^1}]
	This is a direct consequence of Lemma~\ref{lem:diff:induction}. The proof of Lemma~\ref{lem:f_C^1} is thus completed.
\end{proof}

\begin{lemma}\label{lem:f_C^n}
	Assume Setting~\ref{sec:setting_1}, let $n\in\N$, assume for all $s\in S$ that $(\R^d\ni y\mapsto G(y,s)\in \R^d)\in C^n(\R^d,\R^d)$, and assume for all $x\in\R^d$ that
	\begin{equation}\label{eq:f_C^n_0}
	\max_{i \in \{1,2, \ldots, n\}} \inf_{\delta\in (0,\infty)}\sup_{u\in [-\delta,\delta]^d}  \E \Big[ \| (\tfrac{\partial^i}{\partial x^i }G)(x+u,Z_1)\|_{L^{(i)}(\R^d,\R^d)}^{1+\delta}\Big]<\infty.
	\end{equation}
	Then  
	\begin{enumerate}[(i)]
		\item\label{item:f_C^n_1} we have that $g\in C^n(\R^d,\R^d)$ and
		\item\label{item:f_C^n_2} we have for all $x\in \R^d$ that $\E[\|(\tfrac{\partial^n}{\partial x^n}G)(x,Z_1)\|_{L^{(n)}(\R^d,\R^d)}]< \infty$ and
		\begin{equation}\label{eq:lem_f_C^n_1}
		g^{(n)}(x)=\E\!\left[ (\tfrac{\partial^n}{\partial x^n}G)(x,Z_1) \right].
		\end{equation}
	\end{enumerate}
\end{lemma}
\begin{proof}[Proof of Lemma~\ref{lem:f_C^n}]
	Lemma~\ref{lem:diff^n_under_int} (with $d=d$, $n=n$, $(S,\mathcal{S},\mu) = (\Omega,\F,\P) $, $ F = (\R^d \times \Omega \ni (y, \omega) \mapsto G(y, Z_1(\omega)) \in \R^d)$, $f=g$, $k =n$    in the notation of Lemma~\ref{lem:diff^n_under_int}) establishes item~\eqref{item:f_C^n_1} and item~\eqref{item:f_C^n_2}. The proof of Lemma~\ref{lem:f_C^n} is thus completed.
\end{proof}

\section{Spatial regularity results for flows of deterministic ODEs}
\label{subsec:spatial}
\sectionmark{}

\begin{lemma}\label{lem:flow_deriv_1}
	Let $d \in \N$, $L \in \R$, $T \in (0, \infty)$,   let $g \in C^1(\R^d, \R^d)$ satisfy for all $x,y\in \R^d$   that 
	\begin{equation}\label{eq:condition_f}
	\langle g(x)-g(y), x-y\rangle_{\R^d} \leq L \|x-y\|_{\R^d}^2,
	\end{equation}
	and let  $\theta^{\vartheta}\in C( [0, T],\R^d)$, $\vartheta\in\R^d$,  satisfy for all $t\in [0, T]$, $\vartheta\in\R^d$ that
	\begin{equation}\label{eq:theta_f}
	\theta_t^{\vartheta}= \vartheta + \int_0^t g(\theta_s^{\vartheta})\, ds.
	\end{equation} 	
	Then we have for all $x,y\in \R^d$, $t\in [0, T]$ that
	$
	\|\theta_t^{x}-\theta_t^{y}\|_{\R^d}\leq \|x-y\|_{\R^d}\exp(Lt).
	$
\end{lemma}
\begin{proof}[Proof of Lemma~\ref{lem:flow_deriv_1}]
	Throughout this proof  let $E^{x,y} = (E^{x,y}(t))_{t \in [0,T]} =(E^{x,y}_t)_{t \in [0,T]}  \colon \allowbreak [0, T] \to [0,\infty)$, $x,y\in\R^d$,  satisfy for all $x,y\in\R^d$, $t\in [0,T]$ that
	\begin{equation}\label{eq:def_e}
	E^{x,y}_t=\|\theta_t^{x}-\theta_t^{y}\|_{\R^d}^2.
	\end{equation}
	Note that the assumption that $g \in C^1(\R^d, \R^d)$, \eqref{eq:theta_f}, and Lemma~\ref{lem:regularity_flow} ensure that for all $x \in \R^d$ we have that $\theta^x \in C^1([0,T], \R^d)$. This and \eqref{eq:condition_f} prove that for all $x, y \in \R^d$, $t\in [0,T]$ we have that $E^{x,y} \in C^1([0, T], [0,\infty))$ and
	\begin{equation}
	\begin{split}
	( E^{x,y})'(t)& = 2\langle g(\theta_t^{x}) - g(\theta_t^{y}),\theta_t^{x} - \theta_t^{y}  \rangle_{\R^d} \leq 2L \| \theta_t^{x} - \theta_t^{y} \|_{\R^d}^2 = 2L E^{x,y}_t.
	\end{split}
	\end{equation}
	The Gronwall differential inequality in Lemma~\ref{lem:Gronwall_differential} hence assures that for all $x, y \in \R^d$, $t\in [0,T]$ we have that
	\begin{equation}
	\|\theta_t^{x}-\theta_t^{y}\|_{\R^d}^2 = E^{x,y}_t  \leq E^{x,y}_0 \, e^{2Lt} = \|x-y\|_{\R^d}^2e^{2Lt}.
	\end{equation}
	The proof of Lemma~\ref{lem:flow_deriv_1} is thus completed.
\end{proof}

\begin{lemma}\label{lem:flow_deriv_2}
	Let $d \in \N$, $L \in \R$, $T \in (0, \infty)$, $g \in C^1(\R^d, \R^d)$
	and let  $\theta^{\vartheta} \in C( [0, T], \R^d)$, $\vartheta\in\R^d$,  satisfy for all $t\in [0, T]$, $x, y \in\R^d$ that
	\begin{equation}\label{eq:theta_f_2}
	\theta_t^{x}= x + \int_0^t g(\theta_s^{x})\, ds
	\end{equation} 	
	and
	$
	\|\theta_t^{x}-\theta_t^{y}\|_{\R^d}\leq \|x-y\|_{\R^d}\exp(Lt)
	$.
	Then we have for all $x,y\in \R^d$   that 
	$
	\langle g(x)-g(y), x-y\rangle_{\R^d} \leq L \|x-y\|_{\R^d}^2
	$.
\end{lemma}
\begin{proof}[Proof of Lemma~\ref{lem:flow_deriv_2}]
	Throughout this proof  let $E^{x,y} = (E^{x,y}(t))_{t \in [0,T]} =(E^{x,y}_t)_{t \in [0,T]}  \colon \allowbreak [0, T] \to [0,\infty)$, $x,y\in\R^d$,  satisfy for all $x,y\in\R^d$, $t\in [0, T]$ that
	\begin{equation}
	E^{x,y}_t=\|\theta_t^{x}-\theta_t^{y}\|_{\R^d}^2.
	\end{equation}
	Observe  that for all $x,y\in\R^d$, $h\in (0,T]$ we have that
	\begin{equation}\label{eq:theta_h}
	\begin{split}
	&\frac{\|\theta_{h}^x-\theta_h^y\|_{\R^d}^2-\norm{\theta_0^x-\theta_0^y}_{\R^d}^2}{h}  = \frac{\|\theta_{h}^x-\theta_h^y\|_{\R^d}^2-\|x-y\|_{\R^d}^2}{h}\\
	& \leq \frac{\|x-y\|_{\R^d}^2e^{2Lh}-\|x-y\|_{\R^d}^2e^{2L\cdot 0}}{h}
	 = \|x-y\|_{\R^d}^2 \! \left[ \frac{e^{2Lh}-e^{2L\cdot 0}}{h} \right].
	\end{split}
	\end{equation}
	Next note that the assumption that $g \in C^1(\R^d, \R^d)$, \eqref{eq:theta_f_2}, and Lemma~\ref{lem:regularity_flow} imply that for all $x \in \R^d$ we have that $\theta^x \in C^1([0, T], \R^d)$. This and \eqref{eq:theta_h} prove that for all $x, y \in \R^d$, $t\in [0, T]$ we have that $E^{x,y} \in C^1([0, T], [0,\infty))$ and
	\begin{equation}
	\begin{split}
	(  E^{x,y})'(0) &= \lim_{\substack{h \to 0\\
			h \in (0, \infty)}} \frac{\|\theta_{h}^x-\theta_h^y\|_{\R^d}^2-\norm{\theta_0^x-\theta_0^y}_{\R^d}^2}{h}\\
		  &\leq \|x-y\|_{\R^d}^2 \! \left[ \lim_{\substack{h \to 0\\
			h \in (0, \infty)}} \frac{e^{2Lh}-e^{2L\cdot 0}}{h} \right] 
	= 2L\|x-y\|_{\R^d}^2.
	\end{split}
	\end{equation}
	This reveals that for all $x, y \in \R^d$  it holds that
	\begin{equation}
	\begin{split}
	2\langle x-y, g(x)-g(y) \rangle_{\R^d}
	& = \Big[ 2\langle \theta_t^x-\theta_t^y,g(\theta_t^x)-g(\theta_t^y) \rangle_{\R^d}\Big]_{t=0} \\
	&= (E^{x,y})'(0) \leq 2L \|x-y\|_{\R^d}^2.
	\end{split}
	\end{equation}
	The proof of Lemma~\ref{lem:flow_deriv_2} is thus completed.
\end{proof}

\begin{cor}\label{cor:flow_deriv}
	Let $d \in \N$, $L \in \R$, $T \in (0, \infty)$,  $g \in C^1(\R^d, \R^d)$
	and let  $\theta^{\vartheta}\in C( [0, T],\R^d)$, $\vartheta\in\R^d$,  satisfy for all $t\in [0, T]$, $\vartheta \in\R^d$ that
	\begin{equation}
	\theta_t^{\vartheta}= \vartheta + \int_0^t g(\theta_s^{\vartheta})\, ds.
	\end{equation} 	
	Then the following two statements are equivalent:
	\begin{enumerate}[(i)]
		\item\label{item:special_condition_f_1} It holds for all $x,y\in \R^d$   that 
		$
		\langle g(x)-g(y), x-y\rangle_{\R^d} \leq L \|x-y\|_{\R^d}^2
		$.
		\item\label{item:special_condition_f_2} It holds for all $x,y\in \R^d$, $t\in [0, T]$ that
		$
		\|\theta_t^{x}-\theta_t^{y}\|_{\R^d}\leq \|x-y\|_{\R^d}\exp(Lt)
		$.
	\end{enumerate}
\end{cor}
\begin{proof}[Proof of Corollary~\ref{cor:flow_deriv}]
	Observe that Lemma~\ref{lem:flow_deriv_1} ensures that $ (\eqref{item:special_condition_f_1} \Rightarrow \eqref{item:special_condition_f_2})$. Moreover, note that Lemma \ref{lem:flow_deriv_2} establishes  that $ (\eqref{item:special_condition_f_2} \Rightarrow \eqref{item:special_condition_f_1})$.
	The proof of Corollary~\ref{cor:flow_deriv} is thus completed.
\end{proof}

\begin{lemma}\label{lem:exp_decay_partial_vartheta}
	Assume Setting~\ref{sec:setting_1}  and let $L\in \R$  satisfy for all $x,y\in\R^d$ that
	\begin{equation}\label{eq:exp_decay_partial_vartheta_1}
	\langle g(x)-g(y), x-y\rangle_{\R^d} \leq L \|x-y\|_{\R^d}^2.
	\end{equation}
	Then 
	\begin{enumerate}[(i)]
		\item\label{item:exp_decay_partial_vartheta_1} we have that
		$
		( [0,\infty)\times \R^d\ni (t,\vartheta)\mapsto \theta^{\vartheta}_t\in \R^d )\in C^{1}([0,\infty)\times\R^d,\R^d)
		$
		and
		\item\label{item:exp_decay_partial_vartheta_2} we have for all $\vartheta\in\R^d$, $t\in [0,\infty)$ that
		$
		\| (\tfrac{\partial}{\partial \vartheta}\theta_t^{\vartheta})  \|_{L(\R^d, \R^d)} \leq \exp(Lt)
		$.
	\end{enumerate}
\end{lemma}
\begin{proof}[Proof of Lemma~\ref{lem:exp_decay_partial_vartheta}]
	First, observe that  Corollary~\ref{cor:regularity_flow} and Lemma~\ref{lem:f_C^1} prove item~\eqref{item:exp_decay_partial_vartheta_1}.
	Next note that item~\eqref{item:exp_decay_partial_vartheta_1} and the triangle inequality imply that for all $\vartheta\in\R^d$, $t\in [0,\infty)$ we have that
	\begin{align*}
	\label{eq:exp_decay_partial_vartheta_3}
	&\limsup_{\substack{h \to 0\\
			h \in \R^d \backslash \{0\}}}\Big\| \tfrac{1}{\|h\|_{\R^d}}( \pv \theta_t^{\vartheta} )h \Big\|_{\R^d}\\
	&\leq\limsup_{\substack{h \to 0\\
			h \in \R^d \backslash \{0\}}}\Big[\norm{ \tfrac{1}{\|h\|_{\R^d}}( \pv \theta_t^{\vartheta} )h - \tfrac{1}{\|h\|_{\R^d}}(\theta_t^{\vartheta + h}-\theta_t^{\vartheta})}_{\R^d} + \norm{ \tfrac{1}{\|h\|_{\R^d}}(\theta_t^{\vartheta + h}-\theta_t^{\vartheta})}_{\R^d}\Big]\\
	&\leq\limsup_{\substack{h \to 0\\
			h \in \R^d \backslash \{0\}}}\norm{ \tfrac{1}{\|h\|_{\R^d}}( \pv \theta_t^{\vartheta} )h - \tfrac{1}{\|h\|_{\R^d}}(\theta_t^{\vartheta + h}-\theta_t^{\vartheta})}_{\R^d} 
	+ \limsup_{\substack{h \to 0\\
			h \in \R^d \backslash \{0\}}}\norm{ \tfrac{1}{\|h\|_{\R^d}}(\theta_t^{\vartheta + h}-\theta_t^{\vartheta})}_{\R^d}\\
	& = \limsup_{\substack{h \to 0\\
			h \in \R^d \backslash \{0\}}}\norm{ \tfrac{1}{\|h\|_{\R^d}}(\theta_t^{\vartheta + h}-\theta_t^{\vartheta})}_{\R^d}.  \numberthis
	\end{align*}
	This, \eqref{eq:exp_decay_partial_vartheta_1}, and Corollary~\ref{cor:flow_deriv} assure that for all $\vartheta\in\R^d$, $t\in [0,\infty)$ we have that
	\begin{equation}\label{eq:exp_decay_partial_vartheta_4}
	\begin{split}
	\limsup_{\substack{h \to 0\\
			h \in \R^d \backslash \{0\}}}\norm{ \tfrac{1}{\|h\|_{\R^d}}(\pv \theta_t^{\vartheta})h }_{\R^d}
	& \leq e^{Lt}.
	\end{split}
	\end{equation}
	This reveals that for all $\vartheta\in\R^d$, $v\in\R^d\backslash\lb 0\rb$, $t\in [0,\infty)$  it holds that
	\begin{equation}\label{eq:exp_decay_partial_vartheta_5}
	\begin{split}
	\| (\tfrac{\partial}{\partial \vartheta}\theta_t^{\vartheta}) v \|_{\R^d}
	 &= \limsup_{\substack{\lambda \to 0\\
			\lambda \in \R \backslash \{0\}}}\left\| \tfrac{1}{\lambda}( \pv\theta_t^{\vartheta}) \lambda v \right\|_{\R^d}\\
	& = \norm{v}_{\R^d}\limsup_{\substack{\lambda \to 0\\
			\lambda \in \R \backslash \{0\}}}\left\| \tfrac{1}{\lambda}( \pv\theta_t^{\vartheta}) \lambda \tfrac{v}{\norm{v}}_{\R^d} \right\|_{\R^d} \leq \norm{v}_{\R^d} e^{Lt}.
	\end{split}
	\end{equation}
	This establishes item~\eqref{item:exp_decay_partial_vartheta_2}. 
	The proof of Lemma~\ref{lem:exp_decay_partial_vartheta} is thus completed.
\end{proof}

\begin{lemma}\label{lem:Markov_ppty}
	Assume Setting~\ref{sec:setting_1}. Then we have for all $a,b\in [0,\infty)$, $\vartheta\in\R^d$ that
	$
	\theta_{b}^{(\theta_{a}^{\vartheta})}=\theta_{a+b}^{\vartheta}
	$.
\end{lemma}
\begin{proof}[Proof of Lemma~\ref{lem:Markov_ppty}]
	Throughout this proof let $a,b\in [0,\infty)$, $\vartheta\in \R^d$ and let $e\colon [0,b] \allowbreak \to \R$ satisfy for all $s\in [0,b]$ that
	\begin{equation}\label{eq:Markov_ppty_2}
	e(s)=\big\|\theta_{a+s}^{\vartheta}-\theta_s^{(\theta_{a}^{\vartheta})}\big\|_{\R^d}.
	\end{equation}
	Observe that the fact that  $ ([0,b]\ni t\mapsto \theta_{a+t}^{\vartheta} \in \R^d)$ and $([0,b]\ni t\mapsto \theta_t^{(\theta_a^{\vartheta})} \in \R^d)$ are continuous
	implies that there exists a non-empty convex compact set $K\subseteq \R^d$ which satisfies   for all $s\in [0,b]$  that
	\begin{equation}\label{eq:Markov_ppty_3}
	\theta_{a+s}^{\vartheta}\in K \qquad \text{and} \qquad \theta_s^{(\theta_{a}^{\vartheta})}\in K.
	\end{equation}
	Moreover, note that \Cref{prop:cont_fct_bnd_on_cmpcts} and Lemma~\ref{lem:f_C^1} assure that 
	\begin{equation}\label{eq:Markov_ppty_3_p}
	\sup_{x\in K}\|g'(x)\|_{L(\R^d,\R^d)}<\infty.
	\end{equation}
	In the next step we observe that \eqref{eq:equation theta} proves that for all $s\in [0,b]$ we have that
	\begin{align}\label{eq:Markov_ppty_4}
	\begin{split}
	e(s) & = \big\|\theta_{a+s}^{\vartheta}-\theta_s^{(\theta_{a}^{\vartheta})}\big\|_{\R^d}\\
	& = \norm{ \vartheta + \int_0^{a+s}g\big(\theta_u^{\vartheta}\big)\, du - \left( \theta_{a}^{\vartheta}+\int_0^s g\big(\theta_u^{(\theta_{a}^{\vartheta})}\big)\, du \right) }_{\R^d}\\
	& = \norm{ \vartheta + \int_0^{a+s}g\big(\theta_u^{\vartheta}\big)\, du - \left( \vartheta+\int_0^{a}g\big(\theta_u^{\vartheta}\big)\,du +\int_0^s g(\theta_u^{(\theta_{a}^{\vartheta})})\, du \right) }_{\R^d}\\
	& = \norm{ \int_{a}^{a+s}g\big(\theta_u^{\vartheta}\big)\, du - \int_0^s g\big(\theta_u^{(\theta_{a}^{\vartheta})}\big)\, du  }_{\R^d}\\
	& = \norm{\int_{0}^{s}g(\theta_{a+u}^{\vartheta}) - g\big(\theta_u^{(\theta_{a}^{\vartheta})}\big)\, du }_{\R^d}.
	\end{split}
	\end{align}
	This, the triangle inequality for the Bochner integral, and the mean value inequality demonstrate that for all $s\in [0,b]$ we have that
	\begin{equation}
	\begin{split}
	e(s)  & \leq  \int_{0}^{s} \big\| g\big(\theta_{a+u}^{\vartheta}\big) - g\big(\theta_u^{(\theta_{a}^{\vartheta})}\big)\big\|_{\R^d} \, du  \\
	& \leq \int_0^s \left[ \sup_{x\in K}\|g'(x)\|_{L(\R^d,\R^d)} \right] \big\|\theta_{a+u}^{\vartheta}-\theta_u^{(\theta_{a}^{\vartheta})}\big\|_{\R^d}\, du\\
	& = \sup_{x\in K}\|g'(x)\|_{L(\R^d,\R^d)} \int_0^s e(u)\, du.
	\end{split}
	\end{equation}
	The Gronwall integral inequality in Lemma~\ref{lem:Gronwall} and \eqref{eq:Markov_ppty_3_p} hence assure that  for all $s\in [0,b]$ we have that
	\begin{equation}\label{eq:Markov_ppty_5}
	e(s)=0.
	\end{equation}
	The proof of Lemma~\ref{lem:Markov_ppty} is thus completed.
\end{proof}

\begin{lemma}\label{lem:motionless_pt_f}
	Assume Setting~\ref{sec:setting_1}  and let $L\in (0,\infty)$ satisfy for all $x,y\in \R^d$ that
	\begin{equation}\label{eq:motionless_pt_f_1}
	\langle g(x)-g(y), x-y\rangle_{\R^d} \leq -L \|x-y\|_{\R^d}^2.
	\end{equation}
	Then  
	\begin{enumerate}[(i)]
		\item\label{item:motionless_pt_f_1} we have that there exists a unique $\Xi\in \R^d$ which satisfies  that
		$
		g(\Xi)=0
		$,
		\item\label{item:motionless_pt_f_1_2} we have for all $t\in [0,\infty)$ that
		$
		\theta_t^{\Xi}=\Xi
		$,
		\item\label{item:motionless_pt_f_2} we have for all $h\in \R^d$, $t\in[0,\infty)$ that
		$
		\|\theta_t^{\Xi+h}-\Xi\|_{\R^d}\leq \|h\|_{\R^d} \exp(-Lt)
		$,
		and
		\item\label{item:motionless_pt_f_3} we have for all $x\in\R^d$ that
		$
		\limsup_{t\to\infty}\|\theta_t^{x}-\Xi\|_{\R^d}=0
		$.
	\end{enumerate}
\end{lemma}
\begin{proof}[Proof of Lemma~\ref{lem:motionless_pt_f}]
	First, observe that Lemma~\ref{lem:f_C^1}, Corollary~\ref{cor:flow_deriv}, and \eqref{eq:motionless_pt_f_1} imply that for all $t\in [0,\infty)$, $x, y \in \R^d$ we have that
	\begin{equation}\label{eq:motionless_pt_f_5}
	\| \theta_t^x - \theta_t^y \|_{\R^d} \leq \|x-y\|_{\R^d} e^{-Lt}.
	\end{equation}
	This and the Banach fixed point theorem demonstrate that there exists a unique function $\beta\colon (0,\infty)\to\R^d$ which satisfies for all $t\in (0,\infty)$ that
	\begin{equation}\label{eq:motionless_pt_f_6}
	\theta^{\beta_t}_t=\beta_t.
	\end{equation}
	Next we claim that for all $n\in\N$, $t\in (0,\infty)$ we have that
	\begin{equation}\label{eq:motionless_pt_f_9}
	\beta_{nt}=\beta_t.
	\end{equation}
	We establish this by induction on $n\in\N$. The base case $n=1$ is clear. For the induction step $\N\ni n\to n+1 \in\N$ observe that Lemma~\ref{lem:Markov_ppty} and the induction hypothesis imply that for all $t\in (0,\infty)$ we have that
	\begin{equation}
	\theta_{(n+1)t}^{\beta_t}=\theta_{nt+t}^{\beta_t}=\theta_{nt}^{\theta_t^{\beta_t}}=\theta_{nt}^{\beta_t}=\theta_{nt}^{\beta_{nt}}=\beta_{nt}=\beta_t.
	\end{equation}
	This finishes the proof of the induction step. Induction hence establishes \eqref{eq:motionless_pt_f_9}.
	Observe that \eqref{eq:motionless_pt_f_9} implies that for all $m, n\in \N$, $t\in (0,\infty)$ we have that
	\begin{equation}\label{eq:motionless_pt_f_10}
	\beta_{\frac{m}{n}t}=\beta_{n\frac{m}{n}t}=\beta_{mt}=\beta_t.
	\end{equation}
	This reveals that for all $t\in \Q\cap (0,\infty)$  it holds that
	\begin{equation}\label{eq:motionless_pt_f_11}
	\beta_{t}=\beta_1.
	\end{equation}
	This proves that for all $t\in \Q\cap (0,\infty)$ we have that
	\begin{equation}
	\theta_{t}^{\beta_1} = \theta_{t}^{\beta_t} = \beta_t = \beta_1.
	\end{equation}
	Therefore, we obtain  that for all $t \in [0, \infty)$, $n\in \N$ and all functions  $q = (q_k)_{k \in \N} \colon \N \to \Q\cap (0,\infty)$ with $\limsup_{k \to \infty} |q_k -t|=0$ we have that
	\begin{equation}\label{eq:motionless_pt_f_12}
	\theta_{q_n}^{\beta_1} =\beta_1.
	\end{equation}
	Moreover, observe that Corollary~\ref{cor:regularity_flow} and Lemma~\ref{lem:f_C^1} assure that 
	\begin{equation}\label{eq:motionless_pt_f_12_2}
	\big( [0,\infty)\ni t\mapsto \theta_t^{\beta_1} \in \R^d\big)\in C^1([0,\infty),\R^d).
	\end{equation}
	Combining this and \eqref{eq:motionless_pt_f_12} proves that for all $t \in [0, \infty)$ we have that
	\begin{equation}\label{eq:motionless_pt_f_13}
	\theta^{\beta_1}_{t}=\beta_1.
	\end{equation}
	This, \eqref{eq:equation theta}, and \eqref{eq:motionless_pt_f_12_2} ensure that for all $t\in [0,\infty)$ we have that
	\begin{equation}\label{eq:motionless_pt_f_15}
	0=\pt (\theta_t^{\beta_1})=g(\theta_t^{\beta_1})=g(\beta_1).
	\end{equation}
	Combining this and \eqref{eq:motionless_pt_f_1}  implies that for all $x\in\lb y\in\R^d\colon g(y)=0\rb$ we have that
	\begin{equation}\label{eq:motionless_pt_f_16}
	\begin{split}
	0 = \langle 0,x-\beta_1 \rangle_{\R^d} = \langle g(x)-g(\beta_1),x-\beta_1\rangle_{\R^d} \leq -L \|x-\beta_1\|_{\R^d}^2.
	\end{split}
	\end{equation}
	The assumption that $L>0$ and \eqref{eq:motionless_pt_f_15} therefore prove item~\eqref{item:motionless_pt_f_1}. Moreover, note that \eqref{eq:motionless_pt_f_13} establishes item~\eqref{item:motionless_pt_f_1_2}.
	Corollary~\ref{cor:flow_deriv} hence demonstrates that for all $h\in \R^d$, $t\in [0,\infty)$ we have that
	\begin{equation}
	\|\theta_t^{\beta_1+h}-\beta_1\|_{\R^d}= \|\theta_t^{\beta_1+h}-\theta_t^{\beta_1}\|_{\R^d}\leq \|h\|_{\R^d} e^{-Lt}.
	\end{equation}
	This establishes item~\eqref{item:motionless_pt_f_2}. Next observe that item~\eqref{item:motionless_pt_f_2} implies item~\eqref{item:motionless_pt_f_3}. The proof of Lemma~\ref{lem:motionless_pt_f} is thus completed.
\end{proof}

\begin{lemma}\label{lem:estimate_diff_phi}
	Assume Setting~\ref{sec:setting_1}, let $L\in \R$, $\psi\in C^1(\R^d,\R)$,  and assume  for all $x,y\in \R^d$ that
	\begin{equation}\label{eq:estimate_diff_phi_1}
	\langle g(x)-g(y), x-y\rangle_{\R^d} \leq L \|x-y\|_{\R^d}^2.
	\end{equation}
	Then we have for all $t\in [0,\infty)$, $x,y\in\R^d$ that
	\begin{align}
	\label{eq:estimate_diff_phi_4}
	\begin{split}
	|\psi(\theta_t^x)-\psi(\theta_t^y)|  &\leq \sup\!\left\lbrace \| \psi'(\lambda\theta_t^x+(1-\lambda)\theta_t^y) \|_{L(\R^d,\R)} \in \R\colon \lambda\in [0,1] \right\rbrace \\
	& \quad \cdot \|x-y\|_{\R^d}\exp(Lt).
	\end{split}
	\end{align}
\end{lemma}
\begin{proof}[Proof of Lemma~\ref{lem:estimate_diff_phi}]
	Throughout this proof let $M\colon [0,\infty)\times\R^d\times\R^d\to [0,\infty]$ satisfy for all $t\in[0,\infty)$, $x,y\in \R^d$  that
	\begin{equation}\label{eq:estimate_diff_phi_2}
	M(t,x,y)=\sup\!\left\lbrace \| \psi'(\lambda\theta_t^x+(1-\lambda)\theta_t^y) \|_{L(\R^d,\R)} \in \R\colon \lambda\in [0,1] \right\rbrace.
	\end{equation}
	Note that the fundamental theorem of calculus, Lemma~\ref{lem:f_C^1}, and Corollary~\ref{cor:flow_deriv} assure that for all $t\in [0,\infty)$, $x,y\in \R^d$ we have that
	\begin{equation}\label{eq:estimate_diff_phi_6}
	\begin{split}
	|\psi(\theta_t^x)-\psi(\theta_t^y)| 
	& = \left| \int_0^1 \psi'(\lambda\theta_t^x+(1-\lambda)\theta_t^y)(\theta_t^x-\theta_t^y)\, d\lambda \right|\\
	& \leq  \int_0^1 | \psi'(\lambda\theta_t^x+(1-\lambda)\theta_t^y)(\theta_t^x-\theta_t^y) |\, d\lambda \\
	& \leq  \int_0^1 \| \psi'(\lambda\theta_t^x+(1-\lambda)\theta_t^y)\|_{L(\R^d,\R)} \|\theta_t^x-\theta_t^y \|_{\R^d} \, d\lambda \\
	& \leq M(t,x,y) \| \theta_t^x-\theta_t^y \|_{\R^d} \leq M(t,x,y) \|x-y\|_{\R^d}\exp(Lt).
	\end{split}
	\end{equation}
	The proof of Lemma~\ref{lem:estimate_diff_phi} is thus completed.
\end{proof}

\begin{lemma}\label{lem:Kolm_back_eq_main}
	Assume Setting~\ref{sec:setting_1}, let $\psi\in C^1(\R^d,\R)$, $T\in (0,\infty)$, and let $u  = (u(t,\vartheta))_{(t, \vartheta) \in [0, T] \times \R^d} \colon \allowbreak [0,T]\times \R^d\to \R$ satisfy for all $t\in [0,T]$, $\vartheta\in\R^d$ that $	u(t,\vartheta)=\psi(\theta^{\vartheta}_{T-t})$.
	Then 
	\begin{enumerate}[(i)]
		\item\label{item:Kolm_back_eq_main_1} we have that $u\in C^{1}([0,T]\times\R^d,\R)$ and
		\item\label{item:Kolm_back_eq_main_2} we have for all $t\in [0,T]$, $\vartheta\in \R^d$ that
		$
		(\pt u )(t,\vartheta) = - (\pv u )(t,\vartheta)g(\vartheta)
	$.
	\end{enumerate}
\end{lemma}
\begin{proof}[Proof of Lemma~\ref{lem:Kolm_back_eq_main}]
	Throughout this proof let $v \colon [0,T]\times \R^d\to\R$ satisfy for all $t\in [0,T]$, $\vartheta\in \R^d$ that
	\begin{equation}\label{eq:Kolm_back_eq_main_4}
	v(t,\vartheta)=\psi(\theta^{\vartheta}_t).
	\end{equation}
	Note that combining Lemma~\ref{lem:f_C^1} and item~\eqref{Kolm_back_eq_item_1} in \Cref{thm:Kolm_back_eq} proves item~\eqref{item:Kolm_back_eq_main_1}.
	Next observe that item~\eqref{Kolm_back_eq_item_2} in \Cref{thm:Kolm_back_eq} assures that for all $t\in [0,T]$, $\vartheta\in \R^d$ we have that
	\begin{equation}\label{eq:Kolm_back_eq_main_5}
	(\tfrac{\partial}{\partial t}v )(t,\vartheta) = (\tfrac{\partial}{\partial \vartheta} v)(t,\vartheta) g(\vartheta).
	\end{equation}
	This reveals that for all $t\in [0,T]$, $\vartheta\in \R^d$  it holds that
	\begin{equation}\label{eq:Kolm_back_eq_main_6}
	\begin{split}
	(\tfrac{\partial}{\partial t}u )(t,\vartheta) 
	= - (\tfrac{\partial}{\partial t} v )(T-t,\vartheta)= -(\tfrac{\partial}{\partial \vartheta} v)(T-t,\vartheta) g(\vartheta)
= -(\tfrac{\partial}{\partial \vartheta} u)(t,\vartheta) g(\vartheta).
	\end{split}
	\end{equation}
	This establishes item~\eqref{item:Kolm_back_eq_main_2}. The proof of Lemma \ref{lem:Kolm_back_eq_main} is thus completed.
\end{proof}

\begin{lemma}\label{lem:estimate_u_01}
	Assume Setting~\ref{sec:setting_1}, let $\psi\in C^1(\R^d,\R)$, $T\in (0,\infty)$, $L\in\R$, assume  for all $x,y\in \R^d$ that
	\begin{equation}\label{eq:estimate_u_01_1}
	\langle g(x)-g(y), x-y\rangle_{\R^d} \leq L \|x-y\|_{\R^d}^2,
	\end{equation}
	and let $u  = (u(t,\vartheta))_{(t, \vartheta) \in [0, T] \times \R^d}  \colon [0,T]\times \R^d\to\R$ satisfy for all $t\in [0,T]$, $\vartheta\in\R^d$ that
	$
	u(t,\vartheta)=\psi(\theta_{T-t}^{\vartheta})
	$.
	Then
	\begin{enumerate}[(i)]
		\item\label{item:estimate_u_01_1_2} we have that $u\in C^{1}([0,T]\times\R^d,\R)$
		and
		\item\label{item:estimate_u_01_3} we have for all $t\in [0,T]$, $\vartheta\in\R^d$ that 
		\begin{equation}
		\|(\pv u)(t,\vartheta)\|_{L(\R^d,\R)}\leq \| \psi'(\theta_{T-t}^{\vartheta}) \|_{L(\R^d,\R)} \exp(L(T-t)).
		\end{equation}
	\end{enumerate}
\end{lemma}
\begin{proof}[Proof of Lemma~\ref{lem:estimate_u_01}] 
	First, note that item~\eqref{item:Kolm_back_eq_main_1} in Lemma~\ref{lem:Kolm_back_eq_main} proves item~\eqref{item:estimate_u_01_1_2}.
	Next observe that item~\eqref{item:exp_decay_partial_vartheta_2} in Lemma~\ref{lem:exp_decay_partial_vartheta} implies that for all $t\in [0,T]$, $\vartheta\in\R^d$ we have that
	\begin{equation}
	\begin{split}
	\|(\pv u)(t,\vartheta)\|_{L(\R^d,\R)}
	 &= \|\psi'(\theta_{T-t}^{\vartheta})(\pv \theta_{T-t}^{\vartheta})\|_{L(\R^d,\R)}\\
	& \leq \|\psi'(\theta_{T-t}^{\vartheta})\|_{L(\R^d,\R)}\|(\pv \theta_{T-t}^{\vartheta})\|_{L(\R^d,\R^d)}\\
	& \leq \|\psi'(\theta_{T-t}^{\vartheta})\|_{L(\R^d,\R)}e^{L(T-t)}.
	\end{split}
	\end{equation}
	This establishes item~\eqref{item:estimate_u_01_3}. The proof of Lemma~\ref{lem:estimate_u_01} is thus completed.
\end{proof}

\section{Upper bounds for second-order spatial derivatives of certain deterministic flows}
\label{subsec:upper}
\sectionmark{}

\begin{lemma}\label{lem:exp_decay_pty_f'}
	Let $d\in\N$, $g\in C^1(\R^d,\R^d)$, $L\in \R$, and  assume  for all $x,y\in \R^d$ that
	\begin{equation}\label{eq:exp_decay_pty_f'_1}
	\langle g(x)-g(y), x-y\rangle_{\R^d} \leq L \|x-y\|^2_{\R^d}.
	\end{equation}
	Then we have for all $x,v\in\R^d$ that
	$
	\langle g'(v)x,x \rangle_{\R^d} \leq L\|x\|^2_{\R^d}
	$.
\end{lemma}
\begin{proof}[Proof of Lemma~\ref{lem:exp_decay_pty_f'}]
	Throughout this proof let $x,v\in\R^d$. Note that
	\begin{equation}
	\limsup_{\substack{h \to 0\\
			h \in \R^d \backslash \{0\}}} \bigg\|\frac{1}{\|h\|_{\R^d}}(g(v+h)-g(v)-g'(v)h)\bigg\|_{\R^d}=0.
	\end{equation}
	This demonstrates that
	\begin{equation}
	\limsup_{\substack{\lambda \to 0\\
			\lambda \in \R \backslash \{0\}}}\bigg\|\frac{1}{\lambda}(g(v+\lambda x)-g(v)-\lambda g'(v)x)\bigg\|_{\R^d}=0.
	\end{equation}
	This reveals that
	\begin{equation}
	\lim_{\substack{\lambda \to 0\\
			\lambda \in \R \backslash \{0\}}} \left[ \frac{1}{\lambda}(g(v+\lambda x)-g(v)) \right] = g'(v)x.
	\end{equation}
	This and \eqref{eq:exp_decay_pty_f'_1} assure that
	\begin{align*}
	&\langle g'(v)x,x \rangle_{\R^d} \numberthis\\
	&= \left\langle \lim_{\substack{\lambda \to 0\\
			\lambda \in \R \backslash \{0\}}} \left[\frac{1}{\lambda} (g(v+\lambda x)-g(v)) \right], x \right\rangle_{\!\!\R^d}
	= \lim_{\substack{\lambda \to 0\\
			\lambda \in \R \backslash \{0\}}} \left[\frac{1}{\lambda}\langle  g(v+\lambda   x)-g(v), x \rangle_{\R^d} \right]\\
	&= \lim_{\substack{\lambda \to 0\\
			\lambda \in \R \backslash \{0\}}} \left[ \frac{1}{\lambda^2}\langle  g(v+\lambda x)-g(v), \lambda x \rangle_{\R^d} \right]
	\leq \limsup_{\substack{\lambda \to 0\\
			\lambda \in \R \backslash \{0\}}} \left[ \frac{1}{\lambda^2}L\norm{\lambda x}_{\R^d}^2 \right]
	 =  L\|x\|_{\R^d}^2.
	\end{align*}
	The proof of Lemma~\ref{lem:exp_decay_pty_f'} is thus completed.
\end{proof}

\begin{lemma}\label{lem:unicity_sol_Ax+b}
	Let $t,L\in\R$, $T\in (t,\infty)$, $d\in\N$, $b\in C([t,T],\R^d)$, let $A\in C([t,T],L(\R^d,\R^d))$ satisfy for all $s\in [t,T]$, $u\in \R^d$ that
	\begin{equation}\label{eq:unicity_sol_Ax+b_1}
	\langle A(s)u, u\rangle_{\R^d} \leq L\norm{u}_{\R^d}^2,
	\end{equation}
	and let $y_1, y_2 \in C^1([t,T],\R^d)$ satisfy for all $i \in \{1, 2\}$, $s\in [t,T]$ that
	\begin{equation}
	\label{eq:unicity_sol_Ax+b_2}
	y_i(t)=0  \qquad \text{and} \qquad 
	(y_i)'(s) = A(s)y_i(s)+b(s).
	\end{equation}
	Then we have that $y_1=y_2$.
\end{lemma}
\begin{proof}[Proof of Lemma~\ref{lem:unicity_sol_Ax+b}]
	Throughout this proof let $\varphi \in C^1([t,T],\R)$ satisfy for all $s\in [t,T]$ that
	\begin{equation}\label{eq:unicity_sol_Ax+b_4}
	\varphi(s)=\norm{y_1(s)-y_2(s)}_{\R^d}^2.
	\end{equation}
	Observe that \eqref{eq:unicity_sol_Ax+b_4}, \eqref{eq:unicity_sol_Ax+b_2}, and \eqref{eq:unicity_sol_Ax+b_1} imply that for all $s\in [t,T]$ we have that
	\begin{equation}\label{eq:unicity_sol_Ax+b_5}
	\begin{split}
	\varphi'(s)
	&= 2\langle y_1(s)-y_2(s), (y_1)'(s)-(y_2)'(s)\rangle_{\R^d}\\
	&= 2\langle y_1(s)-y_2(s), A(s)(y_1(s)-y_2(s))\rangle_{\R^d}\\
	& \leq 2 L\norm{y_1(s)-y_2(s)}_{\R^d}^2  = 2L \varphi(s).
	\end{split}
	\end{equation}
	This and the Gronwall differential inequality in Lemma~\ref{lem:Gronwall_differential} prove that for all $s\in [t,T]$ we have that
	\begin{equation}
	\varphi(s)\leq \varphi(t)e^{2L(s-t)}.
	\end{equation}
	This and \eqref{eq:unicity_sol_Ax+b_2} assure that for all $s\in [t,T]$ we have that $\varphi(s)=0$. The proof of Lemma~\ref{lem:unicity_sol_Ax+b} is thus completed.
\end{proof}

\begin{lemma}\label{lem:twice:diff}
	Assume Setting~\ref{sec:setting_1}.
	Then 
	\begin{enumerate}[(i)]
		\item\label{item:twice:diff:f} we have that $g \in C^2(\R^d,\R^d)$,
		\item\label{item:twice:diff:theta} we have that
		\begin{equation}\label{eq:twice:diff:theta}
		\big( [0,\infty)\times \R^d\ni (t,\vartheta)\mapsto \theta^{\vartheta}_t\in \R^d \big)\in C^{2}([0,\infty)\times\R^d,\R^d),
		\end{equation}
		\item\label{item:twice:diff:t} we have for all  $\vartheta \in \R^d $ that
		\begin{equation}\label{eq:twice:diff:t}
		\big( [0, \infty) \ni t \mapsto \tfrac{\partial^2}{\partial \vartheta^2}\theta_t^{\vartheta}  \in  L^{(2)}(\R^d, \R^d) \big)  \in C^1([0, \infty), L^{(2)} (\R^d, \R^d)),
		\end{equation}
		and
		\item\label{item:twice:diff:dt} we have for all $t \in [0, \infty)$, $\vartheta \in \R^d$ that
		\begin{equation}
		\tfrac{\partial}{\partial t}(\tfrac{\partial^2}{\partial \vartheta^2}\theta_t^{\vartheta} ) = \tfrac{\partial^2}{\partial \vartheta^2} (g(\theta_t^{\vartheta} )).
		\end{equation}
	\end{enumerate}
\end{lemma}
\begin{proof}[Proof of Lemma~\ref{lem:twice:diff}]
	First, note that Lemma~\ref{lem:f_C^n} proves item~\eqref{item:twice:diff:f}. This and Corollary~\ref{cor:regularity_flow} demonstrate item~\eqref{item:two_expressions_2}.
	Next observe that for all $t \in [0, \infty)$, $\vartheta \in \R^d$ we have that
	\begin{equation}
	\label{eq:theta:f}
	\tfrac{\partial}{\partial t} \theta_t^{\vartheta} = g(\theta_t^{\vartheta}).
	\end{equation}
	This, item~\eqref{item:twice:diff:f}, and item~\eqref{item:two_expressions_2}
	ensure that
	\begin{equation}
	\big( [0,\infty)\times \R^d\ni (t,\vartheta)\mapsto \tfrac{\partial}{\partial t} \theta_t^{\vartheta} \in \R^d \big)\in C^{2}([0,\infty)\times\R^d,\R^d).
	\end{equation}
	This reveals that
	\begin{align}
	\label{eq:mixed:deriv}
	\begin{split}
	\big( [0,\infty)\times \R^d\ni (t,\vartheta)\mapsto \tfrac{\partial^2}{\partial \vartheta^2} ( \tfrac{\partial}{\partial t} \theta_t^{\vartheta}) 
	\in L^{(2)}(\R^d, \R^d) \big) \\
	\in C([0,\infty)\times\R^d, L^{(2)}(\R^d, \R^d)).
	\end{split}
	\end{align}
	Schwarz's theorem (cf., e.g.,  K\"onigsberger~\cite[Section~2.3]{Koenigsberger2004}) hence proves that for all $t \in [0, \infty)$, $\vartheta \in \R^d$ we have that
	\begin{equation}
	\label{eq:mixed:deriv:2}
	\tfrac{\partial^2}{\partial \vartheta^2} ( \tfrac{\partial}{\partial t} \theta_t^{\vartheta}) = \tfrac{\partial}{\partial \vartheta} ( \tfrac{\partial^2}{\partial \vartheta \partial t} \theta_t^{\vartheta}) = \tfrac{\partial}{\partial \vartheta} ( \tfrac{\partial^2}{\partial t \partial \vartheta} \theta_t^{\vartheta})= \tfrac{\partial^2}{\partial \vartheta \partial t} ( \tfrac{\partial}{\partial \vartheta} \theta_t^{\vartheta}).
	\end{equation}
	This  and  \eqref{eq:mixed:deriv} assure that 
	\begin{align}
	\begin{split}
	\big( [0,\infty)\times \R^d\ni (t,\vartheta)\mapsto \tfrac{\partial^2}{\partial \vartheta \partial t} ( \tfrac{\partial}{\partial \vartheta} \theta_t^{\vartheta})
	\in L^{(2)}(\R^d, \R^d) \big) \\
	\in C([0,\infty)\times\R^d, L^{(2)}(\R^d, \R^d)).
	\end{split}
	\end{align}
	Schwarz's theorem (cf., e.g.,  K\"onigsberger~\cite[Section~2.3]{Koenigsberger2004}) therefore implies that for all $t \in [0, \infty)$, $\vartheta \in \R^d$ we have that $\tfrac{\partial^2}{\partial t \partial \vartheta } ( \tfrac{\partial}{\partial \vartheta} \theta_t^{\vartheta})$ exists and 
	\begin{equation}
	\tfrac{\partial^2}{\partial t \partial \vartheta } ( \tfrac{\partial}{\partial \vartheta} \theta_t^{\vartheta}) = \tfrac{\partial^2}{\partial \vartheta \partial t} ( \tfrac{\partial}{\partial \vartheta} \theta_t^{\vartheta}).
	\end{equation}
	This and \eqref{eq:mixed:deriv:2} 
	establish item~\eqref{item:twice:diff:t}  and that for all $t \in [0, \infty)$, $\vartheta \in \R^d$ we have that
	\begin{equation}
	\tfrac{\partial^2}{\partial \vartheta^2} ( \tfrac{\partial}{\partial t} \theta_t^{\vartheta}) = \tfrac{\partial^2}{\partial t \partial \vartheta } ( \tfrac{\partial}{\partial \vartheta} \theta_t^{\vartheta}) = \tfrac{\partial}{\partial t} ( \tfrac{\partial^2}{ \partial \vartheta^2} \theta_t^{\vartheta}).
	\end{equation}
	Combining this with \eqref{eq:theta:f} establishes item~\eqref{item:twice:diff:dt}. The proof of Lemma~\ref{lem:twice:diff} is thus completed.
\end{proof}

\begin{lemma}\label{lem:two_expressions}
	Assume Setting~\ref{sec:setting_1}, let $T\in (0,\infty)$, $L\in \R$, and  assume  for all $x,y\in\R^d$  that
	\begin{equation}\label{eq:two_expressions_1}
	\langle g(x)-g(y),x-y \rangle_{\R^d} \leq L\|x-y\|_{\R^d}^2.
	\end{equation}
	Then 
	\begin{enumerate}[(i)]
		\item\label{item:two_expressions_1} we have that $g\in C^2(\R^d,\R^d)$,
		\item\label{item:two_expressions_2} we have that
		$
		( [0,\infty)\times \R^d\ni (t,\vartheta)\mapsto \theta^{\vartheta}_t\in \R^d )\in C^{2}([0,\infty)\times\R^d,\R^d)
		$,
		\item\label{item:two_expressions_2_02} we have that there exist unique  $\chi^{\vartheta}\in C( \lb (s,t)\in [0,T]^2\colon s\leq t\rb \times \R^d,\R^d)$, $\vartheta \in\R^d$,  which satisfy  for all $\vartheta \in\R^d$, $s\in [0,T]$, $t\in [s,T]$, $x \in \R^d$ that
		\begin{equation}
		\chi^{\vartheta}(s,t,x) = x + \int_{s}^t g'(\theta_u^{\vartheta}) \, \chi^{\vartheta}(s,u,x)\, du,
		\end{equation}
		and
		\item\label{item:two_expressions_3} we have for all $t\in [0,T]$, $\vartheta, v,w\in\R^d$ that
		\begin{equation}\label{eq:two_expressions_2}
		(\tfrac{\partial^2}{\partial \vartheta^2}\theta_t^{\vartheta} )(v,w)
		=\int_0^t \chi^{\vartheta}\big(s,t,g''(\theta_s^{\vartheta})\big(( \pv \theta_s^{\vartheta}) v, (\pv \theta_s^{\vartheta})w\big)\big) \, ds.
		\end{equation}
	\end{enumerate}
\end{lemma}
\begin{proof}[Proof of Lemma~\ref{lem:two_expressions}] 
	First, observe that Lemma~\ref{lem:twice:diff} proves item~\eqref{item:two_expressions_1} and item~\eqref{item:two_expressions_2}.
	Next note that item~\eqref{item:two_expressions_1}, item~\eqref{item:two_expressions_2}, the fact that the set $[0,T] \subseteq \R$ is compact, and \Cref{prop:cont_fct_bnd_on_cmpcts} assure that  there exists $c \colon \R^d \to (0,\infty)$ which satisfies for all $\vartheta \in \R^d$, $t\in[0,T]$ that 
	\begin{equation}
	\|g'(\theta_{t}^{\vartheta})\|_{L(\R^d,\R^d)}\leq c_{\vartheta}.
	\end{equation}
	This proves that for all  $\vartheta \in \R^d$, $x,y\in\R^d$, $s\in [0,T]$ we have that
	\begin{equation}
	\|g'(\theta_{s}^{\vartheta})x-g'(\theta_{s}^{\vartheta})y\|_{\R^d}\leq c_{\vartheta} \|x-y\|_{\R^d}.
	\end{equation}
	\Cref{thm:regularity_flow_dynamical_system} hence ensures that there exist unique $\chi^{\vartheta}\in C( \lb (s,t)\in [0,T]^2\colon s\leq t\rb \times \R^d, \R^d)$, $\vartheta \in \R^d$, which satisfy for all $\vartheta \in \R^d$, $s\in [0,T]$, $t\in [s,T]$,  $ x\in\R^d$ that
	\begin{equation}\label{eq:two_expressions_3-1}
	\chi^{\vartheta}(s,t,x) = x + \int_{s}^t g'(\theta_u^{\vartheta}) \, \chi^{\vartheta}(s,u,x)\, du.
	\end{equation}
	This proves item~\eqref{item:two_expressions_2_02}.
	In the next step let $\vartheta, v,w\in\R^d$, let $A \in C([0,T], L(\R^d,\R^d))$ satisfy for all $t\in [0,T]$ that
	\begin{equation}\label{eq:two_expressions_3}
	A(t)=g'(\theta_t^{\vartheta}),
	\end{equation}
	let $b \in C([0,T], \R^d)$ satisfy for all $t\in [0,T]$ that
	\begin{equation}\label{eq:two_expressions_4}
	b(t)=g''(\theta_t^{\vartheta})\big((\pv \theta_t^{\vartheta})v, (\pv \theta_t^{\vartheta})w \big),
	\end{equation}
	let $y\in C^1([0,T],\R^d)$ satisfy for all $t\in [0,T]$ that
	\begin{equation}\label{eq:two_expressions_4_p}
	y(t)=(\tfrac{\partial^2}{\partial \vartheta^2}\theta_t^{\vartheta} )(v,w)
	\end{equation}
	(cf.~Lemma~\ref{lem:twice:diff}), and let $z\in C^1([0,T],\R^d)$ satisfy for all $t\in [0,T]$ that
	\begin{equation}\label{eq:two_expressions_5}
	z(t)
	=\int_0^t \chi^{\vartheta}\big(s,t,b(s)\big) \, ds.
	\end{equation}
	Note that \eqref{eq:two_expressions_4_p} and \eqref{eq:equation theta} imply that
	\begin{equation}\label{eq:two_expressions_6}
	y(0)  = (\tfrac{\partial^2}{\partial \vartheta^2}\theta_0^{\vartheta} )(v,w)
	= (\tfrac{\partial^2}{\partial \vartheta^2}\vartheta)(v,w)
	= 0.
	\end{equation}
	Moreover, observe that  \eqref{eq:two_expressions_4_p},  Lemma~\ref{lem:twice:diff}, and the chain rule ensure that for all $t\in [0,T]$ we have that
	\begin{equation}\label{eq:two_expressions_7}
	\begin{split}
	y'(t) &= \tfrac{\partial}{\partial t}\big((\tfrac{\partial^2}{\partial \vartheta^2}\theta_t^{\vartheta} )(v,w)\big)\\
	& = \big(\tfrac{\partial^2}{\partial \vartheta^2} (g(\theta_t^{\vartheta}))\big)(v,w)\\
	& = \big(\pv( g'(\theta_t^{\vartheta})(\pv \theta^{\vartheta}_t))\big)(v)(w)\\
	& = \big(\pv( g'(\theta_t^{\vartheta})(\pv \theta^{\vartheta}_t)v)\big)(w)\\
	& = g''(\theta_t^{\vartheta})\big((\pv \theta_t^{\vartheta})v, (\pv \theta_t^{\vartheta})w\big) + g'(\theta_t^{\vartheta})\big(\pv \big( (\tfrac{\partial}{\partial \vartheta} \theta_t^{\vartheta}) v \big) (w)\big)\\
	& = g''(\theta_t^{\vartheta})\big((\pv \theta_t^{\vartheta})v, (\pv \theta_t^{\vartheta})w\big) + g'(\theta_t^{\vartheta})\big((\tfrac{\partial^2}{\partial \vartheta^2} \theta_t^{\vartheta})(v,w)\big)\\
	& =  g'(\theta_t^{\vartheta})\big((\tfrac{\partial^2}{\partial \vartheta^2} \theta_t^{\vartheta})(v,w)\big)
	+ g''(\theta_t^{\vartheta})\big((\pv \theta_t^{\vartheta})v, (\pv \theta_t^{\vartheta})w \big)\\
	& = A(t)y(t)+b(t).
	\end{split}
	\end{equation}
	In addition, note that \eqref{eq:two_expressions_5} assures that
	\begin{equation}\label{eq:two_expressions_8}
	z(0)=0.
	\end{equation}
	In the next step we combine \eqref{eq:two_expressions_5}, \eqref{eq:two_expressions_3}, and \eqref{eq:two_expressions_3-1} to obtain that for all $t\in [0,T]$ we have that
	\begin{equation}\label{eq:two_expressions_9}
	\begin{split}
	z'(t) & = \chi(t,t,b(t)) + \int_0^tA(t)\chi^{\vartheta}\big(s,t,b(s)\big)\, ds\\
	& = b(t) + A(t)z(t)
	 = A(t)z(t) + b(t).
	\end{split}
	\end{equation}
	Furthermore, observe that Lemma~\ref{lem:exp_decay_pty_f'} and \eqref{eq:two_expressions_1} imply that for all $t\in [0,T]$, $u \in\R^d$ we have that
	\begin{equation}
	\langle A(t)u,u\rangle_{\R^d}\leq  L\|u\|_{\R^d}^2.
	\end{equation}
	Combining this, \eqref{eq:two_expressions_6}--\eqref{eq:two_expressions_9}, and Lemma~\ref{lem:unicity_sol_Ax+b} demonstrates that  for all $t\in [0,T]$ we have that
	\begin{equation}
	y(t)=z(t).
	\end{equation}
	This establishes item~\eqref{item:two_expressions_3}. The proof of Lemma~\ref{lem:two_expressions} is thus completed.
\end{proof}

\begin{lemma}\label{lem:special_cond_dyn_syst}
	Let $d\in\N$, $a,L\in\R$, $b\in (a,\infty)$, let $f \in C( [a,b]\times\R^d, \R^d)$  satisfy for all $x,y\in \R^d$, $s\in [a,b]$ that
	\begin{equation}\label{eq:special_cond_dyn_syst_1}
	\langle f(s,x)-f(s,y),x-y\rangle_{\R^d}\leq L \|x-y\|_{\R^d}^2,
	\end{equation}
	and let $\chi_{t,\cdot}^{x}\in C([t,b],\R^d)$, $x\in\R^d$, $t\in [a,b]$,  satisfy for all $t\in [a,b]$, $x\in\R^d$, $s\in [t,b]$ that
	\begin{equation}\label{eq:special_cond_dyn_syst_2}
	\chi_{t,s}^x = x + \int_{t}^s f(u,\chi_{t,u}^x)\, du.
	\end{equation}
	Then we have for all $x,y\in\R^d$, $t\in [a,b]$, $s\in [t,b]$ that
	\begin{equation}\label{eq:special_cond_dyn_syst_3}
	\| \chi_{t,s}^{x} - \chi_{t,s}^{y} \|_{\R^d}\leq \|x-y\|_{\R^d} e^{L(s-t)}.
	\end{equation}
\end{lemma}
\begin{proof}[Proof of Lemma~\ref{lem:special_cond_dyn_syst}]
	Throughout this proof let $t\in [a,b]$, $x,y\in\R^d$, let $E\colon [t,b]\to\R$ satisfy for all $s\in [t,b]$ that
	\begin{equation}\label{eq:special_cond_dyn_syst_4}
	E(s)=\| \chi_{t,s}^{x} - \chi_{t,s}^{y} \|_{\R^d}^2.
	\end{equation}
	Note that \eqref{eq:special_cond_dyn_syst_4}, \eqref{eq:special_cond_dyn_syst_2}, and \eqref{eq:special_cond_dyn_syst_1} assure that for all $s\in [t,b]$ we have that
	\begin{equation}\label{eq:special_cond_dyn_syst_5}
	\begin{split}
	E'(s)
	& = 2\langle f(s,\chi_{t,s}^{x}) - f(s,\chi_{t,s}^{y}), \chi_{t,s}^{x} - \chi_{t,s}^{y} \rangle_{\R^d}
	 \leq 2 L \| \chi_{t,s}^{x} - \chi_{t,s}^{y} \|_{\R^d}^2
	 = 2L E(s).
	\end{split}
	\end{equation}
	The Gronwall differential inequality in Lemma~\ref{lem:Gronwall_differential} hence implies that for all $s\in [t,b]$ we have that
	\begin{equation}\label{eq:special_cond_dyn_syst_6}
	\| \chi_{t,s}^{x} - \chi_{t,s}^{y} \|_{\R^d}=|E(s)|^{1/2}\leq |E(t)|^{1/2} e^{L(s-t)} = \|x-y\|_{\R^d} e^{L(s-t)}.
	\end{equation}
	The proof of Lemma~\ref{lem:special_cond_dyn_syst} is thus completed.
\end{proof}

\begin{lemma}\label{lem:bound_partial_theta^2}
	Assume Setting~\ref{sec:setting_1}, let $T\in (0,\infty)$, $L \in \R$, and  assume  for all $x,y\in \R^d$ that 
	\begin{equation}\label{eq:bound_partial_theta^2_p}
	\langle g(x)-g(y), x-y\rangle_{\R^d} \leq L \|x-y\|_{\R^d}^2.
	\end{equation}
	Then 
	\begin{enumerate}[(i)]
		\item\label{item:bound_partial_theta^2_1} we have that $g\in C^2(\R^d,\R^d)$,
		\item\label{item:bound_partial_theta^2_2} we have that
		$
		( [0,\infty)\times \R^d\ni (t,\vartheta)\mapsto \theta^{\vartheta}_t\in \R^d )\in C^{2}([0,\infty)\times\R^d,\R^d)
		$,
		and
		\item\label{item:bound_partial_theta^2_3} we have for all $\vartheta\in\R^d$, $t\in [0,T]$ that
		\begin{equation}\label{eq:bound_partial_theta^2_3_3}
		\big\|  \tfrac{\partial^2}{\partial \vartheta^2} \theta_t^{\vartheta}  \big\|_{L^{(2)}(\R^d,\R^d)}\leq  \int_0^t \exp(L(t+s))\big\|g''(\theta_s^{\vartheta})\|_{L^{(2)}(\R^d,\R^d)} \, ds.
		\end{equation}
	\end{enumerate}
\end{lemma}
\begin{proof}[Proof of Lemma~\ref{lem:bound_partial_theta^2}]
	First, observe that Lemma~\ref{lem:two_expressions} proves item~\eqref{item:bound_partial_theta^2_1} and item~\eqref{item:bound_partial_theta^2_2}.
	Next let $\vartheta,v,w\in \R^d$, let $A  \colon  [0,T]\to L(\R^d,\R^d)$  satisfy for all $t\in [0,T]$ that
	\begin{equation}\label{eq:bound_partial_theta^2_1}
	A(t)=g'(\theta_t^{\vartheta}),
	\end{equation}
	let $b \colon [0,T] \to \R^d$  satisfy for all $t\in [0,T]$ that
	\begin{equation}\label{eq:bound_partial_theta^2_2}
	b(t)=g''(\theta_t^{\vartheta})\big((\pv \theta_t^{\vartheta})v,(\pv \theta_t^{\vartheta})w\big),
	\end{equation} 
	and let $\chi \in C( \lb (s,t)\in [0,T]^2\colon s\leq t\rb \times \R^d,\R^d)$  satisfy for all  $s\in [0,T]$, $t\in [s,T]$, $ x\in\R^d$ that
	\begin{equation}\label{eq:bound_partial_theta^2_3}
	\chi(s,t,x) = x + \int_{s}^tA(u)\chi(s,u,x)\, du
	\end{equation} 
	(cf.~Lemma~\ref{lem:two_expressions}).
	Note that Lemma~\ref{lem:exp_decay_pty_f'} implies that for all $x,y\in \R^d$, $t\in [0,T]$ we have that
	\begin{equation}\label{eq:bound_partial_theta^2_5}
	\langle A(t)x-A(t)y,x-y\rangle_{\R^d}\leq L \|x-y\|_{\R^d}^2.
	\end{equation}
	Furthermore, observe that item~\eqref{item:two_expressions_2_02} in Lemma~\ref{lem:two_expressions} assures that for all $s\in [0,T]$, $t\in [s,T]$ we have that
	\begin{equation}
	\chi(s,t,0)=0.
	\end{equation}
	Combining this, \eqref{eq:bound_partial_theta^2_5}, \eqref{eq:bound_partial_theta^2_3}, and Lemma~\ref{lem:special_cond_dyn_syst} proves that for all $s\in [0,T]$, $t\in [s,T]$ we have that
	\begin{equation}\label{eq:bound_partial_theta^2_6}
	\begin{split}
	\|\chi(s,t,b(s))\|_{\R^d} &= \|\chi(s,t,b(s))-0\|_{\R^d}\\
	&= \|\chi(s,t,b(s))-\chi(s,t,0)\|_{\R^d}\leq e^{L(t-s)}\|b(s)\|_{\R^d}.
	\end{split}
	\end{equation}
	This, Lemma~\ref{lem:two_expressions}, and the triangle inequality for the Bochner integral ensure that for all $t\in [0,T]$ we have that
	\begin{equation}\label{eq:bound_partial_theta^2_7}
	\begin{split}
	\| ( \tfrac{\partial^2}{\partial \vartheta^2} \theta_t^{\vartheta} )(v,w) \|_{\R^d}
	& = \bigg\|\int_0^t \chi\big(s,t,g''(\theta_s^{\vartheta})\big( (\pv \theta_s^{\vartheta}) v,(\pv \theta_s^{\vartheta}) w\big)\big) \, ds\bigg\|_{\R^d}\\
	& = \norm{\int_0^t\chi(s,t,b(s))\, ds}_{\R^d} 
	 \leq \int_0^t\big\|\chi(s,t,b(s))\big\|_{\R^d} \, ds \\
	 &\leq \int_0^t e^{L(t-s)}\norm{ b(s) }_{\R^d}\, ds. 
	\end{split}
	\end{equation}
	Next observe that for all $s\in [0,T]$ we have that
	\begin{equation}\label{eq:bound_partial_theta^2_8}
	\begin{split}
	\norm{b(s)}_{\R^d} & = \big\|g''(\theta_s^{\vartheta})((\pv \theta_s^{\vartheta})v,(\pv \theta_s^{\vartheta})w)\big\|_{\R^d}\\
	& \leq \|(\pv \theta_s^{\vartheta})v\|_{\R^d}\|(\pv \theta_s^{\vartheta})w\|_{\R^d} \|g''(\theta_s^{\vartheta})\|_{L^{(2)}(\R^d,\R^d)}.
	\end{split}
	\end{equation}
	Lemma~\ref{lem:exp_decay_partial_vartheta} hence implies that for all $s\in [0,T]$ we have that
	\begin{equation}\label{eq:bound_partial_theta^2_9}
	\begin{split}
	\norm{b(s)}_{\R^d} & \leq \norm{v}_{\R^d}\norm{w}_{\R^d} \|g''(\theta_s^{\vartheta})\|_{L^{(2)}(\R^d,\R^d)} e^{2Ls}.
	\end{split}
	\end{equation}
	Combining this and \eqref{eq:bound_partial_theta^2_7} proves that for all $t\in [0,T]$ we have that
	\begin{equation}\label{eq:bound_partial_theta^2_10}
	\begin{split}
	\| ( \tfrac{\partial^2}{\partial \vartheta^2} \theta_t^{\vartheta} )(v,w) \|_{\R^d}
	& \leq \norm{v}_{\R^d}\norm{w}_{\R^d} \int_0^t e^{L(t-s)}e^{2Ls}\|g''(\theta_s^{\vartheta})\|_{L^{(2)}(\R^d,\R^d)} \, ds\\
	& =  \norm{v}_{\R^d}\norm{w}_{\R^d} \int_0^t e^{L(t+s)}\|g''(\theta_s^{\vartheta})\|_{L^{(2)}(\R^d,\R^d)} \, ds.
	\end{split}
	\end{equation}
	This establishes item~\eqref{item:bound_partial_theta^2_3}. The proof of Lemma~\ref{lem:bound_partial_theta^2} is thus completed.
\end{proof}

\section{Temporal regularity results for SAAs  in the case of general learning rates}
\label{subsec:temporal}
\sectionmark{}

\begin{lemma}\label{lem:partial_t_of_Theta}
	Assume Setting~\ref{sec:setting_1} and let $T\in (0,\infty)$.  Then
	\begin{enumerate}[(i)]
		\item\label{item:partial_t_of_Theta_1} we have for all $\omega\in \Omega$, $t \in [0, T)$ with $\gamma(t) = \emptyset$  that  $[0,T]\ni u\mapsto\Theta_u(\omega) \in \R^d$   is differentiable at $t$
		and
		\item\label{item:partial_t_of_Theta_2} we have for all $\omega\in \Omega$, $t \in [0, T)$ with $\gamma(t) = \emptyset$   that
		\begin{equation}\label{eq:partial_t_of_Theta}
		\tfrac{\partial}{\partial t}\Theta_t(\omega) = \tfrac{1}{\#_{\gamma(\floorgrid{t})}}\Big[ \textstyle\sum_{j\in \gamma(\floorgrid{t})} G(\Theta_{\floorgrid{t}}(\omega),Z_j(\omega)) \Big].
		\end{equation}
	\end{enumerate}
\end{lemma}
\begin{proof}[Proof of Lemma~\ref{lem:partial_t_of_Theta}]
	Combining the assumption that
	\begin{equation}
	\forall\, t\in [0,\infty)\colon 0<\#_{\lb s\in [0,t]\colon \gamma(s)\neq \emptyset \rb}<\infty
	\end{equation} 
	with \eqref{eq: equation Theta 2}
	establishes item~\eqref{item:partial_t_of_Theta_1} and item~\eqref{item:partial_t_of_Theta_2}. The proof of Lemma~\ref{lem:partial_t_of_Theta} is thus completed.
\end{proof}

\section{A priori estimates for SAAs in the case of general learning rates}
\label{subsec:apriori}
\sectionmark{}

\begin{lemma}\label{lem:bnd_(sum_xi)^p}
	Let $n\in\N$, $p\in\lb 0\rb\cup [1,\infty)$, $x_1, x_2, \dots,x_n\in\R$. Then we have that
	\begin{equation}\label{eq:bnd_(sum_xi)^p_1}
	\left|\textstyle\sum\limits_{i=1}^n x_i\right|^p\leq n^{p-1} \left[\textstyle\sum\limits_{i=1}^n|x_i|^p \right].
	\end{equation}
\end{lemma}
\begin{proof}[Proof of Lemma~\ref{lem:bnd_(sum_xi)^p}]
	Throughout this proof assume w.l.o.g.~that $p\geq 1$.
Observe that the triangle inequality and Hölder's inequality imply that 
	\begin{equation}
	\begin{split}
	\left|\textstyle\sum\limits_{i=1}^n x_i\right|
	\leq \textstyle\sum\limits_{i=1}^n |x_i|
	\leq  \left| \textstyle\sum\limits_{i=1}^n 1\right|^{\frac{p-1}{p}}\left|\textstyle\sum\limits_{i=1}^n|x_i|^p\right|^{\frac{1}{p}} 
	=  n^{\frac{p-1}{p}}\left|\textstyle\sum\limits_{i=1}^n|x_i|^p\right|^{\frac{1}{p}}. 
	\end{split}
	\end{equation}
	This  establishes \eqref{eq:bnd_(sum_xi)^p_1}. The proof of Lemma~\ref{lem:bnd_(sum_xi)^p} is thus completed.
\end{proof}

\begin{lemma}\label{lem:Q_t_leq_Theta_floor_t}
	Assume Setting~\ref{sec:setting_1}, assume for all $v,w\in [0,\infty)$ with $v\neq w$ that 
	$
	\gamma(v)\cap \gamma(w)=\emptyset,
	$ 
	and let 
	$
	c\in [0,\infty)$, $p\in  [1,\infty)
	$, $\mathfrak{m} \in \R$
	satisfy for all $x\in\R^d$ that
	\begin{equation}\label{eq:Q_t_leq_Theta_floor_t_1}
	\E\big[ \|G(x,Z_1)\|_{\R^d}^p\big] 
	\leq 
	c (1 + \|x\|_{\R^d}^{\mathfrak{m}p}).
	\end{equation}
	Then we have for all $t\in [0,\infty)$ that
	\begin{equation}\label{eq:Q_t_leq_Theta_floor_t_2}
	\E\Big[\big\| \tfrac{1}{\#_{\gamma(\floorgrid{t})}}\textstyle\sum_{j\in \gamma(\floorgrid{t})}G(\Theta_{\floorgrid{t}}, Z_j)\big\|_{\R^d}^p\Big]
	\leq
	c \big( 1 + \E\big[\|\Theta_{\floorgrid{t}}\|_{\R^d}^{\mathfrak{m}p}\big]\big).
	\end{equation}
\end{lemma}
\begin{proof}[Proof of Lemma~\ref{lem:Q_t_leq_Theta_floor_t}]
	Throughout this proof let $t\in [0,\infty)$, $\mathfrak{j}\in\gamma(\floorgrid{t})$.
	Note that Lemma~\ref{lem:bnd_(sum_xi)^p} and the triangle inequality assure that 
	\begin{equation}\label{eq:Q_t_leq_Theta_floor_t_3}
	\begin{split}
	&\E\Big[ \big\|\tfrac{1}{\#_{\gamma(\floorgrid{t})}}\textstyle\sum_{j\in\gamma(\floorgrid{t})}G(\Theta_{\floorgrid{t}},Z_j)\big\|_{\R^d}^p \Big]\\
	 &\leq \tfrac{1}{|\#_{\gamma(\floorgrid{t})}|^p}\E\Big[ \Big(\textstyle\sum_{j\in\gamma(\floorgrid{t})}\|G(\Theta_{\floorgrid{t}},Z_j)\|_{\R^d}\Big)^p \Big]\\
	& \leq \tfrac{1}{|\#_{\gamma(\floorgrid{t})}|^p}\E\Big[ |\#_{\gamma(\floorgrid{t})}|^{p-1}\textstyle\sum_{j\in\gamma(\floorgrid{t})}\|G(\Theta_{\floorgrid{t}},Z_j)\|_{\R^d}^p \Big]\\
	&  = \tfrac{1}{\#_{\gamma(\floorgrid{t})}}\E\Big[\textstyle\sum_{j\in\gamma(\floorgrid{t})}\|G(\Theta_{\floorgrid{t}},Z_j)\|_{\R^d}^p \Big].
	\end{split}
	\end{equation}
	Moreover, observe that combining the assumption that $Z_j$, $j \in \N$, are i.i.d.\ random variables, the assumption that $\forall\, (v,w)\in\lb (a,b)\in [0,\infty)^2\colon  a\neq b \rb\colon \gamma(v)\cap \gamma(w)=\emptyset $, and \eqref{eq: equation Theta 2} proves that for all  $j\in\gamma(\floorgrid{t})$ we have that $Z_{j}$ and $\Theta_{\floorgrid{t}}$ are independent. This, \eqref{eq:Q_t_leq_Theta_floor_t_3}, the assumption that $\mathfrak{j}\in\gamma(\floorgrid{t})$, and the assumption that $Z_j$, $j\in\N$, are i.i.d.\ random variables ensure that 
	\begin{equation}\label{eq:Q_t_leq_Theta_floor_t_4}
	\begin{split}
	& \E\Big[  \Big\|\tfrac{1}{\#_{\gamma(\floorgrid{t})}}\textstyle\sum_{j\in\gamma(\floorgrid{t})}G(\Theta_{\floorgrid{t}},Z_j)\Big\|_{\R^d}^p \Big]\\
	& \leq \tfrac{1}{\#_{\gamma(\floorgrid{t})}}\textstyle\sum_{j\in\gamma(\floorgrid{t})}\E\big[\|G(\Theta_{\floorgrid{t}},Z_{\mathfrak{j}})\|_{\R^d}^p \big]\\
	& = \E\big[\|G(\Theta_{\floorgrid{t}},Z_{\mathfrak{j}})\|_{\R^d}^p \big]\\
	& = \int_{\Omega} \|G(\Theta_{\floorgrid{t}}(\omega),Z_{\mathfrak{j}}(\omega))\|_{\R^d}^p\, \P(d\omega)\\
	& = \int_{\Omega}\int_{\Omega} \|G(\Theta_{\floorgrid{t}}(\omega),Z_{\mathfrak{j}}(\tilde{\omega}))\|_{\R^d}^p\, \P(d\tilde{\omega})\, \P(d\omega).
	\end{split}
	\end{equation}
	Combining this and \eqref{eq:Q_t_leq_Theta_floor_t_1}  demonstrates that
	\begin{align}
	 \label{eq:Q_t_leq_Theta_floor_t_5}
	 \begin{split}
	\E\Big[ \Big\|\tfrac{1}{\#_{\gamma(\floorgrid{t})}}\textstyle\sum_{j\in\gamma(\floorgrid{t})}G(\Theta_{\floorgrid{t}},Z_j)\Big\|_{\R^d}^p \Big]
	& \leq \int_{\Omega} c(1 + \|\Theta_{\floorgrid{t}}(\omega)\|_{\R^d}^{\mathfrak{m}p})\, \P(d\omega)\\
	& = \E\big[c(1 + \|\Theta_{\floorgrid{t}}\|_{\R^d}^{\mathfrak{m}p})\big]\\
	& = c\big(1  +\E\big[\|\Theta_{\floorgrid{t}}\|_{\R^d}^{\mathfrak{m}p}\big]\big). 
	\end{split}
	\end{align}
	This establishes \eqref{eq:Q_t_leq_Theta_floor_t_2}. The proof of Lemma~\ref{lem:Q_t_leq_Theta_floor_t} is thus completed.
\end{proof}

\begin{lemma}\label{lem:Theta_L^2_sp_case}
	Assume Setting~\ref{sec:setting_1},
	assume for all $v,w\in [0,\infty)$ with $v\neq w$ that 
	$
	\gamma(v)\cap \gamma(w)=\emptyset,
	$  let $p\in  [1,\infty)$, 
	and assume that
	\begin{equation}\label{eq:Theta_L^2_sp_case_0}
	\sup_{x\in\R^d}\left(\frac{\E\big[\|G(x,Z_1)\|_{\R^d}^p\big]}{\big[ 1 + \|x\|_{\R^d} \big]^p}\right)<\infty.
	\end{equation}
	Then we have for all $t\in [0,\infty)$ that
	\begin{equation}\label{eq:Theta_L^2_sp_case_1_2}
	\E\big[\|\Theta_{\floorgrid{t}}\|_{\R^d}^p\big]<\infty.
	\end{equation}
\end{lemma}
\begin{proof}[Proof of Lemma~\ref{lem:Theta_L^2_sp_case}]
	Throughout this proof 
	let $\mathfrak{t}\colon \N_0\to [0,\infty)$ be a non-decreasing function which satisfies that
	\begin{equation}\label{eq:Theta_L^2_sp_case_1_3}
	\lb t\in [0,\infty)\colon \gamma(t)\neq \emptyset\rb = \lb \mathfrak{t}_n \colon n \in \N_0 \rb. 
	\end{equation}
	Observe that \eqref{eq:Theta_L^2_sp_case_0} implies that there exists  $c\in [0,\infty)$ which satisfies for all $x\in\R^d$  that
	\begin{equation}\label{eq:Theta_L^2_sp_case_1}
	\E\big[\|G(x,Z_1)\|_{\R^d}^p \big]\leq c(1+\|x\|_{\R^d}^p).
	\end{equation}
	Next note that the assumption that $\forall\, t\in [0,\infty)\colon 0<\#_{\lb s\in [0,t]\colon \gamma(s)\neq \emptyset \rb}$ assures that $\mathfrak{t}_0=0$. This and the assumption that $\Theta_0=\xi$ imply that
	\begin{equation}\label{eq:Theta_L^2_sp_case_4}
	\E\big[\|\Theta_{\mathfrak{t}_0}\|_{\R^d}^p\big]
	=\E\big[\|\Theta_0\|_{\R^d}^p\big]=\|\xi\|_{\R^d}^p<\infty.
	\end{equation}
	Furthermore, observe that the Minkowski inequality and \eqref{eq: equation Theta 2} ensure that for all $n \in \N_0$ we have that
	\begin{align}
	\begin{split}
	&\big|\E\big[\|\Theta_{\mathfrak{t}_{n+1}}\|_{\R^d}^p\big]\big|^{1/p}
	 = \Big|\E\Big[\|\Theta_{\mathfrak{t}_{n}}+\tfrac{(\mathfrak{t}_{n+1}-\mathfrak{t}_n)}{\#_{\gamma(\mathfrak{t}_n)}} \textstyle\sum_{j\in \gamma(\mathfrak{t}_n)} G(\Theta_{\mathfrak{t}_n},Z_j)\|_{\R^d}^p\Big]\Big|^{1/p}    \\
	& \leq \big|\E\big[\|\Theta_{\mathfrak{t}_{n}}\|_{\R^d}^p\big]\big|^{1/p} +  (\mathfrak{t}_{n+1}-\mathfrak{t}_n)\Big|\E\Big[\big\|\tfrac{1}{\#_{\gamma(\mathfrak{t}_n)}} \textstyle\sum_{j\in \gamma(\mathfrak{t}_n)}G(\Theta_{\mathfrak{t}_n},Z_j)\big\|_{\R^d}^p\Big]\Big|^{1/p}.
	\end{split}
	\end{align}
	This, \eqref{eq:Theta_L^2_sp_case_1}, and Lemma~\ref{lem:Q_t_leq_Theta_floor_t} prove that  for all $n \in \N_0$ we have that
	\begin{equation}
	\begin{split}
	\big|\E\big[\|\Theta_{\mathfrak{t}_{n+1}}\|_{\R^d}^p\big]\big|^{1/p}
	& \leq \big|\E\big[\|\Theta_{\mathfrak{t}_{n}}\|_{\R^d}^p\big]\big|^{1/p} +  (\mathfrak{t}_{n+1}-\mathfrak{t}_n)\big|c\big(1  +\E\big[\|\Theta_{\mathfrak{t}_n}\|_{\R^d}^{p}\big]\big)\big|^{1/p}.
	\end{split}
	\end{equation}
	Induction and \eqref{eq:Theta_L^2_sp_case_4} therefore assure that for all $n \in \N_0$ we have that
	\begin{equation}\label{eq:Theta_L^2_sp_case_2}
	\E\big[\|\Theta_{\mathfrak{t}_n}\|_{\R^d}^p\big]<\infty.
	\end{equation}
	This establishes  \eqref{eq:Theta_L^2_sp_case_1_2}. The proof of Lemma~\ref{lem:Theta_L^2_sp_case} is thus completed.
\end{proof}

\begin{cor}\label{cor:sup_Theta_L^2_bnd_on_0_T}
	Assume Setting~\ref{sec:setting_1},
	assume for all $v,w\in [0,\infty)$ with $v\neq w$ that 
	$
	\gamma(v)\cap \gamma(w)=\emptyset
	$, 
	and let $p\in [1,\infty)$, $c\in [0,\infty)$ satisfy for all $x\in\R^d$  that
	\begin{equation}\label{eq:sup_Theta_L^2_bnd_on_0_T_0}
	\E\big[\|G(x,Z_1)\|_{\R^d}^p\big]
	\leq c\big( 1 + \|x\|_{\R^d}^p \big).
	\end{equation}
	Then 
	\begin{enumerate}[(i)]
		\item\label{item:sup_Theta_L^2_bnd_on_0_T_1} we have for all $t\in [0,\infty)$ that
		\begin{equation}
		\E\big[\|\Theta_t\|_{\R^d}^p\big]
		\leq 2^{p - 1} c (t - \floorgrid{t})^p + 2^{p - 1}(1 + c(t - \floorgrid{t})^p)\E\big[\|\Theta_{\floorgrid{t}}\|_{\R^d}^p\big]\\
		\end{equation}
		and
		\item\label{item:sup_Theta_L^2_bnd_on_0_T_2} we have for all $T\in [0,\infty)$ that
		\begin{equation}\label{eq:sup_Theta_L^2_bnd_on_0_T_2}
		\sup_{t\in[0,T]}\E\big[ \|\Theta_t\|_{\R^d}^p \big]<\infty.
		\end{equation}
	\end{enumerate}
\end{cor}
\begin{proof}[Proof of Corollary~\ref{cor:sup_Theta_L^2_bnd_on_0_T}]
	First, note that \eqref{eq: equation Theta 2} and the Minkowski inequality imply that for all $t\in [0,\infty)$ we have that
	\begin{align}
	\label{eq:sup_Theta_L^2_bnd_on_0_T_3}
	\begin{split}
	&\big|\E\big[\|\Theta_t\|_{\R^d}^p\big]\big|^{1/p}
	 = \Big|\E\Big[\big\|\Theta_{\floorgrid{t}} + \tfrac{(t-\floorgrid{t})}{\#_{\gamma(\floorgrid{t})}}\textstyle\sum_{j\in \gamma(\floorgrid{t})}G(\Theta_{\floorgrid{t}},Z_j)\big\|_{\R^d}^p\Big]\Big|^{1/p}  \\
	& \leq \big|\E\big[\|\Theta_{\floorgrid{t}}\|_{\R^d}^p\big]\big|^{1/p} 
	+ (t-\floorgrid{t})\Big|\E\Big[\big\|\tfrac{1}{\#_{\gamma(\floorgrid{t})}} \textstyle\sum_{j\in \gamma(\floorgrid{t})}G(\Theta_{\floorgrid{t}},Z_j)\big\|_{\R^d}^p\Big]\Big|^{1/p}. 
	\end{split}
	\end{align}
	Lemma~\ref{lem:Q_t_leq_Theta_floor_t} and \eqref{eq:sup_Theta_L^2_bnd_on_0_T_0} hence prove that for all $t\in [0,\infty)$ we have that
	\begin{equation}\label{eq:sup_Theta_L^2_bnd_on_0_T_4}
	\begin{split}
	\big|\E\big[\|\Theta_t\|_{\R^d}^p\big]\big|^{1/p}
	& \leq \big|\E\big[\|\Theta_{\floorgrid{t}}\|_{\R^d}^p\big]\big|^{1/p} 
	+ (t-\floorgrid{t})\big|c \big( 1 + \E\big[\|\Theta_{\floorgrid{t}}\|_{\R^d}^{p}\big]\big)\big|^{1/p}.
	\end{split}
	\end{equation}
	Lemma~\ref{lem:bnd_(sum_xi)^p} therefore demonstrates that for all $t\in [0,\infty)$ we have that
	\begin{equation}\label{eq:sup_Theta_L^2_bnd_on_0_T_5}
	\begin{split}
	\E\big[\|\Theta_t\|_{\R^d}^p\big]
	& \leq 2^{p-1}\E\big[\|\Theta_{\floorgrid{t}}\|_{\R^d}^p\big] 
	+ (t-\floorgrid{t})^p2^{p-1}c \big( 1 + \E\big[\|\Theta_{\floorgrid{t}}\|_{\R^d}^{p}\big]\big)\\
	& = 
	2^{p-1} c (t-\floorgrid{t})^p + 2^{p-1}(1+c(t-\floorgrid{t})^p)\E\big[\|\Theta_{\floorgrid{t}}\|_{\R^d}^p\big].
	\end{split}
	\end{equation}
	This proves item~\eqref{item:sup_Theta_L^2_bnd_on_0_T_1}. Next note that \eqref{eq:sup_Theta_L^2_bnd_on_0_T_5}  ensures that for all $T\in [0,\infty)$ we have that
	\begin{equation}\label{eq:sup_Theta_L^2_bnd_on_0_T_6}
	\sup_{t\in[0,T]}\E\big[ \|\Theta_t\|_{\R^d}^p \big]
	\leq  2^{p - 1} c T^p + 2^{p - 1}(1 + cT^p)\sup_{t\in[0,T]}\E\big[\|\Theta_{\floorgrid{t}}\|_{\R^d}^p\big].
	\end{equation}
	The assumption that $\forall \, T\in [0,\infty)\colon 0< \#_{\lb s \in [0,T]\colon\gamma(s)\neq \emptyset\rb} <\infty$ and Lemma~\ref{lem:Theta_L^2_sp_case} hence imply that for all $T\in [0,\infty)$ we have that
	\begin{equation}\label{eq:sup_Theta_L^2_bnd_on_0_T_7}
	\sup_{t\in[0,T]}\E\big[ \|\Theta_t\|_{\R^d}^p \big]
	< \infty.
	\end{equation}
	This establishes item~\eqref{item:sup_Theta_L^2_bnd_on_0_T_2}. The proof of Corollary~\ref{cor:sup_Theta_L^2_bnd_on_0_T} is thus completed.
\end{proof}

\begin{lemma}\label{lem:equiv_sup_L^p_Theta_t}
	Assume Setting~\ref{sec:setting_1}, assume that $\#_{\lb t\in [0,\infty)\colon \gamma(t)\neq \emptyset\rb}=\infty$,
	and let $p\in [1,\infty)$.
	Then the following two statements are equivalent:
	\begin{enumerate}[(i)]
		\item\label{item:equiv_sup_L^p_Theta_t_1} It holds that 
		\begin{equation}
		\sup_{t\in [0,\infty)}\E\big[\|\Theta_{t}\|_{\R^d}^p\big]<\infty.
		\end{equation}
		\item\label{item:equiv_sup_L^p_Theta_t_2} It holds that 
		\begin{equation}
		\sup_{t\in [0,\infty)}\E\big[\|\Theta_{\floorgrid{t}}\|_{\R^d}^p\big]<\infty.
		\end{equation}
	\end{enumerate}
\end{lemma}
\begin{proof}[Proof of Lemma~\ref{lem:equiv_sup_L^p_Theta_t}]
	Throughout  this proof let $\ceil{\cdot}\colon [0,\infty)\to [0,\infty]$ satisfy for all $t\in [0,\infty)$ that
	\begin{equation}\label{eq:equiv_sup_L^p_Theta_t_2}
	\ceil{t}=\inf \big(\lb s\in (t,\infty)\colon \gamma(s)\neq\emptyset \rb\cup \lb \infty \rb \big).
	\end{equation}
	Note that it is clear that $(\eqref{item:equiv_sup_L^p_Theta_t_1} \Rightarrow \eqref{item:equiv_sup_L^p_Theta_t_2})$. Next we prove that $(\eqref{item:equiv_sup_L^p_Theta_t_2} \Rightarrow \eqref{item:equiv_sup_L^p_Theta_t_1})$. 
	Observe that the assumption that $\forall\, t\in[0,\infty)\colon \#_{\lb s \in [0,t]\colon\gamma(s)\neq \emptyset\rb}<\infty$ and the assumption that $\#_{\lb t\in [0,\infty)\colon \gamma(t)\neq \emptyset\rb}=\infty$ assure that for all $t\in [0,\infty)$ we have that
	\begin{equation}\label{eq:equiv_sup_L^p_Theta_t_3}
	\ceil{t}<\infty.
	\end{equation}
	This and \eqref{eq: equation Theta 2} prove that for all $t\in [0,\infty)$ we have that
	\begin{equation}\label{eq:equiv_sup_L^p_Theta_t_4}
	\Theta_{\ceil{t}}
	=\Theta_{\floorgrid{t}} + (\ceil{t} - \floorgrid{t})\tfrac{1}{\#_{ \gamma(\floorgrid{t})}}\textstyle\sum_{j\in \gamma(\floorgrid{t})} G(\Theta_{\floorgrid{t}}, Z_j).
	\end{equation}
	Combining this and \eqref{eq: equation Theta 2}  ensures that for all $t\in [0,\infty)$ we have that
	\begin{align}
	\label{eq:equiv_sup_L^p_Theta_t_5}
	\begin{split}
	&\tfrac{\ceil{t} - t}{\ceil{t} - \floorgrid{t}}\Theta_{\floorgrid{t}} + \tfrac{t - \floorgrid{t}}{\ceil{t} - \floorgrid{t}}\Theta_{\ceil{t}}\\
	& = \tfrac{\ceil{t} - t}{\ceil{t} - \floorgrid{t}}\Theta_{\floorgrid{t}} + \tfrac{t - \floorgrid{t}}{\ceil{t} - \floorgrid{t}}\big(\Theta_{\floorgrid{t}} + (\ceil{t} - \floorgrid{t})\tfrac{1}{\#_{\gamma(\floorgrid{t})}}\textstyle\sum_{j\in \gamma(\floorgrid{t})} G(\Theta_{\floorgrid{t}}, Z_j)\big)\\
	& = \tfrac{\ceil{t} - t}{\ceil{t} - \floorgrid{t}}\Theta_{\floorgrid{t}} + \tfrac{t - \floorgrid{t}}{\ceil{t} - \floorgrid{t}}\Theta_{\floorgrid{t}} + (t - \floorgrid{t})\tfrac{1}{\#_{\gamma(\floorgrid{t})}}\textstyle\sum_{j\in \gamma(\floorgrid{t})} G(\Theta_{\floorgrid{t}}, Z_j)\\
	& = \Theta_{\floorgrid{t}} + (t - \floorgrid{t})\tfrac{1}{\#_{\gamma(\floorgrid{t})}}\textstyle\sum_{j\in \gamma(\floorgrid{t})} G(\Theta_{\floorgrid{t}}, Z_j) = \Theta_t. 
	\end{split}
	\end{align}
	This and the triangle inequality imply that  for all $t\in[0,\infty)$ we have that
	\begin{equation}
	\| \Theta_t \|_{\R^d}\leq \|\Theta_{\floorgrid{t}}\|_{\R^d} + \|\Theta_{\ceil{t}}\|_{\R^d}.
	\end{equation}
	Lemma~\ref{lem:bnd_(sum_xi)^p} therefore demonstrates that
	\begin{equation}
	\begin{split}
	\sup_{t\in [0,\infty)}\E\big[\|\Theta_{t}\|_{\R^d}^p\big]
	&\leq \sup_{t\in [0,\infty)}\E\big[\big(\|\Theta_{\floorgrid{t}}\|_{\R^d} + \|\Theta_{\ceil{t}}\|_{\R^d}\big)^p\big]\\
	&\leq 2^{p-1}\sup_{t\in [0,\infty)}\E\big[\|\Theta_{\floorgrid{t}}\|_{\R^d}^p + \|\Theta_{\ceil{t}}\|_{\R^d}^p\big]\\
	&\leq 2^{p}\sup_{t\in [0,\infty)}\E\big[\|\Theta_{\floorgrid{t}}\|_{\R^d}^p \big].
	\end{split}
	\end{equation}
	This reveals that $(\eqref{item:equiv_sup_L^p_Theta_t_2} \Rightarrow \eqref{item:equiv_sup_L^p_Theta_t_1})$. The proof of Lemma~\ref{lem:equiv_sup_L^p_Theta_t} is thus completed.
\end{proof}

\section{Weak error estimates for SAAs in the case of general learning rates with mini-batches}
\label{subsec:weak:SAA:general}
\sectionmark{}

\begin{prop}\label{prop:main_theorem}
Assume Setting~\ref{sec:setting_1},  assume for all $v,w\in [0,\infty)$ with $v\neq w$ that 
$
\gamma(v)\cap \gamma(w)=\emptyset,
$ 
let $\psi\in C^2(\R^d,\R)$, $ L\in (0,\infty)$,  assume for all $y,z\in\R^d$ that
\begin{equation}
\langle g(y)-g(z), y-z\rangle_{\R^d} \leq -L \|y-z\|_{\R^d}^2,
\end{equation}
assume that
\begin{equation}\label{eq:main_theorem_001}
\sup_{x\in\R^d}\Bigg( \frac{\E\big[ \|G(x,Z_1)\|_{\R^d}^2 \big]}{\big[1+\|x\|_{\R^d}\big]^2}
+ \frac{\big\|\E\big[(\tfrac{\partial}{\partial x}G)(x,Z_1)\big]\big\|_{L(\R^d,\R^d)}}{\big[1+\|x\|_{\R^d}\big]}
+ \|\psi'(x)\|_{L(\R^d, \R)} \Bigg) < \infty,
\end{equation}
and let $Q\colon [0,\infty)\times\Omega\to\R^d$ be the stochastic process which satisfies for all $t\in [0,\infty)$ that
\begin{equation}\label{eq:main_theorem_09}
Q_t=\tfrac{1}{\#_{\gamma(\floorgrid{t})}} \Big[\textstyle\sum_{j\in \gamma(\floorgrid{t})} G(\Theta_{\floorgrid{t}},Z_j) \Big].
\end{equation}
Then 
\begin{enumerate}[(i)]
\item\label{item:main_theorem_1} we have that $g \in C^2(\R^d,\R^d)$,
\item\label{item:main:limit} we have that there exists a unique  $\Xi \in \R^d$ which satisfies that
\begin{equation}
\limsup_{t \to \infty} \|\theta_t^{\xi} - \Xi\|_{\R^d} = 0,
\end{equation}
and
\item\label{item:main_theorem_2} we have for all $T\in (0,\infty)$ that
\end{enumerate}
\begin{align*}
& | \E[\psi(\Theta_T)]- \psi(\Xi)|  \leq  \sup_{s,v\in [0,T]}\E \bigg[ \|Q_s-g(\Theta_{\floorgrid{s}})\|_{\R^d}\norm{Q_s}_{\R^d} \\
& \cdot \bigg(\int_0^1  \exp(-L(T-s)) \big\|\psi''( \theta_{T-s}^{\lambda \Theta_s + (1-\lambda)\Theta_{\floorgrid{s}}})\big\|_{L^{(2)}(\R^d,\R)} \numberthis \label{eq:main_theorem} \\
&+ \big\|\psi'( \theta_{T-s}^{\lambda \Theta_s + (1-\lambda)\Theta_{\floorgrid{s}}} )\big\|_{L(\R^d,\R)} \int_0^{T-s} \exp(-Lu)\big\|g''(\theta_u^{\lambda \Theta_s + (1-\lambda)\Theta_{\floorgrid{s}}})\big\|_{L^{(2)}(\R^d,\R^d)}\, du\, d\lambda \bigg)  \\
& + \big\|\psi'(\theta_{T-s}^{\Theta_s})\big\|_{L(\R^d,\R)} \big\|g'(\Theta_{v})Q_v \big\|_{\R^d}\bigg] \int_0^T \exp(-L(T-t))(t-\floorgrid{t})\, dt\\
& + \sup\left\lbrace \big\| \psi'(\lambda\theta_T^{\xi}+(1-\lambda)\Xi) \big\|_{L(\R^d,\R)} \in \R\colon \lambda\in [0,1] \right\rbrace \! \|\xi-\Xi\|_{\R^d} \exp(-LT).
\end{align*}
\end{prop}
\begin{proof}[Proof of Proposition~\ref{prop:main_theorem}]
Throughout this proof
let $T\in (0,\infty)$,  let $E \subseteq [0,T]$ be the set given by
\begin{equation}
E=\lb t\in [0,T]\colon \gamma(t)\neq \emptyset \rb\cup\lb T\rb,
\end{equation} 
let $u  = (u(t,\vartheta))_{(t, \vartheta) \in [0, T] \times \R^d} \in C([0,T]\times \R^d, \R)$  satisfy for all $t\in [0,T]$, $\vartheta \in \R^d$ that
\begin{equation}\label{eq:def_u_main_theorem}
u(t,\vartheta)=\psi (\theta_{T-t}^{\vartheta}),
\end{equation}
let $u_{1,0}  = (u_{1,0}(t,\vartheta))_{(t, \vartheta) \in [0, T] \times \R^d}  \in C( [0,T]\times\R^d, L(\R,\R))$  satisfy for all $t\in [0,T]$, $\vartheta\in \R^d$ that
\begin{equation}\label{eq:main_theorem_01}
u_{1,0}(t,\vartheta)=(\pt u)(t,\vartheta)
\end{equation}
(cf.~item \eqref{item:Kolm_back_eq_main_1} in Lemma~\ref{lem:Kolm_back_eq_main}),
let $u_{0,1}  = (u_{0,1}(t,\vartheta))_{(t, \vartheta) \in [0, T] \times \R^d} \in C( [0,T]\times \R^d, \allowbreak L(\R^d,\R))$  satisfy for all $t\in [0,T]$, $\vartheta\in \R^d$ that
\begin{equation}\label{eq:main_theorem_02}
u_{0,1}(t,\vartheta)=(\pv u)(t,\vartheta)
\end{equation}
(cf.~item \eqref{item:Kolm_back_eq_main_1} in Lemma~\ref{lem:Kolm_back_eq_main}),
let $\delta\colon [0,T]\times\Omega\to \R^d$ be the stochastic process which satisfies for all $t\in [0,T]$ that
\begin{equation}\label{eq:main_theorem_05}
\delta_t = \Theta_t-\Theta_{\floorgrid{t}},
\end{equation}
let $\theta^{1,\vartheta}\in C( [0,\infty), L(\R^d, \R^d))$, $\vartheta\in\R^d$,  satisfy for all $t\in [0,\infty)$, $\vartheta\in\R^d$ that 
\begin{equation}\label{eq:main_theorem_07}
\theta_t^{1,\vartheta}=\pv \theta_t^{\vartheta}
\end{equation}
(cf.~item \eqref{item:two_expressions_2} in  Lemma~\ref{lem:two_expressions}),
let $\theta^{2,\vartheta}\in C( [0,\infty), L^{(2)}(\R^d, \R^d))$, $\vartheta\in\R^d$,   satisfy for all $t\in [0,\infty)$, $\vartheta\in\R^d$ that 
\begin{equation}\label{eq:main_theorem_08}
\theta_t^{2,\vartheta}=\tfrac{\partial^2}{\partial \vartheta^2}\theta_t^{\vartheta}
\end{equation}
(cf.~item \eqref{item:two_expressions_2} in  Lemma~\ref{lem:two_expressions}),
let $a^\lambda\colon [0,\infty)\times\Omega\to\R^d$, $\lambda\in[0,1]$, be the  stochastic processes which satisfy for all $\lambda\in [0,1]$, $t\in [0,\infty)$ that
\begin{equation}\label{eq:main_theorem_04}
a_t^{\lambda}=\lambda \Theta_t + (1-\lambda)\Theta_{\floorgrid{t}},
\end{equation}
and let $\Delta\colon [0,\infty)\times\Omega\to\R^d$ be the stochastic process which satisfies for all $t\in [0,\infty)$ that
\begin{equation}\label{eq:main_theorem_06}
\Delta_t = \tfrac{1}{\#_{\gamma(\floorgrid{t})}}\Big[ \textstyle\sum_{j\in \gamma(\floorgrid{t})} G(\Theta_{\floorgrid{t}},Z_j) \Big] -g(\Theta_{\floorgrid{t}}) = Q_t-g(\Theta_{\floorgrid{t}}).
\end{equation}
Observe that item~\eqref{item:twice:diff:f} in  Lemma~\ref{lem:twice:diff} establishes item~\eqref{item:main_theorem_1}. 
Next note that  Lemma~\ref{lem:motionless_pt_f} ensures that there exists a unique $\Xi \in \R^d$ which satisfies that $g(\Xi)=0$ and 
\begin{equation}
\limsup_{t \to \infty} \|\theta_t^{\xi} - \Xi\|_{\R^d} = 0.
\end{equation}
This establishes item~\eqref{item:main:limit}.
Next observe that the assumption that $\forall\, t\in [0,\infty) \colon 0<\#_{\lb s\in [0,t]\colon \gamma(s)\neq \emptyset \rb}<\infty$ ensures that there exist $k\in \N$, $\mathfrak{t}_1, \mathfrak{t}_2, \dots,\mathfrak{t}_k \in [0,T]$ which satisfy that 
\begin{equation}\label{eq:main_theorem_8}
0=\mathfrak{t}_1<\mathfrak{t}_2<\dots<\mathfrak{t}_k=\floorgrid{T}\quad \text{and}\quad \lb t\in [0,T]\colon \gamma(t)\neq \emptyset\rb = \lb \mathfrak{t}_1, \mathfrak{t}_2, \dots,\mathfrak{t}_k\rb.
\end{equation}
Note that \eqref{eq:main_theorem_8} implies that there exists  $\mathfrak{j}\colon\lb 1, 2, \dots,k\rb\to\N$ which satisfies for all $n\in\lb 1, 2, \dots,k\rb$ that
\begin{equation}\label{eq:main_theorem_8p}
\mathfrak{j}_n\in\gamma(\mathfrak{t}_n).
\end{equation}
Next note that \eqref{eq:main_theorem_001} and Lemma~\ref{lem:f_C^1} assure that there exists  $c\in (0,\infty)$ which satisfies for all $x\in\R^d$  that
\begin{equation}\label{eq:main_theorem_0012}
\begin{split}
& \max\big\lb \E\big[ \|G(x,Z_1)\|_{\R^d}^2 \big]^{1/2}, \|g(x)\|_{\R^d}, \|g'(x)\|_{L(\R^d,\R^d)}\big\rb \leq c(1+\|x\|_{\R^d}),
\end{split}
\end{equation}
\begin{equation}
\label{eq:main:pol:growth}
\E\big[\|G(x,Z_1)\|_{\R^d}^2\big]\leq c(1+\|x\|_{\R^d}^2),
\end{equation}
and 
\begin{equation}
\label{eq:main:psi}
 \|\psi'(x)\|_{L(\R^d,\R)}\leq c.
\end{equation}
Item~\eqref{item:estimate_u_01_3} in Lemma~\ref{lem:estimate_u_01} hence ensures that for all $t\in [0, T]$, $\vartheta\in\R^d$ we have that
\begin{equation}\label{eq:main_theorem_3_2_2}
\begin{split}
\|u_{0,1}(t,\vartheta)\|_{L(\R^d,\R)} &= \|(\pv u)(t,\vartheta)\|_{L(\R^d,\R)}\leq \|\psi'(\theta_{T-t}^{\vartheta})\|_{L(\R^d,\R)}e^{-L(T-t)}\leq c.
\end{split}
\end{equation}
Next note  Lemma~\ref{lem:Q_t_leq_Theta_floor_t}  and \eqref{eq:main:pol:growth} imply that for all $t \in [0,T]$ we have that
\begin{equation}
\E\big[\|Q_t\|_{\R^d}^2\big] \leq c\big(1+\E\big[\|\Theta_{\floorgrid{t}}\|_{\R^d}^2\big]\big).
\end{equation}
Jensen's inequality therefore proves that for all $t \in [0,T]$ we have that
\begin{equation}
\label{eq:main:Q}
 \E\big[\|Q_t\|_{\R^d} \big]  \leq \big| \E\big[\|Q_t\|_{\R^d}^2 \big] \big|^{1/2} \leq \big|c \big(1+\E\big[\|\Theta_{\floorgrid{t}}\|_{\R^d}^2\big]\big) \big|^{1/2}.
\end{equation}
Moreover, observe that \eqref{eq:main_theorem_0012} and Jensen's inequality ensure that for all $t \in [0,T]$ we have that
\begin{equation}
\begin{split}
\E\big[  \|g(\Theta_{t})\|_{\R^d} \big] \leq c\big(1 + \E\big[\|\Theta_{t}\|_{\R^d}\big]\big)  \leq c\big(1 + \big|\E\big[\|\Theta_{t}\|_{\R^d}^2\big]\big|^{1/2}\big).
\end{split}
\end{equation}
This, \eqref{eq:main_theorem_3_2_2}, and \eqref{eq:main:Q}  assure that
\begin{align*}
\label{eq:main:u01}
& \int_0^T\E\big[(\|u_{0,1}(t,\Theta_t)\|_{L(\R^d,\R)} + \|u_{0,1}(t,\Theta_{\floorgrid{t}})\|_{L(\R^d,\R)}) \\
& \quad \cdot (\|Q_t\|_{\R^d} + \|g(\Theta_{\floorgrid{t}})\|_{\R^d} + \|g(\Theta_{t})\|_{\R^d} ) \big]\, dt\\
& \leq \int_0^T 2c \Big(\big|c\big(1 + \E\big[\|\Theta_{\floorgrid{t}}\|_{\R^d}^2\big]\big)\big|^{1/2} + c\big(1 + \big|\E\big[\|\Theta_{\floorgrid{t}}\|_{\R^d}^2\big]\big|^{1/2}\big) \numberthis \\
& \quad + c\big(1 + \big|\E\big[\|\Theta_{t}\|_{\R^d}^2\big]\big|^{1/2}\big)\Big) \, dt\\
& \leq  2cT \Big(\big|c\big(1 + \big(\sup_{t\in [0,T]}\E\big[\|\Theta_{\floorgrid{t}}\|_{\R^d}^2\big]\big)\big)\big|^{1/2} + 2c\big(1 + \big|\big(\sup_{t\in [0,T]}\E\big[\|\Theta_{t}\|_{\R^d}^2\big]\big)\big|^{1/2}\big) \Big).
\end{align*}
In addition, observe that Corollary~\ref{cor:sup_Theta_L^2_bnd_on_0_T} and \eqref{eq:main:pol:growth}  prove that
\begin{equation}
\label{eq:main:sup}
\sup_{t\in [0,T]}\E\big[\|\Theta_{t}\|_{\R^d}^2\big] < \infty.
\end{equation}
This and \eqref{eq:main:u01} demonstrate that
\begin{equation}\label{eq:main_theorem_7_02}
\begin{split}
& \int_0^T\E\big[ (\|u_{0,1}(t,\Theta_t)\|_{L(\R^d,\R)} + \|u_{0,1}(t,\Theta_{\floorgrid{t}})\|_{L(\R^d,\R)})\\
 & \quad \cdot (\|Q_t\|_{\R^d} + \|g(\Theta_{\floorgrid{t}})\|_{\R^d} + \|g(\Theta_{t})\|_{\R^d} ) \big]\, dt <\infty.
\end{split}
\end{equation}
Furthermore, note that \eqref{eq:main_theorem_3_2_2} and Jensen's inequality imply that for all $s,t\in [0,T]$, $j\in \N$ we have that
\begin{align}
\label{eq:main_theorem_7_03}
\begin{split}
& \E\big[ | u_{0,1}(t,\Theta_{s})  G(\Theta_{s},Z_j)| + |u_{0,1}(t,\Theta_{s}) g(\Theta_{s})| \big]\\
& \leq \E\big[ \| u_{0,1}(t,\Theta_{s})\|_{L(\R^d,\R)}\|  G(\Theta_{s},Z_j)\|_{\R^d} + \|u_{0,1}(t,\Theta_{s})\|_{L(\R^d,\R)} \| g(\Theta_{s})\|_{\R^d} \big]\\
& \leq c\,\E\big[ \|  G(\Theta_{s},Z_j)\|_{\R^d}\big] + c\,\E\big[ \| g(\Theta_{s})\|_{\R^d} \big]\\
& \leq c\,\big|\E\big[ \|  G(\Theta_{s},Z_j)\|_{\R^d}^2\big]\big|^{1/2} + c\,\E\big[ \| g(\Theta_{s})\|_{\R^d} \big]. 
\end{split}
\end{align}
Moreover, observe that the assumption that $Z_j$, $j \in \N$, are i.i.d.\ random variables, the assumption that $\forall\, (v,w)\in\lb (a,b)\in [0,\infty)^2\colon  a\neq b \rb\colon \gamma(v)\cap \gamma(w)=\emptyset $, and \eqref{eq: equation Theta 2} prove that for all  $s\in [0,T]$, $j\in\gamma(\floorgrid{s})$ we have that $Z_j$ and $\Theta_{\floorgrid{s}}$
are independent. This and the assumption that $Z_j$, $j \in \N$, are i.i.d.\ random variables ensure that for all $s\in [0,T]$, $j\in\gamma(\floorgrid{s})$ we have that
\begin{equation}\label{eq:main_theorem_7_04}
\begin{split}
 \E\big[ \| G(\Theta_{\floorgrid{s}},Z_j) \|_{\R^d}^2\big]
& = \int_{\Omega}\| G(\Theta_{\floorgrid{s}}(\omega),Z_j(\omega)) \|_{\R^d}^2\, \P(d\omega)\\
& = \int_{\Omega}\int_{\Omega}\| G(\Theta_{\floorgrid{s}}(\omega),Z_j(\tilde{\omega})) \|_{\R^d}^2\, \P(d\tilde{\omega})\, \P(d\omega)\\
& = \int_{\Omega}\int_{\Omega}\| G(\Theta_{\floorgrid{s}}(\omega),Z_1(\tilde{\omega})) \|_{\R^d}^2\, \P(d\tilde{\omega})\, \P(d\omega).
\end{split}
\end{equation}
Combining this with \eqref{eq:main:pol:growth} implies that for all $s\in [0,T]$, $j\in\gamma(\floorgrid{s})$ we have that
\begin{equation}\label{eq:main_theorem_7_05}
\begin{split}
 \E\big[ \| G(\Theta_{\floorgrid{s}},Z_j) \|_{\R^d}^2\big]
& \leq \int_{\Omega}c\big(1 + \| \Theta_{\floorgrid{s}}(\omega) \|_{\R^d}^2\big) \, \P(d\omega) = c\big(1 + \E\big[\| \Theta_{\floorgrid{s}} \|_{\R^d}^2\big]\big).
\end{split}
\end{equation}
This and \eqref{eq:main_theorem_7_03} demonstrate that for all $s, t \in [0,T]$, $j\in\gamma(\floorgrid{s})$ we have that
\begin{equation}\label{eq:main_theorem_7_06}
\begin{split}
 &\E\big[ | u_{0,1}(t,\Theta_{\floorgrid{s}})  G(\Theta_{\floorgrid{s}},Z_j)| + |u_{0,1}(t,\Theta_{\floorgrid{s}}) g(\Theta_{\floorgrid{s}})| \big]\\
& \leq c^{3/2}\big|1 + \E\big[\| \Theta_{\floorgrid{s}} \|_{\R^d}^2\big]\big|^{1/2} + c\,\E\big[ \| g(\Theta_{\floorgrid{s}})\|_{\R^d} \big].
\end{split}
\end{equation}
Combining this and \eqref{eq:main_theorem_0012}  assures that for all $s, t \in [0,T]$, $j\in\gamma(\floorgrid{s})$ we have that
\begin{equation}\label{eq:main_theorem_7_07}
\begin{split}
& \E\big[ | u_{0,1}(t,\Theta_{\floorgrid{s}})  G(\Theta_{\floorgrid{s}},Z_j)| + |u_{0,1}(t,\Theta_{\floorgrid{s}}) g(\Theta_{\floorgrid{s}})| \big]\\
& \leq c^{3/2}\big|1 + \E\big[\| \Theta_{\floorgrid{s}} \|_{\R^d}^2\big]\big|^{1/2} + c^2\big(1 + \E\big[ \| \Theta_{\floorgrid{s}} \|_{\R^d} \big]\big).
\end{split}
\end{equation}
Jensen's inequality and \eqref{eq:main:sup} hence prove that for all $s, t\in [0,T]$, $j\in\gamma(\floorgrid{s})$ we have that 
\begin{align}
\label{eq:main_theorem_7_08}
\begin{split}
& \E\big[ | u_{0,1}(t,\Theta_{\floorgrid{s}})  G(\Theta_{\floorgrid{s}},Z_j)| + |u_{0,1}(t,\Theta_{\floorgrid{s}}) g(\Theta_{\floorgrid{s}})| \big]   \\
& \leq c^{3/2}\big|1 + \sup_{u\in [0,T]}\E\big[\| \Theta_{u} \|_{\R^d}^2\big]\big|^{1/2} + c^2\big(1 + \sup_{u\in [0,T]} \big| \E\big[ \| \Theta_{u} \|_{\R^d}^2 \big] \big|^{1/2}\big)<\infty.
\end{split}
\end{align}
Next note that item~\eqref{item:motionless_pt_f_1_2} in Lemma~\ref{lem:motionless_pt_f} and Lemma~\ref{lem:estimate_diff_phi} ensure that
\begin{align}
\label{eq:main_theorem_2}
\begin{split}
&  |\psi(\theta_T^{\xi})- \psi(\Xi)| = |\psi(\theta_T^{\xi})- \psi(\theta_T^{\Xi})|  \\
&\leq \sup\left\lbrace \| \psi'(\lambda\theta_T^{\xi}+(1-\lambda)\theta_T^{\Xi}) \|_{L(\R^d,\R)} \in \R\colon \lambda\in [0,1] \right\rbrace \! \|\xi-\Xi\|_{\R^d}e^{-LT}\\
& = \sup\left\lbrace \| \psi'(\lambda\theta_T^{\xi}+(1-\lambda)\Xi) \|_{L(\R^d,\R)} \in \R\colon \lambda\in [0,1] \right\rbrace \! \|\xi-\Xi\|_{\R^d}e^{-LT}.
\end{split}
\end{align}
In the next step we combine \eqref{eq:def_u_main_theorem} and  \eqref{eq:equation theta} to obtain that for all $\vartheta\in\R^d$ we have that
\begin{equation}\label{eq:main_theorem_u_psi}
\psi(\vartheta)=\psi(\theta_0^{\vartheta})=u(T,\vartheta).
\end{equation}
This and the assumption that $\forall\, \omega\in \Omega\colon\Theta_0(\omega)=\xi$ prove that for all $\omega\in \Omega$ we have that
\begin{equation}\label{eq:main_theorem_2_02}
\psi(\Theta_T(\omega)) = u(T,\Theta_T(\omega))
\end{equation}
and
\begin{equation}\label{eq:main_theorem_2_03}
\psi(\theta^{\xi}_T)=\psi(\theta^{\Theta_0(\omega)}_T)=u(0,\Theta_0(\omega)).
\end{equation}
Next observe that Lemma~\ref{lem:partial_t_of_Theta} and \eqref{eq:main_theorem_09} assure that for all $\omega\in\Omega$, $t\in [0,T]\backslash E$ we have that 
\begin{equation}\label{eq:main_theorem_3_0_2}
[0,T]\backslash E \subseteq \left\lbrace u\in [0,T]\colon [0,T]\ni s\mapsto\Theta_s(\omega) \in \R^d \text{  is differentiable at }u \right\rbrace
\end{equation}
and
\begin{equation}\label{eq:main_theorem_3_0_22}
\pt \Theta_t(\omega) = Q_t(\omega).
\end{equation}
This, item~\eqref{item:Kolm_back_eq_main_1} in Lemma~\ref{lem:Kolm_back_eq_main}, and the fact that $\psi\in C^1(\R^d,\R)$ imply that for all $\omega\in\Omega$ we have that 
\begin{equation}\label{eq:main_theorem_3_0_3}
[0,T]\backslash E \subseteq \left\lbrace t\in [0,T]\colon [0,T]\ni s\mapsto u(s,\Theta_s(\omega)) \in \R \text{  is differentiable at }t \right\rbrace.
\end{equation}
This reveals that for all $\omega\in\Omega$, $t\in [0,T]\backslash E$  it holds that
\begin{equation}\label{eq:main_theorem_3_2}
\pt [u(t,\Theta_t(\omega))]= (\pt u)(t,\Theta_t(\omega))+(\pv u)(t,\Theta_t(\omega))\pt \Theta_t(\omega).
\end{equation}
Next note that \eqref{eq:main_theorem_3_0_22} and \eqref{eq:main_theorem_02} demonstrate that
\begin{equation}
\begin{split}
& \int_{[0,T]\backslash E} \E\big[|(\pv u)(t,\Theta_t)\pt \Theta_t|\big]\, dt\\
& \leq \int_{[0,T]\backslash E} \E\big[\|(\pv u)(t,\Theta_t)\|_{L(\R^d,\R)}\|\pt \Theta_t\|_{\R^d}\big]\, dt\\
& = \int_{[0,T]\backslash E} \E\big[\|u_{0,1}(t,\Theta_t)\|_{L(\R^d,\R)}\|Q_t\|_{\R^d}\big]\, dt.
\end{split}
\end{equation}
This and \eqref{eq:main_theorem_7_02} assure that
\begin{equation}\label{eq:main_theorem_4_02}
\int_{[0,T]\backslash E} \E\big[|(\pv u)(t,\Theta_t)\pt \Theta_t|\big]\, dt < \infty.
\end{equation}
Next observe that \eqref{eq:def_u_main_theorem} ensures that for all $t\in [0,T],\vartheta\in\R^d$ we have that
\begin{equation}
\begin{split}
|(\pt u)(t,\vartheta)|& =|\psi'(\theta_{T-t}^{\vartheta})\pt(\theta_{T-t}^{\vartheta})|
=|\psi'(\theta_{T-t}^{\vartheta})g(\theta_{T-t}^{\vartheta})|
\\
&\leq \|\psi'(\theta_{T-t}^{\vartheta})\|_{L(\R^d,\R)}\|g(\theta_{T-t}^{\vartheta})\|_{\R^d}.
\end{split}
\end{equation}
This, \eqref{eq:main:psi}, and \eqref{eq:main_theorem_0012} imply that for all $t\in [0,T],\vartheta\in\R^d$ we have that
\begin{equation}
|(\pt u)(t,\vartheta)| \leq  c \|g(\theta_{T-t}^{\vartheta})\|_{\R^d}\leq c^2(1+\|\theta_{T-t}^{\vartheta}\|_{\R^d}).
\end{equation}
The triangle inequality hence proves that for all $t\in [0,T]$ we have that 
\begin{equation}
\begin{split}
\E\big[|(\pt u)(t,\Theta_t)|\big]
& \leq c^2\big(1+\E\big[\|\theta_{T-t}^{\Theta_t}\|_{\R^d}\big]\big)\\
& = c^2\big(1+\E\big[\|\theta_{T-t}^{\Xi +(\Theta_t-\Xi)}\|_{\R^d}\big]\big)\\
& \leq c^2\big(1+\|\Xi\|_{\R^d}+\E\big[\|\theta_{T-t}^{\Xi +(\Theta_t-\Xi)}-\Xi\|_{\R^d}\big]\big).
\end{split}
\end{equation}
This, item~\eqref{item:motionless_pt_f_2} in Lemma~\ref{lem:motionless_pt_f}, and the triangle inequality ensure that for all $t\in [0,T]$ we have that 
\begin{equation}
\begin{split}
\E\big[|(\pt u)(t,\Theta_t)|\big]
& \leq c^2\big(1+\|\Xi\|_{\R^d}+\E\big[\|\Theta_t-\Xi\|_{\R^d}\big]e^{-L(T-t)}\big)\\
& \leq c^2\big(1+\|\Xi\|_{\R^d}+\E\big[\|\Theta_t - \Xi\|_{\R^d}\big]\big)\\
& \leq c^2\big(1+2\|\Xi\|_{\R^d}+\E\big[\|\Theta_t\|_{\R^d}\big]\big).
\end{split}
\end{equation}
This reveals that
\begin{equation}\label{eq:main_theorem_finite_nb_step}
\begin{split}
\int_0^T\E\big[|(\pt u)(t,\Theta_t)| \big]\, dt
&\leq \int_0^T c^2\big(1+2\|\Xi\|_{\R^d}+\E\big[\|\Theta_t\|_{\R^d}\big]\big) \, dt\\
& = Tc^2(1+2\|\Xi\|_{\R^d})+c^2\int_0^T\E\big[\|\Theta_t\|_{\R^d}\big] \, dt\\
&\leq Tc^2(1+2\|\Xi\|_{\R^d})+Tc^2\sup_{t\in [0,T]}\E\big[\|\Theta_t\|_{\R^d}\big]. 
\end{split}
\end{equation}
Jensen's inequality and \eqref{eq:main:sup} hence imply that
\begin{equation}
\begin{split}
&\int_0^T\E\big[|(\pt u)(t,\Theta_t)| \big]\, dt\\
&\leq Tc^2(1+2\|\Xi\|_{\R^d})+Tc^2\sup_{t\in [0,T]}\big(\E\big[\|\Theta_t\|_{\R^d}^2\big]\big)^{1/2} < \infty.
\end{split}
\end{equation}
Combining this, \eqref{eq:main_theorem_4_02}, \eqref{eq:main_theorem_3_2}, and the triangle inequality demonstrates that
\begin{align}
\begin{split}
&\int_{[0,T]\backslash E}\int_{\Omega}\big| \pt [u(t,\Theta_t(\omega))]\big|\, \P(d\omega)\, dt  \\
&\leq \int_{[0,T]\backslash E}\int_{\Omega} \big|(\pt u)(t,\Theta_t(\omega))\big|+\big|(\pv u)(t,\Theta_t(\omega))\pt \Theta_t(\omega)\big| \, \P(d\omega)\, dt <\infty.
\end{split}
\end{align} 
This, \eqref{eq:main_theorem_2_02}, \eqref{eq:main_theorem_2_03}, the fundamental theorem of calculus, the fact that $\#_E<\infty$, and Tonelli's theorem prove
that
\begin{equation}
\begin{split}
\E\big[|\psi(\Theta_T) - \psi(\theta_T^{\xi})|\big]
& = \int_{\Omega} |u(T,\Theta_T(\omega)) - u(0,\Theta_0(\omega))|\, \P(d\omega)\\
& \leq \int_{\Omega} \int_{[0,T]\backslash E}\big|\pt[u(t,\Theta_t(\omega))]\big|\,dt\, \P(d\omega)\\
& = \int_{[0,T]\backslash E}\int_{\Omega} \big|\tfrac{\partial}{\partial t}[u(t,{\Theta_t(\omega)})]\big| \, \P(d\omega)\, dt < \infty.
\end{split}
\end{equation}
The fundamental theorem of calculus, \eqref{eq:main_theorem_2_02}, \eqref{eq:main_theorem_2_03},  the fact that $\#_E<\infty$, and Fubini's theorem therefore assure that 
\begin{equation}\label{eq:main_theorem_4_2}
\begin{split}
\E[\psi(\Theta_T)]- \psi(\theta_T^{\xi}) 
& = \int_{\Omega} u(T,\Theta_T(\omega)) - u(0,\Theta_0(\omega))\, \P(d\omega)\\
& = \int_{\Omega} \int_{[0,T]\backslash E}\pt[u(t,\Theta_t(\omega))]\,dt\, \P(d\omega)\\
& = \int_{[0,T]\backslash E}\int_{\Omega} \tfrac{\partial}{\partial t}[u(t,{\Theta_t(\omega)})]\, \P(d\omega)\, dt .
\end{split}
\end{equation}
Furthermore, note that item~\eqref{item:Kolm_back_eq_main_2} in Lemma~\ref{lem:Kolm_back_eq_main} implies that for all $t\in[0,T]$, $\vartheta\in \R^d$ we have that
\begin{equation}\label{eq:main_theorem_5}
u_{1,0}(t,\vartheta)=-u_{0,1}(t,\vartheta)g(\vartheta).
\end{equation}
This and \eqref{eq:main_theorem_3_2}  assure that for all $t\in[0,T]\backslash E$, $\omega\in \Omega$ we have that
\begin{equation}\label{eq:main_theorem_6}
\begin{split}
\tfrac{\partial}{\partial t}[u(t,{\Theta_t(\omega)})] &= u_{1,0}(t,\Theta_t(\omega)) + u_{0,1}(t,\Theta_t(\omega)) \pt \Theta_t(\omega)\\
&= u_{0,1}(t,\Theta_t(\omega))\left( (\pt \Theta_t(\omega)) -g(\Theta_t(\omega)) \right).
\end{split}
\end{equation}
Combining this and \eqref{eq:main_theorem_3_0_22} demonstrates that for all $t\in[0,T]\backslash E$,  $\omega\in \Omega$ we have that
\begin{equation}
\tfrac{\partial}{\partial t}[u(t,{\Theta_t(\omega)})]
= u_{0,1}(t,\Theta_t(\omega))\left( Q_t(\omega) -g(\Theta_t(\omega)) \right).
\end{equation}
The fact that  $\#_E<\infty$ and \eqref{eq:main_theorem_4_2} hence imply that
\begin{align*}
\E[\psi(\Theta_T)]- \psi(\theta_T^{\xi}) 
& = \int_{[0,T]\backslash E}\int_{\Omega} \tfrac{\partial}{\partial t}[u(t,{\Theta_t(\omega)})]\, \P(d\omega)\, dt \numberthis \\
&=  \int_{[0,T]\backslash E} \int_{\Omega} u_{0,1}(t,\Theta_t(\omega))\left( Q_t(\omega) -g(\Theta_t(\omega)) \right) \P(d\omega)\, dt \\
&=  \int_{[0,T]} \int_{\Omega} u_{0,1}(t,\Theta_t(\omega))\left( Q_t(\omega) -g(\Theta_t(\omega)) \right) \P(d\omega)\, dt. 
\end{align*}
This reveals that
\begin{equation}\label{eq:main_theorem_7}
\begin{split}
\E[\psi(\Theta_T)]- \psi(\theta_T^{\xi}) 
&=  \int_0^T \E\big[ u_{0,1}(t,\Theta_t)\left( Q_t -g(\Theta_t) \right)\! \big]\, dt.
\end{split}
\end{equation}
This and \eqref{eq:main_theorem_7_02} ensure that
\begin{align}
\label{eq:main_theorem_15_03}
&\E[\psi(\Theta_T)]- \psi(\theta_T^{\xi})  \\
&= \int_0^T \E\big[ u_{0,1}(t,\Theta_t)( Q_t -g(\Theta_{\floorgrid{t}})) \big]\, dt + 
\int_0^T \E\big[ u_{0,1}(t,\Theta_t)( g(\Theta_{\floorgrid{t}}) - g(\Theta_t)) \big]\, dt. \nonumber
\end{align}
Combining this with \eqref{eq:main_theorem_06}  assures that
\begin{equation}\label{eq:main_theorem_15_2}
\begin{split}
&\E[\psi(\Theta_T)]- \psi(\theta_T^{\xi})\\
&= \int_0^T \E[ u_{0,1}(t,\Theta_t) \Delta_t ]\, dt 
+ \int_0^T \E\big[ u_{0,1}(t,\Theta_t) \big( g(\Theta_{\floorgrid{t}}) - g(\Theta_t) \big) \big]\, dt.
\end{split}
\end{equation}
Next note that \eqref{eq:main_theorem_8} and \eqref{eq:setting_2} demonstrate that for all $i\in \lbrace 1, 2, \dots,k-1\rbrace$, $t\in[\mathfrak{t}_i,\mathfrak{t}_{i+1})$ we have that $\floorgrid{t}=\mathfrak{t}_i$. This assures that
\begin{equation}
\begin{split}
 &\int_0^T \E\big[ u_{0,1}(t,\Theta_{\floorgrid{t}})(Q_t - g(\Theta_{\floorgrid{t}})) \big]\, dt\\
& =\Bigg[\textstyle\sum\limits_{i=1}^{k-1} \displaystyle \int_{\mathfrak{t}_i}^{\mathfrak{t}_{i+1}} \E\big[ u_{0,1}(t,\Theta_{\floorgrid{t}})(Q_t - g(\Theta_{\floorgrid{t}})) \big]\, dt\Bigg]\\
& \quad+ \int_{\mathfrak{t}_k}^{T} \E\!\left[ u_{0,1}(t,\Theta_{\floorgrid{t}})\left(Q_t - g(\Theta_{\floorgrid{t}})\right) \right] dt\\
& =\Bigg[\textstyle\sum\limits_{i=1}^{k-1} \displaystyle \int_{\mathfrak{t}_i}^{\mathfrak{t}_{i+1}} \E\big[ u_{0,1}(t,\Theta_{\mathfrak{t}_i})(Q_{\mathfrak{t}_i} - g(\Theta_{\mathfrak{t}_i})) \big]\, dt\Bigg] \\
&\quad + \int_{\mathfrak{t}_k}^{T} \E\big[ u_{0,1}(t,\Theta_{\mathfrak{t}_k})( Q_{\mathfrak{t}_k} - g(\Theta_{\mathfrak{t}_k})) \big]\, dt.
\end{split}
\end{equation}
Combining \eqref{eq:main_theorem_09} and \eqref{eq:main_theorem_7_08} hence proves that
\begin{align*}\label{eq:main_theorem_12}
& \int_0^T \E\big[ u_{0,1}(t,\Theta_{\floorgrid{t}})(Q_t - g(\Theta_{\floorgrid{t}})) \big]\, dt \numberthis\\
& =\Bigg[\textstyle\sum\limits_{i=1}^{k-1} \displaystyle \int_{\mathfrak{t}_i}^{\mathfrak{t}_{i+1}}\bigg(\bigg(\frac{1}{\#_{\gamma(\mathfrak{t}_i)}}\textstyle\sum\limits_{j\in \gamma(\mathfrak{t}_i)} \E\big[ u_{0,1}(t,\Theta_{\mathfrak{t}_i})  G(\Theta_{\mathfrak{t}_i},Z_j)\big]\bigg) - \E\big[u_{0,1}(t,\Theta_{\mathfrak{t}_i}) g(\Theta_{\mathfrak{t}_i}) \big] \bigg)\, dt\Bigg]\\
& \quad + \int_{t_k}^{T}\bigg(\bigg(\frac{1}{\#_{\gamma(\mathfrak{t}_k)}}\textstyle\sum\limits_{j\in \gamma(\mathfrak{t}_k)} \displaystyle \E\big[ u_{0,1}(t,\Theta_{\mathfrak{t}_k})  G(\Theta_{\mathfrak{t}_k},Z_j)\big]\bigg) - \E\big[u_{0,1}(t,\Theta_{\mathfrak{t}_k}) g(\Theta_{\mathfrak{t}_k}) \big] \bigg)\, dt.
\end{align*}
Furthermore, note that the assumption that $Z_j$, $j\in\N$, are i.i.d.\ random variables, the fact that  for all  $s\in [0,T]$, $j\in\gamma(\floorgrid{s})$ it holds that $Z_j$ and $\Theta_{\floorgrid{s}}$
are independent, and  \eqref{eq:main_theorem_8p} ensure that for all $i \in \{1, 2, \ldots, k\}$, $t \in [0,T]$ we have that
\begin{equation}\label{eq:main_theorem_13}
\begin{split}
&\frac{1}{\#_{\gamma(\mathfrak{t}_i)}}\textstyle\sum\limits_{j\in \gamma(\mathfrak{t}_i)} \E\big[ u_{0,1}(t,\Theta_{\mathfrak{t}_i})  G(\Theta_{\mathfrak{t}_i},Z_j)\big]
 = \frac{1}{\#_{\gamma(\mathfrak{t}_i)}}\textstyle\sum\limits_{j\in \gamma(\mathfrak{t}_i)} \E\big[ u_{0,1}(t,\Theta_{\mathfrak{t}_i}) G(\Theta_{\mathfrak{t}_i},Z_{\mathfrak{j}_i})\big]\\
&= \E\big[ u_{0,1}(t,\Theta_{\mathfrak{t}_i})  G(\Theta_{\mathfrak{t}_i},Z_{\mathfrak{j}_i})\big]
= \int_{\Omega} u_{0,1}(t,\Theta_{\mathfrak{t}_i}(\omega)) G(\Theta_{\mathfrak{t}_i}(\omega),Z_{\mathfrak{j}_i}(\omega))\, \P(d\omega)\\
&= \int_{\Omega}\int_{\Omega} u_{0,1}(t,\Theta_{\mathfrak{t}_i}(\omega)) G(\Theta_{\mathfrak{t}_i}(\omega),Z_{\mathfrak{j}_i}(\omega'))\, \P(d\omega')\, \P(d\omega)\\
&= \int_{\Omega} u_{0,1}(t,\Theta_{\mathfrak{t}_i}(\omega))\bigg(\int_{\Omega} G(\Theta_{\mathfrak{t}_i}(\omega),Z_{\mathfrak{j}_i}(\omega'))\, \P(d\omega')\bigg)\P(d\omega)\\
&= \int_{\Omega} u_{0,1}(t,\Theta_{\mathfrak{t}_i}(\omega))\bigg(\int_{\Omega}G(\Theta_{\mathfrak{t}_i}(\omega),Z_1(\omega'))\, \P(d\omega')\bigg)\, \P(d\omega)\\
&= \int_{\Omega} u_{0,1}(t,\Theta_{\mathfrak{t}_i}(\omega)) \E[G(\Theta_{\mathfrak{t}_i}(\omega),Z_1)]\, \P(d\omega)\\
&= \int_{\Omega} u_{0,1}(t,\Theta_{\mathfrak{t}_i}(\omega)) g(\Theta_{\mathfrak{t}_i}(\omega)) \, \P(d\omega)= \E\big[u_{0,1}(t,\Theta_{\mathfrak{t}_i}) g(\Theta_{\mathfrak{t}_i}) \big].
\end{split}
\end{equation}
This and \eqref{eq:main_theorem_12} imply that 
\begin{equation}\label{eq:main_theorem_14}
\int_0^T \E\big[ u_{0,1}(t,\Theta_{\floorgrid{t}})( Q_t - g(\Theta_{\floorgrid{t}})) \big]\, dt=0.
\end{equation}
Next observe that the fundamental theorem of calculus, Lemma~\ref{lem:f_C^1}, the fact that $\#_E<\infty$, and \eqref{eq:main_theorem_3_0_22}  ensure that
\begin{equation}\label{eq:main_theorem_16}
\begin{split}
& \int_0^T \E\big[ u_{0,1}(t,\Theta_t) \big( 		 g(\Theta_{\floorgrid{t}}) - g(\Theta_t) \big)\big]\, dt\\
&  = - \int_0^T \E\Bigg[ u_{0,1}(t,\Theta_t)	\int_{[\floorgrid{t},t]\backslash E} g'(\Theta_{s})Q_s\, ds \Bigg]\, dt\\
 & = - \int_0^T \E\Bigg[ u_{0,1}(t,\Theta_t)	\int_{\floorgrid{t}}^t g'(\Theta_{s})Q_s\, ds \Bigg]\, dt\\
 &  = - \int_0^T \E\Bigg[\int_{\floorgrid{t}}^t u_{0,1}(t,\Theta_t)	 g'(\Theta_{s})Q_s\, ds \Bigg]\, dt.
\end{split}
\end{equation}
This and Tonelli's theorem assure that
\begin{equation}\label{eq:main_theorem_18}
\begin{split}
&\left| \int_0^T \E\!\left[ u_{0,1}(t,\Theta_t) \big( g(\Theta_{\floorgrid{t}}) - g(\Theta_t) \big)\right] dt \right| \\
 &= \Bigg| \int_0^T \E\Bigg[\int_{\floorgrid{t}}^t u_{0,1}(t,\Theta_t)	 g'(\Theta_{s})Q_s\, ds \Bigg]\, dt \Bigg| \\
& \leq \int_0^T \E\Bigg[\int_{\floorgrid{t}}^t | u_{0,1}(t,\Theta_t)	 g'(\Theta_{s})Q_s |\, ds \Bigg]\, dt \\
 &=\int_0^T \int_{\floorgrid{t}}^t\E\big[| u_{0,1}(t,\Theta_t) g'(\Theta_{s})Q_s |\big] \, ds\, dt\\
& \leq \int_0^T \int_{\floorgrid{t}}^t\E\big[\| u_{0,1}(t,\Theta_t)\|_{L(\R^d,\R)} \|g'(\Theta_{s})Q_s \|_{\R^d}\big] \, ds\, dt.
\end{split}
\end{equation}
Item~\eqref{item:estimate_u_01_3} in Lemma~\ref{lem:estimate_u_01} hence implies that
\begin{equation}\label{eq:main_theorem_19}
\begin{split}
&\left| \int_0^T \E\big[ u_{0,1}(t,\Theta_t) \big( g(\Theta_{\floorgrid{t}}) - g(\Theta_t) \big)\big]\, dt \right|\\
&\leq  \int_0^T\! e^{-L(T-t)}\int_{\floorgrid{t}}^t\!\E\big[\|\psi'(\theta_{T-t}^{\Theta_t})\|_{L(\R^d,\R)} \|g'(\Theta_{s})Q_s \|_{\R^d}\big] \, ds\, dt\\
&\leq \int_0^T\! e^{-L(T-t)}(t-{\floorgrid{t}})\sup_{s\in [0,T]}\E\big[\|\psi'(\theta_{T-t}^{\Theta_t})\|_{L(\R^d,\R)} \|g'(\Theta_{s})Q_s \|_{\R^d}\big] \, dt.
\end{split}
\end{equation}
Next note that the fundamental theorem of calculus for the Bochner integral (see, e.g., \cite[Lemma 2.1]{JentzenBochner}) demonstrates that for all $t\in [0,T]$ we have that
\begin{equation}\label{eq:main_theorem_20}
\begin{split}
&u_{0,1}(t,\Theta_t)-u_{0,1}(t,\Theta_{\floorgrid{t}})\\
&=\int_0^1 \big( (\pv u_{0,1}) (t,\lambda \Theta_t + (1-\lambda)\Theta_{\floorgrid{t}}) \big)(\Theta_t - \Theta_{\floorgrid{t}}) \, d\lambda.
\end{split}
\end{equation}
This, \eqref{eq:main_theorem_05}, and \eqref{eq:main_theorem_04} ensure that for all $t\in [0,T]$ we have that
\begin{equation}
\label{eq:u_0_1}
u_{0,1}(t,\Theta_t)-u_{0,1}(t,\Theta_{\floorgrid{t}})=\int_0^1 \big( (\pv u_{0,1})(t,a_t^{\lambda}) \big) \delta_t\, d\lambda.
\end{equation}
Moreover, observe that the chain rule,  \eqref{eq:main_theorem_07}, and \eqref{eq:main_theorem_08} assure that for all $t \in [0, T]$, $\vartheta, y, z \in \R^d$ we have that
\begin{equation}
\begin{split}
\big((\pv u_{0,1})(t,\vartheta)\big)(y,z) &= \big((\tfrac{\partial^2}{\partial \vartheta^2} u )(t,\vartheta)\big)(y,z) =  \big(\tfrac{\partial^2}{\partial \vartheta^2} \psi(\theta^{\vartheta}_{T-t})\big)(y,z)\\
& = \big(\pv (\psi'(\theta^{\vartheta}_{T-t}) \, \theta^{1,\vartheta}_{T-t} y)\big)(z) \\
& = \psi''(\theta^{\vartheta}_{T-t})(\theta^{1,\vartheta}_{T-t}z, \theta^{1,\vartheta}_{T-t}y) + \psi'(\theta^{\vartheta}_{T-t})(\theta^{2,\vartheta}_{T-t}(z,y)).
\end{split}
\end{equation}
This and \eqref{eq:u_0_1} imply that for all $t\in [0,T]$ we have that
\begin{equation}\label{eq:main_theorem_22}
\begin{split}
&(u_{0,1}(t,\Theta_t)-u_{0,1}(t,\Theta_{\floorgrid{t}})) \Delta_t\\
& = \int_0^1  \psi''( \theta_{T-t}^{a_t^{\lambda}} ) ( \theta_{T-t}^{1,a_t^{\lambda}} \Delta_t ,\theta_{T-t}^{1,a_t^{\lambda}} \delta_t )
 + \psi'( \theta_{T-t}^{a_t^{\lambda}} ) (\theta_{T-t}^{2,a_t^{\lambda}}(\Delta_t, \delta_t)) \, d\lambda.
\end{split}
\end{equation}
Furthermore, note that \eqref{eq:main_theorem_07} and Lemma~\ref{lem:exp_decay_partial_vartheta} assure that for all $t\in [0,T]$, $\lambda\in [0,1]$ we have that
\begin{equation}\label{eq:main_theorem_23}
\begin{split}
\big| \psi''( \theta_{T-t}^{a_t^{\lambda}} ) ( \theta_{T-t}^{1,a_t^{\lambda}} \Delta_t ,\theta_{T-t}^{1, a_t^{\lambda}} \delta_t )\big| 
& \leq \big\| \psi''( \theta_{T-t}^{a_t^{\lambda}})\big\|_{L^{(2)}(\R^d,\R)}\big\|\theta_{T-t}^{1,a_t^{\lambda}} \Delta_t\big\|_{\R^d}\big\|\theta_{T-t}^{1,a_t^{\lambda}} \delta_t\big\|_{\R^d}\\
& \leq \big\|\psi''( \theta_{T-t}^{a_t^{\lambda}})\big\|_{L^{(2)}(\R^d,\R)}\big\|\theta_{T-t}^{1,a_t^{\lambda}}\big\|_{L(\R^d,\R^d)}^2 \|\Delta_t\|_{\R^d}\big\| \delta_t\big\|_{\R^d}\\
& \leq e^{-2L(T-t)} \big\|\psi''( \theta_{T-t}^{a_t^{\lambda}})\big\|_{L^{(2)}(\R^d,\R)} \|\Delta_t\|_{\R^d}\|\delta_t\|_{\R^d}.
\end{split}
\end{equation}
Next observe that Lemma~\ref{lem:bound_partial_theta^2}   and \eqref{eq:main_theorem_08} ensure that for all $\lambda\in [0,1]$, $t\in [0,T]$ we have that
\begin{align*}
&\big|\psi'\big( \theta_{T-t}^{a_t^{\lambda}} \big) \big( \theta_{T-t}^{2,a_t^{\lambda}}\big(\Delta_t, \delta_t\big)\big)\big|\\
& \leq \big\|\psi'( \theta_{T-t}^{a_t^{\lambda}} )\big\|_{L(\R^d,\R)} \big\|\theta_{T-t}^{2, a_t^{\lambda}}(\Delta_t, \delta_t)\big\|_{\R^d} \numberthis \label{eq:main_theorem_24}\\
& \leq \big\|\psi'( \theta_{T-t}^{a_t^{\lambda}} )\big\|_{L(\R^d,\R)} \big\|\theta_{T-t}^{2, a_t^{\lambda}}\big\|_{L^{(2)}(\R^d,\R^d)} \|\Delta_t\|_{\R^d}\|\delta_t\|_{\R^d} \\
& \leq e^{-L(T-t)}\big\|\psi'( \theta_{T-t}^{a_t^{\lambda}} )\big\|_{L(\R^d,\R)} \|\Delta_t\|_{\R^d}\|\delta_t\|_{\R^d} \int_0^{T-t} e^{-Ls}\big\|g''(\theta_s^{a_t^{\lambda}})\big\|_{L^{(2)}(\R^d,\R^d)}\, ds.
\end{align*}
Combining this, \eqref{eq:main_theorem_14}, Jensen's inequality, \eqref{eq:main_theorem_22}, and \eqref{eq:main_theorem_23} demonstrates that
\begin{align*}
\label{eq:main_theorem_25}
& \left|\int_0^T \E\!\left[ u_{0,1}(t,\Theta_t) \Delta_t \right] dt\right| = \left|\int_0^T\E\big[ (u_{0,1}(t,\Theta_t)-u_{0,1}(t,\Theta_{\floorgrid{t}})) \Delta_t \big]\, dt\right|\\
& \leq \int_0^T\big|\E \big[ (u_{0,1}(t,\Theta_t)-u_{0,1}(t,\Theta_{\floorgrid{t}})) \Delta_t \big]\big|\, dt\\
& \leq \int_0^T\E\big[ \big| (u_{0,1}(t,\Theta_t)-u_{0,1}(t,\Theta_{\floorgrid{t}})) \Delta_t \big|\big]\, dt \\
& = \int_0^T\E \bigg[ \bigg| 
\int_0^1 \psi''( \theta_{T-t}^{a_t^{\lambda}} ) ( \theta_{T-t}^{1, a_t^{\lambda}} \Delta_t ,\theta_{T-t}^{1,a_t^{\lambda}} \delta_t )
 + \psi'( \theta_{T-t}^{a_t^{\lambda}} ) ( \theta_{T-t}^{2,a_t^{\lambda}}(\Delta_t, \delta_t)) \, d\lambda
 \bigg|\bigg]\, dt \numberthis\\
& \leq \int_0^T\E\bigg[  
\int_0^1\big| \psi''( \theta_{T-t}^{a_t^{\lambda}} ) ( \theta_{T-t}^{1, a_t^{\lambda}} \Delta_t ,\theta_{T-t}^{1,a_t^{\lambda}} \delta_t )\big|
 + \big|\psi'( \theta_{T-t}^{a_t^{\lambda}} ) ( \theta_{T-t}^{2,a_t^{\lambda}}(\Delta_t, \delta_t)) \big| \, d\lambda
\bigg]\, dt\\
& \leq \int_0^T\E\bigg[  
\int_0^1  e^{-2L(T-t)} \big\|\psi''( \theta_{T-t}^{a_t^{\lambda}})\big\|_{L^{(2)}(\R^d,\R)} \norm{\Delta_t}_{\R^d}\norm{\delta_t}_{\R^d}\\
& \quad + e^{-L(T-t)}\big\|\psi'( \theta_{T-t}^{a_t^{\lambda}} )\big\|_{L(\R^d,\R)} \norm{\Delta_t}_{\R^d} \norm{\delta_t}_{\R^d} \int_0^{T-t} e^{-Ls}\big\|g''(\theta_s^{a_t^{\lambda}})\big\|_{L^{(2)}(\R^d,\R^d)}\, ds\, d\lambda
 \bigg]\, dt.
\end{align*}
Next note that \eqref{eq: equation Theta 2} and \eqref{eq:main_theorem_09} assure that for all $t\in [0,T]$, $\omega\in\Omega$ we have that
\begin{equation}\label{eq:main_theorem_26}
\begin{split}
\norm{\delta_t(\omega)}_{\R^d} 
 &= \big\|\Theta_t(\omega)-\Theta_{\floorgrid{t}}(\omega)\big\|_{\R^d}\\
 &= (t-\floorgrid{t})\bigg\| \frac{1}{\#_{\gamma(\floorgrid{t})}} \textstyle\sum\limits_{j\in\gamma(\floorgrid{t})} G(\Theta_{\floorgrid{t}}(\omega),Z_j(\omega))\bigg\|_{\R^d}\\
& = (t-\floorgrid{t})\norm{Q_t(\omega)}_{\R^d}.
\end{split}
\end{equation}
Combining this and \eqref{eq:main_theorem_25} proves that 
\begin{align}
\begin{split}
&\left|\int_0^T \E\!\left[ u_{0,1}(t,\Theta_t) \Delta_t \right] dt\right|\\
 & \leq \int_0^T e^{-L(T-t)}(t-\floorgrid{t}) \E \bigg[ \norm{\Delta_t}_{\R^d}\norm{Q_t}_{\R^d}\bigg( 
\int_0^1  e^{-L(T-t)} \big\|\psi''( \theta_{T-t}^{a_t^{\lambda}})\big\|_{L^{(2)}(\R^d,\R)} \\
&\quad + \big\|\psi'( \theta_{T-t}^{a_t^{\lambda}})\big\|_{L(\R^d,\R)} \int_0^{T-t} e^{-Ls}\big\|g''(\theta_s^{a_t^{\lambda}})\big\|_{L^{(2)}(\R^d,\R^d)}\, ds\, d\lambda \bigg)
 \bigg]\, dt. 
 \end{split}
\end{align}
This, \eqref{eq:main_theorem_15_2}, and  \eqref{eq:main_theorem_19} imply that
\begin{align*}
& \big|\E[\psi(\Theta_T)]- \psi(\theta_T^{\xi})\big| \\
&\leq \left| \int_0^T \E\!\left[ u_{0,1}(t,\Theta_t) \Delta_t \right] dt\right|  + \left|\int_0^T \E\!\left[ u_{0,1}(t,\Theta_t) \left( g(\Theta_{\floorgrid{t}}) - g(\Theta_t) \right) \right] dt\right|\\
& \leq \int_0^T e^{-L(T-t)}(t-\floorgrid{t}) \E \bigg[ \norm{\Delta_t}_{\R^d}\norm{Q_t}_{\R^d}\bigg( 
\int_0^1  e^{-L(T-t)} \big\|\psi''( \theta_{T-t}^{a_t^{\lambda}})\big\|_{L^{(2)}(\R^d,\R)} \\
& \quad+ \big\|\psi'( \theta_{T-t}^{a_t^{\lambda}} )\big\|_{L(\R^d,\R)} \int_0^{T-t} e^{-Lu}\big\|g''(\theta_u^{a_t^{\lambda}})\big\|_{L^{(2)}(\R^d,\R^d)}\, du\, d\lambda \bigg)
 \bigg]\, dt \numberthis\\
&\quad + \int_0^T\! e^{-L(T-t)}(t-{\floorgrid{t}})\sup_{v\in [0,T]}\E\Big[\big\|\psi'(\theta_{T-t}^{\Theta_t})\big\|_{L(\R^d,\R)} \big\|g'(\Theta_{v})Q_v \big\|_{\R^d}\Big] \, dt \\
& = \int_0^T e^{-L(T-t)}(t-\floorgrid{t}) \bigg(\E \bigg[ \norm{\Delta_t}_{\R^d}\norm{Q_t}_{\R^d}\bigg( 
\int_0^1  e^{-L(T-t)} \big\|\psi''( \theta_{T-t}^{a_t^{\lambda}})\big\|_{L^{(2)}(\R^d,\R)} \\
& \quad+ \big\|\psi'( \theta_{T-t}^{a_t^{\lambda}} )\big\|_{L(\R^d,\R)} \int_0^{T-t} e^{-Lu}\big\|g''(\theta_u^{a_t^{\lambda}})\big\|_{L^{(2)}(\R^d,\R^d)}\, du\, d\lambda \bigg)
 \bigg]\\
 & \quad+ \sup_{v\in [0,T]}\E\Big[\big\|\psi'(\theta_{T-t}^{\Theta_t})\big\|_{L(\R^d,\R)} \big\|g'(\Theta_{v})Q_v \big\|_{\R^d}\Big]\bigg) \, dt.
\end{align*}
This reveals that
\begin{align}
\begin{split}
 &|\E[\psi(\Theta_T)]- \psi(\theta_T^{\xi})|  \leq \int_0^T e^{-L(T-t)}(t-\floorgrid{t})\\
& \quad \cdot \sup_{s,v\in [0,T]}\E \bigg[ \norm{\Delta_s}_{\R^d}\norm{Q_s}_{\R^d}\bigg( 
\int_0^1  e^{-L(T-s)} \big\|\psi''( \theta_{T-s}^{a_s^{\lambda}})\big\|_{L^{(2)}(\R^d,\R)} \\
& \quad + \big\|\psi'( \theta_{T-s}^{a_s^{\lambda}} )\big\|_{L(\R^d,\R)} \int_0^{T-s} e^{-Lu}\big\|g''(\theta_u^{a_s^{\lambda}})\big\|_{L^{(2)}(\R^d,\R^d)}\, du\, d\lambda \bigg)\\
& \quad+  \big\|\psi'(\theta_{T-s}^{\Theta_s})\big\|_{L(\R^d,\R)} \big\|g'(\Theta_{v})Q_v \big\|_{\R^d}\bigg] \, dt \\
& =  \sup_{s,v\in [0,T]}\E \bigg[ \norm{\Delta_s}_{\R^d}\norm{Q_s}_{\R^d}\bigg( 
\int_0^1  e^{-L(T-s)} \big\|\psi''( \theta_{T-s}^{a_s^{\lambda}})\big\|_{L^{(2)}(\R^d,\R)} \\
& \quad + \big\|\psi'( \theta_{T-s}^{a_s^{\lambda}})\big\|_{L(\R^d,\R)} \int_0^{T-s} e^{-Lu}\big\|g''(\theta_u^{a_s^{\lambda}})\big\|_{L^{(2)}(\R^d,\R^d)}\, du\, d\lambda \bigg)\\
& \quad+ \big\|\psi'(\theta_{T-s}^{\Theta_s})\big\|_{L(\R^d,\R)} \big\|g'(\Theta_{v})Q_v \big\|_{\R^d}\bigg]\int_0^T e^{-L(T-t)}(t-\floorgrid{t}) \, dt. 
\end{split}
\end{align}
Combining this, the triangle inequality, \eqref{eq:main_theorem_04}, and \eqref{eq:main_theorem_2} proves that
\begin{align*}
 & |\E[\psi(\Theta_T)]- \psi(\Xi)|  \leq |\E[\psi(\Theta_T)]- \psi(\theta_T^{\xi})| + |\psi(\theta_T^{\xi})- \psi(\Xi)|\\
& \leq \sup_{s,v\in [0,T]}\E \bigg[ \|Q_s-g(\Theta_{\floorgrid{s}})\|_{\R^d}\norm{Q_s}_{\R^d} \bigg(\int_0^1  e^{-L(T-s)} \big\|\psi''( \theta_{T-s}^{\lambda \Theta_s + (1-\lambda)\Theta_{\floorgrid{s}}})\big\|_{L^{(2)}(\R^d,\R)} \\
& \quad + \big\|\psi'( \theta_{T-s}^{\lambda \Theta_s + (1-\lambda)\Theta_{\floorgrid{s}}} )\big\|_{L(\R^d,\R)} \int_0^{T-s} e^{-Lu}\big\|g''(\theta_u^{\lambda \Theta_s + (1-\lambda)\Theta_{\floorgrid{s}}})\big\|_{L^{(2)}(\R^d,\R^d)}\, du\, d\lambda \bigg)\\
& \quad + \big\|\psi'(\theta_{T-s}^{\Theta_s})\big\|_{L(\R^d,\R)} \big\|g'(\Theta_{v})Q_v \big\|_{\R^d}\bigg] \int_0^T e^{-L(T-t)}(t-\floorgrid{t})\, dt\\
& \quad + \sup\left\lbrace \big\| \psi'(\lambda\theta_T^{\xi}+(1-\lambda)\Xi) \big\|_{L(\R^d,\R)} \in \R\colon \lambda\in [0,1] \right\rbrace \! \big\|\xi-\Xi\big\|_{\R^d}e^{-LT}. \numberthis
\end{align*}
This establishes item~\eqref{item:main_theorem_2}. This completes the proof of Proposition~\ref{prop:main_theorem}.
\end{proof}

\section{Upper bounds for integrals of certain exponentially decaying functions}
\label{subsec:upper:integrals}
\sectionmark{}

\begin{lemma}\label{lem:comp_int}
	Assume Setting~\ref{sec:setting_1}, let $L\in [0,\infty)$, let $\mathfrak{t}\colon\N \to [0,\infty)$ be  non-decreasing, and assume that
	\begin{equation}
	\label{eq:sequence}
	\lb t\in [0,\infty)\colon \gamma(t)\neq \emptyset\rb = \lb \mathfrak{t}_n\colon n\in\N\rb.
	\end{equation}
	Then we have for all $k\in\lb 2,3,\dots\rb$ that
	\begin{equation}
	\int_0^{\mathfrak{t}_k}\exp(-L(\mathfrak{t}_k-t))(t-\floorgrid{t})\, dt 
	\leq \frac{1}{2}\textstyle\sum\limits_{n=1}^{k-1} \exp(-L(\mathfrak{t}_k-\mathfrak{t}_{n+1}))(\mathfrak{t}_{n+1}-\mathfrak{t}_n)^2.
	\end{equation}
\end{lemma}
\begin{proof}[Proof of Lemma~\ref{lem:comp_int}]
	First, observe that the assumption that	$\forall \, t \in [0, \infty) \colon 0< \allowbreak \#_{\lb s \in [0,t]\colon\gamma(s)\neq \emptyset\rb}<\infty$ ensures that
	\begin{equation}
	\gamma(0) \neq  \emptyset.
	\end{equation}
	This, \eqref{eq:sequence}, and the assumption that 	$\mathfrak{t}\colon\N \to [0,\infty)$ is  non-decreasing  imply that
	\begin{equation}
	\mathfrak{t}_1 = 0.
	\end{equation}
	This reveals that for all 	$k\in\lb 2,3,\dots\rb$  it holds  that 
	\begin{equation}
	\begin{split}
	&\int_0^{\mathfrak{t}_k}e^{-L(\mathfrak{t}_k-t)}(t-\floorgrid{t})\, dt
	= \textstyle\sum\limits_{n=1}^{k-1} \displaystyle \int_{\mathfrak{t}_n}^{\mathfrak{t}_{n+1}}e^{-L(\mathfrak{t}_k-t)}(t-\floorgrid{t})\, dt \\
	&\leq  \textstyle\sum\limits_{n=1}^{k-1} \displaystyle e^{-L(\mathfrak{t}_k-\mathfrak{t}_{n+1})}\int_{\mathfrak{t}_n}^{\mathfrak{t}_{n+1}}(t-\mathfrak{t}_n)\, dt\\
	&= \textstyle\sum\limits_{n=1}^{k-1} \displaystyle e^{-L(\mathfrak{t}_k-\mathfrak{t}_{n+1})}\left[\frac{1}{2}(t-\mathfrak{t}_n)^2\right]_{t=\mathfrak{t}_n}^{t=\mathfrak{t}_{n+1}}= \frac{1}{2}\textstyle\sum\limits_{n=1}^{k-1} \displaystyle e^{-L(\mathfrak{t}_k-\mathfrak{t}_{n+1})}(\mathfrak{t}_{n+1}-\mathfrak{t}_n)^2.
	\end{split}
	\end{equation}
	The proof of Lemma~\ref{lem:comp_int} is thus completed.
\end{proof}

\begin{lemma}\label{lem:bounds_sum}
	Let $\nu\in [0,1)$. Then 
	\begin{enumerate}[(i)]
		\item\label{item:bounds_sum_1} we have for all $l\in\N$ that
		\begin{equation}\label{eq:bounds_sum_1_1}
		\textstyle\sum\limits_{n=1}^l \displaystyle \frac{1}{n^{\nu}}\geq \frac{1}{1-\nu}\left( (l+1)^{1-\nu}-1 \right)
		\end{equation}
		and
		\item\label{item:bounds_sum_2} we have for all $l\in\N$ that 
		\begin{equation}\label{eq:bounds_sum_1_2}
	\textstyle\sum\limits_{n=1}^l \displaystyle \frac{1}{n^{\nu}}
		\leq 1 + \frac{1}{1-\nu}\left(l^{1-\nu} -1\right) = \frac{1}{1-\nu}\left( l^{1-\nu}-\nu \right) .
		\end{equation}
	\end{enumerate}
\end{lemma}
\begin{proof}[Proof of Lemma~\ref{lem:bounds_sum}]
	First, observe that for all $l\in \N$ we have that
	\begin{equation}\label{eq:bounds_sum_3}
	\begin{split}
		\textstyle\sum\limits_{n=1}^l \displaystyle\frac{1}{n^{\nu}}
	& \geq 	\textstyle\sum\limits_{n=1}^l \displaystyle\int_{n}^{n+1} \frac{1}{x^{\nu}}\, dx = 	\textstyle\sum\limits_{n=1}^l \displaystyle\left[ \frac{1}{1-\nu}x^{1-\nu} \right]_{x=n}^{x=n+1}\\
	& = \frac{1}{1-\nu}	\textstyle\sum\limits_{n=1}^l \displaystyle\left[ (n+1)^{1-\nu} - n^{1-\nu}\right] = \frac{1}{1-\nu}\left( (l+1)^{1-\nu}-1 \right).
	\end{split}
	\end{equation}
	This proves item~\eqref{item:bounds_sum_1}.
	Moreover, note that for all $l\in \lb 2,3,\dots\rb$ we have that
	\begin{equation}\label{eq:bounds_sum_2}
	\begin{split}
		\textstyle\sum\limits_{n=1}^l \displaystyle\frac{1}{n^{\nu}} 
	& = 1 + 	\textstyle\sum\limits_{n=2}^l \displaystyle\frac{1}{n^{\nu}}
	\leq 1 + 	\textstyle\sum\limits_{n=2}^l \displaystyle\int_{n-1}^n \frac{1}{x^{\nu}}\, dx
	= 1 + 	\textstyle\sum\limits_{n=2}^l \displaystyle\left[ \frac{1}{1-\nu}x^{1-\nu} \right]_{x=n-1}^{x=n}\\
	& = 1 + 	\textstyle\sum\limits_{n=2}^l \displaystyle\left[ \frac{1}{1-\nu}n^{1-\nu} - \frac{1}{1-\nu}(n-1)^{1-\nu} \right]\\
	&= 1 + \frac{1}{1-\nu}\left(l^{1-\nu} -1 \right)
	= \frac{1}{1-\nu}\left(l^{1-\nu} -\nu\right).
	\end{split}
	\end{equation}
	Next observe that
	\begin{equation}
	\textstyle\sum\limits_{n=1}^1 \displaystyle \frac{1}{n^{\nu}}=1=\frac{1}{1-\nu}\left(1^{1-\nu} -\nu\right).
	\end{equation}
	This and \eqref{eq:bounds_sum_2} establish item~\eqref{item:bounds_sum_2}. The proof of Lemma~\ref{lem:bounds_sum} is thus completed.
\end{proof}

\begin{lemma}\label{lem:bnd_special_case}
	Assume Setting~\ref{sec:setting_1}, let $L\in[0,\infty)$, $\eta\in (0,\infty)$, $\nu\in [0,1)$, 
	let $\mathfrak{t}\colon \N\to [0,\infty)$  satisfy for all $m\in \lb 2, 3,\dots\rb$ that $\mathfrak{t}_1=0$ and 
	\begin{equation}
	\mathfrak{t}_m = \textstyle\sum\limits_{n=1}^{m-1} \displaystyle \frac{\eta}{n^{\nu}},
	\end{equation}
	and assume that
	\begin{equation}
	\lb s\in [0,\infty)\colon \gamma(s)\neq \emptyset\rb = \lb \mathfrak{t}_m\colon m\in\N \rb .
	\end{equation}
	Then we have for all $k\in \lb 2,3,\dots\rb$ that
	\begin{equation}\label{eq:bnd_special_case_4}
	\begin{split}
	&\int_0^{\mathfrak{t}_k} \exp(-L(\mathfrak{t}_k-t))(t-\floorgrid{t})\, dt\\
	&\leq  \frac{\eta^2 \exp(L\eta)}{2}\left[\textstyle\sum\limits_{n=1}^{k-2} \displaystyle \frac{\exp\!\left( -\tfrac{L\eta}{1-\nu}(k^{1-\nu}-n^{1-\nu}) \right) }{n^{2\nu}}\right] + \frac{\eta^2 (k-1)^{-2\nu}}{2}.
	\end{split}
	\end{equation}
\end{lemma}
\begin{proof}[Proof of Lemma~\ref{lem:bnd_special_case}] Throughout this proof let $k\in\lb 2,3,\dots\rb$. 
	Observe that \linebreak Lemma~\ref{lem:bounds_sum} implies that  for all $n\in\lbrace 1, 2, \dots, k-1\rbrace$ we have that
	\begin{equation}\label{eq:bnd_special_case_6}
	\mathfrak{t}_{n+1}\leq \eta +\frac{\eta}{1-\nu}(n^{1-\nu}-1) \qquad \text{and}\qquad \mathfrak{t}_k\geq \frac{\eta}{1-\nu}\left( k^{1-\nu}-1 \right).
	\end{equation}
	Lemma~\ref{lem:comp_int} hence ensures that
	\begin{align*}
	\label{eq:bnd_special_case_8}
	&\int_0^{\mathfrak{t}_k}e^{-L(\mathfrak{t}_k-t)}(t-\floorgrid{t})\, dt
	\leq \frac{1}{2}\textstyle\sum\limits_{n=1}^{k-1} \displaystyle e^{-L(\mathfrak{t}_k-\mathfrak{t}_{n+1})}(\mathfrak{t}_{n+1}-\mathfrak{t}_n)^2 \numberthis\\
	& \leq \frac{1}{2}\left[\textstyle\sum\limits_{n=1}^{k-2} \displaystyle \exp\!\left(-L\left( \tfrac{\eta}{1-\nu}(k^{1-\nu}-1)-\big( \eta+\tfrac{\eta}{1-\nu}(n^{1-\nu}-1) \big) \right)\right)\frac{\eta^2}{n^{2\nu}}\right]
	+ \frac{\eta^2}{2(k-1)^{2\nu}}\\
	& = \frac{\eta^2 e^{L\eta}}{2}\left[\textstyle\sum\limits_{n=1}^{k-2} \displaystyle \frac{\exp\!\left( \tfrac{-L\eta}{1-\nu}(k^{1-\nu}-n^{1-\nu}) \right)}{n^{2\nu}}\right]+ \frac{\eta^2 (k-1)^{-2\nu}}{2}.
	\end{align*}
	This establishes \eqref{eq:bnd_special_case_4}. The proof of Lemma~\ref{lem:bnd_special_case} is thus completed.
\end{proof}

\begin{lemma}\label{lem:min_fct_bnd}
	Let $c\in (0,\infty)$, $\epsilon\in (0,\nicefrac{1}{2})$, $\alpha=\left( \frac{2-2\epsilon}{c\epsilon} \right)^{1/\epsilon}$ and let $v\colon (0,\infty)\to \R$ satisfy for all $x\in (0,\infty)$ that
	\begin{equation}
	v(x)=x^{2\epsilon-2}\exp(cx^{\epsilon}).
	\end{equation}
	Then $v$ is non-increasing on $(0,\alpha]$ and non-decreasing on $[\alpha,\infty)$.
\end{lemma}
\begin{proof}[Proof of Lemma~\ref{lem:min_fct_bnd}] 
	First, observe that for all $x\in (0,\infty)$  we have that
	\begin{equation}
	\begin{split}
	v'(x) &=c\epsilon x^{\epsilon-1}\exp(cx^{\epsilon})x^{2\epsilon-2} + \exp(cx^{\epsilon})(-2+2\epsilon)x^{2\epsilon-3}\\
	& = c\epsilon \exp(cx^{\epsilon})x^{3\epsilon-3}-(2-2\epsilon)\exp(cx^{\epsilon})x^{2\epsilon-3}\\
	& = \exp(cx^{\epsilon})x^{2\epsilon-3}\big[ c\epsilon x^{\epsilon} -(2-2\epsilon)\big].
	\end{split}
	\end{equation}
	This reveals that
	\begin{equation}
	\label{eq:derivative:zero}
	\{ x \in (0, \infty) \colon v'(x) = 0 \} = \{\alpha\}.
	\end{equation}
	Next note that for all $x\in (0,\infty)$ we have that
	\begin{align}
	\begin{split}
	v''(x)&=(c\epsilon)^2 x^{\epsilon-1}\exp(cx^{\epsilon})x^{3\epsilon-3}-c\epsilon \exp(cx^{\epsilon})(3-3\epsilon)x^{3\epsilon-4}\\
	& \quad-(2-2\epsilon)c\epsilon x^{\epsilon-1}\exp(cx^{\epsilon}) x^{2\epsilon-3}+(2-2\epsilon)\exp(cx^{\epsilon})(3-2\epsilon)x^{2\epsilon-4}\\
	& = (c\epsilon)^2 \exp(cx^{\epsilon})x^{4\epsilon-4} - c\epsilon \exp(cx^{\epsilon}) (3-3\epsilon)x^{3\epsilon-4}\\
	& \quad -(2-2\epsilon)c\epsilon \exp(cx^{\epsilon})x^{3\epsilon-4}+(2-2\epsilon)(3-2\epsilon)\exp(cx^{\epsilon})x^{2\epsilon-4} \\
	& = (c\epsilon)^2 \exp(cx^{\epsilon})x^{4\epsilon-4} - (5-5\epsilon)c\epsilon \exp(cx^{\epsilon}){x^{3\epsilon-4}} \\
	& \quad +(2-2\epsilon)(3-2\epsilon)\exp(cx^{\epsilon})x^{2\epsilon-4}\\
	& = c\epsilon \exp(cx^{\epsilon})x^{2\epsilon-4}\left( c\epsilon x^{2\epsilon} -  (5-5\epsilon)x^{\epsilon} +\tfrac{(2-2\epsilon)(3-2\epsilon)}{c\epsilon} \right).
	\end{split}
	\end{align}
	This implies that
	\begin{align}
	\begin{split}
	v''(\alpha) & = c\epsilon e^{c\alpha^{\epsilon}}\alpha^{2\epsilon-4}\left( c\epsilon \alpha^{2\epsilon} -  (5-5\epsilon)\alpha^{\epsilon} +\tfrac{(2-2\epsilon)(3-2\epsilon)}{c\epsilon} \right)\\
	& = c\epsilon e^{c\alpha^{\epsilon}}\alpha^{2\epsilon-4}\left( c\epsilon \left( \tfrac{2-2\epsilon}{c\epsilon} \right)^{2} -  (5-5\epsilon)\left( \tfrac{2-2\epsilon}{c\epsilon} \right) +\tfrac{(2-2\epsilon)(3-2\epsilon)}{c\epsilon} \right)\\
	& =  e^{c\alpha^{\epsilon}}\alpha^{2\epsilon-4}\left(  ( 4-8\epsilon+4\epsilon^2 ) -  (10-20\epsilon + 10\epsilon^2) +(6-10\epsilon + 4\epsilon^2) \right)\\
	& =  e^{c\alpha^{\epsilon}}\alpha^{2\epsilon-4}(2\epsilon-2\epsilon^2)>0. 
	\end{split}
	\end{align}
	Combining this with \eqref{eq:derivative:zero} verifies that $v$ is non-increasing on $(0,\alpha]$ and non-decreasing on $[\alpha,\infty)$. The proof of Lemma~\ref{lem:min_fct_bnd} is thus completed.
\end{proof}

\begin{lemma}\label{lem:int_bnd_on_sum}
	Let $a\in (0,\infty)$, $\epsilon\in (0,\nicefrac{1}{2})$.
	Then we have for all $l\in \N$ that
	\begin{equation}\label{eq:int_bnd_on_sum_4}
	\textstyle\sum\limits_{n=1}^{l}n^{2\epsilon-2} \displaystyle \exp(an^{\epsilon})\leq e^a + \int_{1}^{l+1}x^{2\epsilon-2}\exp(ax^{\epsilon})\, dx.
	\end{equation}
\end{lemma}
\begin{proof}[Proof of Lemma~\ref{lem:int_bnd_on_sum}] 
	Throughout this proof let $v\colon (0,\infty)\to \R$ satisfy for all $x\in (0,\infty)$ that
	\begin{equation}\label{eq:int_bnd_on_sum_1}
	v(x)=x^{2\epsilon-2}\exp(ax^{\epsilon})
	\end{equation}
	and let
	\begin{equation}
	\mathfrak{N}=\max\!\left\{ \Big( -\infty, \left| \tfrac{2-2\epsilon}{a\epsilon} \right|^{1/\epsilon} \Big] \cap \Z \right\}.
	\end{equation}
	Note that for all $l\in \N$ we have that
	\begin{equation}\label{eq:int_bnd_on_sum_5}
	\begin{split}
	\textstyle\sum\limits_{n=1}^{l} \displaystyle n^{2\epsilon-2}\exp(an^{\epsilon}) 
	& = 	\textstyle\sum\limits_{n=1}^l \displaystyle v(n) \leq                                                                                                                                                                                                                                                                                                                                                                                                                                                                                                                                                                                           v(1)+\textstyle\sum\limits_{n=2}^{\min\{\mathfrak{N}, l\}} \displaystyle v(n)+\textstyle\sum\limits_{n=\min\{\mathfrak{N}, l\}+1}^l \displaystyle v(n).
	\end{split}
	\end{equation}
	Combining this and Lemma~\ref{lem:min_fct_bnd} assures that for all $l\in \N$ we have that
	\begin{equation}\label{eq:int_bnd_on_sum_6}
	\begin{split}
	&\textstyle\sum\limits_{n=1}^{l} \displaystyle n^{2\epsilon-2}\exp(an^{\epsilon})\\
	&  \leq v(1)+\textstyle\sum\limits_{n=2}^{\min\{\mathfrak{N}, l\}} \displaystyle \int_{n-1}^n v(n)\, dx+\textstyle\sum\limits_{n=\min\{\mathfrak{N}, l\}+1}^l \displaystyle \int_{n}^{n+1} v(n)\, dx\\
	& \leq v(1) + \textstyle\sum\limits_{n=2}^{\min\{\mathfrak{N}, l\}} \displaystyle \int_{n-1}^n v(x)\, dx + \textstyle\sum\limits_{n=\min\{\mathfrak{N}, l\}+1}^l \displaystyle \int_{n}^{n+1} v(x)\, dx \\
	& = v(1) + \textstyle\sum\limits_{n=1}^{\min\{\mathfrak{N}, l\}-1} \displaystyle \int_{n}^{n+1} v(x)\, dx + \textstyle\sum\limits_{n=\min\{\mathfrak{N}, l\}+1}^l \displaystyle \int_{n}^{n+1} v(x)\, dx \\
	& \leq v(1) + \textstyle\sum\limits_{n=1}^{l} \displaystyle \int_{n}^{n+1} v(x)\, dx = v(1) + \int_{1}^{l+1} v(x)\, dx .
	\end{split}
	\end{equation}
	This demonstrates \eqref{eq:int_bnd_on_sum_4}. The proof of Lemma~\ref{lem:int_bnd_on_sum} is thus completed.
\end{proof}

\begin{lemma}\label{lem:rate_conv}
	Let $n\in \N$, $a\in[0,\infty)$.
	Then we have for all $\epsilon\in (0,\nicefrac{1}{2})$, $\lambda\in (0,1]$ that
	\begin{equation}\label{eq:rate_conv_2}
	\int_1^nx^{2\epsilon-2}\exp(a(x^{\epsilon}-n^{\epsilon}))\, dx\leq \frac{1}{1-2\epsilon}\big(\exp(a(\lambda^{\epsilon}n^{\epsilon}-n^{\epsilon})) + (\lambda n)^{2\epsilon-1}\big).
	\end{equation}
\end{lemma}
\begin{proof}[Proof of Lemma~\ref{lem:rate_conv}]
	Observe that for all $\epsilon\in (0,\nicefrac{1}{2})$, $\lambda\in (0,1]$ we have that
	\begin{equation}
	\begin{split}
	&\int_1^n {x^{2\epsilon-2}}\exp(a(x^{\epsilon}-n^{\epsilon}))\, dx\\
	& = \int_1^{\lambda n}{x^{2\epsilon-2}}\exp(a(x^{\epsilon}-n^{\epsilon}))\, dx + \int_{\lambda n}^n{x^{2\epsilon-2}}\exp(a(x^{\epsilon}-n^{\epsilon}))\, dx\\
	& \leq \exp(a(\lambda^{\epsilon}n^{\epsilon}-n^{\epsilon}))\int_1^{\lambda n} {x^{2\epsilon-2}}\, dx + \int_{\lambda n}^{n} {x^{2\epsilon-2}}\, dx\\
	& \leq \exp(a(\lambda^{\epsilon}n^{\epsilon}-n^{\epsilon}))\int_1^{\infty}{x^{2\epsilon-2}}\, dx + \left[ \frac{1}{2\epsilon-1}{x^{2\epsilon-1}}\right]_{x=\lambda n}^{x=\infty}\\
	& = \exp(a(\lambda^{\epsilon}n^{\epsilon}-n^{\epsilon}))\frac{1}{1-2\epsilon} + \frac{1}{1-2\epsilon}(\lambda n)^{2\epsilon-1}\\
	& = \frac{1}{1-2\epsilon}\big(\exp(a(\lambda^{\epsilon}n^{\epsilon}-n^{\epsilon})) + (\lambda n)^{2\epsilon-1}\big).
	\end{split}
	\end{equation}
	The proof of Lemma~\ref{lem:rate_conv} is thus completed.
\end{proof}

\begin{lemma}\label{lem:bnd_sum_main}
	Assume Setting~\ref{sec:setting_1}, let $L,\eta\in (0,\infty)$,  $\epsilon\in (0,\nicefrac{1}{2})$,   let $ K \colon (0,1)\to (0,\infty]$ satisfy for all $\lambda\in (0,1)$ that
	\begin{align*}
	\label{eq:bnd_sum_main_6}
	&K(\lambda) \numberthis \\
	&= \sup_{n \in \N \cap [2,\infty)} \bigg[ \frac{\eta^2 \exp(L\eta+\frac{L\eta}{\epsilon}) }{2(1-2\epsilon)} \Big(  n^{1-2\epsilon} \Big[ 2\exp\!\big(\!-\tfrac{L\eta}{\epsilon}(1-\lambda^{\epsilon})n^{\epsilon} \big) + (n-1)^{2\epsilon-2} \Big] + \lambda^{2\epsilon-1}\Big)  \bigg],
	\end{align*}
	and assume that
	\begin{equation}
	\lb s\in [0,\infty)\colon \gamma(s)\neq \emptyset\rb = \lb 0 \rb \cup \left\lb \textstyle\sum\limits_{n=1}^m \displaystyle \frac{\eta}{n^{1-\epsilon}}\colon m\in\N  \right\rb.
	\end{equation}
	Then we have for all $\lambda\in (0,1)$, $k\in\N$ that
	\begin{equation}\label{eq:bnd_sum_main_10}
	\int_0^{\sum_{n=1}^{k-1}  \frac{\eta}{n^{1-\epsilon}}} (t-\floorgrid{t}) \exp\!\big(\!-L\big(\textstyle\sum_{n=1}^{k-1}  \frac{\eta}{n^{1-\epsilon}}\big) +Lt \big) \, dt\leq K(\lambda){k^{2\epsilon-1}} < \infty.
	\end{equation}
\end{lemma}
\begin{proof}[Proof of Lemma~\ref{lem:bnd_sum_main}]
	Throughout this proof let $\lambda\in (0,1)$, $k\in\N\cap [2,\infty)$, let $a=\frac{L\eta}{\epsilon} \in (0, \infty)$, let $v\colon (0,\infty)\to \R$ satisfy for all $x\in (0,\infty)$ that
	\begin{equation}
	v(x)=x^{2\epsilon-2}\exp(ax^{\epsilon}),
	\end{equation} 
	let $\kappa\colon \N \cap [2,\infty)\to (0,\infty)$ satisfy for all $n \in \N \cap [2,\infty)$ that
	\begin{equation}
	\kappa(n) = \frac{\eta^2e^{L\eta+a} }{2(1-2\epsilon)} \big(  n^{1-2\epsilon} \big[ 2e^{-a(1-\lambda^{\epsilon})n^{\epsilon}} + (n-1)^{2\epsilon-2} \big] + \lambda^{2\epsilon-1}\big),
	\end{equation}
	and let $\mathfrak{t}\colon \N\to[0,\infty)$ satisfy for all $n\in\lb 2,3,\dots\rb$ that $\mathfrak{t}_1=0$ and
	\begin{equation}
	\mathfrak{t}_n=\textstyle\sum\limits_{m=1}^{n-1} \displaystyle \frac{\eta}{m^{1-\epsilon}}.
	\end{equation}
	Note that Lemma~\ref{lem:bnd_special_case} implies that
	\begin{equation}\label{eq:bnd_sum_main_11}
	\begin{split}
	&\int_0^{\mathfrak{t}_k}e^{-L(\mathfrak{t}_k-t)}(t-\floorgrid{t})\, dt \\
	& \leq \left[ \frac{\eta^2 e^{L\eta}}{2}\textstyle\sum\limits_{n=1}^{k-2} \displaystyle \frac{\exp\!\left( -\tfrac{L\eta}{\epsilon}(k^{\epsilon}-n^{\epsilon}) \right)}{n^{2-2\epsilon}}\right] + \frac{\eta^2 (k-1)^{2\epsilon-2}}{2}\\
	& = \left[ \frac{\eta^2 e^{L\eta}}{2}\textstyle\sum\limits_{n=1}^{k-2} \displaystyle n^{2\epsilon-2}\exp({a(n^{\epsilon}-k^{\epsilon})})\right] + \frac{\eta^2 (k-1)^{2\epsilon-2}}{2}.
	\end{split}
	\end{equation}
	Next observe that Lemma~\ref{lem:int_bnd_on_sum} ensures that
	\begin{equation}\label{eq:bnd_sum_main_12}
	\begin{split}
	\textstyle\sum\limits_{n=1}^{k-2} \displaystyle n^{2\epsilon-2}\exp({a(n^{\epsilon}-k^{\epsilon})})
	& \leq e^{-ak^{\epsilon}}\left( e^a +\int_1^{k-1} v(x)\, dx \right).
	\end{split}
	\end{equation}
	Combining this and \eqref{eq:bnd_sum_main_11} demonstrates that
	\begin{equation}\label{eq:bnd_sum_main_13}
	\begin{split}
	\int_0^{\mathfrak{t}_k}e^{-L(\mathfrak{t}_k-t)}(t-\floorgrid{t})\, dt 
	&\leq \frac{\eta^2  e^{L\eta}}{2} e^{-ak^{\epsilon}}\left( e^a +\int_1^{k-1} v(x)\, dx \right) + \frac{\eta^2 (k-1)^{2\epsilon-2}}{2}\\
	& \leq \frac{\eta^2  e^{L\eta -ak^{\epsilon}}}{2}\left( e^a +\int_1^{k} v(x)\, dx \right) + \frac{\eta^2 (k-1)^{2\epsilon-2}}{2}.
	\end{split}
	\end{equation}
	Lemma~\ref{lem:rate_conv} hence assures that
	\begin{align}
	\begin{split}
	&\int_0^{\mathfrak{t}_k}e^{-L(\mathfrak{t}_k-t)}(t-\floorgrid{t})\, dt \\
	&\leq \frac{\eta^2 e^{L\eta-ak^{\epsilon}+a} }{2}
	+ \frac{\eta^2e^{L\eta}}{2}\int_1^kx^{2\epsilon-2}\exp(a(x^{\epsilon}-k^{\epsilon}))\, dx 
	+ \frac{\eta^2 (k-1)^{2\epsilon-2}}{2}\\
	& \leq \frac{\eta^2 e^{L\eta-ak^{\epsilon}+a} }{2}
	+ \frac{\eta^2e^{L\eta}}{2(1-2\epsilon)}\big(\exp(a(\lambda^{\epsilon}k^{\epsilon}-k^{\epsilon})) + (\lambda k)^{2\epsilon-1}\big) 
	+ \frac{\eta^2 (k-1)^{2\epsilon-2}}{2}\\
	& \leq \frac{\eta^2e^{L\eta+a}}{2(1-2\epsilon)} \Big(\!\exp(-ak^{\varepsilon}) + \exp(-a(1-\lambda^{\epsilon})k^{\epsilon}) + (\lambda k)^{2\epsilon-1} + (k-1)^{2\epsilon-2} \Big) .
	\end{split}
	\end{align}
	This reveals that
	\begin{align}
	\label{eq:bnd_sum_main_14} 
	\begin{split}
	&\int_0^{\mathfrak{t}_k}e^{-L(\mathfrak{t}_k-t)}(t-\floorgrid{t})\, dt    \\
	& \leq \frac{\eta^2e^{L\eta+a}}{2(1-2\epsilon)} \Big( 2\exp(-a(1-\lambda^{\epsilon})k^{\epsilon}) + (\lambda k)^{2\epsilon-1} + (k-1)^{2\epsilon-2} \Big)\\
	& = \frac{\eta^2e^{L\eta+a} k^{2\epsilon-1}}{2(1-2\epsilon)} \Big(  k^{1-2\epsilon} \Big[ 2\exp(-a(1-\lambda^{\epsilon})k^{\epsilon}) + (k-1)^{2\epsilon-2} \Big] + \lambda^{2\epsilon-1}\Big)\\
	& \leq  K(\lambda) {k^{2\epsilon-1}}.
	\end{split}
	\end{align}
	In addition, note that the fact that
	\begin{equation}
	\limsup_{n \to \infty} \Big( n^{1-2\epsilon}\exp(-a(1-\lambda^{\epsilon})n^{\epsilon}) + n^{1-2\epsilon}(n-1)^{2\epsilon-2}\Big) = 0
	\end{equation}
	ensures that
	\begin{equation}
	\limsup_{n\to\infty}\kappa(n)=\frac{\eta^2e^{L\eta+a}\lambda^{2\epsilon-1}}{2(1-2\epsilon)} < \infty. 
	\end{equation}
	This reveals that
	\begin{equation}
	\sup_{n \in \N \cap [2,\infty) } \kappa(n) < \infty.
	\end{equation}
	Combining this and \eqref{eq:bnd_sum_main_14}  establishes \eqref{eq:bnd_sum_main_10}.   The proof of Lemma~\ref{lem:bnd_sum_main} is thus completed.
\end{proof}

\section{Weak error estimates for SAAs in the case of polynomially decaying learning rates with mini-batches}
\label{subsec:SAA:poly}
\sectionmark{}

\begin{cor}\label{cor:sp_case_gamma}
Assume Setting~\ref{sec:setting_1},  assume for all $v,w\in [0,\infty)$ with $v\neq w$  that 
$
\gamma(v)\cap \gamma(w)=\emptyset
$, let $\psi\in C^2(\R^d,\R)$,  $\epsilon\in (0,\nicefrac{1}{2})$,  $L,\eta \in (0,\infty)$, assume  for all $y, z\in\R^d$  that
\begin{equation}
\langle g(y)-g(z), y-z\rangle_{\R^d} \leq -L \|y-z\|_{\R^d}^2,
\end{equation}
\begin{equation}
\sup_{x\in\R^d}\Bigg( \frac{\E\big[ \|G(x,Z_1)\|_{\R^d}^2 \big]}{\big[1+\|x\|_{\R^d}\big]^2}
+ \frac{\big\|\E\big[\big(\tfrac{\partial}{\partial x}G\big)(x,Z_1)\big]\big\|_{L(\R^d,\R^d)}}{\big[1+\|x\|_{\R^d}\big]}
+ \|\psi'(x)\|_{L(\R^d,\R)} \Bigg) < \infty,
\end{equation}
and
\begin{equation}
\lb s\in [0,\infty)\colon \gamma(s)\neq \emptyset\rb = \lb 0 \rb \cup \left\lb \textstyle\sum\limits_{n=1}^m \displaystyle \frac{\eta}{n^{1-\epsilon}}\colon m\in\N \right\rb,
\end{equation}
let $Q\colon [0,\infty)\times\Omega\to\R^d$ be the stochastic process which satisfies for all $t\in [0,\infty)$ that
\begin{equation}
Q_t=\tfrac{1}{\#_{\gamma(\floorgrid{t})}}\textstyle\sum_{j\in \gamma(\floorgrid{t})} G(\Theta_{\floorgrid{t}},Z_j) ,
\end{equation}\\
let $K\colon (0,1)\to (0,\infty]$ satisfy for all $\lambda\in (0,1)$ that
\begin{align*}
&K(\lambda) \numberthis \\
&= \sup_{n \in \N \cap [2,\infty)} \bigg[ \frac{\eta^2\exp(L\eta+\frac{L\eta}{\epsilon}) }{2(1-2\epsilon)} \Big(  n^{1-2\epsilon} \Big[ 2 \exp\! \big(-\tfrac{L\eta}{\epsilon}(1-\lambda^{\epsilon})n^{\epsilon} \big) + (n-1)^{2\epsilon-2} \Big] + \lambda^{2\epsilon-1}\Big)  \bigg],
\end{align*}
and let $C\colon [0,\infty)\to [0,\infty]$ satisfy for all $T\in [0,\infty)$ that
\begin{align*}
\label{eq:cor:C(T)}
& C(T)
 = \sup_{s,v\in [0,T]}\E \bigg[ \|Q_s-g(\Theta_{\floorgrid{s}})\|_{\R^d} \norm{Q_s}_{\R^d} \\
 &   \cdot \bigg(\int_0^1  \exp(-L(T-s)) \|\psi''( \theta_{T-s}^{\lambda \Theta_s + (1-\lambda)\Theta_{\floorgrid{s}}})\|_{L^{(2)}(\R^d,\R)} \numberthis \\
&  + \|\psi'( \theta_{T-s}^{\lambda \Theta_s + (1-\lambda)\Theta_{\floorgrid{s}}} )\|_{L(\R^d,\R)} \int_0^{T-s} \exp(-Lu)\big\|g''(\theta_u^{\lambda \Theta_s + (1-\lambda)\Theta_{\floorgrid{s}}})\big\|_{L^{(2)}(\R^d,\R^d)}\, du\, d\lambda \bigg)\\
&  + \|\psi'(\theta_{T-s}^{\Theta_s})\|_{L(\R^d,\R)} \big\|g'(\Theta_{v})Q_v \big\|_{\R^d}\bigg]  
\end{align*}
(cf.~item~\eqref{item:twice:diff:f} in  Lemma~\ref{lem:twice:diff}).
Then
\begin{enumerate}[(i)]
\item\label{item:cor:limit} we have that there exists a unique  $\Xi \in \R^d$ which satisfies that
\begin{equation}
\limsup_{t \to \infty} \|\theta_t^{\xi} - \Xi\|_{\R^d} = 0
\end{equation}
and
\item\label{item:cor:error} we have for all $\lambda\in (0,1)$, $k\in\N$ that
\begin{align}
\label{eq:sp_case_gamma_13}
\begin{split}
& \la\E[\psi(\Theta_{\sum_{n=1}^{k-1}\frac{\eta}{n^{1-\epsilon}}})]- \psi(\Xi)\ra   \\
& \leq k^{2\epsilon-1} \Bigg[ K(\lambda)C\big(\textstyle\sum_{n=1}^{k-1}\frac{\eta}{n^{1-\epsilon}}\big)
  + k^{1-2\epsilon} \exp\! \big(-L \big(\textstyle\sum_{n=1}^{k-1}\frac{\eta}{n^{1-\epsilon}}\big)\big)\\
 & \quad  \cdot\sup_{\alpha \in [0,1]}\Big( \| \psi'(\alpha\theta_{\sum_{n=1}^{k-1}\frac{\eta}{n^{1-\epsilon}}}^{\xi}+(1-\alpha)\Xi) \|_{L(\R^d,\R)}\Big) \|\xi-\Xi\|_{\R^d}\Bigg].
 \end{split}
\end{align}
\end{enumerate}
\end{cor}
\begin{proof}[Proof of Corollary~\ref{cor:sp_case_gamma}]
First, observe that item~\eqref{item:main:limit} in Proposition~\ref{prop:main_theorem} implies that  there exists a unique  $\Xi \in \R^d$ which satisfies that
\begin{equation}
\limsup_{t \to \infty} \|\theta_t^{\xi} - \Xi\|_{\R^d} = 0.
\end{equation}
This establishes item~\eqref{item:cor:limit}. Next note that item~\eqref{item:main_theorem_2} in Proposition~\ref{prop:main_theorem} and Lemma~\ref{lem:estimate_diff_phi} demonstrate that for all $T\in [0,\infty)$ we have that
\begin{align*}
& \la\E\big[\psi(\Theta_T)\big]- \psi(\Xi)\ra  \leq  \sup_{s,v\in [0,T]}\E \bigg[ \|Q_s-g(\Theta_{\floorgrid{s}})\|_{\R^d}\norm{Q_s}_{\R^d}\\
& \quad \cdot \bigg(\int_0^1  e^{-L(T-s)} \big\|\psi''( \theta_{T-s}^{\lambda \Theta_s + (1-\lambda)\Theta_{\floorgrid{s}}})\big\|_{L^{(2)}(\R^d,\R)} + \big\|\psi'( \theta_{T-s}^{\lambda \Theta_s + (1-\lambda)\Theta_{\floorgrid{s}}} )\big\|_{L(\R^d,\R)}\\
& \quad \cdot \int_0^{T-s} e^{-Lu}\big\|g''(\theta_u^{\lambda \Theta_s + (1-\lambda)\Theta_{\floorgrid{s}}})\big\|_{L^{(2)}(\R^d,\R^d)}\, du\, d\lambda \bigg) \numberthis  \\
& \quad + \big\|\psi'(\theta_{T-s}^{\Theta_s})\big\|_{L(\R^d,\R)} \big\|g'(\Theta_{v})Q_v \big\|_{\R^d}\bigg] \int_0^T e^{-L(T-t)}(t-\floorgrid{t})\, dt\\
& \quad + \sup\left\lbrace \big\| \psi'(\alpha\theta_T^{\xi}+(1-\alpha)\Xi) \big\|_{L(\R^d,\R)} \in \R\colon \alpha \in [0,1] \right\rbrace \! \|\xi-\Xi\|_{\R^d}e^{-LT}.
\end{align*}
This and \eqref{eq:cor:C(T)} imply that for all $T\in [0,\infty)$ we have that
\begin{align}
\begin{split}
& \la\E\big[\psi(\Theta_T)\big]- \psi(\Xi)\ra  \leq C(T) \int_0^T e^{-L(T-t)}(t-\floorgrid{t})\, dt\\
& \quad+ \sup_{\alpha \in [0,1]} \Big( \big\| \psi'(\alpha\theta_T^{\xi}+(1-\alpha)\Xi) \big\|_{L(\R^d,\R)}\Big) \|\xi-\Xi\|_{\R^d}e^{-LT}.
\end{split}
\end{align}
Lemma~\ref{lem:bnd_sum_main} therefore ensures that for all $\lambda\in (0,1)$, $k\in\N$ we have that
\begin{align}
\begin{split}
&\la\E[\psi(\Theta_{\sum_{n=1}^{k-1}\frac{\eta}{n^{1-\epsilon}}})]- \psi(\Xi)\ra   \leq C\big(\textstyle\sum_{n=1}^{k-1}\frac{\eta}{n^{1-\epsilon}}\big) K(\lambda) k^{2\epsilon-1} \\
& \quad +   \sup_{\alpha \in [0,1]}\Big( \| \psi'(\alpha\theta_{\sum_{n=1}^{k-1}\frac{\eta}{n^{1-\epsilon}}}^{\xi}+(1-\alpha)\Xi) \|_{L(\R^d,\R)}\Big) \|\xi-\Xi\|_{\R^d} e^{-L(\sum_{n=1}^{k-1}\frac{\eta}{n^{1-\epsilon}})}\\
& = k^{2\epsilon-1} \Bigg[ K(\lambda)C\big(\textstyle\sum_{n=1}^{k-1}\frac{\eta}{n^{1-\epsilon}}\big)
+ k^{1-2\epsilon}e^{-L(\sum_{n=1}^{k-1}\frac{\eta}{n^{1-\epsilon}})}\\
& \quad \cdot\sup_{\alpha \in [0,1]}\Big( \| \psi'(\alpha\theta_{\sum_{n=1}^{k-1}\frac{\eta}{n^{1-\epsilon}}}^{\xi}+(1-\alpha)\Xi) \|_{L(\R^d,\R)}\Big) \|\xi-\Xi\|_{\R^d}\Bigg]. 
\end{split}
\end{align}
This establishes item~\eqref{item:cor:error}. The proof of Corollary~\ref{cor:sp_case_gamma} is thus completed.
\end{proof}

\chapter{Weak error estimates for SAAs in the case of polynomially decaying learning rates}
\label{sec:SAA:poly}
\chaptermark{}

In 
this chapter we specialize the weak error analysis for SAAs in the case of general learning rates from Chapter~\ref{sec:SAA:general} to accomplish weak error estimates for SAAs in the case of polynomially decaying learning rates. In particular, we present and prove in this chapter the main result of this paper, Theorem~\ref{cor:no_setting_discrete_version} in Section~\ref{subsec:weak:SAA:poly} below, which establishes weak convergence rates for SAAs in the case of polynomially decaying learning rates with  mini-batches. In Section~\ref{subsec:SAA:without} we apply Theorem~\ref{cor:no_setting_discrete_version} to establish in \Cref{cor:no_setting_discrete_version_one_sample} in Section~\ref{subsec:SAA:without}  below weak convergence rates for SAAs in the case of polynomially decaying learning rates without mini-batches. In Section~\ref{subsec:example} below
 we illustrate   \Cref{cor:no_setting_discrete_version_one_sample} by means of an elementary example.
Our proof of Theorem~\ref{cor:no_setting_discrete_version}  employs the weak error analysis result in
Corollary~\ref{cor:sp_case_gamma} in Section~\ref{subsec:SAA:poly} above,
the elementary results on suitable sequences of uniformly bounded functions in Section~\ref{subsec:sequence} below,   the elementary result on differentiable functions with bounded derivatives in \Cref{lem:diff_bnded_map_linear_growth} in Section~\ref{subsec:apriori:constant} below, and the a priori estimates for suitable approximation error constants associated to SAAs in \Cref{lem:bnd_C_case_f''_bnded} in Section~\ref{subsec:apriori:constant} below.
Our proof of \Cref{lem:bnd_C_case_f''_bnded}, in turn, uses
 the result on the possibility of interchanging derivatives and expectations in Lemma~\ref{lem:f_C^n} in Section~\ref{subsec:sufficient:exp} above, the a  priori estimates for SAAs in the case of general learning rates in  Lemma~\ref{lem:Theta_L^2_sp_case} in Section~\ref{subsec:apriori} above, 
 the  a priori estimates for SAAs in the case of polynomially decaying learning rates in Section~\ref{subsec:apriori:poly} below, 
 and the a posteriori estimates for conditional variances associated to SAAs in Lemma~\ref{lem:iterated_error} in Section~\ref{subsec:aposteriori}  below. 
In the scientific literature a posteriori estimates similar to the ones as in Lemma~\ref{lem:iterated_error}   can, e.g., be found in~\cite[(3) in Theorem 1.1, (169) in Corollary~3.5, and (217) in Theorem~3.7]{Wurstemberger2018}). Our proof of Lemma~\ref{lem:iterated_error}, in turn,   employs the elementary growth bound estimates  in Lemma~\ref{lem:(1)+(2)=>(*)} in Section~\ref{subsec:aposteriori}   below. In the scientific literature
 similar results to Lemma~\ref{lem:(1)+(2)=>(*)}   can, e.g., be found in  Dereich \& M\"uller-Gronbach \cite[Remark 2.1]{DereichMuellerGronbach2019}.
 In Setting~\ref{sec:setting_2}  in Section~\ref{subsec:setting:4} below we present a mathematical framework for describing SAAs  in the case of  polynomially decaying learning rates. In the results of this chapter  we frequently employ Setting~\ref{sec:setting_2}.

\section{Mathematical description for SAAs in the case of  polynomially decaying learning rates}
\label{subsec:setting:4}
\sectionmark{}

\begin{setting}
\label{sec:setting_2}
 Let $d\in \N$, $ \xi,\,\Xi\in \R^d$, $\epsilon\in (0,\nicefrac{1}{2})$, $\eta,L \in(0,\infty)$,
$(\mathfrak{M}_n)_{n \in \N_0} \subseteq \N$, 
let $(S,\mathcal{S})$ be a measurable space, let $(\Omega,\F,\P)$ be a probability space, let $Z_{m, n}\colon\Omega \to S$, $(m, n) \in  \N^2$, be i.i.d.\ random variables,  let $\mathfrak{t}\colon \N_0 \to [0,\infty)$ satisfy for all $m \in \N_0$ that
$
\mathfrak{t}_m = \eta [\sum_{n=1}^{m} \displaystyle n^{\epsilon-1}],
$
let  $G  = (G(x, s))_{(x, s) \in \R^d \times S} \colon \R^d\times S\to \R^d$ be   $(\mathcal{B}(\R^d)\otimes \mathcal{S})$/$\mathcal{B}(\R^d)$-measurable,  let $g\colon \mathbb{R}^d \to \R^d$ be a function, assume for all $s \in S$ that $( \R^d\ni x \mapsto G(x,s)\in\R^d)\in C^2(\R^d,\R^d)$, assume for all $x, y\in\R^d$ that 
\begin{equation}
\label{eq:setting_2_2}
\vspace{-.8cm}
\max_{i \in \{1,2\}} \inf_{\delta\in(0,\infty)}\sup_{u\in [-\delta,\delta]^d}\E\Big[ \|G(x,Z_{1,1})\|_{\R^d} + \| (\tfrac{\partial^i}{\partial x^i}G)(x+u,Z_{1,1})\|_{L^{(i)}(\R^d,\R^d)}^{1+\delta} \Big]<\infty,
\end{equation} 
\begin{equation}\label{eq:setting_2_3_2}
 \langle x-\Xi,g(x)\rangle_{\R^d} \leq -L \|g(x)\|_{\R^d}^2,
\qquad
\langle x-y,g(x)-g(y)\rangle_{\R^d} \leq -L \| x-y\|_{\R^d}^2,
\end{equation}
and $g(x)=\E[ G(x,Z_{1,1}) ]$
(cf.~Corollary~\ref{cor:derivative:gen}),
let $\theta^{\vartheta}\in C( [0,\infty), \R^d)$, $\vartheta\in \R^d$,  satisfy for all $t\in [0,\infty)$, $\vartheta\in\R^d$ that
\begin{equation}\label{eq:setting_2_4}
\theta_t^{\vartheta}= \vartheta + \int_0^t g(\theta_s^{\vartheta})\, ds
\end{equation}
(cf.~item~\eqref{item1:diff:induction} in Lemma~\ref{lem:diff:induction}), and
 let $\Theta \colon [0,\infty)\times \Omega \to \R^d$ be the stochastic process w.c.s.p.\ which satisfies for all
$m \in \N_0$, $t \in [\mathfrak{t}_m, \mathfrak{t}_{m+1})$ that
 $\Theta_0 = \xi$ and 
\begin{equation}\label{eq:setting_2_6}
\Theta_t = \Theta_{\mathfrak{t}_m} + \frac{(t-\mathfrak{t}_m)}{\mathfrak{M}_m} \left[ \textstyle\sum\limits_{n=1}^{\mathfrak{M}_m} \displaystyle G(\Theta_{\mathfrak{t}_m},Z_{m+1, n}) \right].
\end{equation}
\end{setting}

\section{On a sequence of uniformly bounded functions}
\label{subsec:sequence}
\sectionmark{}

\begin{lemma}\label{lem:cont_max}
	Let $d_1,d_2,d_3\in\N$, $f\in C(\R^{d_1}\times \R^{d_2},\R^{d_3})$ and let $K\subseteq\R^{d_2}$ be a non-empty compact set.
	Then 
	\begin{enumerate}[(i)]
		\item\label{item:cont_max_1} we have for all $x\in\R^{d_1}$ that
		$
		\sup_{y \in K} \norm{f(x,y)}_{\R^{d_3}} <\infty
		$
		and
		\item\label{item:cont_max_2} we have that 
		$
		\R^{d_1}\ni x\mapsto \sup_{y \in K} \norm{f(x,y)}_{\R^{d_3}} \in \R
		$
		is continuous.
	\end{enumerate}
\end{lemma}
\begin{proof}[Proof of  Lemma~\ref{lem:cont_max}]
	Throughout this proof let $g\colon\R^{d_1}\to\R\cup\lb\infty\rb$ satisfy for all $x\in\R^{d_1}$ that
	\begin{equation}\label{eq:cont_max_1_2}
	g(x)=\sup_{y \in K} \norm{f(x,y)}_{\R^{d_3}},
	\end{equation}
	let $x = (x_n)_{n \in \N_0} \colon \N_0 \to\R^{d_1}$ satisfy that
	\begin{equation}\label{eq:cont_max_1_3}
	\limsup_{n \to \infty}\| x_n-x_0 \|_{\R^{d_1}}=0,
	\end{equation}
	and let $k\colon \N\to\N$ be  strictly increasing.
	Note that the assumption that $f$ is continuous ensures that for all $z \in \R^{d_1}$ we have that 
	\begin{equation}
	\label{eq:prop:cont}
	\big( \R^{d_2} \ni y \mapsto f(z, y) \in \R^{d_3}\big) \in C(\R^{d_2}, \R^{d_3}).
	\end{equation}
	\Cref{prop:cont_fct_bnd_on_cmpcts} and the assumption that $K$ is a non-empty compact set hence establish item~\eqref{item:cont_max_1}.
	Next observe that \eqref{eq:prop:cont} and the assumption that $K$ is a non-empty compact set assure that there exists   $y = (y_n)_{n \in \N_0} \colon\N_0\to K$ which satisfies for all $n\in\N$ that 
	\begin{equation}\label{eq:cont_max_2}
	g(x_0)=\norm{f(x_0,y_0)}_{\R^{d_3}}
	\end{equation}
	and
	\begin{equation}\label{eq:cont_max_3}
	g(x_{k(n)})=\| f(x_{k(n)},y_n) \|_{\R^{d_3}}
	\end{equation}
	(see, e.g., Coleman~\cite[Theorem 1.3]{Coleman2012}).
	Furthermore, observe that the assumption that $K$ is a non-empty compact set and the Bolzano-Weierstrass theorem demonstrate that there exist   $\mathbbm{y} \in K$ and strictly increasing  $l\colon \N\to\N$   which satisfy that
	\begin{equation}
	\limsup_{n \to \infty}  \| y_{l(n)} - \mathbbm{y}\|_{\R^{d_2}} = 0.
	\end{equation}
	This and \eqref{eq:cont_max_1_3} imply that $ (x_{k(l(n))}, y_{l(n)} )\in \R^{d_1}\times \R^{d_2},n\in\N,$ is a convergent sequence in $\R^{d_1}\times \R^{d_2}$. The  assumption that $f$ is continuous,  \eqref{eq:cont_max_1_3}, and \eqref{eq:cont_max_3} hence assure that $g(x_{k(l(n))})\in \R, n\in\N,$ is a convergent sequence in $\R$ and 
	\begin{equation}\label{eq:cont_max_3_2}
	\lim_{n\to\infty}g(x_{k(l(n))})=\lim_{n\to\infty}\| f(x_{k(l(n))},y_{l(n)})\|_{\R^{d_3}}=\| f(x_0,\mathbbm{y} )\|_{\R^{d_3}}.
	\end{equation}
	This and \eqref{eq:cont_max_1_2} prove that
	\begin{equation}\label{eq:cont_max_g1}
	\begin{split}
	\lim_{n\to\infty}g(x_{k(l(n))})  = \| f(x_0,\mathbbm{y} )\|_{\R^{d_3}}
	\leq g(x_0). 
	\end{split}
	\end{equation}
	Moreover, note that \eqref{eq:cont_max_2}, \eqref{eq:cont_max_1_3}, the assumption that $f$ is continuous, and \eqref{eq:cont_max_1_2} imply that
	\begin{equation}\label{eq:cont_max_4}
	\begin{split}
	g(x_0) & = \norm{ f(x_0,y_0) }_{\R^{d_3}}\\
	& = \big\| f\big(\lim_{n\to\infty}x_{k(l(n))},y_0\big)  \big\|_{\R^{d_3}}\\
	& =  \big\| \lim_{n\to\infty}f(x_{k(l(n))},y_0)  \big\|_{\R^{d_3}}\\
	& = \lim_{n\to\infty}\| f(x_{k(l(n))},y_0) \|_{\R^{d_3}}\\
	& \leq \lim_{n\to\infty}g(x_{k(l(n))}). 
	\end{split}
	\end{equation}
	Combining this and \eqref{eq:cont_max_g1} assures that 
	\begin{equation}
	\limsup_{n\to\infty} | g(x_{k(l(n))}) - g(x_0)| = 0.
	\end{equation}
	This and, e.g., \cite[Lemma 3.2]{JentzenBochner}
	prove that the sequence $g(x_n)\in\R$, $n\in\N$,  converges to $g(x_0)$. This reveals that  $g$ is continuous at $x_0$. 
	This establishes item~\eqref{item:cont_max_2}. The proof of  Lemma~\ref{lem:cont_max} is thus completed.
\end{proof}

\begin{lemma}\label{lem:bnd_exp_term}
Assume Setting~\ref{sec:setting_2} and let $\psi\in C^2(\R^d,\R)$.
Then 
\begin{enumerate}[(i)]
\item\label{item:lem:lim} we have that 
$
\limsup_{t \to \infty} \|\theta_t^{\xi} - \Xi\|_{\R^d} = 0
$
and 
\item  we have that
\begin{equation}\label{eq:bnd_exp_term_1}
\sup_{n\in\N}\sup_{\lambda\in[0,1]}\Big( n^{1-2\epsilon} \exp(-L\mathfrak{t}_{n-1})  \| \psi'(\lambda\theta_{\mathfrak{t}_{n-1}}^{\xi}+(1-\lambda)\Xi) \|_{L(\R^d,\R)}\Big) < \infty.
\end{equation}
\end{enumerate}
\end{lemma}
\begin{proof}[Proof of Lemma~\ref{lem:bnd_exp_term}]
First, note that  \eqref{eq:setting_2_3_2} assures that
\begin{equation}
\{x \in \R^d \colon g(x) =0 \} = \{\Xi\}.
\end{equation}
 This and item~\eqref{item:motionless_pt_f_3} in Lemma~\ref{lem:motionless_pt_f} establish item~\eqref{item:lem:lim}. In the next step observe that item~\eqref{item:bounds_sum_1} in Lemma~\ref{lem:bounds_sum} assures that for all $k\in \lb 2,3,\dots\rb$ we have that
\begin{equation}
\mathfrak{t}_{k-1} = \textstyle\sum\limits_{n=1}^{k-1} \displaystyle \frac{\eta}{n^{1-\epsilon}} \geq \frac{\eta}{\varepsilon} \big( k^{\varepsilon} -1\big).
\end{equation}
This implies that for all $k\in \lb 2,3,\dots\rb$ we have that
\begin{equation}\label{eq:bnd_exp_term_3}
\begin{split}
k^{1-2\epsilon}e^{-L\mathfrak{t}_{k-1}}
 &\leq k^{1-2\epsilon}e^{-\frac{L \eta}{\epsilon}(k^{\epsilon}-1)}
 = k^{1-2\epsilon}e^{-\frac{L \eta}{\epsilon}k^{\epsilon}}e^{\frac{L \eta}{\epsilon}} = (k^{\epsilon})^{\frac{1-2\epsilon}{\epsilon}}e^{-\frac{L \eta}{\epsilon}k^{\epsilon}}e^{\frac{L \eta}{\epsilon}}.
\end{split}
\end{equation}
This reveals that
\begin{equation}\label{eq:bnd_exp_term_4}
\limsup_{k\to\infty} k^{1-2\epsilon}e^{-L\mathfrak{t}_{k-1}}=0.
\end{equation}
Moreover, note that for all $t\in [0,\infty)$ we have that
\begin{equation}\label{eq:bnd_exp_term_5_02}
\sup_{\lambda\in[0,1]} \| \psi'(\lambda\theta_{t}^{\xi}+(1-\lambda)\Xi) \|_{L(\R^d,\R)}
 = \sup_{\lambda\in[0,1]} \| \nabla\psi(\lambda\theta_{t}^{\xi}+(1-\lambda)\Xi) \|_{\R^d}.
\end{equation}
In the next step we combine item~\eqref{item:lem:lim} and the fact that $\liminf_{k\to\infty}\mathfrak{t}_{k}=\infty$ to obtain that 
\begin{equation}\label{eq:bnd_exp_term_6}
\limsup_{k\to\infty}\|\theta_{\mathfrak{t}_{k}}^{\xi} - \Xi\|_{\R^d} =0.
\end{equation}
Furthermore, observe that  Lemma~\ref{lem:cont_max} and the fact that $\nabla\psi$ is continuous assure that for all $x \in \R^d$  we have that
\begin{equation}
\sup_{\lambda\in[0,1]} \| \nabla\psi(\lambda x+(1-\lambda)\Xi) \|_{\R^d} < \infty
\end{equation}
and
\begin{equation}
\Big(\R^d\ni y \mapsto \sup_{\lambda\in[0,1]} \| \nabla\psi(\lambda y+(1-\lambda)\Xi) \|_{\R^d}\in \R \Big)\in C(\R^d,\R).
\end{equation}
This and \eqref{eq:bnd_exp_term_6} prove that the sequence
\begin{equation}
\Big(\sup\nolimits_{\lambda\in[0,1]} \| \nabla\psi(\lambda\theta_{\mathfrak{t}_{k-1}}^{\xi}+(1-\lambda)\Xi) \|_{\R^d}\Big)_{k\in\N}
\end{equation}
is convergent in $\R$ and
\begin{equation}\label{eq:bnd_exp_term_7}
\begin{split}
 \limsup_{k\to\infty}\sup_{\lambda\in[0,1]} \| \nabla\psi(\lambda\theta_{\mathfrak{t}_{k-1}}^{\xi}+(1-\lambda)\Xi) \|_{\R^d} &= \sup_{\lambda\in[0,1]} \| \nabla\psi(\lambda\Xi+(1-\lambda)\Xi) \|_{\R^d} \\
 &= \|\nabla\psi(\Xi)\|_{\R^d}.
\end{split}
\end{equation}
This, \eqref{eq:bnd_exp_term_5_02}, and \eqref{eq:bnd_exp_term_4} establish \eqref{eq:bnd_exp_term_1}. The proof of Lemma~\ref{lem:bnd_exp_term} is thus completed.
\end{proof}

\section{A priori estimates for SAAs in the case of polynomially decaying learning rates}
\label{subsec:apriori:poly}
\sectionmark{}

\begin{lemma}\label{lem:Theta_adapted}
Assume Setting~\ref{sec:setting_2} and let $\mathbb{G}_t \subseteq \mathcal{F}$, $t\in [0,\infty)$, satisfy for all $t\in(0,\infty)$ that  
\begin{equation}\label{eq:Theta_adapted_1}
\mathbb{G}_0= \{\emptyset, \Omega \} \qquad \text{and} \qquad \mathbb{G}_t=\sigma_{\Omega}( Z_{m+1, n} \colon (m, n) \in \N_0 \times \N,\, n \leq \mathfrak{M}_m,\, \mathfrak{t}_m <t ).
\end{equation}
Then we have for all $t\in [0,\infty)$ that $\Theta_t$ is $\mathbb{G}_{t}/\mathcal{B}(\R^d)$-measurable.
\end{lemma}
\begin{proof}[Proof of Lemma~\ref{lem:Theta_adapted}]
First, observe that \eqref{eq:setting_2_6} ensures that for all $m \in \N_0$, $t \in (\mathfrak{t}_m, \mathfrak{t}_{m+1})$ we have that
\begin{equation}\label{eq:Theta_adapted_2}
\begin{split}
\Theta_t
&=\xi + \bigg[\textstyle\sum\limits_{n=0}^{m-1} \displaystyle \frac{(\mathfrak{t}_{n+1}-\mathfrak{t}_{n})}{\mathfrak{M}_n}\textstyle\sum\limits_{j=1}^{\mathfrak{M}_n} \displaystyle G(\Theta_{\mathfrak{t}_n},Z_{n+1, j})\bigg]\\
& \quad +\frac{(t-\mathfrak{t}_m)}{\mathfrak{M}_m} \textstyle\sum\limits_{n=1}^{\mathfrak{M}_m} \displaystyle G(\Theta_{\mathfrak{t}_m},Z_{m+1, n}).
\end{split}
\end{equation}
Next we claim that for all $m\in\N_0$, $t\in (\mathfrak{t}_m, \mathfrak{t}_{m+1}] $ we have that 
$\Theta_t$ is $\mathbb{G}_{t}/\mathcal{B}(\R^d)$-measurable. We prove this by induction on $m \in \N_0$. For the base case $m=0$ note that \eqref{eq:Theta_adapted_1}  assures that for all $t\in (\mathfrak{t}_0,\mathfrak{t}_1]$ we have that
\begin{equation}\label{eq:Theta_adapted_2_02}
\mathbb{G}_t=\sigma_{\Omega}( Z_{1,n} \colon  n \in \{1, 2, \ldots, \mathfrak{M}_0\}).
\end{equation}
Furthermore, observe that \eqref{eq:Theta_adapted_2} and the assumption that $\Theta \colon [0, \infty) \times \Omega \to \R^d$ is a stochastic process w.c.s.p.\  prove that for all $t\in (\mathfrak{t}_0,\mathfrak{t}_1]$ we have that
\begin{equation}\label{eq:Theta_adapted_2_03}
\Theta_t
=\xi 
+\frac{t}{\mathfrak{M}_0}\textstyle\sum\limits_{n=1}^{\mathfrak{M}_0} \displaystyle G(\xi,Z_{1, n}).
\end{equation}
This and \eqref{eq:Theta_adapted_2_02} prove that for all $t\in (\mathfrak{t}_0,\mathfrak{t}_1]$ we have that $\Theta_t$ is $\mathbb{G}_{t}/\mathcal{B}(\R^d)$-measurable. For the induction step $\N_0\ni m\to m+1\in\N_0$ observe that \eqref{eq:Theta_adapted_1}  assures that for all $t\in (\mathfrak{t}_{m+1}, \mathfrak{t}_{m+2}] $ we have that
\begin{equation}\label{eq:Theta_adapted_2_04}
\mathbb{G}_t= \sigma_{\Omega}( Z_{k+1, n} \colon (k, n) \in \N_0 \times \N,\, n \leq \mathfrak{M}_k,\, k \leq m+1 ).
\end{equation}
Moreover, note that \eqref{eq:Theta_adapted_2} and the assumption that $\Theta \colon [0, \infty) \times \Omega \to \R^d$ is a stochastic process w.c.s.p.\ demonstrate that for all $t\in (\mathfrak{t}_{m+1}, \mathfrak{t}_{m+2}] $ we have that
\begin{equation}\label{eq:Theta_adapted_2_05}
\begin{split}
\Theta_t
&= \xi + \bigg[\textstyle\sum\limits_{n=0}^{m} \displaystyle \frac{(\mathfrak{t}_{n+1}-\mathfrak{t}_{n})}{\mathfrak{M}_n} \textstyle\sum\limits_{j=1}^{\mathfrak{M}_n} \displaystyle G(\Theta_{\mathfrak{t}_n},Z_{n+1, j})\bigg]\\
& \quad +\frac{(t-\mathfrak{t}_{m+1})}{\mathfrak{M}_{m+1}} \textstyle\sum\limits_{n=1}^{\mathfrak{M}_{m+1}} \displaystyle G(\Theta_{\mathfrak{t}_{m+1}},Z_{m+2, n}).
\end{split}
\end{equation}
The induction hypothesis and  \eqref{eq:Theta_adapted_2_04} hence imply that for all $t\in (\mathfrak{t}_{m+1}, \mathfrak{t}_{m+2}] $ we have that $\Theta_t$ is $\mathbb{G}_{t}/\mathcal{B}(\R^d)$-measurable. 
This finishes the proof of the induction step. 
Induction therefore establishes that for all $m\in\N_0$, $t\in (\mathfrak{t}_m, \mathfrak{t}_{m+1}] $ we have that 
$\Theta_t$ is $\mathbb{G}_{t}/\mathcal{B}(\R^d)$-measurable. Combining this with the fact that $\Theta_0$ is $\mathbb{G}_{0}/\mathcal{B}(\R^d)$-measurable ensures that for all $t\in [0,\infty)$ we have that $\Theta_t$ is $\mathbb{G}_{t}/\mathcal{B}(\R^d)$-measurable. The proof of Lemma~\ref{lem:Theta_adapted} is thus completed.
\end{proof}

\begin{lemma}\label{lem:diff_martingale}
Assume Setting~\ref{sec:setting_2},  let $\mathbb{F}_n \subseteq \mathcal{F}$, $n\in\N_0$, be the sigma-algebras which satisfy for all $n\in\N$ that 
\begin{equation}\label{eq:diff_martingale_1}
\mathbb{F}_0= \{\emptyset, \Omega \} \qquad \text{and} \qquad \mathbb{F}_n= \sigma_{\Omega}( Z_{m+1, j} \colon (m, j) \in \N_0 \times \N,\, j \leq \mathfrak{M}_m,\, m \leq n-1 ),
\end{equation}
and  assume for all $n \in \N_0$, $j \in \{1, 2, \ldots, \mathfrak{M}_n\}$  that
\begin{equation}
\label{eq:regularity:F}
\E\big[ \|G(\Theta_{\mathfrak{t}_{n}},Z_{n+1, j})\|_{\R^d}\big] < \infty. 
\end{equation}
Then we have for all $n\in\N_0$, $A \in \mathbb{F}_{n}$  that
\begin{equation}\label{eq:diff_martingale_2}
\E\Big[\tfrac{1}{\mathfrak{M}_n} \textstyle\sum_{j=1}^{\mathfrak{M}_n}(G(\Theta_{\mathfrak{t}_{n}},Z_{n+1,j})-g(\Theta_{\mathfrak{t}_{n}})) \mathbbm{1}_A \Big]=0.
\end{equation}
\end{lemma}
\begin{proof}[Proof of Lemma~\ref{lem:diff_martingale}]
Throughout this proof let $D\colon \N_0\times\Omega\to\R^d$ be the  stochastic process which satisfies for all $n\in \N_0$ that
\begin{equation}\label{eq:diff_martingale_3}
\begin{split}
D_n &=\frac{1}{\mathfrak{M}_n} \textstyle\sum\limits_{j=1}^{\mathfrak{M}_n} \displaystyle (G(\Theta_{\mathfrak{t}_{n}},Z_{n+1,j})-g(\Theta_{\mathfrak{t}_{n}}))\\
\end{split}
\end{equation}
and let $Y\colon\N_0\times\Omega\to\R^d$ be the  stochastic process which satisfies for all $n\in \N_0$ that
\begin{equation}\label{eq:diff_martingale_4}
Y_n=\Theta_{\mathfrak{t}_n}.
\end{equation} 
Observe that the assumption that $Z_{m, j}$, $(m, j)\in \N^2$, are i.i.d.\ random variables and \eqref{eq:diff_martingale_1} ensure that for all $n \in \N_0$, $j \in \{1, 2, \ldots, \mathfrak{M}_n\}$ we have that 
\begin{equation}
\label{eq:Zj:independent}
Z_{n+1, j} \quad \text{is independent of} \quad \mathbb{F}_n.
\end{equation}
Moreover, note that Lemma~\ref{lem:Theta_adapted} proves that $Y$ is an $(\mathbb{F}_n)_{n\in\N_0}/\mathcal{B}(\R^d)$-adapted stochastic process. This, \eqref{eq:Zj:independent}, the fact that $\forall \, x \in \R^d, (n, j)\in \N^2 \colon \E[ \|G(x, Z_{n, j})\|_{\R^d}] < \infty$, the fact that $\forall \, x \in \R^d, (n, j)\in \N^2 \colon g(x) = \E[G(x, Z_{n, j})]$, \eqref{eq:regularity:F}, and, e.g.,  \cite[Corollary~2.9]{Wurstemberger2018} establish that for all $n \in \N_0$, $j \in \{1, 2, \ldots, \mathfrak{M}_n\}$, $ A \in \mathbb{F}_{n}$  we have that  
\begin{equation}
\E[G(Y_{n},Z_{n+1, j}) \mathbbm{1}_A]=\E[g(Y_{n}) \mathbbm{1}_A].
\end{equation}
Hence, we obtain  that  for all $n\in\N_0$, $ A \in \mathbb{F}_{n}$  it holds that
\begin{equation}
\begin{split}
\E[D_n \mathbbm{1}_A]
 &= \frac{1}{\mathfrak{M}_{n}} \textstyle\sum\limits_{j=1}^{\mathfrak{M}_{n}}\E\big[ \big(G(\Theta_{\mathfrak{t}_{n}},Z_{n+1,j})-g(\Theta_{\mathfrak{t}_{n}}) \big) \mathbbm{1}_A\big]\\
& = \frac{1}{\mathfrak{M}_{n}} \textstyle\sum\limits_{j=1}^{\mathfrak{M}_{n}}\big(\E[G(\Theta_{\mathfrak{t}_{n}},Z_{n+1, j}) \mathbbm{1}_A]-\E[g(\Theta_{\mathfrak{t}_{n}}) \mathbbm{1}_A]\big)\\
& = \frac{1}{\mathfrak{M}_{n}} \textstyle\sum\limits_{j=1}^{\mathfrak{M}_{n}}\big(\E[G(Y_{n},Z_{n+1, j}) \mathbbm{1}_A]-\E[g(Y_{n}) \mathbbm{1}_A]\big)  =0.
\end{split}
\end{equation}
The proof of Lemma~\ref{lem:diff_martingale} is thus completed.
\end{proof}

\begin{prop}\label{prop:PvW_L^2}
Let $d\in\N$, $c,\kappa\in (0,\infty)$, $\Xi\in\R^d$, let  $\gamma = (\gamma_n)_{n \in \N_0} \colon \N_0 \to (0, \infty)$ be a function, let $g\colon\R^d\to\R^d$ be  $\mathcal{B}(\R^d)/\mathcal{B}(\R^d)$-measurable,  let $(\Omega,\mathcal{F},\mathbb{P}, (\mathbb{F}_n)_{n\in\N_0})$ be a filtered probability space, let $\Theta\colon \N_0\times\Omega\to\R^d$ be an $(\mathbb{F}_n)_{n\in\N_0}/\mathcal{B}(\R^d)$-adapted stochastic process, and assume for all $x\in\R^d$, $n\in\N$, $ A \in \mathbb{F}_{n-1}$ that
\begin{equation}
\E\big[ \|\Theta_0\|_{\R^d}^2+\|\Theta_n-(\Theta_{n-1}+\gamma_ng(\Theta_{n-1}))\|_{\R^d}^2 \big]<\infty,
\end{equation}
\begin{equation}\label{eq:PvW_L^2_1}
\E\big[ \big(\Theta_n-(\Theta_{n-1}+\gamma_ng(\Theta_{n-1})) \big) \mathbbm{1}_A \big]=0,
\end{equation}
\begin{equation}\label{eq:PvW_L^2_2}
\langle x-\Xi,g(x)\rangle_{\R^d} \leq -c \max\big\{\|x-\Xi\|_{\R^d}^2, \|g(x)\|_{\R^d}^2 \big\},
\end{equation}
\begin{equation}\label{eq:PvW_L^2_4}
\E\big[\|\Theta_n-(\Theta_{n-1}+\gamma_ng(\Theta_{n-1}))\|_{\R^d}^2\big]  \leq (\gamma_n)^2\kappa\big(1+\E\big[\|\Theta_{n-1}-\Xi\|_{\R^d}^2\big]\big),
\end{equation}
and
\begin{equation}\label{eq:PvW_L^2_5}
\limsup_{k \to \infty} \gamma_k = 0 < \liminf_{k\to\infty} \left[ \tfrac{\gamma_k-\gamma_{k-1}}{(\gamma_k)^2} + \tfrac{c\gamma_{k-1}}{2\gamma_k} \right].
\end{equation}
Then there exists $C\in(0,\infty)$ such that for all $n\in\N_0$ we have that
\begin{equation}
\E\big[\|\Theta_n-\Xi\|_{\R^d}^2\big]\leq C\gamma_n
\end{equation}
and 
\begin{equation}\label{eq:PvW_L^2_8}
\sup_{m\in\N_0}\E\big[\|\Theta_m-\Xi\|_{\R^d}^2\big]<\infty.
\end{equation}
\end{prop}
\begin{proof}[Proof of Proposition~\ref{prop:PvW_L^2}]
First, note that   \cite[Corollary~3.5]{Wurstemberger2018} (with $d=d$, $(\gamma_n)_{n \in \N_0} \allowbreak = (\gamma_n)_{n \in \N_0}$, $g = g$, $(\Omega,\mathcal{F},\mathbb{P}, (\mathbb{F}_n)_{n\in\N_0}) = (\Omega,\mathcal{F},\mathbb{P}, (\mathbb{F}_n)_{n\in\N_0})$,  $D_k = \allowbreak  \nicefrac{[\Theta_k-(\Theta_{k-1}+\gamma_k g(\Theta_{k-1}))]}{\gamma_k}$, $\Theta = \Theta$,
$\langle \cdot, \cdot \rangle = \langle \cdot, \cdot \rangle_{\R^d}$,
 $c = c$, $\kappa = \kappa$, $\vartheta = \Xi$   for $k \in \N$ in the notation of   \cite[Corollary~3.5]{Wurstemberger2018}) implies that  there exists $C\in(0,\infty)$ such that for all $n\in\N_0$ we have that
\begin{equation}\label{eq:PvW_L^2_7}
\E\big[\|\Theta_n-\Xi\|_{\R^d}^2\big]\leq C\gamma_n.
\end{equation}
The assumption that $\limsup_{n \to \infty} \gamma_n = 0$ hence ensures that
\begin{equation}
\sup_{m\in\N_0}\E\big[\|\Theta_m-\Xi\|_{\R^d}^2\big]<\infty.
\end{equation}
Combining this with \eqref{eq:PvW_L^2_7} completes the proof of
 Proposition~\ref{prop:PvW_L^2}.
\end{proof}

\begin{lemma}
\label{lem:liminf}
It holds for all $\varepsilon \in (-\infty,1)$ that
\begin{equation}
\label{eq:lem:liminf}
\limsup_{n\to\infty} |n^{\varepsilon} - (n-1)^{\varepsilon}| = 0.
\end{equation}
\end{lemma}
\begin{proof}[Proof of Lemma~\ref{lem:liminf}]
First, observe that for all $\varepsilon \in (- \infty, 0)$ we have that
\begin{equation}
\label{eq:liminf:neg}
\limsup_{n\to\infty} n^\varepsilon = 0.
\end{equation}
This reveals that for all $\varepsilon \in (-\infty,0)$  it holds that
\begin{equation}
\label{eq:liminf:1}
\limsup_{n\to\infty} |n^{\varepsilon} - (n-1)^{\varepsilon}| = 0.
\end{equation}
Next note that for all  $\varepsilon \in [0,1)$, $n \in \N \cap [2, \infty)$ we have that
\begin{equation}
\begin{split}
0 & \leq |n^{\varepsilon} - (n-1)^{\varepsilon}| = n^{\varepsilon} - (n-1)^{\varepsilon} = \frac{n^{1-\varepsilon}(n^{\varepsilon} - (n-1)^{\varepsilon})}{n^{1-\varepsilon}}\\
& =  \frac{n - n^{1-\varepsilon} (n-1)^{\varepsilon}}{n^{1-\varepsilon}} \leq \frac{n-(n-1)}{n^{1-\varepsilon}} = \frac{1}{n^{1-\varepsilon}} = n^{\varepsilon -1}.
\end{split}
\end{equation}
This and \eqref{eq:liminf:neg} imply that for all $\varepsilon \in [0,1)$ we have that 
\begin{equation}
\limsup_{n\to\infty} |n^{\varepsilon} - (n-1)^{\varepsilon}| = 0.
\end{equation}
Combining this and \eqref{eq:liminf:1} establishes \eqref{eq:lem:liminf}.  The proof of Lemma~\ref{lem:liminf} is thus completed.
\end{proof}

\begin{cor}\label{cor:PvW_bnd_L^2_Theta-Xi}
Assume Setting~\ref{sec:setting_2},  let $\kappa\in(0,\infty)$, and assume for all $n\in\N$, $m \in \N_0$, $j \in \{1, 2, \ldots, \mathfrak{M}_m\}$
 that
\begin{equation}\label{eq:PvW_bnd_L^2_Theta-Xi_1}
\E\big[\|\Theta_{\mathfrak{t}_{n}}-(\Theta_{\mathfrak{t}_{n-1}}+(\mathfrak{t}_{n}-\mathfrak{t}_{n-1})g(\Theta_{\mathfrak{t}_{n-1}}))\|_{\R^d}^2 + \|G(\Theta_{\mathfrak{t}_{m}},Z_{m+1, j})\|_{\R^d}  \big]<\infty
\end{equation}
and
\begin{equation}
\label{eq:PvW_bnd_L^2_Theta-Xi_2}
\begin{split}
&\E\big[\|\Theta_{\mathfrak{t}_n}-(\Theta_{\mathfrak{t}_{n-1}}+(\mathfrak{t}_{n}-\mathfrak{t}_{n-1})g(\Theta_{\mathfrak{t}_{n-1}}))\|_{\R^d}^2\big]  \leq (\mathfrak{t}_{n}-\mathfrak{t}_{n-1})^2\kappa\big(1+\E\big[\|\Theta_{\mathfrak{t}_{n-1}}-\Xi\|_{\R^d}^2\big] \big).
\end{split}
\end{equation}
Then we have that
\begin{equation}\label{eq:PvW_bnd_L^2_Theta-Xi_4}
\sup_{n\in\N_0}\E\big[\|\Theta_{\mathfrak{t}_n}-\Xi\|_{\R^d}^2\big]<\infty.
\end{equation}
\end{cor}
\begin{proof}[Proof of Corollary~\ref{cor:PvW_bnd_L^2_Theta-Xi}]
Throughout this proof let $\alpha = (\alpha_n)_{n \in \N_0} \colon \N_0 \to (0, \infty)$ satisfy for all $n \in \N$ that
\begin{equation}
\label{eq:gamma}
\alpha_n= \mathfrak{t}_n - \mathfrak{t}_{n-1} = \frac{\eta}{n^{1-\varepsilon}},
\end{equation} 
let $\mathbb{F}_n \subseteq \mathcal{F}$, $n \in \N_0$, be the sigma-algebras which satisfy for all $n\in\N$ that 
\begin{equation}\label{eq:PvW_bnd_L^2_Theta-Xi_5}
\mathbb{F}_0= \{\emptyset, \Omega \} \qquad \text{and} \qquad \mathbb{F}_n= \sigma_{\Omega}( Z_{m+1, j} \colon (m, j) \in \N_0 \times \N,\, j \leq \mathfrak{M}_m,\, m \leq n-1 ),
\end{equation}
and let $Y\colon\N_0\times\Omega\to\R^d$ be the stochastic process which satisfies for all $n\in \N_0$ that
\begin{equation}\label{eq:PvW_bnd_L^2_Theta-Xi_7}
Y_n=\Theta_{\mathfrak{t}_n}.
\end{equation} 
Note that Lemma~\ref{lem:Theta_adapted} ensures that $Y$ is an $(\mathbb{F}_n)_{n\in\N_0}/\mathcal{B}(\R^d)$-adapted stochastic process. Next observe that \eqref{eq:PvW_bnd_L^2_Theta-Xi_1}, \eqref{eq:gamma}, and  \eqref{eq:PvW_bnd_L^2_Theta-Xi_7} imply that  for all $n\in\N$ we have that
\begin{equation}\label{eq:PvW_bnd_L^2_Theta-Xi_8}
\begin{split}
&\E\big[\|Y_n-(Y_{n-1}+\alpha_ng(Y_{n-1}))\|^2_{\R^d}\big]  \\
&= \E\big[\|\Theta_{\mathfrak{t}_{n}}-(\Theta_{\mathfrak{t}_{n-1}}+(\mathfrak{t}_{n}-\mathfrak{t}_{n-1})g(\Theta_{\mathfrak{t}_{n-1}}))\|_{\R^d}^2\big] <\infty.
\end{split}
\end{equation}
Moreover, note that combining Lemma~\ref{lem:diff_martingale} and \eqref{eq:PvW_bnd_L^2_Theta-Xi_1} assures that for all $n\in\N$, $A \in \mathbb{F}_{n-1}$ we have that
\begin{equation}\label{eq:PvW_bnd_L^2_Theta-Xi_9}
\begin{split}
&\E\big[ \big(Y_n-(Y_{n-1}+\alpha_ng(Y_{n-1})) \big) \mathbbm{1}_A\big]\\
&= \alpha_n \, \E\Big[\tfrac{1}{\mathfrak{M}_{n-1}}\textstyle\sum_{j=1}^{\mathfrak{M}_{n-1}} \big( G(\Theta_{\mathfrak{t}_{n-1}},Z_{n, j}) - g(\Theta_{\mathfrak{t}_{n-1}}) \big) \mathbbm{1}_A\Big]=0.
\end{split}
\end{equation}
Next observe that  \eqref{eq:setting_2_3_2} ensures that $g(\Xi) =0$. This and  \eqref{eq:setting_2_3_2} establish that for all  $x \in \R^d$ we have that
\begin{equation}
\label{eq:Xi:c}
\begin{split}
\langle x-\Xi,g(x)\rangle_{\R^d} & = \langle x-\Xi,g(x) - g(\Xi)\rangle_{\R^d}  \leq -L \|x-\Xi\|_{\R^d}^2.
\end{split}
\end{equation}
Combining this with  \eqref{eq:setting_2_3_2} implies that for all $x \in \R^d$ we have that
\begin{equation}
\label{eq:with:c}
\langle x-\Xi,g(x)\rangle_{\R^d} \leq -L \max\big\{\|x-\Xi\|_{\R^d}^2, \|g(x)\|_{\R^d}^2 \big\}.
\end{equation}
Next note that  \eqref{eq:gamma} assures that for all $n \in \N \cap [2, \infty)$ we have that
\begin{equation}
\begin{split}
\frac{\alpha_n-\alpha_{n-1}}{(\alpha_n)^2} + \frac{L\alpha_{n-1}}{2\alpha_n}  &= \frac{\frac{\eta}{n^{1-\epsilon}}-\frac{\eta}{(n-1)^{1-\epsilon}}}{\big(\frac{\eta}{n^{1-\epsilon}}\big)^2} + \frac{\frac{L\eta}{2(n-1)^{1-\epsilon}}}{\frac{\eta}{n^{1-\epsilon}}} \\
& = \frac{n^{1-\varepsilon} \big[(n-1)^{1-\varepsilon} - n^{1-\varepsilon}\big]}{\eta (n-1)^{1-\varepsilon}}  +\frac{L n^{1-\varepsilon}}{2(n-1)^{1-\varepsilon}} \\
& = \tfrac{1}{\eta}\! \left[1+ \tfrac{1}{n-1}\right]^{1-\varepsilon} \big[(n-1)^{1-\varepsilon} - n^{1-\varepsilon}\big] + \tfrac{L}{2} \!\left[1+ \tfrac{1}{n-1}\right]^{1-\varepsilon}.
\end{split}
\end{equation}
Lemma~\ref{lem:liminf} (with $\varepsilon = (1-\varepsilon) \in (-\infty,1)$ in the notation of  Lemma~\ref{lem:liminf}) hence ensures that 
\begin{equation}
\begin{split}
&\liminf_{n\to\infty} \left[ \frac{\alpha_n-\alpha_{n-1}}{(\alpha_n)^2} + \frac{L\alpha_{n-1}}{2\alpha_n} \right]  = \frac{L}{2} > 0 = \limsup_{n \to \infty} \alpha_n.
\end{split}
\end{equation}
Combining this, \eqref{eq:PvW_bnd_L^2_Theta-Xi_8}, \eqref{eq:PvW_bnd_L^2_Theta-Xi_9}, \eqref{eq:with:c}, and Proposition~\ref{prop:PvW_L^2}  (with $d=d$, $\gamma_n=\alpha_n$, $c=L$, $\kappa = \kappa$, $\Xi=\Xi$, $g=g$,  $(\Omega,\mathcal{F},\mathbb{P}, (\mathbb{F}_k)_{k\in\N_0})=(\Omega,\mathcal{F},\mathbb{P}, (\mathbb{F}_k)_{k\in\N_0})$,   $\Theta_n = Y_n$ for $n \in \N_0$ in the notation of Proposition~\ref{prop:PvW_L^2}) establishes \eqref{eq:PvW_bnd_L^2_Theta-Xi_4}. The proof of Corollary~\ref{cor:PvW_bnd_L^2_Theta-Xi} is thus completed.
\end{proof}

\section{A posteriori estimates for conditional variances associated to SAAs}
\label{subsec:aposteriori}
\sectionmark{}

\begin{lemma}\label{lem:(1)+(2)=>(*)}
	Let $d\in\N$, $\Xi\in \R^d$, $M,L\in (0,\infty)$ and let $f\colon \R^d\to\R^d$ satisfy for all $x\in\R^d$ that
	\begin{equation}\label{eq:(1)+(2)=>(*)_1}
	\langle x-\Xi,f(x)\rangle_{\R^d}\leq -\max\lb L\|x-\Xi\|_{\R^d}^2, M\|f(x)\|_{\R^d}^2\rb.
	\end{equation}
	Then we have for all $x\in\R^d$ that
	\begin{equation}\label{eq:(1)+(2)=>(*)_3}
	\begin{split}
	L\|x\|_{\R^d}-L\|\Xi\|_{\R^d}&\leq L\|x-\Xi\|_{\R^d} \leq\|f(x)\|_{\R^d} \\
	&\leq \tfrac{1}{M}\|x-\Xi\|_{\R^d}\leq \tfrac{\max\lb 1,\norm{\Xi}_{\R^d}\rb}{M}(1+\|x\|_{\R^d}).
	\end{split}
	\end{equation}
\end{lemma}
\begin{proof}[Proof of Lemma~\ref{lem:(1)+(2)=>(*)}]
	First, note that \eqref{eq:(1)+(2)=>(*)_1} assures that  for all $x\in\R^d\backslash \lb\Xi\rb$ we have that
	\begin{equation}\label{eq:(1)+(2)=>(*)_4}
	f(x)\neq 0
	\end{equation}
	and
	\begin{equation}\label{eq:lem:xi:zero}
	f(\Xi)=0.
	\end{equation}
	Furthermore, observe that \eqref{eq:(1)+(2)=>(*)_1} and the Cauchy-Schwartz inequality imply that for all $x\in\R^d$ we have that
	\begin{equation}\label{eq:(1)+(2)=>(*)_5}
	\begin{split}
	\|f(x)\|_{\R^d}^2 & \leq -\tfrac{1}{M}\langle x-\Xi,f(x)\rangle_{\R^d}\leq \tfrac{1}{M}\|x-\Xi\|_{\R^d}\|f(x)\|_{\R^d}.
	\end{split}
	\end{equation}
	This, \eqref{eq:(1)+(2)=>(*)_4}, \eqref{eq:lem:xi:zero}, and the triangle inequality ensure that for all $x\in\R^d$ we have that
	\begin{equation}\label{eq:(1)+(2)=>(*)_6}
	\begin{split}
	\|f(x)\|_{\R^d} &\leq \tfrac{1}{M}\|x-\Xi\|_{\R^d}\leq \tfrac{1}{M}(\|x\|_{\R^d} +\|\Xi\|_{\R^d})
	\leq \tfrac{\max\lb 1,\norm{\Xi}_{\R^d}\rb}{M}(1+\|x\|_{\R^d}).
	\end{split}
	\end{equation}
	Moreover, note that \eqref{eq:(1)+(2)=>(*)_1} and the Cauchy-Schwartz inequality demonstrate that for all $x\in\R^d$ we have that
	\begin{equation}\label{eq:(1)+(2)=>(*)_7}
	\begin{split}
	\|x-\Xi\|_{\R^d}^2 & \leq -\tfrac{1}{L}\langle x-\Xi,f(x)\rangle_{\R^d}
	 \leq \tfrac{1}{L}\| x-\Xi\|_{\R^d}\|f(x)\|_{\R^d}.
	\end{split}
	\end{equation}
	This reveals that for all $x\in\R^d$  it holds that 
	\begin{equation}\label{eq:(1)+(2)=>(*)_8}
	L\|x-\Xi\|_{\R^d} \leq \|f(x)\|_{\R^d}.
	\end{equation}
	This, \eqref{eq:(1)+(2)=>(*)_6}, and the triangle inequality establish \eqref{eq:(1)+(2)=>(*)_3}. The proof of Lemma~\ref{lem:(1)+(2)=>(*)} is thus completed.
\end{proof}

\begin{lemma}\label{lem:iterated_error}
Assume Setting~\ref{sec:setting_2} and assume that
\begin{equation}\label{eq:iterated_error_0}
\sup_{x\in\R^d}\left(\frac{\E\big[\|G(x,Z_{1,1})\|_{\R^d}^2 \big]}{\big[1+\|x\|_{\R^d}\big]^2}\right)<\infty.
\end{equation}
Then there exists  $\kappa\in (0,\infty)$ such that for all $n\in\N$ we have that
\begin{equation}\label{eq:iterated_error_3}
\begin{split}
& \E\big[\|\Theta_{\mathfrak{t}_n}-(\Theta_{\mathfrak{t}_{n-1}}+(\mathfrak{t}_n - \mathfrak{t}_{n-1})g(\Theta_{\mathfrak{t}_{n-1}}))\|_{\R^d}^2\big] \\
& \leq (\mathfrak{t}_{n}-\mathfrak{t}_{n-1})^2\kappa(1+\E\big[\|\Theta_{\mathfrak{t}_{n-1}}-\Xi\|_{\R^d}^2\big]) < \infty.
\end{split}
\end{equation}
\end{lemma}
\begin{proof}[Proof of Lemma~\ref{lem:iterated_error}]
First, note that \eqref{eq:setting_2_3_2} ensures that $g(\Xi) =0$. This and \eqref{eq:setting_2_3_2} establish that for all $x\in\R^d$ we have that
\begin{equation}
\langle x-\Xi,g(x)\rangle_{\R^d}\leq -\max\lb L\|x-\Xi\|_{\R^d}^2, L\|g(x)\|_{\R^d}^2\rb.
\end{equation}
Lemma~\ref{lem:(1)+(2)=>(*)} hence assures that  for all $x\in\R^d$ we have that
\begin{equation}
\label{eq:f:bound}
\|g(x)\|_{\R^d} \leq \tfrac{\max\lb 1,\norm{\Xi}_{\R^d}\rb}{L}(1+\|x\|_{\R^d}).
\end{equation}
Next observe that \eqref{eq:iterated_error_0} and the fact that $ \forall \, a,b \in \R \colon (a+b)^2 \leq 2|a|^2 + 2|b|^2$ imply that
\begin{equation}
\sup_{x\in\R^d}\left(\frac{\E\big[\|G(x,Z_{1,1})\|_{\R^d}^2 \big]}{1+\|x\|_{\R^d}^2}\right)<\infty.
\end{equation}
This  and \eqref{eq:f:bound} demonstrate that there exists  $c\in (0,\infty)$ which satisfies for all $x\in\R^d$  that
\begin{equation}\label{eq:iterated_error_5}
\|g(x)\|_{\R^d}\leq c(1+\|x\|_{\R^d})\quad\text{and}\quad \E\big[\|G(x,Z_{1,1})\|_{\R^d}^2 \big]\leq c(1+\|x\|_{\R^d}^2).
\end{equation}
Moreover, note that \eqref{eq:setting_2_6} and the fact that $\forall \, x,y \in \R^d \colon \|x+y \|_{\R^d}^2 \leq 2 \|x\|_{\R^d}^2 + 2\|y\|_{\R^d}^2$ ensure that for all $n \in \N$ we have that
\begin{align}
\label{eq:iterated_error_6}
\begin{split}
&\E\big[\|\Theta_{\mathfrak{t}_n}-(\Theta_{\mathfrak{t}_{n-1}}+(\mathfrak{t}_n - \mathfrak{t}_{n-1})g(\Theta_{\mathfrak{t}_{n-1}}))\|_{\R^d}^2 \big]\\
& = (\mathfrak{t}_n-\mathfrak{t}_{n-1})^2 \, \E\Big[\big\| \tfrac{1}{\mathfrak{M}_{n-1}} \big[ \textstyle\sum_{j=1}^{\mathfrak{M}_{n-1}} G(\Theta_{\mathfrak{t}_{n-1}},Z_{n, j}) \big] -g(\Theta_{\mathfrak{t}_{n-1}})\big\|_{\R^d}^2\Big]\\
& \leq 2(\mathfrak{t}_n-\mathfrak{t}_{n-1})^2 \, \E\Big[\big\|\tfrac{1}{\mathfrak{M}_{n-1}}\textstyle\sum_{j=1}^{\mathfrak{M}_{n-1}} G(\Theta_{\mathfrak{t}_{n-1}},Z_{n, j})\big\|_{\R^d}^2\Big] \\
& \quad + 2(\mathfrak{t}_n-\mathfrak{t}_{n-1})^2 \, \E\big[\|g(\Theta_{\mathfrak{t}_{n-1}})\|_{\R^d}^2\big].
\end{split}
\end{align}
This, \eqref{eq:iterated_error_5}, and Lemma~\ref{lem:Q_t_leq_Theta_floor_t} assure that for all $n \in \N$ we have that
\begin{align}
\begin{split}
&\E\big[\|\Theta_{\mathfrak{t}_n}-(\Theta_{\mathfrak{t}_{n-1}}+(\mathfrak{t}_n - \mathfrak{t}_{n-1})g(\Theta_{\mathfrak{t}_{n-1}}))\|_{\R^d}^2 \big]  \\
& \leq 2 (\mathfrak{t}_n-\mathfrak{t}_{n-1})^2  \Big(c\big(1+\E\big[\|\Theta_{\mathfrak{t}_{n-1}}\|_{\R^d}^2\big]\big)
 + \E\big[(c(1+\|\Theta_{\mathfrak{t}_{n-1}}\|_{\R^d}))^2\big]  \Big).
 \end{split}
\end{align}
Lemma~\ref{lem:bnd_(sum_xi)^p} hence implies that for all $n \in \N$ we have that
\begin{align}
\begin{split}
& \E\big[\|\Theta_{\mathfrak{t}_n}-(\Theta_{\mathfrak{t}_{n-1}}+(\mathfrak{t}_n - \mathfrak{t}_{n-1})g(\Theta_{\mathfrak{t}_{n-1}}))\|_{\R^d}^2\big]  \\
& \leq 2(\mathfrak{t}_n-\mathfrak{t}_{n-1})^2\Big(c\big(1+\E\big[\|\Theta_{\mathfrak{t}_{n-1}}\|_{\R^d}^2\big] \big)
+ \E\big[2c^2\big(1+\|\Theta_{\mathfrak{t}_{n-1}}\|_{\R^d}^2\big)\big]\Big).
\end{split}
\end{align}
Hence, we obtain  that for all $n \in \N$  it holds that
\begin{equation}
\begin{split}
&\E\big[\|\Theta_{\mathfrak{t}_n}-(\Theta_{\mathfrak{t}_{n-1}}+(\mathfrak{t}_n - \mathfrak{t}_{n-1})g(\Theta_{\mathfrak{t}_{n-1}}))\|_{\R^d}^2\big]\\
& \leq  2(\mathfrak{t}_n-\mathfrak{t}_{n-1})^2\Big(c\big(1+\E\big[\|\Theta_{\mathfrak{t}_{n-1}}\|_{\R^d}^2\big]\big)
+ 2c^2\big(1+\E\big[\|\Theta_{\mathfrak{t}_{n-1}}\|_{\R^d}^2\big]\big)\Big)\\
& = 2(\mathfrak{t}_n-\mathfrak{t}_{n-1})^2(c+ 2c^2)\big(1+\E\big[\|\Theta_{\mathfrak{t}_{n-1}}\|_{\R^d}^2\big] \big).
\end{split}
\end{equation}
The fact that $\forall \, x,y \in \R^d \colon \|x+y \|_{\R^d}^2 \leq 2 \|x\|_{\R^d}^2 + 2\|y\|_{\R^d}^2$ therefore demonstrates that for all $n \in \N$ we have that
\begin{equation}
\label{eq:iterated_error_4}
\begin{split}
&\E\big[\|\Theta_{\mathfrak{t}_n}-(\Theta_{\mathfrak{t}_{n-1}}+(\mathfrak{t}_n - \mathfrak{t}_{n-1})g(\Theta_{\mathfrak{t}_{n-1}}))\|_{\R^d}^2 \big]\\
& \leq 2(\mathfrak{t}_n-\mathfrak{t}_{n-1})^2(c+ 2c^2)\big(1+2\|\Xi\|_{\R^d}^2+2\E\big[\|\Theta_{\mathfrak{t}_{n-1}}-\Xi\|_{\R^d}^2\big]\big)\\
& \leq (\mathfrak{t}_n-\mathfrak{t}_{n-1})^2\big[2(c+ 2c^2)(2+2\|\Xi\|_{\R^d}^2)\big]\big(1+\E\big[\|\Theta_{\mathfrak{t}_{n-1}}-\Xi\|_{\R^d}^2\big] \big).
\end{split}
\end{equation}
Next note that \eqref{eq:iterated_error_0} and Lemma~\ref{lem:Theta_L^2_sp_case} imply that for all $n \in \N_0$ we have that
\begin{equation}
\E\big[\|\Theta_{\mathfrak{t}_n}\|_{\R^d}^2\big] < \infty.
\end{equation}
The fact that $\forall \, x,y \in \R^d \colon \|x+y \|_{\R^d}^2 \leq 2 \|x\|_{\R^d}^2 + 2\|y\|_{\R^d}^2$ hence ensures that for all $n \in \N$ we have that
\begin{equation}
\E\big[\|\Theta_{\mathfrak{t}_{n-1}}-\Xi\|_{\R^d}^2\big]\leq 2 \big(\E\big[\|\Theta_{\mathfrak{t}_{n-1}}\|_{\R^d}^2\big]+\|\Xi\|_{\R^d}^2\big)<\infty. 
\end{equation}
Combining this and \eqref{eq:iterated_error_4} establishes \eqref{eq:iterated_error_3}. The proof of Lemma~\ref{lem:iterated_error} is thus completed.
\end{proof}

\section{A priori estimates for suitable approximation error constants associated to SAAs}
\label{subsec:apriori:constant}
\sectionmark{}

\begin{lemma}\label{lem:bnd_EF^p}
Assume Setting~\ref{sec:setting_2} and let  $p, \mathfrak{m}\in\lb 0\rb \cup [1,\infty)$ satisfy that 
\begin{equation}\label{eq:bnd_EF^p_0}
\sup_{x\in\R^d}\left(\frac{\E\big[\|G(x,Z_{1,1})\|_{\R^d}^p \big]}{\big[1+\|x\|_{\R^d}^{\mathfrak{m}}\big]^p}\right) + \sup_{n\in\N_0}\E\big[\norm{\Theta_{\mathfrak{t}_n}}_{\R^d}^{\mathfrak{m} p}\big]<\infty.
\end{equation}
Then we have that 
\begin{equation}\label{eq:bnd_EF^p_3}
\sup_{n\in\N_0} \E\Big[ \big\|\tfrac{1}{\mathfrak{M}_{n}}\textstyle\sum_{j=1}^{\mathfrak{M}_{n}} G(\Theta_{\mathfrak{t}_{n}},Z_{n+1, j})\big\|_{\R^d}^p \Big]<\infty.
\end{equation}
\end{lemma}
\begin{proof}[Proof of Lemma~\ref{lem:bnd_EF^p}]
First, observe that \eqref{eq:bnd_EF^p_0} and Lemma~\ref{lem:bnd_(sum_xi)^p} imply that there exists  $c\in [0,\infty)$ which satisfies for all $x\in\R^d$  that
\begin{equation}\label{eq:bnd_EF^p_2}
\E\big[\|G(x,Z_{1,1})\|_{\R^d}^p \big]\leq c(1+\|x\|_{\R^d}^{\mathfrak{m}p}).
\end{equation}
Lemma~\ref{lem:Q_t_leq_Theta_floor_t} therefore assures that for all $n\in\N_0$ we have that
\begin{equation}\label{eq:bnd_EF^p_7}
\begin{split}
\E\Big[ \big\|\tfrac{1}{\mathfrak{M}_{n}}\textstyle\sum_{j=1}^{\mathfrak{M}_{n}} G(\Theta_{\mathfrak{t}_{n}},Z_{n+1, j})\big\|_{\R^d}^p \Big]
& \leq c\big(1 + \E\big[\|\Theta_{\mathfrak{t}_{n}}\|_{\R^d}^{\mathfrak{m}p}\big]\big).
\end{split}
\end{equation}
Combining this and \eqref{eq:bnd_EF^p_0} establishes \eqref{eq:bnd_EF^p_3}. The proof of Lemma~\ref{lem:bnd_EF^p} is thus completed.
\end{proof}

\begin{lemma}\label{lem:bnd_Ef^p}
Assume Setting~\ref{sec:setting_2} and let $p,\mathfrak{m}\in\lb 0\rb \cup [1,\infty)$ satisfy that
\begin{equation}\label{eq:bnd_Ef^p_0}
\sup_{x\in\R^d} \left( \frac{\|g(x)\|_{\R^d}}{1+\|x\|_{\R^d}^{\mathfrak{m}}} \right) + \sup_{n\in\N_0}\E\big[\norm{\Theta_{\mathfrak{t}_n}}_{\R^d}^{\mathfrak{m} p}\big]<\infty.
\end{equation}
Then we have that
\begin{equation}\label{eq:bnd_Ef^p_3}
\sup_{t\in [0,\infty)}\E\big[\|g(\Theta_{t})\|_{\R^d}^p\big]<\infty.
\end{equation}
\end{lemma}
\begin{proof}[Proof of Lemma~\ref{lem:bnd_Ef^p}]
First, observe that \eqref{eq:bnd_Ef^p_0} implies that there exists  $c \in [0, \infty)$ which satisfies for all $x\in\R^d$   that
\begin{equation}\label{eq:bnd_Ef^p_2}
\norm{g(x)}_{\R^d}\leq c(1+\|x\|_{\R^d}^\mathfrak{m}).
\end{equation}
This and Lemma~\ref{lem:bnd_(sum_xi)^p} ensure that for all $t\in [0,\infty)$  we have that
\begin{equation}
\begin{split}
\E\big[\|g(\Theta_{t})\|_{\R^d}^p\big]
& \leq \E\big[(c(1+\|\Theta_{t}\|_{\R^d}^\mathfrak{m}))^p\big]
 \leq \E\big[2^{p-1}c^p(1+\|\Theta_{t}\|_{\R^d}^{\mathfrak{m}p})\big]\\
& = 2^{p-1}c^p+2^{p-1}c^p \, \E\big[\|\Theta_{t}\|_{\R^d}^{\mathfrak{m}p}\big].
\end{split}
\end{equation}
Lemma~\ref{lem:equiv_sup_L^p_Theta_t} and \eqref{eq:bnd_Ef^p_0} hence demonstrate that 
\begin{equation}
\sup_{t\in [0,\infty)}\E\big[\|g(\Theta_{t})\|_{\R^d}^p\big] 
\leq 2^{p-1}c^p+2^{p-1}c^p\sup_{t\in [0,\infty)}\E\big[\|\Theta_{t}\|_{\R^d}^{\mathfrak{m}p}\big]<\infty.
\end{equation} 
This establishes \eqref{eq:bnd_Ef^p_3}.
The proof of Lemma~\ref{lem:bnd_Ef^p} is thus completed.
\end{proof}

\begin{lemma}\label{lem:bnd_Ef'^p}
Assume Setting~\ref{sec:setting_2} and let $p,\mathfrak{m}\in\lb 0\rb \cup [1,\infty)$ satisfy that
\begin{equation}\label{eq:bnd_Ef'^p_0}
\sup_{x\in\R^d}\left(\frac{\big\|\E\big[(\tfrac{\partial}{\partial x} G)(x,Z_{1,1})\big]\big\|_{L(\R^d,\R^d)}}{1+\|x\|_{\R^d}^\mathfrak{m}}\right) + \sup_{n\in\N_0}\E\big[\norm{\Theta_{\mathfrak{t}_n}}_{\R^d}^{\mathfrak{m} p}\big] <\infty.
\end{equation}
Then 
\begin{enumerate}[(i)]
\item\label{item:bnd_Ef'^p_1} we have that $g\in C^1(\R^d,\R^d)$ and
\item\label{item:bnd_Ef'^p_2} we have that
\begin{equation}\label{eq:bnd_Ef'^p_3}
\sup_{t\in [0,\infty)}\E\big[\|g'(\Theta_{t})\|_{L(\R^d,\R^d)}^p\big]<\infty.
\end{equation}
\end{enumerate}
\end{lemma}
\begin{proof}[Proof of Lemma~\ref{lem:bnd_Ef'^p}]
First, note that item~\eqref{item:f_C^1_1} in Lemma~\ref{lem:f_C^1} proves item~\eqref{item:bnd_Ef'^p_1}. Furthermore, observe that \eqref{eq:bnd_Ef'^p_0} and  item~\eqref{item:f_C^1_2} in Lemma~\ref{lem:f_C^1} demonstrate that there exists  $c\in [0,\infty)$ which satisfies  for all $x\in\R^d$  that
\begin{equation}
\|g'(x)\|_{L(\R^d,\R^d)}=
\big\|\E\big[(\tfrac{\partial}{\partial x} G)(x,Z_{1,1})\big]\big\|_{L(\R^d,\R^d)}
\leq c(1+\|x\|_{\R^d}^\mathfrak{m}).
\end{equation}
This and Lemma~\ref{lem:bnd_(sum_xi)^p} imply that  for all $t\in [0,\infty)$ we have that
\begin{equation}
\begin{split}
\E\big[\|g'(\Theta_{t})\|_{L(\R^d,\R^d)}^p\big]
& \leq \E\big[(c(1+\|\Theta_{t}\|_{\R^d}^\mathfrak{m}))^p\big]\\
& \leq \E\big[2^{p-1}c^p(1+\|\Theta_{t}\|_{\R^d}^{\mathfrak{m}p})\big]\\
& = 2^{p-1}c^p+2^{p-1}c^p \, \E\big[\|\Theta_{t}\|_{\R^d}^{\mathfrak{m}p}\big].
\end{split}
\end{equation}
Lemma~\ref{lem:equiv_sup_L^p_Theta_t} and \eqref{eq:bnd_Ef'^p_0} hence demonstrate that 
\begin{equation}
\sup_{t\in [0,\infty)}\E\big[\|g'(\Theta_{t})\|_{L(\R^d,\R^d)}^p\big] 
\leq 2^{p-1}c^p+2^{p-1}c^p\sup_{t\in [0,\infty)}\E\big[\|\Theta_{t}\|_{\R^d}^{\mathfrak{m}p}\big]<\infty.
\end{equation}
The proof of Lemma~\ref{lem:bnd_Ef'^p} is thus completed.
\end{proof}

\begin{lemma}\label{lem:diff_bnded_map_linear_growth}
	Let $(V,\left\|\cdot\right\|_V)$ be a non-trivial $\R$-Banach space, let $(W,\left\|\cdot\right\|_W)$ be an $\R$-Banach space, and let $f\in C^1(V,W)$, $c\in \R$ satisfy for all $x\in V$ that
	\begin{equation}\label{eq:diff_bnded_map_linear_growth_1}
	\|f'(x)\|_{L(V,W)}\leq c.
	\end{equation}
	Then we have for all $x\in V$ that
	\begin{equation}\label{eq:diff_bnded_map_linear_growth_2}
	\|f(x)\|_W\leq (c+ \|f(0)\|_W )(1+\|x\|_V).
	\end{equation}
\end{lemma}
\begin{proof}[Proof of Lemma~\ref{lem:diff_bnded_map_linear_growth}]
	First, note that the  fundamental theorem of calculus for the Bochner integral (see, e.g., \cite[Lemma 2.1]{JentzenBochner}) proves that for all $x\in V$ we have that
	\begin{equation}
	\begin{split}
	\|f(x)-f(0)\|_W & = \bigg\|\int_0^1f'(\lambda x)x\, d\lambda \bigg\|_W.
	\end{split}
	\end{equation}
	This and the triangle inequality for the Bochner integral demonstrate that for all $x\in V$ we have that
	\begin{equation}
	\begin{split}
	\|f(x)-f(0)\|_W 
	 \leq \int_0^1\|f'(\lambda x)x\|_{W}\, d\lambda 
	& \leq \int_0^1\|f'(\lambda x)\|_{L(V,W)} \|x\|_{V}\, d\lambda \\
	& \leq \int_0^1c \|x\|_{V}\, d\lambda 
	= c\|x\|_V.
	\end{split}
	\end{equation}
	This reveals that for all $x\in V$  it holds that
	\begin{equation}
	\begin{split}
	\|f(x)\|_{W} &\leq \|f(x)-f(0)\|_W + \|f(0)\|_W \\
	& \leq c\|x\|_V + \|f(0)\|_W \\
	& \leq (c+ \|f(0)\|_W )(1+\|x\|_V). 
	\end{split}
	\end{equation}
	This establishes \eqref{eq:diff_bnded_map_linear_growth_2}. The proof of Lemma~\ref{lem:diff_bnded_map_linear_growth} is thus completed.
\end{proof}

\begin{lemma}\label{lem:bnd_C_case_f''_bnded}
Assume Setting~\ref{sec:setting_2},  let $\psi\in C^2(\R^d,\R)$ satisfy that
\begin{equation}\label{eq:bnd_C_case_f''_bnded_0}
\sup_{x\in\R^d}\left(\frac{\E\big[\|G(x,Z_{1,1})\|_{\R^d}^2 \big]}{\big[1+\|x\|_{\R^d} \big]^2} + \big\|\E\big[(\tfrac{\partial^2}{\partial x^2}G)(x,Z_{1,1})\big]\big\|_{L^{(2)}(\R^d,\R^d)}\right)<\infty
\end{equation}
and
$
\sup_{x\in\R^d} \max_{i \in \{1,2\}}
\|\psi^{(i)}(x)\|_{L^{(i)}(\R^d,\R)}
<\infty,
$
let $Q\colon [0,\infty)\times\Omega\to\R^d$ be the stochastic process which satisfies for all
$m \in \N_0$, $t \in [\mathfrak{t}_m, \mathfrak{t}_{m+1})$ that
\begin{equation}\label{eq:bnd_C_case_f''_bnded_6}
Q_t = \frac{1}{\mathfrak{M}_m} \sum_{n=1}^{\mathfrak{M}_m}G(\Theta_{\mathfrak{t}_m},Z_{m+1, n}),
\end{equation}
and let $C\colon [0,\infty)\to [0,\infty]$ satisfy for all $T\in [0,\infty)$ that
\begin{align*}
\label{eq:bnd_C_case_f''_bnded_7}
 &C(T)  = \sup_{s,v\in [0,T]}\E \bigg[ \|Q_s-g(\Theta_{\floorgrid{s}})\|_{\R^d}\|Q_s\|_{\R^d} \numberthis \\&  \cdot \bigg(\int_0^1  e^{-L(T-s)} \|\psi''( \theta_{T-s}^{\lambda \Theta_s + (1-\lambda)\Theta_{\floorgrid{s}}})\|_{L^{(2)}(\R^d,\R)}  + \|\psi'( \theta_{T-s}^{\lambda \Theta_s + (1-\lambda)\Theta_{\floorgrid{s}}} )\|_{L(\R^d,\R)} \\
&  \cdot \int_0^{T-s} e^{-Lu}\big\|
g''(\theta_u^{\lambda \Theta_s + (1-\lambda)\Theta_{\floorgrid{s}}})
\big\|_{L^{(2)}(\R^d,\R^d)}\, du\, d\lambda \bigg) + \|\psi'(\theta_{T-s}^{\Theta_s})\|_{L(\R^d,\R)} \|g'(\Theta_{v})
Q_v\|_{\R^d}\bigg].
\end{align*}
Then we have that
\begin{equation}\label{eq:bnd_C_case_f''_bnded_8}
\sup_{T\in [0,\infty)}C(T)<\infty.
\end{equation}
\end{lemma}
\begin{proof}[Proof of Lemma~\ref{lem:bnd_C_case_f''_bnded}]
Throughout this proof let $\kappa\in (0,\infty)$ be a real number which satisfies for all $n\in\N$ that
\begin{align}\label{eq:bnd_C_case_f''_bnded_8_2}
\begin{split}
&\E\big[\|\Theta_{\mathfrak{t}_n}-(\Theta_{\mathfrak{t}_{n-1}}+(\mathfrak{t}_n - \mathfrak{t}_{n-1})g(\Theta_{\mathfrak{t}_{n-1}}))\|_{\R^d}^2 \big]\\
& \leq (\mathfrak{t}_{n}-\mathfrak{t}_{n-1})^2\kappa\big(1+\E\big[\|\Theta_{\mathfrak{t}_{n-1}}-\Xi\|_{\R^d}^2\big]\big)
\end{split}
\end{align}
(cf.~Lemma~\ref{lem:iterated_error}).
Note that Lemma~\ref{lem:f_C^n} assures that for all $x\in\R^d$ we have that 
\begin{equation}
g\in C^2(\R^d,\R^d) \qquad\text{and}\qquad g''(x)=\E\big[\big(\tfrac{\partial^2}{\partial x^2}G\big)(x,Z_{1,1})\big].
\end{equation}
This, \eqref{eq:bnd_C_case_f''_bnded_0},  and Lemma~\ref{lem:(1)+(2)=>(*)} prove that there exists  $c\in [0,\infty)$ which satisfies for all $x\in\R^d$  that
\begin{equation}\label{eq:bnd_C_case_f''_bnded_3}
\E\big[\|G(x,Z_{1,1})\|_{\R^d}^2 \big]  \leq c(1+\|x\|_{\R^d})^2,\qquad \|g(x)\|_{\R^d}\leq c(1+\|x\|_{\R^d}),
\end{equation}
and
\begin{equation}\label{eq:bnd_C_case_f''_bnded_3p}
\max\big\lb \| g''(x)\|_{L^{(2)}(\R^d,\R^d)}, \|\psi'(x)\|_{L(\R^d,\R)},\|\psi''(x)\|_{L^{(2)}(\R^d,\R)}\big\rb  \leq c.
\end{equation}
Lemma~\ref{lem:diff_bnded_map_linear_growth} hence implies that for all $x\in\R^d$ we have that
\begin{equation}\label{eq:bnd_C_case_f''_bnded_8_p}
\|g'(x)\|_{L(\R^d,\R^d)} \leq (c+\|g'(0)\|_{L(\R^d,\R^d)} )(1+\|x\|_{\R^d}).
\end{equation}
Moreover, note that Lemma~\ref{lem:iterated_error} proves that  for all $n\in\N$ we have that
\begin{equation}
\E\big[\|\Theta_{\mathfrak{t}_n}-(\Theta_{\mathfrak{t}_{n-1}}+(\mathfrak{t}_n - \mathfrak{t}_{n-1})g(\Theta_{\mathfrak{t}_{n-1}}))\|_{\R^d}^2\big]<\infty.
\end{equation}
Next observe that the assumption that $Z_{m, n}$, $(m, n) \in \N^2$, are i.i.d.\ random variables and \eqref{eq:setting_2_6} ensure that for all $n \in \N_0$, $j \in \{1, 2,\ldots, \mathfrak{M}_n\}$ we have that $Z_{n+1, j}$ and $ \Theta_{\mathfrak{t}_{n}}$
are independent. This and the assumption that $Z_{m, n}$, $(m, n) \in \N^2$, are i.i.d.\ random variables ensure that for all $n \in \N_0$, $j \in \{1, 2, \ldots, \mathfrak{M}_n\}$  we have that
\begin{equation}
\begin{split}
\E\big[ \| G(\Theta_{\mathfrak{t}_{n}},Z_{n+1,j}) \|_{\R^d}^2\big]
& = \int_{\Omega}\| G(\Theta_{\mathfrak{t}_{n}}(\omega),Z_{n+1, j}(\omega)) \|_{\R^d}^2\, \P(d\omega)\\
& = \int_{\Omega}\int_{\Omega}\| G(\Theta_{\mathfrak{t}_{n}}(\omega),Z_{n+1, j}(\tilde{\omega})) \|_{\R^d}^2\, \P(d\tilde{\omega})\, \P(d\omega)\\
& = \int_{\Omega}\int_{\Omega}\| G(\Theta_{\mathfrak{t}_{n}}(\omega),Z_{1,1}(\tilde{\omega})) \|_{\R^d}^2\, \P(d\tilde{\omega})\, \P(d\omega).
\end{split}
\end{equation}
Combining this with \eqref{eq:bnd_C_case_f''_bnded_3} and  Lemma~\ref{lem:Theta_L^2_sp_case} demonstrates that for all $n \in \N_0$, $j \in \{1, 2, \ldots, \mathfrak{M}_n\}$ we have that
\begin{equation}
\begin{split}
\E\big[ \| G(\Theta_{\mathfrak{t}_{n}},Z_{n+1, j}) \|_{\R^d}^2\big]
& \leq \int_{\Omega}c\big(1 + \| \Theta_{\mathfrak{t}_{n}}(\omega) \|_{\R^d}\big)^2 \, \P(d\omega) \\
&\leq 2 c\big(1 + \E\big[\| \Theta_{\mathfrak{t}_{n}} \|_{\R^d}^2\big]\big) < \infty.
\end{split}
\end{equation}
This reveals that for all $n \in \N_0$, $j \in \{1, 2, \ldots, \mathfrak{M}_n\}$  it holds that
\begin{equation}
\E\big[ \|G(\Theta_{\mathfrak{t}_{n}},Z_{n+1, j})\|_{\R^d}\big] \leq \big|\E\big[ \| G(\Theta_{\mathfrak{t}_{n}},Z_{n+1, j}) \|_{\R^d}^2\big] \big|^{1/2} < \infty.
\end{equation}
This and Corollary~\ref{cor:PvW_bnd_L^2_Theta-Xi} imply that
\begin{equation}\label{eq:bnd_C_case_f''_bnded_9}
\sup_{n\in\N_0}\big|\E\big[\norm{\Theta_{\mathfrak{t}_n}-\Xi}_{\R^d}^{2}\big]\big|^{1/2}<\infty.
\end{equation}
The Minkowski inequality hence assures that
\begin{equation}\label{eq:bnd_C_case_f''_bnded_9_2}
\sup_{n\in\N_0}\big|\E\big[\norm{\Theta_{\mathfrak{t}_n}}_{\R^d}^{2}\big]\big|^{1/2}
\leq \|\Xi\|_{\R^d} + \sup_{n\in\N_0}\big|\E\big[\norm{\Theta_{\mathfrak{t}_n}-\Xi}_{\R^d}^{2}\big]\big|^{1/2}<\infty.
\end{equation}
Next observe that \eqref{eq:bnd_C_case_f''_bnded_3p} demonstrates that for all $T\in[0,\infty)$ we have that
\begin{align*}
\label{eq:bnd_C_case_f''_bnded_10}
&  \sup_{s\in [0,T]}\E \bigg[ \|Q_s-g(\Theta_{\floorgrid{s}})\|_{\R^d}\norm{Q_s}_{\R^d} \int_0^1  e^{-L(T-s)} \|\psi''( \theta_{T-s}^{\lambda \Theta_s + (1-\lambda)\Theta_{\floorgrid{s}}})\|_{L^{(2)}(\R^d,\R)}\, d\lambda\bigg]\\
& \leq \sup_{s\in [0,T]}\E \bigg[ \|Q_s-g(\Theta_{\floorgrid{s}})\|_{\R^d}\norm{Q_s}_{\R^d} \int_0^1  c e^{-L(T-s)}\, d\lambda\bigg]\\
& \leq c\sup_{s\in [0,T]}\E \big[ \|Q_s-g(\Theta_{\floorgrid{s}})\|_{\R^d}\norm{Q_s}_{\R^d} \big]\\
& \leq c\sup_{s\in [0,\infty)}\E \big[ \|Q_s-g(\Theta_{\floorgrid{s}})\|_{\R^d}\norm{Q_s}_{\R^d} \big]. \numberthis
\end{align*}
Hölder's inequality therefore assures that for all $T\in [0,\infty)$ we have that
\begin{align*}
\label{eq:bnd_C_case_f''_bnded_11}
&  \sup_{s\in [0,T]}\E \bigg[ \|Q_s-g(\Theta_{\floorgrid{s}})\|_{\R^d}\norm{Q_s}_{\R^d} \int_0^1  e^{-L(T-s)} \|\psi''( \theta_{T-s}^{\lambda \Theta_s + (1-\lambda)\Theta_{\floorgrid{s}}})\|_{L^{(2)}(\R^d,\R)}\, d\lambda\bigg]\\
& \leq c \sup_{s\in [0,\infty)}\Big(\big|\E\big[\|Q_s-g(\Theta_{\floorgrid{s}})\|_{\R^d}^2\big]\big|^{1/2}\big| \E\big[\|Q_s\|_{\R^d}^2\big]\big|^{1/2}\Big)\\
& \leq c \sup_{s\in [0,\infty)}\big|\E \big[\|Q_s-g(\Theta_{\floorgrid{s}})\|_{\R^d}^2\big]\big|^{1/2}\sup_{s\in [0,\infty)} \big| \E\big[\|Q_s\|_{\R^d}^2\big]\big|^{1/2}. \numberthis
\end{align*}
Moreover, note that the Minkowski inequality implies that
\begin{equation}\label{eq:bnd_C_case_f''_bnded_12}
\begin{split}
&\sup_{s\in [0,\infty)}\big|\E\big[\|Q_s-g(\Theta_{\floorgrid{s}})\|_{\R^d}^2\big]\big|^{1/2}\\
&\leq \sup_{s\in [0,\infty)}\big|\E \big[\|Q_s\|_{\R^d}^2\big]\big|^{1/2}+\sup_{s\in [0,\infty)}\big|\E\big[\|g(\Theta_{\floorgrid{s}})\|_{\R^d}^2\big]\big|^{1/2}.
\end{split}
\end{equation}
In the next step observe that  \eqref{eq:bnd_C_case_f''_bnded_3}, \eqref{eq:bnd_C_case_f''_bnded_9_2},  and  Lemma~\ref{lem:bnd_EF^p} assure that
\begin{equation}\label{eq:bnd_C_case_f''_bnded_13}
\begin{split}
\sup_{s\in [0,\infty)}\E\big[\|Q_s\|^2_{\R^d}\big]
 = \sup_{n\in\N_0}\E\big[\| \tfrac{1}{\mathfrak{M}_n}\textstyle\sum_{j=1}^{\mathfrak{M}_n} G(\Theta_{\mathfrak{t}_n},Z_{n+1, j})\|^2_{\R^d}\big]
<\infty.
\end{split}
\end{equation}
Next note that \eqref{eq:bnd_C_case_f''_bnded_3}, \eqref{eq:bnd_C_case_f''_bnded_9_2}, and Lemma~\ref{lem:bnd_Ef^p} demonstrate that
\begin{equation}\label{eq:bnd_C_case_f''_bnded_14}
\sup_{s\in [0,\infty)}\E\big[\|g(\Theta_{\floorgrid{s}})\|_{\R^d}^2\big]
\leq \sup_{t\in [0,\infty)}\E\big[\|g(\Theta_{t})\|_{\R^d}^2\big]
<\infty.
\end{equation}
Combining this, \eqref{eq:bnd_C_case_f''_bnded_12}, and \eqref{eq:bnd_C_case_f''_bnded_13} ensures that
\begin{equation}\label{eq:bnd_C_case_f''_bnded_15}
\sup_{s\in [0,\infty)}\big|\E \big[\|Q_s-g(\Theta_{\floorgrid{s}})\|_{\R^d}^2\big]\big|^{1/2}<\infty.
\end{equation}
This, \eqref{eq:bnd_C_case_f''_bnded_11}, and \eqref{eq:bnd_C_case_f''_bnded_13} prove that
\begin{equation}\label{eq:bnd_C_case_f''_bnded_16}
\begin{split}
&  \sup_{T\in [0,\infty)} \sup_{s\in [0,T]}\E \bigg[  \|Q_s-g(\Theta_{\floorgrid{s}})\|_{\R^d}\norm{Q_s}_{\R^d}  \\
  & \quad \cdot \int_0^1  e^{-L(T-s)} \|\psi''( \theta_{T-s}^{\lambda \Theta_s + (1-\lambda)\Theta_{\floorgrid{s}}})\|_{L^{(2)}(\R^d,\R)}\, d\lambda\bigg] <\infty.
\end{split}
\end{equation}
Next observe that \eqref{eq:bnd_C_case_f''_bnded_3p} implies that for all $T\in(0,\infty)$ we have that
\begin{equation}\label{eq:bnd_C_case_f''_bnded_17}
\begin{split}
&\sup_{s\in [0,T]}\E \bigg[ \|Q_s-g(\Theta_{\floorgrid{s}})\|_{\R^d}\norm{Q_s}_{\R^d}  \int_0^1 \|\psi'( \theta_{T-s}^{\lambda \Theta_s + (1-\lambda)\Theta_{\floorgrid{s}}} )\|_{L(\R^d,\R)}\\ 
&\quad \cdot\int_0^{T-s} e^{-Lu}\|g''(\theta_u^{\lambda \Theta_s + (1-\lambda)\Theta_{\floorgrid{s}}})\|_{L^{(2)}(\R^d,\R^d)}\, du\, d\lambda \bigg]\\
& \leq c^2\sup_{s\in [0,T]}\E \bigg[ \|Q_s-g(\Theta_{\floorgrid{s}})\|_{\R^d}\norm{Q_s}_{\R^d} 
 \int_0^1 \int_0^{T-s} e^{-Lu}\, du\, d\lambda \bigg]\\
& = c^2\sup_{s\in [0,T]}\E \bigg[ \|Q_s-g(\Theta_{\floorgrid{s}})\|_{\R^d}\norm{Q_s}_{\R^d} 
 \int_0^{T-s} e^{-Lu}\, du \bigg]\\
& = \tfrac{c^2}{L}\sup_{s\in [0,T]}\E \big[ \|Q_s-g(\Theta_{\floorgrid{s}})\|_{\R^d}\norm{Q_s}_{\R^d}(1-e^{-L(T-s)})\big]\\
 & \leq \tfrac{c^2}{L}\sup_{s\in [0,\infty)}\E \big[ \|Q_s-g(\Theta_{\floorgrid{s}})\|_{\R^d}\norm{Q_s}_{\R^d} \big].
\end{split}
\end{equation}
This and Hölder's inequality ensure that for all $T\in(0,\infty)$ we have that
\begin{equation}
\begin{split}
&\sup_{s\in [0,T]}\E \bigg[ \|Q_s-g(\Theta_{\floorgrid{s}})\|_{\R^d}\norm{Q_s}_{\R^d} \int_0^1  \|\psi'( \theta_{T-s}^{\lambda \Theta_s + (1-\lambda)\Theta_{\floorgrid{s}}} )\|_{L(\R^d,\R)}\\ 
& \quad \cdot \int_0^{T-s} e^{-Lu}\|g''(\theta_u^{\lambda \Theta_s + (1-\lambda)\Theta_{\floorgrid{s}}})\|_{L^{(2)}(\R^d,\R^d)}\, du\, d\lambda \bigg]\\
& \leq \tfrac{c^2}{L}\sup_{s\in [0,\infty)}\big(\big|\E\big[ \|Q_s-g(\Theta_{\floorgrid{s}})\|_{\R^d}^2\big]\big|^{1/2}\big|\E\big[\norm{Q_s}_{\R^d}^2\big]\big|^{1/2}\big).
\end{split}
\end{equation}
Combining this, \eqref{eq:bnd_C_case_f''_bnded_13}, and \eqref{eq:bnd_C_case_f''_bnded_15} demonstrates that
\begin{align}
\label{eq:bnd_C_case_f''_bnded_18}
\begin{split}
&\sup_{T\in [0,\infty)}\sup_{s\in [0,T]}\E \bigg[ \|Q_s-g(\Theta_{\floorgrid{s}})\|_{\R^d}\norm{Q_s}_{\R^d} \int_0^1 \|\psi'( \theta_{T-s}^{\lambda \Theta_s + (1-\lambda)\Theta_{\floorgrid{s}}} )\|_{L(\R^d,\R)}\\ 
& \quad \cdot \int_0^{T-s} e^{-Lu}\|g''(\theta_u^{\lambda \Theta_s + (1-\lambda)\Theta_{\floorgrid{s}}})\|_{L^{(2)}(\R^d,\R^d)}\, du\, d\lambda \bigg]<\infty. 
\end{split}
\end{align}
Next observe that \eqref{eq:bnd_C_case_f''_bnded_3p} and Hölder's inequality imply that for all $T\in [0,\infty)$ we have that
\begin{equation}\label{eq:bnd_C_case_f''_bnded_19}
\begin{split}
& \sup_{s,v\in [0,T]}\E\big[\|\psi'(\theta_{T-s}^{\Theta_s})\|_{L(\R^d,\R)} \|g'(\Theta_{v})Q_v \|_{\R^d}\big]\\
&
\leq \sup_{v\in [0,\infty)}\E\big[c \|g'(\Theta_{v})Q_v \|_{\R^d}\big]\\
&
\leq c\sup_{v\in [0,\infty)}\E\big[ \|g'(\Theta_{v})\|_{L(\R^d,\R^d)}\|Q_v \|_{\R^d}\big]\\
&
\leq c\sup_{v\in [0,\infty)}\Big(\big|\E\big[ \|g'(\Theta_{v})\|_{L(\R^d,\R^d)}^2\big]\big|^{1/2}\big|\E\big[\|Q_v \|_{\R^d}^2\big]\big|^{1/2}\Big)\\
&
\leq c\sup_{v\in [0,\infty)}\big|\E\big[ \|g'(\Theta_{v})\|_{L(\R^d,\R^d)}^2\big]\big|^{1/2} \sup_{v\in [0,\infty)}\big|\E\big[\|Q_v \|_{\R^d}^2\big]\big|^{1/2}.
\end{split}
\end{equation}
Furthermore, note that Lemma~\ref{lem:bnd_Ef'^p}, \eqref{eq:bnd_C_case_f''_bnded_8_p}, and \eqref{eq:bnd_C_case_f''_bnded_9_2} ensure that 
\begin{equation}\label{eq:bnd_C_case_f''_bnded_20}
\sup_{v\in [0,\infty)}\E\big[ \|g'(\Theta_{v})\|_{L(\R^d,\R^d)}^2\big]<\infty.
\end{equation}
Combining this, \eqref{eq:bnd_C_case_f''_bnded_13}, and \eqref{eq:bnd_C_case_f''_bnded_19} proves that
\begin{equation}\label{eq:bnd_C_case_f''_bnded_21}
\sup_{T\in [0,\infty)}\sup_{s,v\in [0,T]}\E\big[\|\psi'(\theta_{T-s}^{\Theta_s})\|_{L(\R^d,\R)} \|g'(\Theta_{v})Q_v \|_{\R^d}\big]<\infty.
\end{equation}
This, \eqref{eq:bnd_C_case_f''_bnded_16}, and \eqref{eq:bnd_C_case_f''_bnded_18}  establish \eqref{eq:bnd_C_case_f''_bnded_8}. The proof of Lemma~\ref{lem:bnd_C_case_f''_bnded} is thus completed.
\end{proof}

\section{Weak convergence rates for SAAs in the case of polynomially decaying learning rates with  mini-batches}
\label{subsec:weak:SAA:poly}
\sectionmark{}

\begin{prop}\label{prop:rate_sp_case_f''_bnded}
Assume Setting~\ref{sec:setting_2} and 
let $\psi\in C^2(\R^d,\R)$ satisfy that
\begin{equation}\label{eq:rate_sp_case_f''_bnded_0}
\sup_{x\in\R^d}\left(\frac{\E\big[\|G(x,Z_{1,1})\|_{\R^d}^2 \big]}{\big[1+\|x\|_{\R^d} \big]^2} + \big\|\E\big[(\tfrac{\partial^2}{\partial x^2} G )(x,Z_{1,1})\big]\big\|_{L^{(2)}(\R^d,\R^d)}  \right)<\infty
\end{equation}
and $\sup_{x\in\R^d} \max_{i \in \{1,2\}}
\|\psi^{(i)}(x)\|_{L^{(i)}(\R^d,\R)} < \infty$.
Then 
\begin{enumerate}[(i)]
\item\label{item:rate_sp_case_f''_bnded_1} we have that 
$
\lb x\in \R^d\colon g(x)=0\rb=\lb \Xi \rb
$
and
\item\label{item:rate_sp_case_f''_bnded_2} there exists  $C \in \R$ such that for all $n\in\N$ we have that
\begin{equation}
 |\E[\psi(\Theta_{\mathfrak{t}_{n}})]- \psi(\Xi)|  \leq C n^{2\epsilon-1}.
\end{equation}
\end{enumerate}
\end{prop}
\begin{proof}[Proof of Proposition~\ref{prop:rate_sp_case_f''_bnded}]
Throughout this proof let $\lambda\in (0,1)$, 
let $K(\lambda) \in (0,\infty)$ be the  real number given by
\begin{equation}
K(\lambda) 
 = \sup_{n \in \N \cap [2,\infty)} \bigg[ \frac{\eta^2e^{L\eta+\frac{L\eta}{\epsilon}} }{2(1-2\epsilon)} \Big(  n^{1-2\epsilon} \Big[ 2e^{-\frac{L\eta}{\epsilon}(1-\lambda^{\epsilon})n^{\epsilon}} + (n-1)^{2\epsilon-2} \Big] + \lambda^{2\epsilon-1}\Big)  \bigg]
\end{equation}
(cf.~Lemma~\ref{lem:bnd_sum_main}), let $Q\colon [0,\infty)\times\Omega\to\R^d$ be the stochastic process which satisfies for all
$m \in \N_0$, $t \in [\mathfrak{t}_m, \mathfrak{t}_{m+1})$ that
\begin{equation}\label{eq:rate_sp_case_f''_bnded_7_2}
Q_t = \frac{1}{\mathfrak{M}_m} \textstyle\sum\limits_{n=1}^{\mathfrak{M}_m}G(\Theta_{\mathfrak{t}_m},Z_{m+1, n}),
\end{equation}
and let $R\colon [0,\infty)\to [0,\infty]$ satisfy for all $T\in [0,\infty)$ that
\begin{align*}
\label{eq:rate_sp_case_f''_bnded_7_3}
& R(T)  = \sup_{s,v\in [0,T]}\E \bigg[ \|Q_s-g(\Theta_{\floorgrid{s}})\|_{\R^d}\|Q_s\|_{\R^d} \numberthis \\&  \cdot \bigg(\int_0^1  e^{-L(T-s)} \|\psi''( \theta_{T-s}^{\lambda \Theta_s + (1-\lambda)\Theta_{\floorgrid{s}}})\|_{L^{(2)}(\R^d,\R)}  + \|\psi'( \theta_{T-s}^{\lambda \Theta_s + (1-\lambda)\Theta_{\floorgrid{s}}} )\|_{L(\R^d,\R)} \\
&  \cdot \int_0^{T-s} e^{-Lu}\big\|
g''(\theta_u^{\lambda \Theta_s + (1-\lambda)\Theta_{\floorgrid{s}}})
\big\|_{L^{(2)}(\R^d,\R^d)}\, du\, d\lambda \bigg) + \|\psi'(\theta_{T-s}^{\Theta_s})\|_{L(\R^d,\R)} \|g'(\Theta_{v})
Q_v\|_{\R^d}\bigg].
\end{align*}
Note that \eqref{eq:rate_sp_case_f''_bnded_0} and Lemma~\ref{lem:f_C^n} prove that 
\begin{equation}\label{eq:rate_sp_case_f''_bnded_7_4}
\sup_{x\in\R^d}\|g''(x)\|_{L^{(2)}(\R^d,\R^d)}<\infty.
\end{equation}
This, Lemma~\ref{lem:f_C^1}, and Lemma~\ref{lem:diff_bnded_map_linear_growth} demonstrate that
\begin{equation}\label{eq:rate_sp_case_f''_bnded_7_5}
\sup_{x\in\R^d} \left(\frac{\big\|\E\big[(\tfrac{\partial}{\partial x}G)(x,Z_{1,1})\big]\big\|_{L(\R^d,\R^d)}}{\big[1+\|x\|_{\R^d}\big]} \right)
=\sup_{x\in\R^d} \left(\frac{\|g'(x)\|_{L(\R^d,\R^d)}}{\big[1+\|x\|_{\R^d}\big]} \right)<\infty.
\end{equation}
Next observe that \eqref{eq:setting_2_3_2} assures that $\Xi$ is the unique zero of $g$. This proves item~\eqref{item:rate_sp_case_f''_bnded_1}. Item~\eqref{item:motionless_pt_f_3} in Lemma~\ref{lem:motionless_pt_f} therefore ensures that
\begin{equation}
\limsup_{s\to\infty} \|\theta_s^{\xi}-\Xi\|_{\R^d} =0.
\end{equation}
Corollary~\ref{cor:sp_case_gamma}, \eqref{eq:rate_sp_case_f''_bnded_0}, \eqref{eq:rate_sp_case_f''_bnded_7_5}, and \eqref{eq:setting_2_3_2} hence assure that for all  $k\in\N$ we have that
\begin{align}
 \label{eq:rate_sp_case_f''_bnded_8}
 \begin{split}
 &|\E[\psi(\Theta_{\mathfrak{t}_{k}})]- \psi(\Xi)|  \\
&\leq (k+1)^{2\epsilon-1} \bigg[ K(\lambda)R(\mathfrak{t}_{k})\\
&
  + (k+1)^{1-2\epsilon}e^{-L\mathfrak{t}_{k}}
 \sup_{\alpha \in [0,1]}\Big( \big\| \psi'(\alpha \theta_{\mathfrak{t}_{k}}^{\xi}+(1-\alpha)\Xi) \big\|_{L(\R^d,\R)}\Big) \|\xi-\Xi\|_{\R^d}\bigg]\\
&  \leq (k+1)^{2\epsilon-1} \bigg[ K(\lambda)\sup_{T\in [0,\infty)}R(T)\\
&  + (k+1)^{1-2\epsilon}e^{-L\mathfrak{t}_{k}}
  \sup_{\alpha \in[0,1]}\Big( \big\| \psi'(\alpha \theta_{\mathfrak{t}_{k}}^{\xi}+(1-\alpha)\Xi) \big\|_{L(\R^d,\R)}\Big) \|\xi-\Xi\|_{\R^d}\bigg]\\
&    \leq (k+1)^{2\epsilon-1} \bigg[ K(\lambda)\sup_{T\in [0,\infty)}R(T)\\
&  +\sup_{l\in\N_0}\bigg((l+1)^{1-2\epsilon}e^{-L\mathfrak{t}_{l}} \sup_{\alpha \in[0,1]}\Big( \big\| \psi'(\alpha \theta_{\mathfrak{t}_{l}}^{\xi}+(1-\alpha)\Xi) \big\|_{L(\R^d,\R)}\Big)\bigg) \|\xi-\Xi\|_{\R^d}\bigg].
\end{split}
\end{align}
Next note that Lemma~\ref{lem:bnd_exp_term} and Lemma~\ref{lem:bnd_C_case_f''_bnded} imply that
\begin{multline}\label{eq:rate_sp_case_f''_bnded_9}
 \sup_{l\in\N_0}\bigg((l+1)^{1-2\epsilon}e^{-L\mathfrak{t}_{l}} \sup_{\alpha \in[0,1]}\Big( \big\| \psi'(\alpha \theta_{\mathfrak{t}_{l}}^{\xi}+(1-\alpha)\Xi) \big\|_{L(\R^d,\R)}\Big)\bigg) \|\xi-\Xi\|_{\R^d} \\
+ K(\lambda)\sup_{T\in [0,\infty)}R(T)  <\infty .
\end{multline}
Furthermore, observe that for all $k\in\N$ we have that
\begin{equation}
 (k+1)^{2\epsilon-1} \leq  k^{2\epsilon-1}.
\end{equation}
This, \eqref{eq:rate_sp_case_f''_bnded_9}, and \eqref{eq:rate_sp_case_f''_bnded_8} establish item~\eqref{item:rate_sp_case_f''_bnded_2}. The proof of Proposition~\ref{prop:rate_sp_case_f''_bnded} is thus completed.
\end{proof}

\begin{theorem}\label{cor:no_setting_discrete_version}
Let  $d\in \N$, $ \xi,\,\Xi\in \R^d$, $\epsilon\in (0,\nicefrac{1}{2})$,  $\eta, L,  c \in (0,\infty)$, $(\mathfrak{M}_n)_{n \in \N_0} \subseteq \N$,  $\psi\in C^2(\R^d,\R)$, let $(S,\mathcal{S})$ be a measurable space, let $(\Omega,\F,\P)$ be a probability space, let $Z_{m, n} \colon\Omega \to S$, $(m, n) \in \N^2$, be i.i.d.\ random variables,
let  $G = (G(x, s))_{(x, s) \in \R^d \times S}\colon \R^d\times S\to \R^d$ be  $(\mathcal{B}(\R^d)\otimes \mathcal{S})/\mathcal{B}(\R^d)$-measurable, let $g\colon\R^d\to\R^d$ be a function, assume for all $s \in S$ that $( \R^d \ni x \mapsto G(x,s)\in\R^d)\in C^2(\R^d,\R^d)$,
assume for all $x,y\in\R^d$  that 
\begin{equation}
\E\big[\|G(x,Z_{1,1})\|_{\R^d}^2 \big] \leq c \big[1+\|x\|_{\R^d} \big]^2, \quad \langle x-\Xi,g(x)\rangle_{\R^d} \leq -L \|g(x)\|_{\R^d}^2,
\end{equation}
\vspace{-.5cm}
\begin{equation}
g(x)=\E\big[ G(x,Z_{1,1}) \big],\quad \langle x-y,g(x)-g(y)\rangle_{\R^d} \leq -L \|x-y\|_{\R^d}^2,
\end{equation}
\vspace{-.4cm}
\begin{equation}
\max_{i \in \{1,2\}} \inf_{\delta\in(0,\infty)}\sup_{u\in [-\delta,\delta]^d}\E\Big[ \| (\tfrac{\partial^i}{\partial x^i}G)(x+u,Z_{1,1})\|_{L^{(i)}(\R^d,\R^d)}^{1+\delta}\Big]<\infty,
\end{equation}
and
\begin{equation}
\big\|\E\big[ (\tfrac{\partial^2}{\partial x^2}G)(x,Z_{1,1}) \big]\big\|_{L^{(2)}(\R^d,\R^d)} 
+ \max_{i \in \{1,2\}}
\|\psi^{(i)}(x)\|_{L^{(i)}(\R^d,\R)}  < c
\end{equation}
(cf.~Corollary~\ref{cor:derivative:gen}), and 
 let $\Theta \colon \N_0\times \Omega \to \R^d$ be the stochastic process which satisfies for all $n\in\N$ that $\Theta_0 = \xi$ and 
\begin{equation}
\Theta_n = \Theta_{n-1} + \frac{\eta}{n^{1-\epsilon} \mathfrak{M}_{n-1}} \textstyle\sum\limits_{j=1}^{\mathfrak{M}_{n-1}} G(\Theta_{n-1},Z_{n, j}).
\end{equation}
Then 
\begin{enumerate}[(i)]
\item we have that
$
\lb x\in\R^d\colon g(x)=0\rb = \lb \Xi\rb
$
and 
\item there exists  $C\in [0,\infty)$ such that for all $n\in\N$ we have that
\begin{equation}\label{eq:no_setting_discrete_version_10}
|\E[\psi(\Theta_{n})]- \psi(\Xi)|  \leq C n^{2\epsilon-1}.
\end{equation}
\end{enumerate}
\end{theorem}
\begin{proof}[Proof of Theorem~\ref{cor:no_setting_discrete_version}]
This is a direct consequence of Proposition~\ref{prop:rate_sp_case_f''_bnded}. The proof of Theorem~\ref{cor:no_setting_discrete_version} is thus completed.
\end{proof}

\section{Weak convergence rates for SAAs in the case of polynomially decaying learning rates  without mini-batches}
\label{subsec:SAA:without}
\sectionmark{}

\begin{cor}
\label{cor:no_setting_discrete_version_one_sample}
Let  $d\in \N$, $ \xi,\,\Xi\in \R^d$, $\epsilon\in (0,\nicefrac{1}{2})$,  $\eta, L,  c \in (0,\infty)$, $\psi\in C^2(\R^d,\R)$,  let $(S,\mathcal{S})$ be a measurable space, let $(\Omega,\F,\P)$ be a probability space, 
let  $G  = (G(x, s))_{(x, s) \in \R^d \times S} \colon \R^d\times S\to \R^d$ be  $(\mathcal{B}(\R^d)\otimes \mathcal{S})/\mathcal{B}(\R^d)$-measurable,
let $Z_n\colon\Omega \to S$, $n \in \N$, be i.i.d.\ random variables, let $g\colon\R^d\to\R^d$ be a function, assume for all $s \in S$ that $(\R^d\ni x \mapsto G(x,s)\in\R^d)\in C^2(\R^d,\R^d)$,
 assume for all $x,y\in\R^d$  that 
\begin{equation}
\E\big[\|G(x,Z_1)\|_{\R^d}^2 \big] \leq c \big[1+\|x\|_{\R^d} \big]^2, \quad \langle x-\Xi,g(x)\rangle_{\R^d} \leq -L \|g(x)\|_{\R^d}^2,
\end{equation}
\vspace{-.5cm}
\begin{equation}
g(x)=\E\big[ G(x,Z_1) \big],\quad \langle x-y,g(x)-g(y)\rangle_{\R^d} \leq -L \|x-y\|_{\R^d}^2,
\end{equation}
\vspace{-.4cm}
\begin{equation}
\max_{i \in \{1,2\}} \inf_{\delta\in(0,\infty)}\sup_{u\in [-\delta,\delta]^d}\E\Big[ \| (\tfrac{\partial^i}{\partial x^i}G)(x+u,Z_1)\|_{L^{(i)}(\R^d,\R^d)}^{1+\delta}\Big]<\infty,
\end{equation}
 and
\begin{equation}
\big\|\E\big[ (\tfrac{\partial^2}{\partial x^2}G)(x,Z_1) \big]\big\|_{L^{(2)}(\R^d,\R^d)} 
+ \max_{i \in \{1,2\}}
\|\psi^{(i)}(x)\|_{L^{(i)}(\R^d,\R)}  < c
\end{equation}
(cf.~Corollary~\ref{cor:derivative:gen}), and let $\Theta \colon \N_0\times \Omega \to \R^d$ be the stochastic process which satisfies for all $n\in\N$   that $\Theta_0 = \xi$ and 
\begin{equation}
\Theta_n = \Theta_{n-1} + \tfrac{\eta}{n^{1-\epsilon}} G(\Theta_{n-1},Z_n).
\end{equation}
Then 
\begin{enumerate}[(i)]
\item we have that
$
\lb x\in\R^d\colon g(x)=0\rb = \lb \Xi\rb
$
and 
\item there exists  $C \in \R$ such that for all $n\in\N$ we have that
\begin{equation}
|\E[\psi(\Theta_{n})]- \psi(\Xi)|  \leq C n^{2\epsilon-1}.
\end{equation}
\end{enumerate}
\end{cor}
\begin{proof}[Proof of Corollary~\ref{cor:no_setting_discrete_version_one_sample}]
This is a direct consequence of Theorem~\ref{cor:no_setting_discrete_version}. The proof of Corollary~\ref{cor:no_setting_discrete_version_one_sample} is thus completed.
\end{proof}

\section{SAAs for random rotation problems}
\label{subsec:example}
\sectionmark{}

\begin{lemma}\label{example:random_rotation}
Let $(\Omega,\mathcal{F},\mathbb{P})$ be a probability space, let $Z_n\colon \Omega\to [\nicefrac{\pi}{4},\nicefrac{5\pi}{4}]$, $n \in \N$, be  i.i.d.\ random variables, assume that $Z_1$  is continuous uniformly distributed on
 $(\nicefrac{\pi}{4},\nicefrac{5\pi}{4})$, let $A\colon [\nicefrac{\pi}{4},\nicefrac{5\pi}{4}]\to \R^{2\times 2}$ satisfy for all $s\in [\nicefrac{\pi}{4},\nicefrac{5\pi}{4}] $ that
\begin{equation}\label{eq:random_rotation_1}
A(s)=\begin{pmatrix}
\cos(s) & -\sin(s)\\
\sin(s) & \cos(s)\\
\end{pmatrix},
\end{equation}
let $G = (G(x, s))_{(x, s) \in \R^2 \times [\nicefrac{\pi}{4},\nicefrac{5\pi}{4}]} \colon \R^2\times [\nicefrac{\pi}{4},\nicefrac{5\pi}{4}]\to \R^2$ satisfy for all $x\in\R^2$, $s\in [\nicefrac{\pi}{4},\nicefrac{5\pi}{4}]$ that
\begin{equation}\label{eq:random_rotation_2}
G(x,s)=A(s)x,
\end{equation}
and let $g\colon\R^2\to\R^2$ satisfy for all $x\in\R^2$ that
\begin{equation}\label{eq:random_rotation_3}
g(x)=\E\big[G(x,Z_1)\big].
\end{equation}
Then
\begin{enumerate}[(i)]
\item\label{item:random_rotation_1} we have for all $x\in\R^2$ that
\begin{equation}\label{eq:random_rotation_4}
\E\big[A(Z_1)\big]=\tfrac{\sqrt{2}}{\pi}\begin{pmatrix}
-1 & -1 \\
1 & -1 \\
\end{pmatrix}
=\tfrac{2}{\pi}A\big(\tfrac{3\pi}{4}\big)
\end{equation}
and
\begin{equation}\label{eq:random_rotation_4_2}
g(x)=\tfrac{2}{\pi}A\big(\tfrac{3\pi}{4}\big)x,
\end{equation}
\item\label{item:random_rotation_2} we have for all $s\in [\nicefrac{\pi}{4},\nicefrac{5\pi}{4}]$ that
\begin{equation}\label{eq:random_rotation_6}
\big(\R^2\ni x \mapsto A(s)x\in \R^2\big)\in C^2(\R^2,\R^2),
\end{equation}
\item\label{item:random_rotation_3} we have for all $x\in\R^2$ that
\begin{equation}\label{eq:random_rotation_5}
\max_{i \in \{1,2\}} \sup_{u\in[-1,1]^2}\E\Big[ 
\|(\tfrac{\partial^i}{\partial x^i}G)(x+u,Z_1)\|_{L^{(i)}(\R^2,\R^2)}^2
 \Big]<\infty,
\end{equation}
\item\label{item:random_rotation_3_02} we have for all $s\in [\nicefrac{\pi}{4},\nicefrac{5\pi}{4}]$, $x\in\R^2$ that
\begin{equation}\label{eq:random_rotation_5_02}
\|A(s)x\|_{\R^2}=\|x\|_{\R^2}, 
\qquad 
\|g(x)\|_{\R^2}=\tfrac{2}{\pi}\|x\|_{\R^2},
\end{equation}
and 
\begin{equation}\label{eq:random_rotation_5_03}
\E\big[\|G(x,Z_1)\|_{\R^2}^2\big]=\|x\|_{\R^2}^2,
\end{equation}
\item\label{item:random_rotation_4} we have for all $s\in [\nicefrac{\pi}{4},\nicefrac{5\pi}{4}]$, $x\in\R^2$ that
\begin{equation}\label{eq:random_rotation_7}
\langle A(s)x,x\rangle_{\R^2}= \cos(s)\|x\|_{\R^2}^2,
\end{equation}
and 
\item\label{item:random_rotation_5} we have for all  $x,y\in\R^2$ that
\begin{equation}\label{eq:random_rotation_8}
\langle x-y,g(x)-g(y)\rangle_{\R^2}= -\tfrac{\sqrt{2}}{\pi}\|x-y\|_{\R^2}^2.
\end{equation}
\end{enumerate}
\end{lemma}
\begin{proof}[Proof of Lemma~\ref{example:random_rotation}]
First, observe that
\begin{equation}\label{eq:random_rotation_9}
\begin{split}
\E\big[ \!\cos(Z_1) \big]
 = \int_{\nicefrac{\pi}{4}}^{\nicefrac{5\pi}{4}}\cos(s)\tfrac{1}{\pi}\, ds
 = \tfrac{1}{\pi}\big[\sin(s)\big]_{s=\nicefrac{\pi}{4}}^{s=\nicefrac{5\pi}{4}}
 = \tfrac{1}{\pi}(-\tfrac{\sqrt{2}}{2}-\tfrac{\sqrt{2}}{2})=-\tfrac{\sqrt{2}}{\pi}
\end{split}
\end{equation}
and
\begin{equation}\label{eq:random_rotation_10}
\begin{split}
\E\big[ \!\sin(Z_1) \big]
 = \int_{\nicefrac{\pi}{4}}^{\nicefrac{5\pi}{4}}\sin(s)\tfrac{1}{\pi}\, ds
 = -\tfrac{1}{\pi}\big[\cos(s)\big]_{s=\nicefrac{\pi}{4}}^{s=\nicefrac{5\pi}{4}}
 = -\tfrac{1}{\pi}(-\tfrac{\sqrt{2}}{2}-\tfrac{\sqrt{2}}{2})=\tfrac{\sqrt{2}}{\pi}.
\end{split}
\end{equation}
This and \eqref{eq:random_rotation_1}  prove \eqref{eq:random_rotation_4}. Combining this, \eqref{eq:random_rotation_2}, and \eqref{eq:random_rotation_3} demonstrates \eqref{eq:random_rotation_4_2}. This establishes item~\eqref{item:random_rotation_1}.
Moreover, note that item~\eqref{item:random_rotation_2} is obvious.
Next observe that for all $x\in\R^2$, $s\in [\nicefrac{\pi}{4},\nicefrac{5\pi}{4}] $ we have that
\begin{equation}\label{eq:random_rotation_11}
(\tfrac{\partial}{\partial x}G)(x,s)=A(s)\qquad\text{and}\qquad (\tfrac{\partial^2}{\partial x^2}G)(x,s)=0.
\end{equation}
Furthermore, note that for all $s\in [\nicefrac{\pi}{4},\nicefrac{5\pi}{4}]$, $x=(x_1,x_2)\in\R^2$ we have that
\begin{equation}\label{eq:random_rotation_11_02}
\begin{split}
\|A(s)x\|^2_{\R^2} &= (\cos(s) x_1 - \sin(s) x_2)^2 + (\sin(s) x_1 + \cos(s) x_2)^2\\
& = \cos(s)^2x_1^2 - 2 \cos(s)\sin(s)x_1x_2 + \sin(s)^2x_2^2\\
& \quad + \sin(s)^2x_1^2 + 2 \cos(s)\sin(s)x_1x_2 + \cos(s)^2x_2^2\\
& = x_1^2 + x_2^2 = \|x\|_{\R^2}^2.
\end{split}
\end{equation}
This and \eqref{eq:random_rotation_11} establish item~\eqref{item:random_rotation_3}.
Next observe that \eqref{eq:random_rotation_11_02} and \eqref{eq:random_rotation_4_2} prove that for all $x \in \R^2$ we have that
\begin{equation}\label{eq:random_rotation_11_04}
\|g(x)\|_{\R^2}=\tfrac{2}{\pi}\|x\|_{\R^2}.
\end{equation}
In addition, observe that  \eqref{eq:random_rotation_11_02} and \eqref{eq:random_rotation_2} ensure that for all $x\in\R^2$ we have that
\begin{equation}\label{eq:random_rotation_11_05}
\E\big[\|G(x,Z_1)\|_{\R^2}^2\big]
= \E\big[\|A(Z_1)x\|_{\R^2}^2\big]
= \E\big[\|x\|_{\R^2}^2\big]
= \|x\|_{\R^2}^2.
\end{equation}
Combining this, \eqref{eq:random_rotation_11_02}, and \eqref{eq:random_rotation_11_04} establishes item~\eqref{item:random_rotation_3_02}.
Next note that for all $x=(x_1,x_2)\in\R^2$, $s\in [\nicefrac{\pi}{4},\nicefrac{5\pi}{4}] $ we have that
\begin{align*}
\langle A(s)x,x\rangle_{\R^2}
& = \langle (x_1\cos(s) - x_2\sin(s), x_1\sin(s) + x_2\cos(s)),(x_1, x_2)\rangle_{\R^2}\\
& = x_1^2\cos(s) + x_2^2\cos(s)\\
& = \cos(s)\|x\|_{\R^2}^2. \numberthis \label{eq:random_rotation_12}
\end{align*}
This proves item~\eqref{item:random_rotation_4}.
Combining this with~\eqref{eq:random_rotation_4_2}  assures that for all $x,y\in\R^2$ we have that
\begin{equation}\label{eq:random_rotation_13}
\begin{split}
\langle x - y, g(x) - g(y)\rangle_{\R^2}
&= \tfrac{2}{\pi}\langle x - y, A\big(\tfrac{3\pi}{4}\big)(x-y)\rangle_{\R^2} \\
& = \tfrac{2}{\pi}\langle A\big(\tfrac{3\pi}{4}\big)(x-y), (x-y) \rangle_{\R^2} \\
&=  \tfrac{2}{\pi} \cos\!\big(\tfrac{3 \pi}{4}\big) \| x - y \|_{\R^2}^2 = -\tfrac{\sqrt{2}}{\pi}\| x - y \|_{\R^2}^2. 
\end{split}
\end{equation}
This  establishes item~\eqref{item:random_rotation_5}.
The proof of Lemma~\ref{example:random_rotation} is thus completed.
\end{proof}

\begin{cor}
	\label{cor:example}
Let $\xi \in \R^2$, $\varepsilon \in (0, \nicefrac{1}{2})$, $\eta \in (0, \infty)$, $\psi \in C^2(\R^2, \R)$ satisfy that $\sup_{x \in \R^2} \max_{i \in \{1, 2\}} \allowbreak \|\psi^{(i)} (x)\|_{L^{(i)}(\R^2, \R)}< \infty$,
let $(\Omega,\mathcal{F},\mathbb{P})$ be a probability space, let $Z_n\colon \Omega\to [\nicefrac{\pi}{4},\nicefrac{5\pi}{4}]$, $n \in \N$, be  i.i.d.\ random variables,
assume that $Z_1$  is continuous uniformly distributed on 
$(\nicefrac{\pi}{4},\nicefrac{5\pi}{4})$, let $A\colon [\nicefrac{\pi}{4},\nicefrac{5\pi}{4}]\to \R^{2\times 2}$ satisfy for all $s\in [\nicefrac{\pi}{4},\nicefrac{5\pi}{4}] $ that
\begin{equation}
A(s)=\begin{pmatrix}
\cos(s) & -\sin(s)\\
\sin(s) & \cos(s)\\
\end{pmatrix},
\end{equation}
let $G = (G(x, s))_{(x, s) \in \R^2 \times [\nicefrac{\pi}{4},\nicefrac{5\pi}{4}]} \colon \R^2\times [\nicefrac{\pi}{4},\nicefrac{5\pi}{4}]\to \R^2$ satisfy for all $x\in\R^2$, $s\in [\nicefrac{\pi}{4},\nicefrac{5\pi}{4}]$ that
\begin{equation}
G(x,s)=A(s)x,
\end{equation}
let $g\colon\R^2\to\R^2$ satisfy for all $x\in\R^2$ that
\begin{equation}
g(x)=\E\big[G(x,Z_1)\big],
\end{equation}
and let $\Theta \colon \N_0\times \Omega \to \R^2$ be the stochastic process which satisfies for all $n\in\N$   that $\Theta_0 = \xi$ and 
\begin{equation}
\Theta_n = \Theta_{n-1} + \tfrac{\eta}{n^{1-\epsilon}} G(\Theta_{n-1},Z_n).
\end{equation}
Then 
\begin{enumerate}[(i)]
	\item
	\label{item:ex:1}
we have that $\lb x \in \R^2 \colon  g(x) = 0 \rb = \lb 0 \rb$
	and 
	\item
	\label{item:ex:2}
 there exists  $C \in \R$ such that for all $n\in\N$ we have that
	\begin{equation}
	|\E [\psi(\Theta_{n})]- \psi(0)|  \leq C n^{2\epsilon-1}.
	\end{equation}
\end{enumerate}
\end{cor}
\begin{proof}[Proof of Corollary~\ref{cor:example}]
First, note that item~\eqref{item:random_rotation_3_02}  in Lemma~\ref{example:random_rotation} proves that for all $x \in \R^2$ we have that
\begin{align}
\label{eq:example:F}
 \E\big[\|G(x,Z_1)\|_{\R^2}^2\big]= \|x\|_{\R^2}^2 \leq  \big[ 1+ \|x\|_{\R^2}\big]^2
\end{align}
and 
\begin{align}
\label{eq:example:f}
 \|g(x)\|_{\R^2}=\tfrac{2}{\pi}\|x\|_{\R^2}. 
\end{align}
Next observe that item~\eqref{item:random_rotation_3} in Lemma~\ref{example:random_rotation} establishes for all $x \in \R^2$ that 
\begin{align}
\label{eq:example:sup}
\begin{split}
 \max_{i \in \{1,2\}} \sup_{u\in[-1,1]^2}\E\Big[ 
\|(\tfrac{\partial^i}{\partial x^i}G)(x+u,Z_1)\|_{L^{(i)}(\R^2,\R^2)}^2
\Big]<\infty.
\end{split} 
\end{align}
Moreover, note that item~\eqref{item:random_rotation_5} in  Lemma~\ref{example:random_rotation} ensures that for all $x, y \in \R^2$ we have that
\begin{align}
\label{eq:example:scalar}
\begin{split}
& \langle x - y, g(x) - g(y)\rangle_{\R^2}  = -\tfrac{\sqrt{2}}{\pi}\| x - y \|_{\R^2}^2.
\end{split}
\end{align}
This and \eqref{eq:example:f} ensure that for all $x \in \R^2$ we have that
\begin{align}
\label{eq:example:0}
\langle x , g(x) \rangle_{\R^2} = \langle x-0 , g(x) - g(0) \rangle_{\R^2} = -\tfrac{\sqrt{2}}{\pi}\| x  \|_{\R^2}^2 = -\tfrac{\pi \sqrt{2}}{4}\| g(x)  \|_{\R^2}^2.
\end{align}
In addition, observe that for all $x \in \R^2$,  $s\in [\nicefrac{\pi}{4},\nicefrac{5\pi}{4}] $ we have that
\begin{align}
(\tfrac{\partial^2}{\partial x^2} G)(x,s)=0.
\end{align}
This reveals that for all  $s\in [\nicefrac{\pi}{4},\nicefrac{5\pi}{4}] $  it holds that
\begin{align}
\sup_{x \in \R^2} \| (\tfrac{\partial^2}{\partial x^2} G)(x,s) \|_{L^{(2)}(\R^2,\R^2)} < \infty.
\end{align}
Combining this with Corollary~\ref{cor:no_setting_discrete_version_one_sample}, \eqref{eq:example:F}, \eqref{eq:example:sup}, \eqref{eq:example:scalar}, \eqref{eq:example:0}, and the assumption that $\sup_{x \in \R^2} \max_{i \in \{1, 2\}} \allowbreak \|\psi^{(i)} (x)\|_{L^{(i)}(\R^2, \R)}< \infty$ establishes item~\eqref{item:ex:1} and item~\eqref{item:ex:2}. The proof of Corollary~\ref{cor:example} is thus completed.
\end{proof}

\chapter{Weak  error estimates for stochastic gradient descent (SGD) optimization methods}
\label{sec:SGD}
\chaptermark{}

 In this chapter we apply the weak error analysis results for SAAs from Chapter~\ref{sec:SAA:poly} above to establish weak error estimates for SGD optimization methods. In particular, in \Cref{cor:conv:sgd:nof} in Section~\ref{subsec:SGD:linear} below we establish weak  error estimates for SGD optimization methods in the case of objective functions with linearly growing derivatives. In our proof of  \Cref{cor:conv:sgd:nof} we employ the weak  error estimates for SGD optimization methods in the case of coercive objective functions
 in 	Corollary~\ref{cor:conv:sgd}  in Section~\ref{subsec:SGD:coercive} below. Our proof of Corollary~\ref{cor:conv:sgd}, in turn, uses the elementary result on derivatives of gradients of smooth functions in \Cref{lem:der:gradient}  in Section~\ref{subsec:SGD:coercive}  below and  the weak convergence result for SAAs in Corollary~\ref{cor:no_setting_discrete_version_one_sample} in Section~\ref{subsec:SAA:without} above. 

\section{Weak  error estimates for SGD optimization methods in the case of coercive objective functions}
\label{subsec:SGD:coercive}
\sectionmark{}

\begin{lemma}
	\label{lem:der:gradient}
	Let $d, n \in \N$,  $ f \in C^n(\R^d, \R)$ and let $g \colon  \R^d \to \R^d$ satisfy for all $x \in \R^d$ that $
	g(x) = (\nabla f) (x).
	$
	Then 
	\begin{enumerate}[(i)]
		\item
		\label{item:cont:g} we have that 
	$
		g \in C^{(n-1)}(\R^d, \R^d),
		$
		\item
		\label{item:inner:product} we have for all $k \in \{1, 2, \ldots, n\}$,  $x , y_1, y_2, \ldots, y_k \in \R^d$  that
		\begin{equation}
		\label{eq:rel:gradient}
		f^{(k)}(x)(y_1, y_2, \ldots, y_k) = \langle g^{(k-1)}(x)(y_2, y_3, \ldots, y_{k}), y_1 \rangle_{\R^d},
		\end{equation}
		and
		\item
		\label{item:norm} we have for all $k \in \{1, 2, \ldots, n\}$, $x \in \R^d$  that
		\begin{equation}
		\|f^{(k)}(x) \|_{L^{(k)}(\R^d, \R)} = \|g^{(k-1)} (x) \|_{L^{(k-1)}(\R^d, \R^d)}.
		\end{equation}
	\end{enumerate}
\end{lemma}
\begin{proof}[Proof of Lemma~\ref{lem:der:gradient}]
	First, note that the hypothesis that $f \in C^n(\R^d, \R)$ establishes item~\eqref{item:cont:g}.
	Next we prove item~\eqref{item:inner:product} by induction on $k \in \{1, 2, \ldots, n\}$. For the base case $k=1$ note that for all $x, y \in \R^d$ we have that
	\begin{equation}
	f'(x)(y)= \langle g(x), y \rangle_{\R^d}.
	\end{equation}
	This proves \eqref{eq:rel:gradient} in the base case $k=1$. For the induction step $ \{1, 2, \ldots, n-1\} \ni k \to k+1 \in \{2, 3, \ldots, n\}$ let $k \in \{1, 2, \ldots, n-1\}$ satisfy
	for all $x , y_1, y_2, \ldots, y_k \in \R^d$  that
	\begin{equation}
	\label{eq:ind:hypo}
	f^{(k)}(x)(y_1, y_2, \ldots, y_k) = \langle g^{(k-1)}(x)(y_2, y_3, \ldots, y_{k}), y_1 \rangle_{\R^d}.
	\end{equation}
	Next observe that   item~\eqref{item:cont:g} ensures that for all $x,  y_2, y_3, \ldots, y_{k+1} \in \R^d$ we have that
	\begin{align}
	\begin{split}
	\limsup_{\substack{h \to 0\\
		h \in \R \backslash \{0\}}}  &\bigg\| \frac{g^{(k-1)}(x+h y_{k+1})(y_2, y_3, \ldots, y_k) - g^{(k-1)}(x)(y_2, y_3, \ldots, y_k)}{h} \\
	& \quad - g^{(k)}(x)(y_2, y_3, \ldots, y_k, y_{k+1})\bigg\|_{\R^d} = 0. 
	\end{split}
	\end{align}
	The Cauchy-Schwartz inequality hence implies that for all $x, y_1, y_2, \ldots, y_{k+1} \in \R^d$ we have that
	\begin{align}
	\begin{split}
\limsup_{\substack{h \to 0\\
			h \in \R \backslash \{0\}}} &  \bigg|\frac{ \langle g^{(k-1)}(x+h y_{k+1})(y_2, y_3, \ldots, y_k) - g^{(k-1)}(x)(y_2, y_3, \ldots, y_k), y_1 \rangle_{\R^d}}{h} \\
&\quad	-\langle g^{(k)}(x)(y_2, y_3, \ldots, y_k, y_{k+1}), y_1 \rangle_{\R^d} \bigg| \\
\leq  \limsup_{\substack{h \to 0\\
			h \in \R \backslash \{0\}}} & \bigg\| \frac{g^{(k-1)}(x+h y_{k+1})(y_2, y_3, \ldots, y_k) - g^{(k-1)}(x)(y_2, y_3, \ldots, y_k)}{h} \\
&\quad	- g^{(k)}(x)(y_2, y_3, \ldots, y_k, y_{k+1})\bigg\|_{\R^d} \| y_1\|_{\R^d} = 0.
	\end{split}
	\end{align}
	The induction hypothesis (see~\eqref{eq:ind:hypo}) therefore assures that for all $x, y_1, y_2, \ldots, y_{k+1} \in \R^d$ we have that
	\begin{align*}
	 \limsup_{\substack{h \to 0\\
			h \in \R \backslash \{0\}}} & \bigg|\frac{  f^{(k)}(x+h y_{k+1})(y_1, y_2, \ldots, y_k) - f^{(k)}(x)(y_1, y_2, \ldots, y_k)}{h} \\
	& \quad  -\langle g^{(k)}(x)(y_2, y_3, \ldots, y_{k+1}), y_1 \rangle_{\R^d} \bigg| \numberthis \\
	=\limsup_{\substack{h \to 0\\
			h \in \R \backslash \{0\}}} & \bigg|\frac{ \langle g^{(k-1)}(x+h y_{k+1})(y_2, y_3, \ldots, y_k) - g^{(k-1)}(x)(y_2, y_3, \ldots, y_k), y_1 \rangle_{\R^d}}{h}  \\
		& \quad
	-\langle g^{(k)}(x)(y_2, y_3, \ldots, y_{k+1}), y_1 \rangle_{\R^d} \bigg| = 0. 
	\end{align*}
	This and the assumption that $f \in C^{n}(\R^d, \R)$ demonstrates that for all $x, y_1, y_2, \ldots, \allowbreak y_{k+1} \in \R^d$ we have that
	\begin{equation}
	f^{(k+1)}(x)(y_1, y_2, \ldots,  y_{k+1}) = \langle g^{(k)}(x)(y_2, y_3,  \ldots, y_{k+1}), y_1 \rangle_{\R^d}.
	\end{equation}
	Induction thus proves item~\eqref{item:inner:product}. Next observe that item~\eqref{item:inner:product} implies that for all $k \in \{1, 2, \ldots, n\}$, $x \in \R^d$ it holds  that
	\begin{align}
	\begin{split}
	&\|f^{(k)}(x) \|_{L^{(k)}(\R^d, \R)} = \sup_{y_1, y_2, \ldots, y_k \in \R^d \backslash \{0\}} \frac{|f^{(k)}(x)(y_1, y_2, \ldots, y_k)|}{\|y_1\|_{\R^d} \|y_2\|_{\R^d} \ldots \|y_k\|_{\R^d}}\\
	&= \sup_{y_1, y_2,  \ldots, y_k \in \R^d \backslash \{0\}} \frac{|\langle g^{(k-1)}(x)(y_2, y_3, \ldots, y_{k}), y_1 \rangle_{\R^d}|}{\|y_1\|_{\R^d} \|y_2\|_{\R^d} \ldots \|y_k\|_{\R^d}}\\
	& = \sup_{y_2, y_3,  \ldots, y_k \in \R^d \backslash \{0\}}  \left[ \sup_{y_1 \in \R^d \backslash \{0\}}  \frac{|\langle g^{(k-1)}(x)(y_2, y_3, \ldots, y_{k}), y_1 \rangle_{\R^d}|}{\|y_1\|_{\R^d} \|y_2\|_{\R^d} \ldots \|y_k\|_{\R^d}}\right] \\
	& = \sup_{y_2, y_3,  \ldots, y_k \in \R^d \backslash \{0\}}    \frac{\| g^{(k-1)}(x)(y_2, y_3, \ldots, y_{k})\|_{\R^d}}{\|y_2\|_{\R^d} \|y_3\|_{\R^d}  \ldots \|y_k\|_{\R^d}} \\
	&= \|g^{(k-1)} (x) \|_{L^{(k-1)}(\R^d, \R^d)}.
	\end{split}
	\end{align}
	This establishes item~\eqref{item:norm}. The proof of Lemma~\ref{lem:der:gradient} is thus completed.
\end{proof}

\begin{cor}\label{cor:conv:sgd}
	Let  $d\in \N$, $ \xi,\,\Xi\in \R^d$, $\epsilon\in (0,\nicefrac{1}{2})$,  $\eta, L, c \in (0,\infty)$, $\psi\in C^2(\R^d,\R)$,  let $(S,\mathcal{S})$ be a measurable space, let $(\Omega,\F,\P)$ be a probability space, 
	let  $F  =  (F(x, s))_{(x, s) \in \R^d \times S} \colon \R^d\times S\to \R$ be  $(\mathcal{B}(\R^d)\otimes \mathcal{S})/\mathcal{B}(\R)$-measurable,
	let $Z_n\colon\Omega \to S$, $n \in \N$, be i.i.d.\ random variables, let $g\colon\R^d\to\R^d$ be a function, assume for all $s \in S$ that $( \R^d\ni x \mapsto F(x,s)\in\R\big)\in C^3(\R^d,\R)$, assume  for all $x,y\in\R^d$ that 
	\begin{gather}
	\label{eq:gradient:linear}
	\E\big[\|(\nabla_{x} F)(x, Z_1)\|_{\R^d}^2 \big]
	 \leq c \big[1+\|x\|_{\R^d}^2 \big], \qquad 	\langle x-\Xi,g(x)\rangle_{\R^d} \leq -L \|g(x)\|_{\R^d}^2, \\
	\label{eq:gradient:exp}
	g(x)=\E[ ( \nabla_x F)(x,Z_1)],\quad \langle x-y,g(x)-g(y)\rangle_{\R^d} \leq -L \|x-y\|_{\R^d}^2,
	 \\
	\label{eq:cond:bound:G}
	\max_{i \in \{2,3\}} \inf_{\delta\in(0,\infty)}\sup_{u\in [-\delta,\delta]^d}\E\Big[ \| (\tfrac{\partial^i}{\partial x^i}F)(x+u,Z_1)\|_{L^{(i)}(\R^d,\R)}^{1+\delta}\Big]<\infty,
	\end{gather}
	and
	\begin{equation}
	\label{eq:psi:G}
	\begin{split}
	\big\|\E\big[ (\tfrac{\partial^3}{\partial x^3}F)(x,Z_1) \big]\big\|_{L^{(3)}(\R^d,\R)} 
	+ \max_{i \in \{1,2\}}
	\|\psi^{(i)}(x)\|_{L^{(i)}(\R^d,\R)} < c
	\end{split}
	\end{equation}
	(cf.~Corollary~\ref{cor:derivative:gen}),	and let $\Theta \colon \N_0\times \Omega \to \R^d$ be the stochastic process which satisfies for all $n\in\N$   that $\Theta_0 = \xi$ and 
	\begin{equation}
	\label{eq:gradient:Theta}
	\Theta_n = \Theta_{n-1} + \tfrac{\eta}{n^{1-\epsilon}} (\nabla_x F)(\Theta_{n-1},Z_n).
	\end{equation}
	Then 
	\begin{enumerate}[(i)]
		\item we have that
		$
		\lb x\in\R^d\colon g(x)=0\rb = \lb \Xi\rb
	$
		and 
		\item there exists  $C \in \R$ such that for all $n\in\N$ we have that
		\begin{equation}
		|\E[\psi(\Theta_{n})]- \psi(\Xi)|  \leq C n^{2\epsilon-1}.
		\end{equation}
	\end{enumerate}
\end{cor}
\begin{proof}[Proof of Corollary~\ref{cor:conv:sgd}]
	Throughout this proof let $G  = (G(x, s))_{(x, s) \in \R^d \times S}  \colon \R^d \times S \to \R^d$ satisfy for all $x \in \R^d$, $s \in S$ that
	\begin{align}
	\label{eq:gradient:F}
	G(x, s) = (\nabla_x F)(x, s).
	\end{align}
	Observe that the hypothesis that $ \forall \, s \in S \colon ( \R^d\ni x\mapsto F(x,s)\in\R)\in C^3(\R^d,\R)$ ensures that for all $s \in S$ we have that
	\begin{align}
	\label{eq:gradient:C2}
	\big( \R^d\ni x \mapsto G(x,s)\in\R^d\big)\in C^2(\R^d,\R^d).
	\end{align}
	In addition, note that \eqref{eq:gradient:F} and \eqref{eq:gradient:linear} imply that 
	\begin{equation}
\label{eq:gradient:linear:F}
	\sup_{x\in\R^d}\left(\frac{\E\big[\|G(x, Z_1)\|_{\R^d}^2 \big]}{\big[1+\|x\|_{\R^d} \big]^2}\right) \leq c.
	\end{equation}
	Next observe that item~\eqref{item:norm} in Lemma~\ref{lem:der:gradient} (with $d =d$, $n =3$, $f = ( \R^d \ni x \mapsto F(x,s) \in \R) \in C^3(\R^d, \R)$, $g = ( \R^d \ni x \mapsto G(x,s) \in \R^d)$ for $s \in S$ in the notation of Lemma~\ref{lem:der:gradient})  assures that for all $i \in \{1, 2\}$, $x \in \R^d$, $s \in S$ we have that 
	\begin{align}
	\begin{split}
	\| (\tfrac{\partial^i}{\partial x^i}G)(x, s)\|_{L^{(i)}(\R^d,\R^d)} = \| (\tfrac{\partial^{i+1}}{\partial x^{i+1}}F)(x, s)\|_{L^{(i+1)}(\R^d,\R)}.
	\end{split}
	\end{align}
	This and \eqref{eq:cond:bound:G} demonstrate that for all $x \in \R^d$ we have that	
	\begin{equation}
	\label{eq:gradient:max}
	\max_{i \in \{1,2\}} \inf_{\delta\in(0,\infty)}\sup_{u\in [-\delta,\delta]^d}\E\Big[ \| (\tfrac{\partial^i}{\partial x^i}G)(x+u,Z_1)\|_{L^{(i)}(\R^d,\R^d)}^{1+\delta}\Big]<\infty.
	\end{equation}
	Jensen's inequality hence proves that for all $x \in \R^d$ we have that	
	\begin{equation}
	\label{eq:gradient:finite}
	\max_{i \in \{1,2\}} \E\Big[ \| (\tfrac{\partial^i}{\partial x^i}G)(x,Z_1)\|_{L^{(i)}(\R^d,\R^d)}\Big]<\infty.
	\end{equation}
	Moreover, observe that for all $y_1, y_2, y_3 \in \R^d$ we have that $(L^{(3)}(\R^d, \R) \ni A \mapsto \allowbreak A(y_1, y_2, y_3) \in \R)$ is a continuous linear function. This ensures that for all vectors $y_1, y_2, y_3 \in \R^d$ and all random variables $ A \colon \Omega \to L^{(3)}(\R^d, \R)$  with $\E[\|A\|_{L^{(3)}(\R^d, \R)}] < \infty$ we have that $\E[|A(y_1,y_2,y_3)|] < \infty$ and
	\begin{equation}
	\E [A](y_1, y_2, y_3) = \E[A(y_1, y_2, y_3)].
	\end{equation}
	Combining this, Corollary~\ref{cor:derivative:gen}, and item~\eqref{item:inner:product} in Lemma~\ref{lem:der:gradient} (with $d =d$, $n =3$, $f = ( \R^d \ni x \mapsto F(x,Z_1(\omega)) \in \R) \in C^3(\R^d, \R)$, $g = ( \R^d \ni x \mapsto G(x,Z_1(\omega)) \in \R^d)$ for $\omega \in \Omega$ in the notation of Lemma~\ref{lem:der:gradient}) implies that for all $x, y_1, y_2, y_3 \in \R^d$ we have that
	\begin{align}
	\label{eq:G:cube}
	\begin{split}
	\E\big[ (\tfrac{\partial^3}{\partial x^3}F)(x,Z_1) \big](y_1, y_2, y_3) &=  \E\big[ (\tfrac{\partial^3}{\partial x^3}F)(x,Z_1) (y_1, y_2, y_3) \big]\\
	& = \E \big[ \langle (\tfrac{\partial^2}{\partial x^2}G)(x,Z_1) (y_2, y_3), y_1  \rangle_{\R^d} \big].
	\end{split}
	\end{align}
	Moreover, note that \eqref{eq:gradient:finite} assures that for all $x, y_1, y_2, y_3 \in \R^d$ we have that
	\begin{align}
	\begin{split}
	&\E \big[ \big| \langle (\tfrac{\partial^2}{\partial x^2}G)(x,Z_1) (y_2, y_3), y_1  \rangle_{\R^d} \big| \big] 
	\leq \E \big[ \|(\tfrac{\partial^2}{\partial x^2}G)(x,Z_1) (y_2, y_3)\|_{\R^d} \| y_1  \|_{\R^d} \big]\\
	& \leq \E \big[ \|(\tfrac{\partial^2}{\partial x^2}G)(x,Z_1) \|_{L^{(2)}(\R^d, \R^d)} \big] \|y_2\|_{\R^d} \|y_3\|_{\R^d} \| y_1  \|_{\R^d}  < \infty.
	\end{split}
	\end{align}
	This and \eqref{eq:G:cube} prove that for all $x, y_1, y_2, y_3 \in \R^d$ we have that
	\begin{align}
	\begin{split}
	\E\big[ (\tfrac{\partial^3}{\partial x^3}F)(x,Z_1) \big](y_1, y_2, y_3)  = \big\langle \E \big[  (\tfrac{\partial^2}{\partial x^2}G)(x,Z_1) (y_2, y_3)  \big], y_1   \big \rangle_{\R^d}.
	\end{split}
	\end{align}
	This reveals that for all $x \in \R^d$  it holds that
	\begin{align}
	\begin{split}
	&\big\|\E\big[ (\tfrac{\partial^3}{\partial x^3}F)(x,Z_1) \big]\big\|_{L^{(3)}(\R^d,\R)} = \sup_{y_1, y_2, y_3 \in \R^d \backslash \{0\}} \frac{ \big| \E\big[ (\tfrac{\partial^3}{\partial x^3}F)(x,Z_1) \big](y_1, y_2, y_3) \big|}{\|y_1\|_{\R^d} \|y_2\|_{\R^d} \|y_3\|_{\R^d}}\\
	& = \sup_{y_1, y_2, y_3 \in \R^d \backslash \{0\}} \frac{ \big|\big\langle \E \big[  (\tfrac{\partial^2}{\partial x^2}G)(x,Z_1) (y_2, y_3)  \big], y_1   \big \rangle_{\R^d} \big|}{\|y_1\|_{\R^d} \|y_2\|_{\R^d} \|y_3\|_{\R^d}}\\
	& = \sup_{ y_2, y_3 \in \R^d \backslash \{0\}} \left[ \sup_{y_1 \in \R^d \backslash \{0\}}\frac{ \big| \big\langle \E \big[  (\tfrac{\partial^2}{\partial x^2}G)(x,Z_1) (y_2, y_3)  \big], y_1   \big \rangle_{\R^d} \big|}{\|y_1\|_{\R^d} \|y_2\|_{\R^d} \|y_3\|_{\R^d}} \right] \\
	& = \sup_{ y_2, y_3 \in \R^d \backslash \{0\}} \frac{\big\| \E \big[  (\tfrac{\partial^2}{\partial x^2}G)(x,Z_1) (y_2, y_3)  \big] \big\|_{\R^d}}{\|y_2\|_{\R^d} \|y_3\|_{\R^d}} = \big\|\E\big[ (\tfrac{\partial^2}{\partial x^2}G)(x,Z_1) \big]\big\|_{L^{(2)}(\R^d,\R^d)}.
	\end{split}
	\end{align}
	This and \eqref{eq:psi:G} demonstrate that for all $x \in \R^d$ we have that
	\begin{align}
	\label{eq:gradient:sup}
\big\|\E\big[ (\tfrac{\partial^2}{\partial x^2}G)(x,Z_1) \big]\big\|_{L^{(2)}(\R^d,\R^d)} 
	+ \max_{i \in \{1,2\}}
	\|\psi^{(i)}(x)\|_{L^{(i)}(\R^d,\R)} < c.
	\end{align}
	Moreover, note that combining \eqref{eq:gradient:Theta} and \eqref{eq:gradient:F} ensures that for all $n \in \N$ we have that
	$\Theta_0 = \xi$ and 
	\begin{equation}
	\label{eq:gradient:F:Theta}
	\Theta_n = \Theta_{n-1} + \tfrac{\eta}{n^{1-\epsilon}} G(\Theta_{n-1},Z_n).
	\end{equation}
	Next observe that, e.g., \cite[Lemma~4.4]{Wurstemberger2018} proves that $G$ is $(\mathcal{B}(\R^d)\otimes \mathcal{S})/\mathcal{B}(\R^d)$-measurable.
	Corollary~\ref{cor:no_setting_discrete_version_one_sample} (with 
	$d = d$, $\xi = \xi$, $\Xi = \Xi$, $\epsilon = \epsilon$,  $\eta = \eta$, $L = L $,  $c = c$, $\psi = \psi$, $(S,\mathcal{S}) = (S,\mathcal{S})$, $(\Omega,\F,\P) = (\Omega,\F,\P)$, $Z_n = Z_n$,
	$G=G$, $g= g$, $\Theta = \Theta$ 
	for $n \in \N$ in the notation of Corollary~\ref{cor:no_setting_discrete_version_one_sample}), \eqref{eq:gradient:C2}, \eqref{eq:gradient:linear:F}, \eqref{eq:gradient:max}, \eqref{eq:gradient:sup}, \eqref{eq:gradient:F:Theta}, \eqref{eq:gradient:exp}, and \eqref{eq:gradient:linear} therefore assure that $\lb x\in\R^d\colon g(x)=0 \rb = \lb \Xi\rb$ and that there exists  $C \in \R$ such that for all $n\in\N$ we have that
	\begin{equation}
	|\E[\psi(\Theta_{n})]- \psi(\Xi)|  \leq C n^{2\epsilon-1}.
	\end{equation}
	The proof of Corollary~\ref{cor:conv:sgd} is thus completed.
\end{proof}

\section{Weak  error estimates for SGD optimization methods in the case of objective functions with linearly growing derivatives}
\label{subsec:SGD:linear}
\sectionmark{}

\begin{cor}\label{cor:conv:sgd:nof}
	Let  $d\in \N$, $ \xi,\,\Xi\in \R^d$, $\epsilon\in (0,\nicefrac{1}{2})$,  $\eta, L, c \in (0,\infty)$, $\psi\in C^2(\R^d,\R)$,  let $(S,\mathcal{S})$ be a measurable space, let $(\Omega,\F,\P)$ be a probability space,
let  $F  =  (F(\theta, s))_{(\theta, s) \in \R^d \times S} \colon \R^d\times S\to \R$ be  $(\mathcal{B}(\R^d)\otimes \mathcal{S})/\mathcal{B}(\R)$-measurable, let $Z_n\colon\Omega \to S$, $n \in \N$, be i.i.d.\ random variables,
 assume for all $s \in S$ that $(\R^d \ni \theta \mapsto F(\theta,s)\in\R )\in C^3(\R^d,\R)$,  assume  for all $\theta,  \vartheta \in \R^d$  that 
	\begin{gather}
	\label{eq:gradient:linear:nof}
	\E\big[\|(\nabla_{\theta} F)(\theta, Z_1)\|_{\R^d}^2 \big] \leq c \big[1+\|\theta\|_{\R^d} \big]^2,\\
	\label{eq:cond:bound:F:nof}
	\max_{i \in \{2,3\}} \inf_{\delta\in(0,\infty)}\sup_{u\in [-\delta,\delta]^d}\E\big[ |F(\theta, Z_1)| + \| (\tfrac{\partial^i}{\partial \theta^i}F)(\theta+u,Z_1)\|_{L^{(i)}(\R^d,\R)}^{1+\delta}\big]<\infty,\\
	\label{eq:gradient:exp:nof}
	\langle \theta - \vartheta, \E[ ( \nabla_{\theta} F)(\theta, Z_1) ]-\E[ ( \nabla_{\theta} F)(\vartheta,Z_1) ]\rangle_{\R^d} \geq L \|\theta-\vartheta\|_{\R^d}^2,\\
	\label{eq:cube:F}
	\big\|\E\big[ (\tfrac{\partial^3}{\partial {\theta}^3}F)(\theta, Z_1) \big]\big\|_{L^{(3)}(\R^d,\R)} 
	+ \max_{i \in \{1,2\}}
	\|\psi^{(i)}(\theta)\|_{L^{(i)}(\R^d,\R)} < c,
	\end{gather}
	and $\| \E[ ( \nabla_{\theta} F)(\theta, Z_1) ]\|_{\R^d} \leq c \|\theta - \Xi\|_{\R^d}$ (cf.~Corollary~\ref{cor:derivative:gen}),
	and let $\Theta \colon \N_0\times \Omega \to \R^d$ be the stochastic process which satisfies for all $n\in\N$   that $\Theta_0 = \xi$ and 
	\begin{equation}
	\label{eq:gradient:Theta:nof}
	\Theta_n = \Theta_{n-1} - \tfrac{\eta}{n^{1-\epsilon}} (\nabla_{\theta} F)(\Theta_{n-1},Z_n).
	\end{equation}
	Then 
	\begin{enumerate}[(i)]
		\item
		\label{item:sgd:inf} we have that $\lb \theta \in \R^d \colon ( \E[F(\theta, Z_1)] = \inf\nolimits_{\vartheta \in \R^d} \E[F(\vartheta, Z_1)])\rb = \lb \Xi\rb$
		and 
		\item
		\label{item:sgd:conv} there exists  $C \in \R$ such that for all $n\in\N$ we have that
		\begin{equation}
		|\E [\psi(\Theta_{n})]- \psi(\Xi)|  \leq C n^{2\epsilon-1}.
		\end{equation}
	\end{enumerate}
\end{cor}
\begin{proof}[Proof of Corollary~\ref{cor:conv:sgd:nof}]
	Throughout this proof let $f \colon \R^d \to \R$ satisfy for all $\theta \in \R^d$ that
	\begin{align}
	f(\theta) = \E[F(\theta, Z_1)].
	\end{align}
	Observe that \eqref{eq:gradient:linear:nof} ensures that for all $\theta \in \R^d$ we have that
	\begin{align}
	\label{eq:gradient:kappa}
	\E\big[\|(\nabla_{\theta} F)(\theta, Z_1)\|_{\R^d} \big] \leq \sqrt{\E\big[\|(\nabla_{\theta} F)(\theta, Z_1)\|_{\R^d}^2 \big]} \leq \sqrt{c} \big[1 + \|\theta\|_{\R^d} \big].
	\end{align}
	This and \eqref{eq:cond:bound:F:nof} imply that  for all $\theta \in \R^d$ we have that
	\begin{align}
	\label{eq:bound:F}
	\E\big[|F(\theta, Z_1)| + \|(\nabla_{\theta} F)(\theta, Z_1)\|_{\R^d} \big]  < \infty.
	\end{align} 
	Next  note that \eqref{eq:gradient:linear:nof} and Lemma~\ref{lem:bnd_(sum_xi)^p} (with $n=2$, $p=2$ in the notation of Lemma~\ref{lem:bnd_(sum_xi)^p}) demonstrate that for all $\theta \in \R^d$ we have that
	\begin{align}
	\begin{split}
	&\E \big[ \|(\nabla_{\theta} F)(\theta, Z_1) - \E[ (\nabla_{\theta} F)(\theta, Z_1)]\|_{\R^d}^2 ] \\
	&= \E \big[ \|(\nabla_{\theta} F)(\theta, Z_1)\|_{\R^d}^2 ]  -  \|\E[ (\nabla_{\theta} F)(\theta, Z_1)]\|_{\R^d}^2\\
	& \leq \E \big[ \|(\nabla_{\theta} F)(\theta, Z_1)\|_{\R^d}^2 ]  \leq c \big[1 + \|\theta\|_{\R^d}\big]^2 \leq 2c (1 + \|\theta\|_{\R^d}^2).
	\end{split}
	\end{align}
	Combining this, \eqref{eq:bound:F}, and \cite[Lemma~4.8]{Wurstemberger2018} (with $d=d$, $p=2$, $\kappa = 2c$, $(\Omega,\F,\P) = (\Omega,\F,\P)$, $(S, \mathcal{S}) = (S, \mathcal{S})$, $X = Z_1$, $F = F$, $f = f$ in the notation of  \cite[Lemma~4.8]{Wurstemberger2018}) ensures that for all $\theta \in \R^d$  we have that
	\begin{align}
	\label{eq:gradient:f}
	f \in C^1(\R^d, \R) \qquad \text{and} \qquad (\nabla f)(\theta) = \E [(\nabla_\theta F)(\theta, Z_1)].
	\end{align}
	This and the assumption that $\forall \, \theta \in \R^d \colon \| \E[ ( \nabla_{\theta} F)(\theta, Z_1) ]\|_{\R^d} \leq c \|\theta - \Xi\|_{\R^d}$ prove that for all $\theta \in \R^d$ we have that 
	\begin{align}
	\label{eq:linear:growth}
	\| (\nabla f)(\theta)\|_{\R^d} \leq c \|\theta - \Xi\|_{\R^d}.
	\end{align}
	This reveals that
	\begin{align}
	\label{eq:gradient_f_zero}
	(\nabla f)(\Xi) = 0.
	\end{align}	
	Combining this with \eqref{eq:gradient:f} and \eqref{eq:gradient:exp:nof} assures that for all $\theta \in \R^d$ we have that
	\begin{align}
	\label{eq:coercivity}
	\langle \theta - \Xi, (\nabla f)(\theta) \rangle_{\R^d} \geq  L \|\theta- \Xi\|_{\R^d}^2.
	\end{align}
	This proves that for all $\theta \in \R^d$ we have that
	\begin{align}
	\langle \theta , (\nabla f)(\theta+ \Xi) \rangle_{\R^d} \geq  L \|\theta\|_{\R^d}^2.
	\end{align}
	The fundamental theorem of calculus hence demonstrates that for all $\theta \in \R^d$ we have that
	\begin{align}
	\begin{split}
	f(\theta) &= f(\Xi) + \big[ f(\Xi + t(\theta- \Xi)) \big]_{t = 0}^{t = 1}  \\
	&= f(\Xi) + \int_0^1 f'(\Xi + t(\theta-\Xi))(\theta- \Xi) \, dt  \\
	&= f(\Xi)+ \int_0^1 \langle (\nabla f)(\Xi + t(\theta-\Xi)),t(\theta-\Xi)  \rangle_{\R^d} \frac{1}{t} \,dt \\
	& \geq f(\Xi)+  \int_0^1 L \norm{t(\theta-\Xi)}_{\R^d}^2\frac{1}{t} \, dt \\
	&= f(\Xi)+   L \norm{\theta-\Xi}_{\R^d}^2\int_0^1 t\, dt 
	= f(\Xi)+ \tfrac{L}{2} \norm{\theta-\Xi}_{\R^d}^2.
	\end{split}
	\end{align}
	The hypothesis that $L \in (0, \infty)$ therefore ensures that for all $\theta \in \R^d \backslash \{\Xi\}$ we have that
	\begin{align}
	f(\theta) \geq  f(\Xi)+ \tfrac{L}{2} \norm{\theta-\Xi}_{\R^d}^2 > f(\Xi).
	\end{align}
	This establishes item~\eqref{item:sgd:inf}. 
	Moreover, observe that \eqref{eq:linear:growth} and \eqref{eq:coercivity} ensure that for all $\theta \in \R^d$, $r \in (0, \infty)$  we have that
	\begin{align}
	\begin{split}
	  2 \langle \theta - \Xi, -(\nabla f)(\theta) \rangle_{\R^d} + r\|(\nabla f)(\theta) \|_{\R^d}^2
	&\leq -2L \|\theta- \Xi\|_{\R^d}^2 + r c^2  \|\theta- \Xi\|_{\R^d}^2\\
	& = (rc^2 -2L) \|\theta- \Xi\|_{\R^d}^2.
	\end{split}
	\end{align}
	This reveals that
	\begin{align}
	\inf_{r \in (0,\infty) } \left( \sup_{\theta \in \R^d \backslash \{\Xi\}} \left[  \frac{ 2 \langle \theta - \Xi, -(\nabla f)(\theta) \rangle_{\R^d} + r\|(\nabla f)(\theta) \|_{\R^d}^2}{ \|\theta- \Xi\|_{\R^d}^2} \right] \right) < 0.
	\end{align}
	Combining this with \eqref{eq:gradient_f_zero} and, e.g.,
	\cite[Proposition~2.16]{Wurstemberger2018} (with $d=d$,  $\vartheta = \Xi$, $\langle \cdot, \cdot \rangle = \langle \cdot, \cdot \rangle_{\R^d}$, $\left\| \cdot \right\| = \left\| \cdot \right \|_{\R^d}$, $ g = - (\nabla f)$
	in the notation  of \cite[Proposition~2.16]{Wurstemberger2018}) prove that  there exists $M \in (0, \infty)$ which satisfies for all $\theta \in \R^d$  that
	\begin{align}
	\label{eq:coercivity:two}
	\langle \theta - \Xi, (\nabla f)(\theta) \rangle_{\R^d} \geq  M  \max\{ \|\theta - \Xi \|_{\R^d}^2, \|(\nabla f)(\theta)\|_{\R^d}^2\}.
	\end{align}
	This, \eqref{eq:gradient:f}, and \eqref{eq:gradient:exp:nof} assure that for all $\theta, \vartheta \in \R^d$ we have that
	\begin{align}
	\begin{split}
	\langle \theta - \vartheta, -(\nabla f)(\theta) + (\nabla f)(\vartheta)\rangle_{\R^d} &= - \langle \theta - \vartheta, (\nabla f)(\theta) - (\nabla f)(\vartheta)\rangle_{\R^d} \\
	&\leq  -L \|\theta- \vartheta \|_{\R^d}^2 \\
	&\leq - \min\{L, M\} \|\theta- \vartheta \|_{\R^d}^2
	\end{split}
	\end{align}
	and
	\begin{align}
	\begin{split}
	\langle \theta - \Xi, -(\nabla f)(\theta) \rangle_{\R^d} &= - \langle \theta - \Xi, (\nabla f)(\theta) \rangle_{\R^d}
	 \leq -M\|(\nabla f)(\theta)\|_{\R^d}^2 \\
	&\leq - \min\{L, M\} \|(\nabla f)(\theta)\|_{\R^d}^2.
	\end{split}
	\end{align}
	Corollary~\ref{cor:conv:sgd} (with $d=d$, $\xi = \xi$, $\Xi=\Xi$, $\varepsilon = \varepsilon$, $\eta = \eta$, $L = \min\{L, M\} \in (0, \infty)$, $c = c$,  $\psi = \psi$, $(S,\mathcal{S}) = (S,\mathcal{S})$, $(\Omega,\F,\P) = (\Omega,\F,\P)$, $Z_n = Z_n$,
	$F = -F$, $ g= (\R^d \ni \theta \mapsto -(\nabla f)(\theta) \in \R^d)$, $\Theta = \Theta$ 
	for $n \in \N$  in the  notation of Corollary~\ref{cor:conv:sgd}) therefore establishes  item~\eqref{item:sgd:conv}. The proof of  Corollary~\ref{cor:conv:sgd:nof} is thus completed.
\end{proof}

\section*{Acknowledgments}

This article is to a large extent based on the master thesis of AB written in 2017 at ETH Zurich under the supervision of AJ. Special thanks are due to Philipp Grohs for  several instructive suggestions. Mike Giles is also gratefully acknowledged for several useful comments. This work has been funded by the Deutsche Forschungsgemeinschaft (DFG, German Research Foundation) under Germany's Excellence Strategy EXC 2044-390685587, Mathematics M\"unster: Dynamics-Geometry-Structure and by the Swiss National Science Foundation (SNSF) through the research grant  \allowbreak 200020\_175699.

\bibliographystyle{acm}			 
\bibliography{bibmasterthesis}

\def\cprime{$'$} \def\cprime{$'$}
\begin{thebibliography}{100}

\bibitem{AliprantisBorder2006}
{\sc Aliprantis, C.~D., and Border, K.~C.}
\newblock {\em Infinite dimensional analysis}, third~ed.
\newblock Springer, Berlin, 2006.

\bibitem{ZeyuanEtAl2018}
{\sc Allen-Zhu, Z., Li, Y., and Song, Z.}
\newblock {A Convergence Theory for Deep Learning via Over-Parameterization}.
\newblock {\em arXiv:1811.03962\/} (2018), 37 pages.

\bibitem{Bach14}
{\sc Bach, F.}
\newblock Adaptivity of averaged stochastic gradient descent to local strong
  convexity for logistic regression.
\newblock {\em J. Mach. Learn. Res. 15\/} (2014), 595--627.

\bibitem{BachMoulines11}
{\sc Bach, F., and Moulines, E.}
\newblock Non-asymptotic analysis of stochastic approximation algorithms for
  machine learning.
\newblock {\em Advances in Neural Information Processing Systems 24\/} (2011),
  451--459.

\bibitem{BachMoulines13}
{\sc Bach, F., and Moulines, E.}
\newblock Non-strongly-convex smooth stochastic approximation with convergence
  rate {$O(1/n)$}.
\newblock {\em Proceedings of the 26th International Conference on Neural
  Information Processing Systems 1\/} (2013), 773--781.

\bibitem{BayerHorvathEthAl2019}
{\sc Bayer, C., Horvath, B., Muguruza, A., Stemper, B., and Tomas, M.}
\newblock On deep calibration of (rough) stochastic volatility models.
\newblock {\em arXiv:1908.08806\/} (2019), 32 pages.

\bibitem{BeckBeckerCheridito2019}
{\sc Beck, C., Becker, S., Cheridito, P., Jentzen, A., and Neufeld, A.}
\newblock {Deep splitting method for parabolic PDEs}.
\newblock {\em arXiv:1907.03452\/} (2019), 40 pages.

\bibitem{BeckEtAL2018}
{\sc Beck, C., Becker, S., Grohs, P., Jaafari, N., and Jentzen, A.}
\newblock Solving stochastic differential equations and {Kolmogorov} equations
  by means of deep learning.
\newblock {\em arXiv:1806.00421\/} (2018), 56 pages.

\bibitem{BeckJentzenE2019}
{\sc Beck, C., E, W., and Jentzen, A.}
\newblock Machine {L}earning {A}pproximation {A}lgorithms for
  {H}igh-{D}imensional {F}ully {N}onlinear {P}artial {D}ifferential {E}quations
  and {S}econd-order {B}ackward {S}tochastic {D}ifferential {E}quations.
\newblock {\em J. Nonlinear Sci. 29}, 4 (2019), 1563--1619.

\bibitem{becker2018deep}
{\sc Becker, S., Cheridito, P., and Jentzen, A.}
\newblock Deep optimal stopping.
\newblock {\em J. Mach. Learn. Res. 20}, 74 (2019), 1--25.

\bibitem{BeckerEtAl2019}
{\sc Becker, S., Cheridito, P., and Jentzen, A.}
\newblock {Pricing and hedging American-style options with deep learning}.
\newblock {\em arXiv:1912.11060\/} (2019), 12 pages.

\bibitem{BeckerCheriditoJentzen2019}
{\sc Becker, S., Cheridito, P., Jentzen, A., and Welti, T.}
\newblock Solving high-dimensional optimal stopping problems using deep
  learning.
\newblock {\em arXiv:1908.01602\/} (2019), 42 pages.

\bibitem{BengioLewandowskiPascanu2013}
{\sc Bengio, Y., Boulanger{-}Lewandowski, N., and Pascanu, R.}
\newblock Advances in optimizing recurrent networks.
\newblock In {\em {IEEE} International Conference on Acoustics, Speech and
  Signal Processing\/} (2013), pp.~8624--8628.

\bibitem{BercuFort13}
{\sc Bercu, B., and Fort, J.-C.}
\newblock Generic stochastic gradient methods.
\newblock {\em Wiley Encyclopedia of Operations Research and Management
  Science\/} (2013), 1--8.

\bibitem{BordesBottouGallinari09}
{\sc Bordes, A., Bottou, L., and Gallinari, P.}
\newblock {SGD-QN}: Careful quasi-{N}ewton stochastic gradient descent.
\newblock {\em J. Mach. Learn. Res. 10\/} (2009), 1737--1754.

\bibitem{Bottou12}
{\sc Bottou, L.}
\newblock Large-scale machine learning with stochastic gradient descent.
\newblock {\em Proceedings of {COMPSTAT} 2010\/} (2010), 177--186.

\bibitem{BottouBousquet11}
{\sc Bottou, L., and Bousquet, O.}
\newblock The tradeoffs of large scale learning.
\newblock {\em Optimization for Machine Learning, MIT Press\/} (2011),
  351--368.

\bibitem{BottouCurtisNocedal2018}
{\sc Bottou, L., Curtis, F.~E., and Nocedal, J.}
\newblock Optimization methods for large-scale machine learning.
\newblock {\em SIAM Rev. 60}, 2 (2018), 223--311.

\bibitem{BottouLeCun04}
{\sc Bottou, L., and LeCun, Y.}
\newblock Large scale online learning.
\newblock {\em Advances in Neural Information Processing Systems 16\/} (2004),
  217--224.

\bibitem{BrutzkusEtAl2017}
{\sc Brutzkus, A., Globerson, A., Malach, E., and Shalev-Shwartz, S.}
\newblock {SGD Learns Over-parameterized Networks that Provably Generalize on
  Linearly Separable Data}.
\newblock {\em arXiv:1710.10174\/} (2017), 17 pages.

\bibitem{BuehlerGanonTeichmann2019}
{\sc Buehler, H., Gonon, L., Teichmann, J., and Wood, B.}
\newblock Deep hedging.
\newblock {\em Quant. Finance 19}, 8 (2019), 1271--1291.

\bibitem{ChauEtAl2019}
{\sc Chau, N.~H., Moulines, E., Rásonyi, M., Sabanis, S., and Zhang, Y.}
\newblock {On stochastic gradient Langevin dynamics with dependent data
  streams: the fully non-convex case}.
\newblock {\em arXiv:1905.13142\/} (2019), 31 pages.

\bibitem{CheriditoEtAl2020}
{\sc Cheridito, P., Jentzen, A., and Rossmannek, F.}
\newblock Non-convergence of stochastic gradient descent in the training of
  deep neural networks.
\newblock {\em arXiv:2006.07075\/} (2020), 12 pages.

\bibitem{ChouiekhHaj2018}
{\sc Chouiekh, A., and Haj, E. H. I.~E.}
\newblock {ConvNets for Fraud Detection analysis}.
\newblock {\em Procedia Computer Science 127\/} (2018), 133--138.

\bibitem{Coleman2012}
{\sc Coleman, R.}
\newblock {\em Calculus on normed vector spaces}.
\newblock Universitext. Springer, New York, 2012.

\bibitem{DahlSainathHinton2013}
{\sc Dahl, G.~E., Sainath, T.~N., and Hinton, G.~E.}
\newblock Improving deep neural networks for lvcsr using rectified linear units
  and dropout.
\newblock In {\em IEEE International Conference on Acoustics, Speech and Signal
  Processing\/} (2013), pp.~8609--8613.

\bibitem{DahlEtAl2012}
{\sc Dahl, G.~E., Yu, D., Deng, L., and Acero, A.}
\newblock Context-dependent pre-trained deep neural networks for
  large-vocabulary speech recognition.
\newblock {\em IEEE Transactions on audio, speech, and language processing 20},
  1 (2012), 30--42.

\bibitem{DarkenChangMoody92}
{\sc Darken, C., Chang, J., and Moody, J.}
\newblock Learning rate schedules for faster stochastic gradient search.
\newblock {\em Neural Networks for Signal Processing II Proceedings of the 1992
  IEEE Workshop\/} (1992), 1--11.

\bibitem{DauphinDeVriesBengio15}
{\sc Dauphin, Y., {de Vries}, H., and Bengio, Y.}
\newblock Equilibrated adaptive learning rates for non-convex optimization.
\newblock {\em Advances in Neural Information Processing Systems 28\/} (2015),
  1504--1512.

\bibitem{DauphinEtAl2014}
{\sc Dauphin, Y.~N., Pascanu, R., Gulcehre, C., Cho, K., Ganguli, S., and
  Bengio, Y.}
\newblock Identifying and attacking the saddle point problem in
  high-dimensional non-convex optimization.
\newblock {\em Proceedings of the 27th International Conference on Neural
  Information Processing Systems\/} (2014), 2933--2941.

\bibitem{DeanETAL12}
{\sc Dean, J., Corrado, G.~S., Monga, R., Chen, K., Devin, M., Le, Q.~V., Mao,
  M.~Z., Ranzato, M.~A., Senior, A., Tucker, P., Yang, K., and Ng, A.~Y.}
\newblock Large scale distributed deep networks.
\newblock {\em Advances in Neural Information Processing Systems 25\/} (2012),
  1--11.

\bibitem{DefossezBach17}
{\sc D\'efossez, A., and Bach, F.}
\newblock {AdaBatch: Efficient Gradient Aggregation Rules for Sequential and
  Parallel Stochastic Gradient Methods}.
\newblock {\em arXiv:1711.01761\/} (2017), 26 pages.

\bibitem{DengETAL13}
{\sc Deng, L., Li, J., Huang, J.-T., Yao, K., Yu, D., Seide, F., Seltzer, M.,
  Zweig, G., He, X., and Williams, J.}
\newblock Recent advances in deep learning for speech research at {M}icrosoft.
\newblock {\em Acoustics, Speech, and Signal Processing (ICASSP)\/} (2013).

\bibitem{DereichMuellerGronbach2019}
{\sc Dereich, S., and M\"{u}ller-Gronbach, T.}
\newblock General multilevel adaptations for stochastic approximation
  algorithms of {R}obbins-{M}onro and {P}olyak-{R}uppert type.
\newblock {\em Numer. Math. 142}, 2 (2019), 279--328.

\bibitem{DieuleveutDurmusBach17}
{\sc Dieuleveut, A., Durmus, A., and Bach, F.}
\newblock Bridging the gap between constant step size stochastic gradient
  descent and {M}arkov chains.
\newblock {\em arXiv:1707.06386\/} (2017), 49 pages.

\bibitem{Dozat16}
{\sc Dozat, T.}
\newblock {Incorporating Nesterov Momentum into Adam}.
\newblock {\em ICLR Workshop\/} (2016), 2013--2016.

\bibitem{DuchiHazanSinger11}
{\sc Duchi, J., Hazan, E., and Singer, Y.}
\newblock Adaptive subgradient methods for online learning and stochastic
  optimization.
\newblock {\em J. Mach. Learn. Res. 12\/} (2011), 2121--2159.

\bibitem{Durrett2010}
{\sc Durrett, R.}
\newblock {\em Probability: Theory and Examples}.
\newblock Cambridge University Press, 2010.

\bibitem{EHanJentzen2017}
{\sc E, W., Han, J., and Jentzen, A.}
\newblock Deep learning-based numerical methods for high-dimensional parabolic
  partial differential equations and backward stochastic differential
  equations.
\newblock {\em Commun. Math. Stat. 5}, 4 (2017), 349--380.

\bibitem{EYu2018}
{\sc E, W., and Yu, B.}
\newblock The deep {R}itz method: a deep learning-based numerical algorithm for
  solving variational problems.
\newblock {\em Commun. Math. Stat. 6}, 1 (2018), 1--12.

\bibitem{FehrmanGessJentzen2019}
{\sc Fehrman, B., Gess, B., and Jentzen, A.}
\newblock Convergence rates for the stochastic gradient descent method for
  non-convex objective functions.
\newblock {\em Accepted in J. Mach. Learn. Res., arXiv:1904.01517\/} (2019), 52
  pages.

\bibitem{GhadimiLanZhang2016}
{\sc Ghadimi, S., Lan, G., and Zhang, H.}
\newblock Mini-batch stochastic approximation methods for nonconvex stochastic
  composite optimization.
\newblock {\em Math. Program. 155}, 1-2, Ser. A (2016), 267--305.

\bibitem{Graves13}
{\sc Graves, A.}
\newblock {Generating Sequences With Recurrent Neural Networks}.
\newblock {\em arXiv:1308.0850\/} (2013), 43 pages.

\bibitem{GravesMohamedHinton13}
{\sc Graves, A., Mohamed, A.-r., and Hinton, G.}
\newblock Speech recognition with deep recurrent neural networks.
\newblock {\em Acoustics, Speech and Signal Processing (ICASSP)\/} (2013),
  6645--6649.

\bibitem{HanJentzenE2018}
{\sc Han, J., Jentzen, A., and E, W.}
\newblock Solving high-dimensional partial differential equations using deep
  learning.
\newblock {\em Proceedings of the National Academy of Sciences 115}, 34 (2018),
  8505--8510.

\bibitem{Henry2017}
{\sc Henry-Labord\`ere, P.}
\newblock Deep {Primal}-{Dual} {Algorithm} for {BSDEs}: {Applications} of
  {Machine} {Learning} to {CVA} and {IM}.
\newblock (November 15, 2017), 16 pages.
\newblock Available at SSRN: \url{https://ssrn.com/abstract=3071506}.

\bibitem{HintonETAL12}
{\sc Hinton, G., Deng, L., Yu, D., Dahl, G.~E., Mohamed, A.-r., Jaitly, N.,
  Senior, A., Vanhoucke, V., Nguyen, P., and Sainath, T.~N.}
\newblock Deep neural networks for acoustic modeling in speech recognition: The
  shared views of four research groups.
\newblock {\em Signal Processing Magazine, IEEE 29}, 6 (2012), 82--97.

\bibitem{HintonSalakhutdinov06}
{\sc Hinton, G.~E., and Salakhutdinov, R.~R.}
\newblock Reducing the dimensionality of data with neural networks.
\newblock {\em Science 313}, 5786 (2006), 504--507.

\bibitem{HuLuLiChen2014}
{\sc Hu, B., Lu, Z., Li, H., and Chen, Q.}
\newblock Convolutional neural network architectures for matching natural
  language sentences.
\newblock {\em Proceedings of the 27th International Conference on Neural
  Information Processing Systems 2\/} (2014), 2042--2050.

\bibitem{HuangEtAl2017}
{\sc Huang, G., Liu, Z., van~der Maaten, L., and Weinberger, K.~Q.}
\newblock Densely connected convolutional networks.
\newblock {\em Proceedings of the IEEE Conference on Computer Vision and
  Pattern Recognition\/} (2017), 2261--2269.

\bibitem{InoueParkOkada03}
{\sc Inoue, M., Park, H., and Okada, M.}
\newblock On-line learning theory of soft committee machines with correlated
  hidden units steepest gradient descent and natural gradient descent.
\newblock {\em Journal of the Physical Society of Japan 72}, 4 (2003),
  805--810.

\bibitem{IoffeSzegedy2015}
{\sc Ioffe, S., and Szegedy, C.}
\newblock Batch normalization: Accelerating deep network training by reducing
  internal covariate shift.
\newblock {\em Proceedings of the 32nd International Conference on Machine
  Learning\/} (2015), 448--456.

\bibitem{Wurstemberger2018}
{\sc Jentzen, A., Kuckuck, B., Neufeld, A., and von Wurstemberger, P.}
\newblock Strong error analysis for stochastic gradient descent optimization
  algorithms.
\newblock {\em Accepted in IMA J. Num. Anal., arXiv:1801.09324\/} (2018), 75
  pages.

\bibitem{JentzenBochner}
{\sc Jentzen, A., Salimova, D., and Welti, T.}
\newblock Strong convergence for explicit space-time discrete numerical
  approximation methods for stochastic {B}urgers equations.
\newblock {\em J. Math. Anal. Appl. 469}, 2 (2019), 661--704.

\bibitem{JentzenWurstemberger2018}
{\sc Jentzen, A., and von Wurstemberger, P.}
\newblock Lower error bounds for the stochastic gradient descent optimization
  algorithm: {S}harp convergence rates for slowly and fast decaying learning
  rates.
\newblock {\em J. Complexity 57\/} (2020), 101438, 16.

\bibitem{JohnsonZhangNIPS2013}
{\sc Johnson, R., and Zhang, T.}
\newblock Accelerating stochastic gradient descent using predictive variance
  reduction.
\newblock {\em Advances in Neural Information Processing Systems 26\/} (2013),
  315--323.

\bibitem{KalchbrennerEtAl2014}
{\sc Kalchbrenner, N., Grefenstette, E., and Blunsom, P.}
\newblock A convolutional neural network for modelling sentences.
\newblock {\em Proceedings of the 52nd Annual Meeting of the Association for
  Computational Linguistics\/} (2014), 655--665.

\bibitem{KingmaBa14}
{\sc Kingma, D.~P., and Ba, J.}
\newblock Adam: {A} method for stochastic optimization.
\newblock {\em arXiv:1412.6980\/} (2014), 15 pages.

\bibitem{Klenke2013}
{\sc Klenke, A.}
\newblock {\em Probability theory}, second~ed.
\newblock Universitext. Springer, London, 2014.
\newblock A comprehensive course.

\bibitem{kp92}
{\sc Kloeden, P.~E., and Platen, E.}
\newblock {\em Numerical solution of stochastic differential equations},
  vol.~23 of {\em Applications of Mathematics (New York)}.
\newblock Springer-Verlag, Berlin, 1992.

\bibitem{Koenigsberger2004}
{\sc K{\"o}nigsberger, K.}
\newblock {\em Analysis. 2}, fifth~ed.
\newblock Springer-Lehrbuch. [Springer Textbook]. Springer-Verlag, Berlin,
  2004.

\bibitem{KrizhevskySutskeverHinton12}
{\sc Krizhevsky, A., Sutskever, I., and Hinton, G.~E.}
\newblock Image{N}et classification with deep convolutional neural networks.
\newblock {\em Commun. ACM 60}, 6 (2017), 84--90.

\bibitem{LanZhou2018}
{\sc Lan, G., and Zhou, Y.}
\newblock An optimal randomized incremental gradient method.
\newblock {\em Math. Program. 171}, 1-2, Ser. A (2018), 167--215.

\bibitem{LangfordLiZhang09}
{\sc Langford, J., Li, L., and Zhang, T.}
\newblock Sparse online learning via truncated gradient.
\newblock {\em J. Mach. Learn. Res. 10\/} (2009), 777--801.

\bibitem{LeRouxSchmidtBach12}
{\sc {Le Roux}, N., Schmidt, M., and Bach, F.}
\newblock A stochastic gradient method with an exponential convergence rate for
  strongly-convex optimization with finite training sets.
\newblock {\em Proceedings of the 25th International Conference on Neural
  Information Processing Systems 2\/} (2012), 2663--2671.

\bibitem{LeCunBottouBengioHaffner98}
{\sc LeCun, Y., Bottou, L., Bengio, Y., and Haffner, P.}
\newblock Gradient-based learning applied to document recognition.
\newblock {\em Proceedings of the IEEE 86}, 11 (1998), 2278--2324.

\bibitem{LeCunBottouOrrMuller98}
{\sc LeCun, Y., Bottou, L., Orr, G.~B., and M{\"u}ller, K.~R.}
\newblock Efficient backprop.
\newblock {\em Neural Networks: Tricks of the Trade\/} (1998), 9--50.

\bibitem{LeiHuLiTang2019}
{\sc {Lei}, Y., {Hu}, T., {Li}, G., and {Tang}, K.}
\newblock Stochastic gradient descent for nonconvex learning without bounded
  gradient assumptions.
\newblock {\em IEEE Transactions on Neural Networks and Learning Systems\/}
  (2019), 1--7.

\bibitem{LiTaiE2017}
{\sc Li, Q., Tai, C., and E, W.}
\newblock Stochastic modified equations and adaptive stochastic gradient
  algorithms.
\newblock {\em Proceedings of the 34th International Conference on Machine
  Learning\/} (2017), 2101--2110.

\bibitem{LiLiang2018}
{\sc Li, Y., and Liang, Y.}
\newblock Learning overparameterized neural networks via stochastic gradient
  descent on structured data.
\newblock {\em Proceedings of the 32nd International Conference on Neural
  Information Processing Systems\/} (2018), 8168--8177.

\bibitem{LiuBorovykhEtAl2019}
{\sc Liu, S., Borovykh, A., Grzelak, L.~A., and Oosterlee, C.~W.}
\newblock A neural network-based framework for financial model calibration.
\newblock {\em J. Math. Ind. 9\/} (2019), Paper No. 9.

\bibitem{LovasEtAl2020}
{\sc Lovas, A., Lytras, I., Rásonyi, M., and Sabanis, S.}
\newblock {Taming neural networks with TUSLA: Non-convex learning via adaptive
  stochastic gradient Langevin algorithms}.
\newblock {\em arXiv:2006.14514\/} (2020), 29 pages.

\bibitem{McmahanStreeter14}
{\sc Mcmahan, H.~B., and Streeter, M.}
\newblock Delay-tolerant algorithms for asynchronous distributed online
  learning.
\newblock {\em Advances in Neural Information Processing Systems 27\/} (2014),
  1--9.

\bibitem{Mishra2018}
{\sc Mishra, S.}
\newblock A machine learning framework for data driven acceleration of
  computations of differential equations.
\newblock {\em Mathematics in Engineering 1}, 1 (2019), 118--146.

\bibitem{MizutaniDreyfus10}
{\sc Mizutani, E., and Dreyfus, S.}
\newblock An analysis on negative curvature induced by singularity in
  multi-layer neural-network learning.
\newblock {\em Advances in Neural Information Processing Systems 23\/} (2010),
  1669--1677.

\bibitem{mr08}
{\sc M{\"u}ller-Gronbach, T., and Ritter, K.}
\newblock Minimal errors for strong and weak approximation of stochastic
  differential equations.
\newblock In {\em Monte {C}arlo and quasi-{M}onte {C}arlo methods 2006}.
  Springer, Berlin, 2008, pp.~53--82.

\bibitem{NabianMeidani2018}
{\sc Nabian, M.~A., and Meidani, H.}
\newblock A deep learning solution approach for high-dimensional random
  differential equations.
\newblock {\em Probabilistic Engineering Mechanics 57\/} (2019), 14--25.

\bibitem{NguyenEtAl2018}
{\sc Nguyen, L.~M., Nguyen, N.~H., Phan, D.~T., Kalagnanam, J.~R., and
  Scheinberg, K.}
\newblock When does stochastic gradient algorithm work well?
\newblock {\em arXiv:1801.06159\/} (2018), 21 pages.

\bibitem{NiuRechtChristopherWright11}
{\sc Niu, F., Recht, B., Christopher, R., and Wright, S.~J.}
\newblock {HOGWILD!: A Lock-Free Approach to Parallelizing Stochastic Gradient
  Descent}.
\newblock {\em Proceedings of the 24th International Conference on Neural
  Information Processing Systems\/} (2011), 693--701.

\bibitem{PascanuBengio13}
{\sc Pascanu, R., and Bengio, Y.}
\newblock Revisiting natural gradient for deep networks.
\newblock {\em International Conference on Learning Representations\/} (2014).

\bibitem{Pillaud-VivienRudiBach17}
{\sc Pillaud-Vivien, L., Rudi, A., and Bach, F.}
\newblock Exponential convergence of testing error for stochastic gradient
  methods.
\newblock {\em Proceedings of the 31st Conference On Learning Theory 75\/}
  (2018), 250--296.

\bibitem{Polyak98}
{\sc Polyak, B.~T.}
\newblock A new method of stochastic approximation type.
\newblock {\em Avtomat. i Telemekh. 51}, 7 (1998), 937--1008.

\bibitem{PolyakJuditsky92}
{\sc Polyak, B.~T., and Juditsky, A.~B.}
\newblock Acceleration of stochastic approximation by averaging.
\newblock {\em Automat. Remote Control 30}, 4 (1992), 838--855.

\bibitem{PolyakTsypkin80}
{\sc Polyak, B.~T., and Tsypkin, Y.~Z.}
\newblock Optimal pseudogradient adaptation algorithms.
\newblock {\em Avtomat. i Telemekh. 8\/} (1980), 74--84.

\bibitem{RakhlinShamirSridharan12}
{\sc Rakhlin, A., Shamir, O., and Sridharan, K.}
\newblock Making gradient descent optimal for strongly convex stochastic
  optimization.
\newblock {\em {Proceedings of the 29th International Conference on Machine
  Learning}\/} (2012), 1571--1578.

\bibitem{RattraySaadAmari98}
{\sc Rattray, M., Saad, D., and Amari, S.~I.}
\newblock Natural gradient descent for on-line learning.
\newblock {\em Phys. Rev. Lett. 81}, 24 (1998), 5461--5464.

\bibitem{r03}
{\sc R{\"o}{\ss}ler, A.}
\newblock {\em Runge-Kutta Methods for the Numerical Solution of Stochastic
  Differential Equations}.
\newblock Shaker Verlag, Aachen, 2003.
\newblock Dissertation, Technische Universit{\"a}t Darmstadt, Darmstadt,
  Germany.

\bibitem{RoyEtAl2018}
{\sc Roy, A., Sun, J., Mahoney, R., Alonzi, L., Adams, S., and Beling, P.}
\newblock Deep learning detecting fraud in credit card transactions.
\newblock In {\em Systems and Information Engineering Design Symposium\/}
  (2018), pp.~129--134.

\bibitem{Ruder16}
{\sc Ruder, S.}
\newblock An overview of gradient descent optimization algorithms.
\newblock {\em arXiv:1609.04747\/} (2016), 12 pages.

\bibitem{SankararamanEtAl2019}
{\sc Sankararaman, K.~A., Soham~De, Z.~X., Huang, W.~R., and Goldstein, T.}
\newblock {The Impact of Neural Network Overparameterization on Gradient
  Confusion and Stochastic Gradient Descent}.
\newblock {\em arXiv:1904.06963\/} (2019), 37 pages.

\bibitem{SchaulZhangLeCun12}
{\sc Schaul, T., Zhang, S., and LeCun, Y.}
\newblock No more pesky learning rates.
\newblock {\em Proceedings of the 30th International Conference on Machine
  Learning 28}, 3 (2013), 343--351.

\bibitem{Schraudolph99}
{\sc Schraudolph, N.~N.}
\newblock Local gain adaptation in stochastic gradient descent.
\newblock {\em Artificial Neural Networks, 1999. ICANN 99. Ninth International
  Conference on (Conf. Publ. No. 470) 2\/} (1999), 569--574.

\bibitem{SchraudolphYuGunter07}
{\sc Schraudolph, N.~N., Yu, J., and G{\"u}nter, S.}
\newblock A stochastic quasi-newton method for online convex optimization.
\newblock {\em Proceedings of the 9th International Conference on Artificial
  Intelligence and Statistics (AISTAT)\/} (2007), 433--440.

\bibitem{Shalev-ShwartzShingerSrebroCotter11}
{\sc Shalev-Shwartz, S., Shinger, Y., Srebro, N., and Cotter, A.}
\newblock Pegasos: primal estimated sub-gradient solver for {SVM}.
\newblock {\em Math. Program. 127}, 1 (2011), 3--30.

\bibitem{SimonyanZisserman2014}
{\sc Simonyan, K., and Zisserman, A.}
\newblock Very deep convolutional networks for large-scale image recognition.
\newblock {\em arXiv:1409.1556\/} (2014), 14 pages.

\bibitem{SirignanoCont2019}
{\sc Sirignano, J., and Cont, R.}
\newblock Universal features of price formation in financial markets:
  perspectives from deep learning.
\newblock {\em Quant. Finance 19}, 9 (2019), 1449--1459.

\bibitem{SirignanoSpiliopoulos2017}
{\sc Sirignano, J., and Spiliopoulos, K.}
\newblock D{GM}: a deep learning algorithm for solving partial differential
  equations.
\newblock {\em J. Comput. Phys. 375\/} (2018), 1339--1364.

\bibitem{Sohl-DicksteinPooleGanguli14}
{\sc Sohl-Dickstein, J., Poole, B., and Ganguli, S.}
\newblock Fast large-scale optimization by unifying stochastic gradient and
  quasi-{N}ewton methods.
\newblock {\em Proceedings of the 31st International Conference on Machine
  Learning\/} (2014), 604--612.

\bibitem{SutskeverMartensDahlHinton13}
{\sc Sutskever, I., Martens, J., Dahl, G., and Hinton, G.}
\newblock On the importance of initialization and momentum in deep learning.
\newblock {\em Proceedings of the 30th International Conference on Machine
  Learning\/} (2013), 1139--1147.

\bibitem{TaigmanEtAl2014}
{\sc Taigman, Y., Yang, M., Ranzato, M., and Wolf, L.}
\newblock Deepface: Closing the gap to human-level performance in face
  verification.
\newblock In {\em IEEE Conference on Computer Vision and Pattern Recognition\/}
  (2014), pp.~1701--1708.

\bibitem{TangMonteleoni15}
{\sc Tang, C., and Monteleoni, C.}
\newblock On the convergence rate of stochastic gradient descent for strongly
  convex functions.
\newblock In {\em Regularization, optimization, kernels, and support vector
  machines}, Chapman \& Hall/CRC Mach. Learn. Pattern Recogn. Ser. CRC Press,
  Boca Raton, FL, 2015, pp.~159--175.

\bibitem{TeschlODE}
{\sc Teschl, G.}
\newblock {\em Ordinary differential equations and dynamical systems}, vol.~140
  of {\em Graduate Studies in Mathematics}.
\newblock American Mathematical Society, Providence, RI, 2012.

\bibitem{WangEtAl2017}
{\sc Wang, R., Fu, B., Fu, G., and Wang, M.}
\newblock Deep \& cross network for ad click predictions.
\newblock {\em Proceedings of the 26th ACM SIGKDD International Conference on
  Knowledge Discovery and Data Mining\/} (2017), 1--7.

\bibitem{WangEtAl2015}
{\sc Wang, W., Yang, J., Xiao, J., Li, S., and Zhou, D.}
\newblock Face recognition based on deep learning.
\newblock In {\em Human Centered Computing\/} (2015), pp.~812--820.

\bibitem{WoodworthSrebro2016}
{\sc Woodworth, B.~E., and Srebro, N.}
\newblock Tight complexity bounds for optimizing composite objectives.
\newblock {\em Advances in Neural Information Processing Systems 29\/} (2016),
  3639--3647.

\bibitem{WuEtAl2016}
{\sc Wu, C., Karanasou, P., Gales, M.~J., and Sim, K.~C.}
\newblock Stimulated deep neural network for speech recognition.
\newblock In {\em Interspeech 2016\/} (2016), pp.~400--404.

\bibitem{Xu11}
{\sc Xu, W.}
\newblock Towards optimal one pass large scale learning with averaged
  stochastic gradient descent.
\newblock {\em arXiv:1107.2490\/} (2011), 19 pages.

\bibitem{ZarembaSutskever2014}
{\sc Zaremba, W., and Sutskever, I.}
\newblock Learning to execute.
\newblock {\em arXiv:1410.4615\/} (2014), 25 pages.

\bibitem{Zeiler12}
{\sc Zeiler, M.~D.}
\newblock {ADADELTA: A}n adaptive learning rate method.
\newblock {\em arXiv:1212.5701\/} (2012), 6 pages.

\bibitem{ZhaiEtAl2016}
{\sc Zhai, S., Chang, K.-h., Zhang, R., and Zhang, Z.~M.}
\newblock {DeepIntent}: Learning attentions for online advertising with
  recurrent neural networks.
\newblock {\em Proceedings of the 22nd ACM SIGKDD International Conference on
  Knowledge Discovery and Data Mining\/} (2016), 1295--1304.

\bibitem{ZhangChoromanskaLeCun15}
{\sc Zhang, S., Choromanska, A., and LeCun, Y.}
\newblock Deep learning with elastic averaging {SGD}.
\newblock {\em Proceedings of the 28th International Conference on Neural
  Information Processing Systems\/} (2015), 685--693.

\bibitem{Zhang04}
{\sc Zhang, T.}
\newblock Solving large scale linear prediction problems using stochastic
  gradient descent algorithms.
\newblock {\em Proceedings of the 21st International Conference on Machine
  Learning\/} (2004).

\end{thebibliography}

\end{document}